\documentclass[12pt,notitlepage]{amsart}
\usepackage{latexsym,amsfonts,amssymb,amsmath,amsthm, url, hyperref,xypic}
\usepackage[utf8]{inputenc}
\usepackage[T1]{fontenc}
\usepackage{lmodern}
\usepackage{graphicx}
\usepackage{color}
\usepackage{xcolor}
\usepackage[normalem]{ulem}
\usepackage{enumitem}
\let\turc\c
\usepackage{etoolbox}
\robustify\turc 
\newcommand{\bpm}{\begin{pmatrix}}
\newcommand{\epm}{\end{pmatrix}}
\newcommand{\Z}{\ensuremath{\mathbb{Z}}}
\newcommand{\Q}{\ensuremath{\mathbb Q}}
\newcommand{\R}{\ensuremath{\mathbb R}}
\newcommand{\eps}{\varepsilon}

\DeclareMathOperator{\SL}{SL}
\DeclareMathOperator{\Gal}{Gal}

\usepackage[inner=1.0in,outer=1.0in,bottom=1.0in, top=1.0in]{geometry}

\newcommand{\mz}{\ensuremath{\mathbb Z}}
\newcommand{\mr}{\ensuremath{\mathbb R}}

\newcommand{\mq}{\ensuremath{\mathbb Q}}
\newcommand{\mc}{\ensuremath{\mathbb C}}

\newcommand{\fa}{\ensuremath{\mathfrak a}}

\newcommand{\fc}{\ensuremath{\mathfrak c}}
\newcommand{\fp}{\ensuremath{\mathfrak p}}
\newcommand{\fg}{\ensuremath{\mathfrak g}}

\newcommand{\shortmod}{\ensuremath{\negthickspace \negthickspace \negthickspace \pmod}}

\newcommand{\intR}{\int_{-\infty}^{\infty}}

\newcommand{\sumstar}{\sideset{}{^*}\sum}

\DeclareMathOperator{\GL}{GL}
\DeclareMathOperator{\PGL}{PGL}
\DeclareMathOperator{\SO}{SO}

\DeclareMathOperator{\lcm}{lcm}

\newcommand{\cF}{\mathcal{F}}
\newcommand{\cC}{\mathcal{C}}
\newcommand{\cO}{\mathcal{O}}
\newcommand{\cB}{\mathcal{B}}

\renewcommand{\cH}{\mathcal{H}}
\newcommand{\A}{\mathbb{A}}
\newcommand{\N}{\mathbb{N}}
\newcommand{\C}{\mathbb{C}}
\newcommand{\G}{\mathbb{G}}
\DeclareMathOperator{\St}{St}
\DeclareMathOperator{\Res}{Res}
\DeclareMathOperator{\Ind}{Ind}
\DeclareMathOperator{\cInd}{{\it c}-Ind}
\renewcommand{\P}{\mathbb{P}}
\newcommand{\Hom}{\text{Hom}}
\newcommand{\End}{\text{End}}
\DeclareMathOperator{\Kl}{Kl}

\newcommand{\Span}{\text{Span}}

\DeclareMathOperator{\vol}{vol}

\DeclareMathOperator{\sgn}{sgn}
\DeclareMathOperator{\supp}{supp}
\DeclareMathOperator{\imag}{Im}
\DeclareMathOperator{\real}{Re}
\DeclareMathOperator{\Tr}{Tr}
\DeclareMathOperator{\Nm}{Nm}
\DeclareMathOperator{\disc}{disc}

\newcommand{\zxz}[4]{\begin{pmatrix} #1 & #2 \\ #3 & #4 \end{pmatrix}}
\newcommand{\lb}{\left(}
\newcommand{\rb}{\right) }

\newcommand{\F}{\ensuremath{{\mathbb{F}}}}

\newcommand{\fin}{{\rm fin}}

\theoremstyle{plain}		
	\newtheorem{mytheo}{Theorem} [section]
	
	\newtheorem{myprop}[mytheo]{Proposition}
	\newtheorem{mypropdef}[mytheo]{Proposition/Definition}
	\newtheorem{mycoro}[mytheo]{Corollary}
     \newtheorem{mylemma}[mytheo]{Lemma}
	\newtheorem{mydefi}[mytheo]{Definition}

	\newtheorem{myhypothesis}[mytheo]{Hypothesis}
	\newtheorem*{geom_assumption}{Geometric Assumptions}
	\newtheorem*{spec_assumption}{Spectral Assumption}

\theoremstyle{remark}		
	
	\newtheorem{myrema}[mytheo]{Remark}
	\newtheorem*{myexam}{Example}
	\newtheorem*{myexams}{Examples}

\numberwithin{equation}{section}
\numberwithin{figure}{section}

\title[A generalized P/B-K formula for analytic applications]{A generalized $\PGL(2)$ Petersson/Bruggeman-Kuznetsov formula for analytic applications}
\author{Yueke Hu} 
 \address{Yau Mathematical Sciences Center\\ Tsinghua University\\
	Beijing 100084\\
	China}
\email{yhumath@tsinghua.edu.cn}
 \author{Ian Petrow}
\address{  Department of Mathematics \\
 University College London \\
 25 Gordon Street \\
  London WC1H 0AY \\
   United Kingdom}
\email{i.petrow@ucl.ac.uk}

 \author{Matthew P. Young}
 
 \address{Department of Mathematics \\
 	  Rutgers University \\
 	 Piscataway \\
 	  NJ 08854-8019 \\
 		U.S.A.}		
 \email{mpy4@rutgers.edu}
 
\begin{document}

 \maketitle
  \begin{abstract}
 We develop generalized Petersson/Bruggeman-Kuznetsov (PBK) formulas for specified local components at non-archimedean places. In fact, we introduce two hypotheses on non-archimedean test function pairs $f \leftrightarrow \pi(f)$, called geometric and spectral hypotheses, under which one obtains `nice' PBK formulas by the adelic relative trace function approach. Then, given a supercuspidal representation $\sigma$ of $\PGL_2(\Q_p)$, we study extensively the case that $\pi(f)$ is a projection onto the line of the newform if $\pi$ is isomorphic to $\sigma$ or its unramified quadratic twist, and $\pi(f) = 0$ otherwise. As a first application, we prove an optimal large sieve inequality for families of automorphic representations that arise in our framework. 
 \end{abstract}
 \tableofcontents

\section{Introduction}\label{intro}

\subsection{Motivation}\label{sec:motivation}
 The Bruggeman-Kuznetsov formula, one of the core tools of analytic number theory since the late 1970s, can be stated in its simplest form as follows.   Given a test function $h_\infty(t)$ one defines the Kuznetsov transform of it by \begin{equation}\label{eq:htoH}
 H_\infty(x) = \frac{i}{2} \int_{-\infty}^\infty \frac{J_{2it}(x)}{\cosh (\pi t)} h_\infty(t) t\,dt.
\end{equation}

 For sufficiently well-behaved test functions $h_\infty$ (see \eqref{hinfty_conditions}) and integers $m,n$ with $mn>0$ one has
\begin{multline}
\label{KuzFormClassic} 
\sum_{u \in \cB_0} h_\infty(t_u) a_u(m)\overline{a_u(n)}+ \frac{1}{4\pi} \int_{-\infty}^\infty h_\infty(t) \frac{\pi}{|\zeta(1+2it)|^2} \lambda_t(m) \overline{\lambda_t(n)}\,dt \\
= \delta_{m=n} \frac{1}{4\pi} \intR h_{\infty}(t) t \tanh(\pi t) dt 
+ 
 \sum_{c \in \N} \frac{S(m,n;c)}{c} H_{\infty}\Big(\frac{4 \pi \sqrt{|mn|}}{c}\Big) 
\end{multline}
where $\cB_0$ is an orthonormal basis of Hecke-Maass waveforms $u$ on $\SL_2(\Z) \backslash \mathcal{H}$ normalized by $\vol( \SL_2(\Z) \backslash \mathcal{H})=\pi/3$, 
$1/4 + t_u^{2}$ is the Laplace eigenvalue of $u$, 
$a_u(m)$ are the Fourier coefficients given by
$$ u(x+iy) = 2 \sqrt{y} \left( \frac{\cosh (\pi t_u)}{\pi}\right)^{1/2} \sum_{n \neq 0} a_u(n) K_{it_u}(2 \pi |n| y)e(nx),$$
the $\lambda_t(n) = \sum_{ab=|n|}(b/a)^{it}$, 
and $$S(m,n;c)= \sumstar_{x \shortmod{c}} e\left(\frac{mx + n \overline{x}}{c}\right)$$ is the Kloosterman sum, in which the $*$ indicates that the sum runs over invertible residue classes modulo $c$ and $\overline{x}$ denotes the inverse of $x$ modulo $c$. There is also an opposite-sign version of \eqref{KuzFormClassic} that holds in the case $mn<0$ with the modification that $H_\infty(x)$ is replaced by a function $H_\infty^-(x)$ (see \eqref{Hinfty-}).

Even more classical is the  holomorphic counterpart to \eqref{KuzFormClassic}, i.e. the Petersson formula:
\begin{equation}
\label{PeterssonClassic} 
\sum_{f \in \cB_\kappa} a_{f}(m) \overline{a_{f}(n)}
= \frac{\kappa-1}{4\pi} \left(\delta_{m=n} + 2 \pi i^{-\kappa} \sum_{c\in \N} \frac{S(m,n;c)}{c} J_{\kappa -1}\Big(\frac{4 \pi \sqrt{mn}}{c}\Big)\right),
\end{equation}
where $\cB_\kappa$ is an orthonormal basis of holomorphic cusp forms  $f$ for $\SL_2(\Z)$ of weight $\kappa$ and Fourier coefficients $a_f(m)$ given by $$f(z) = \left( \frac{(4\pi)^\kappa}{\Gamma(\kappa)}\right)^{1/2} \sum_{n\geq 1} a_f(n) n^{\frac{\kappa-1}{2}}e(nz).$$

In the representation-theoretic framework for automorphic forms, the parity $p(u)$ (see \eqref{opposition_sign_parity}) and spectral parameter $t_u$ or weight $\kappa=\kappa_f$ parametrize the possible archimedean local components $\pi_\infty$ of trivial central character cuspidal automorphic representations $\pi$ of $\GL_2/\Q$.  
Therefore, the above Bruggeman-Kuznetsov (both $mn>0$ and $mn<0$ cases) and Petersson formulas can be combined to give a spectral summation device for automorphic forms on $\PGL_2/\Q$ with specified local representation at infinity (and that are unramified at all finite places). 

The goal of this paper is to analogously develop Petersson/Bruggeman-Kuznetsov (PBK) formulas at finite places $p$ that allow control on the associated representations of $\GL_2(\Q_p)$ at those places (as well as at $\infty$).   
To generate such formulas, we use the adelic relative trace formula approach to the PBK formulas of Jacquet \cite{JacquetRelativeTraceFormula} and Zagier \cite{ZagierEisSeriesAndSTF1, Joyner}, as exposited by Knightly and Li \cite{KLPetersson, knightly_kuznetsovs_2013}. We restrict our attention to automorphic forms over $\Q$ in this paper, but many of the local aspects of our work should carry over to more general non-archimedean local fields. 

In this perspective, one chooses a test function $f$ on the group $\GL_2(\A)$ for the pre-trace formula and then integrates along left and right unipotent orbits to obtain the PBK formula. To aid this strategy and to produce a reasonably explicit formula, we introduce two assumptions on the test function $f$ called the geometric and spectral assumptions. The geometric assumptions place a constraint on the support of the local test function $f_p$ on $\GL_2(\Q_p)$ and allow us to establish the standard properties of the geometric side of the formula. The spectral assumption puts a strong constraint on the integral operators $\pi(f)$ 
and allows us to explicate the spectral side of the formula.  The result is Theorem \ref{theoGeomSpec}. 

As an application of Theorem \ref{theoGeomSpec}, we give a harmonically-weighted Weyl-Selberg Law 
for the family of cusp forms $\cF_0(f)$ cut out by our chosen test function $f$ and interpret the leading constant in terms of local Plancherel volumes.  For this result, see Corollary \ref{thmWeylLaw}.  In a similar context, Palm \cite[Thm.\ 3.2.1]{Palm} gave a Weyl law for cusp forms with specified local components as an application of the Selberg trace formula. 

As a second application of Theorem \ref{theoGeomSpec}, we give an axiomatized Large Sieve Inequality for the families $\cF_0(f)$ cut out by $f$. Under additional local hypotheses (stated in Section \ref{sec:intro:LSI}) these large sieve inequalities are optimally strong: the estimate is of the shape $\ll (X|\cF|)^\eps (X + |\cF|)\|\mathbf{a}\|_2^2$, where $X$ is the length of summation of the sequence $\mathbf{a}$ and $|\cF|$ is the cardinality of the family of cusp forms.

Probably the most important part of the paper however are the examples. Most notably, in Section \ref{sec:SCKloostermanSum} we give an elegant expression for the generalized Kloosterman sum that arises from a specified (trivial central character) supercuspidal representation $\sigma$ of $\GL_2(\Q_p)$. 
See Theorem \ref{KlSumThm} for this formula.  

For the specified supercuspidal formula, we set the local test function $f_p$ equal to the diagonal newform matrix coefficient of $\sigma$ restricted to a maximal compact subgroup, building on earlier work of the first author \cite{Hu}. 
This test function $f_p$ generates a generalized PBK formula that selects on the spectral side automorphic forms with local component at $p$ isomorphic to either $\sigma$ or at most two other supercuspidal representations of the same conductor. This is essentially the narrowest possible support on the spectral side under the geometric assumptions. 
For precise statements, see Theorems \ref{cor:SpecAssumption_supercuspidal} and \ref{cor:SpecAssumption_supercuspidal_p2}.

In a parallel fashion, given a primitive Dirichlet character $\chi$ modulo a power of $p$ with $\chi^2 \neq 1$, we construct local test functions $f_\chi$ in Section \ref{sec:egPS} whose generalized PBK formula selects on the spectral side automorphic forms with local component at $p$ isomorphic to a principal series representations $\pi(\chi |\cdot|_p^{it} ,\chi^{-1}|\cdot|_p^{-it})$ for some $t\in \R$. The generalized Kloosterman sum on the geometric side of the formula \eqref{PSgenKLsum2} is in complete analogy with the supercuspidal Kloosterman sum mentioned above. Again, this generalized PBK formula has the narrowest possible support on the spectral side under the geometric assumptions (see Lemma \ref{transport}). 

These examples lay the groundwork for future important analytic applications. That we can produce several interesting examples that satisfy both the geometric and spectral hypotheses shows that while these two hypotheses together may appear fairly restrictive, they nonetheless contain the families of greatest interest to us.

An important feature of the Bruggeman-Kuznetsov (BK) formula is that the integral transform \eqref{eq:htoH} relating the test function $h_\infty$ on the spectral side to the archimedean test function $H_\infty$ on the geometric side is relatively simple and can often be analyzed effectively using 
standard techniques such as stationary phase estimates. The situation (at present) with finite places is not quite as clean: the local test function $f_p$ on $\GL_2(\Q_p)$ continues to play a strong role in the formula whereas the test function $f_\infty$ on $\GL_2(\R)$ can be completely suppressed from the classical BK formula. 

Nonetheless, we develop the sequence of transforms $$h_p \to f_p \to H_p$$ to some extent, where $h_p: \pi \mapsto \pi(f_p)$ is an operator-valued function on the unitary dual of $\PGL_2(\Q_p)$ (assumed to be projections with finite-dimensional image) and $H_p$ are the generalized (local) Kloosterman sums defined in \eqref{eq:KloostermanLocalFormula}. Indeed, Proposition \ref{mainprop} gives an expression for $f_p$ as an integral transform (of sorts) of $h_p$ in terms of matrix coefficients over the unitary dual $\PGL_2(\Q_p)$ with respect to Plancherel measure. Then $H_p$ is by definition an integral of $f_p$ against additive characters. 
Furthermore, Section \ref{sec:examples} gives a list of transform pairs $h_p \to H_p$ for which one can mostly forget about the function $f_p$ on the group entirely. 

As already mentioned, much of this paper builds on previous work of the first author \cite{Hu}. To be definite, our paper contains the following new features:
\begin{itemize}
\item an axiomatized treatment of PBK formulas given by by adelic test functions (see the geometric and spectral assumptions in Section \ref{intro:statement_of_results}),
\item some preliminary applications of the axiomatized PBK formulas (see Section \ref{sec:applications}), 
\item a test function with smallest possible spectral support around a specified supercuspidal representation $\sigma$ with even conductor exponent, i.e.\ we achieve $l_0=0$ in the sense of \cite[Rem.\ 1.4]{Hu}, see the discussion around Theorem \ref{cor:SpecAssumption_supercuspidal} in this paper, 
\item a more natural expression for the generalized Kloosterman sum associated to a specified supercuspidal representation (compare \cite[Def.\ 4.6]{Hu} and Theorem \ref{KlSumThm}), and
\item a detailed treatment of the dihedral $p=2$ case, which is often omitted for convenience but hides many technical complications (see Section \ref{sec:egSupercuspidal_even}).
\end{itemize}

\subsection{Statement of results 
}\label{intro:statement_of_results}
   Let $G=\GL_2$ and $\overline{G} =\PGL_2$. Write $\cH_{\rm fin} =C_c^\infty (\overline{G}(\A_{\rm fin}))$ for the \emph{non-archimedean Hecke algebra} of $\overline{G}$, that is the space of locally constant functions on $G(\A_{\rm fin})$ that are invariant by and compactly supported modulo center $Z(\A_{\rm fin})$. Define the \emph{local Hecke algebra} $\cH_p= C_c^\infty (\overline{G}(\Q_p))$ similarly.

Let $K_p = G(\Z_p)$ and $ZK_p = Z(\Q_p)G(\Z_p)$ for $p<\infty$.  We say that a pure tensor $f= \bigotimes_p f_p \in \cH_\fin$ is \emph{ramified} at a prime $p$ if $f_p$ is not a constant multiple of $1_{ZK_p}$. 
Let $K = \prod_p K_p$ be the standard maximal compact subgroup of $G(\A_{\rm fin})$ and let $K(N)$ be the principal congruence subgroup of $K$.  
The minimal $N \in \N$ such that $f \in \cH_{\rm fin}$ is bi-$K(N)$-invariant is called the \emph{level} of $f$. 

Our generalized Bruggeman-Kuznetsov formula is an equality between a spectral sum of Fourier coefficients/Hecke eigenvalues over a family of automorphic forms and a geometric sum of generalized Kloosterman sums over a set of admissible moduli. In the next several paragraphs, we define these objects.

For an irreducible admissible representation $(\pi, V_\pi)$ with $\pi_{\rm fin}$ the underlying $\overline{G}(\A_{\rm fin})$-representation  and $f\in \cH_\fin$, we write $\pi(f): V_{\pi} \to V_{\pi}$ for the integral operator 
\begin{equation}\label{pif}
\pi(f): v \mapsto \int_{\overline{G}(\A_{\rm fin})} f(g) \pi_{\rm fin}(g)v\,dg.
\end{equation}
Note that $\pi(f):=\pi_{\rm fin}(f)$ but we have dropped the subscript to avoid cluttering the notation.
\begin{mydefi}[Family cut out by $f$]\label{F0(f)def}
We write $\cF_0(f)$ for the set of cuspidal automorphic representations $\pi$ that are spherical at $\infty$ and such that $\pi(f): V_{\pi} \to V_{\pi}$ is not the zero map.  
\end{mydefi}
The family $\cF_0(f)$ has no restrictions on the archimedean spectral parameters of the representation it contains. Such restrictions will be imposed in our formulas in the standard way: by selecting a test function $h_\infty$.  
Note that $\cF_0(f)$ is a harmonic family in the sense of \cite{SarnakShinTemplier}, and at least in spirit every harmonic family on $\PGL_2$ over $\Q$ arises in this way.

For $\pi$ a cuspidal automorphic representation, write $\cB(\pi)$ for an orthonormal basis of $V_\pi$ (with respect to \eqref{aut_inner_prod}). Let $K_\infty = \SO_2(\R)$.
The subspace $V_\pi^{K_\infty \times K(N)}$ of fixed vectors in $\pi$ is finite-dimensional, and for cuspidal $\pi$ let $u=u_\varphi$ be the classical Maass waveform with respect to $\Gamma(N)$ corresponding to $\varphi \in V_\pi^{K_\infty \times K(N)}$ by $u(x+iy) = \varphi( \left( \begin{smallmatrix} y & x \\ & 1\end{smallmatrix}\right) \times 1_{\rm fin}).$ Recall from \eqref{k=0FCclassical} the Fourier coefficients $a_u(m)$ for $m\in \frac{1}{N}\Z$ of a Maass form $u$ for $\Gamma(N)$.

Let $h_\infty(t)$ be a test function as in the classical Kuznetsov formula.  Iwaniec and Kowalski \cite[(15.19)]{IK} give the following sufficient conditions: 
For some $\delta>0$ \begin{equation}\label{hinfty_conditions}
\begin{cases} 
h_\infty(t) \text{ is holomorphic in } |\mathrm{Im}(t)| \leq 1/2+\delta, \\
h_{\infty}(t)\ll (1+|t|)^{-2-\delta}, \text{ and } \\
h_\infty(t)= h_\infty(-t) \text{ for all } t.
\end{cases}
\end{equation}

Let $\psi_p: \Q_p\to \C^\times$ be the standard additive character $\psi_p(x)= e(\{x\}_p)$ and $\psi_{\rm fin} : \A_{\rm fin} \to \C^\times$ be given by $\psi_{\rm fin} = \prod_p \psi_p$.
\begin{mydefi}\label{def:GenKloostermanDef}
For $f \in \cH_\fin$, $m,n \in \Q$ and $c \in \Q_+$, the generalized Kloosterman sum (attached to $f$) that appears in this paper is defined as
\begin{equation}
\label{GenKloostermanDef}
H(m_1,m_2;c) = \iint_{\A_{\rm fin}^2} f\left( \left( \begin{smallmatrix}1 & -t_1 \\ & 1 \end{smallmatrix}\right) \left( \begin{smallmatrix} & -c^{-2} \\ 1 &  \end{smallmatrix}\right)  \left( \begin{smallmatrix}1 & t_2 \\ & 1 \end{smallmatrix}\right) \right) \psi_{\rm fin}(m_1t_1- m_2 t_2)\,dt_1\,dt_2.
\end{equation}
\end{mydefi}
Note that $H(m_1, m_2;c)$ is in fact a finite sum since $f\in \cH_\fin$. While the sum $H(m,n;c)$ is \`a priori defined for all $m,n \in \Q$, it vanishes unless $m,n \in \frac{1}{N}\Z$ (see Theorem \ref{thmKP}(1)). We also define local generalized Kloosterman sums $H_p(m,n;c)$ by the same formula \eqref{GenKloostermanDef} but where $\A_{\rm fin}, \psi_\fin,$ and $f$ are replaced by their local versions $\Q_p, \psi_p,$ and $f_p$ (to be definite, see \eqref{eq:KloostermanLocalFormula}).  

When $f=\bigotimes_p f_p$ is a pure tensor, one has $H(m,n;c) = \prod_p H_p(m,n;c)$. Of course the generalized Kloosterman sum $H(m,n;c)$ depends on $f \in \cH_{\rm fin}$ and the local $H_p(m,n;c)$ depend on $f_p\in \cH_p$, but these are suppressed in the notation. Recall, we also defined the transform $H_\infty(x)$ of $h_\infty(t)$ by \eqref{eq:htoH}, as in the classical Kuznetsov formula. 

Next we define the index set of the sum on the geometric side of our formula. 
\begin{mypropdef}[Admissible moduli]\label{intro:admissiblemodulus}
We say $c \in \Q_+$ is an \emph{admissible modulus} if there exists a pair $(m,n)\in \Q^2$ so that $H(m,n;c)\neq 0,$ 
and write $\mathcal{C}(\cF)\subseteq \Q_+$ for the set of admissible moduli. 
If $f = \bigotimes_p f_p \in \cH_{\rm fin}$ is a pure tensor and $\supp f_p$ is contained in $\{g \in G(\Q_p): \det g \in \Z_p^\times(\Q_p^\times)^2\}$ for all $p$, then there exists $q' \in \Q_+$ such that $\cC(\cF) \subseteq q'\Z$.  
\end{mypropdef}
See Lemma \ref{lem:existance_of_geom_cond_without_hypotheses} for a proof.  Under the mild hypothesis in the second sentence of Proposition/Definition \ref{intro:admissiblemodulus}, it is not too hard to show that one has an ``unrefined'' Bruggeman-Kuznetsov formula of the shape
\begin{multline}\label{intro:GeneralizedKuznetsov1}
\sum_{ \pi \in \cF_0(f)} h_\infty(t_\pi)  \sum_{\varphi \in \mathcal{B}(\pi)}a_{u_{\pi(f)\varphi}}(m_1) \overline{a_{u_\varphi}(m_2) } + (\text{ cts. })  
\\ = (\text{ diag.\ weight })\delta_{m_1=m_2} +  \sum_{c \in \cC(\cF)  } \frac{H(m_1,m_2;c)}{c}H_\infty\left(\frac{4 \pi \sqrt{ m_1m_2}}{c}\right).
\end{multline}
For more details, see Theorem \ref{MT}. On its own, \eqref{intro:GeneralizedKuznetsov1} is not very useful without additional information on $f$. 

We next introduce the geometric and spectral assumptions alluded to in Section \ref{sec:motivation}, which permit a practically useful refinement of \eqref{intro:GeneralizedKuznetsov1}. 
Let $a(y)= \left(\begin{smallmatrix} y & 0 \\ 0 & 1 \end{smallmatrix}\right)$ and $A\subset G$ be as in Section \ref{notation-groupsandsubgroups}. 
\begin{geom_assumption} 
\hfill
\begin{enumerate} 
\item\label{geo2} The function $f \in \cH_{\fin}$ is bi-$A(\widehat{\Z})$-invariant. 
\item\label{geo3} There exists $y\in \Q_+$ such that $\supp f \subseteq a(y)^{-1}ZK  a(y)$. 
\end{enumerate}
\end{geom_assumption}
 We say that a rational number $y \in \Q_+$ for which Geometric Assumption \eqref{geo3} holds ``controls the support of $f$''. Caution: $y$ is not necessarily uniquely determined by  $f$.  Note, Geometric Assumption \eqref{geo3} ensures that the hypothesis on the support of $f$ of Theorem \ref{MT} is satisfied. Another useful fact to keep in mind is that under Geometric Assumption \eqref{geo3}, the function $f$ is ramified at $p$ if and only if $p$ divides the level of $f$, for which see Section \ref{level_vs_ramification}.

\begin{myrema}
Under Geometric Assumption \eqref{geo3}, the test function $f$ has support contained in $ZK'$ for some maximal compact open subgroup $K'$ of $G(\A_{\rm fin})$. One might hope to relax Geometric Assumption \eqref{geo3} to the more natural-sounding condition of being contained in a maximal compact subgroup.
However, we do not pursue this generalization since the relaxed condition that $\supp f \subseteq ZK'$ for some compact open $K'$ together with Geometric Assumption \eqref{geo2} already imply assumption \eqref{geo3} at odd primes, and at $p=2$ there is essentially only a single additional example allowed under the relaxed condition, of which we know no practical application. See Lemma \ref{x=0} for a formal statement.  
\end{myrema}

The geometric assumptions control the set of admissible moduli $\cC(\cF)$ as follows. 
\begin{mypropdef}[Geometric conductor]\label{def:geom_cond}
Suppose Geometric Assumption \eqref{geo3} holds. The set of admissible moduli $\cC(\cF)$ is non-empty. Thus, there exists a unique maximal by divisibility $q' \in \Q_+$ such that $\cC(\cF) \subseteq q'\Z$. We write $k(\cF)$ for the maximal such $q'$ and call it the geometric conductor of $\cF$. The geometric conductor satisfies $k(\cF) \geq y$ for any $y$ controlling the support of $f$.
\end{mypropdef}
For a proof, see Lemmas \ref{ccondition} {and \ref{lem:admmodulus}. With additional information on the support of $f$, the geometric conductor $k(\cF)$ can be determined exactly (see Section \ref{sec:control_on_geom_cond}).  
One also has that $k(\cF) = \prod_p p^{k_p}$ for ``local geometric conductors'' $k_p$ defined analogously by the non-identical vanishing of $H_p$, for which see Theorem \ref{thmKP}(6). 

In addition to controlling $\cC(\cF)$, the geometric assumptions also endow the generalized Kloosterman sums $H(m,n;c)$ with many of the same basic structural properties as the standard Kloosterman sums, as in \cite[Ch.\ 4.3]{IwaniecTopics}. 
For a detailed list of these, see Theorem \ref{thmKP}.

We now move on to the spectral assumption. 
Let $\overline{G}(\Q_p)^{\wedge}$ denote the unitary dual of $\overline{G}(\Q_p)$, i.e. the space of isomorphism classes of smooth irreducible unitary representations of $\overline{G}(\Q_p)$ on a complex vector space.

\begin{mydefi}[Newform projector]\label{def:newformproj}
We say that $f_p \in \cH_p$ is a newform projector if for all generic $(\pi, V) \in \overline{G}(\Q_p)^{\wedge}$ the operator $\pi(f_p):V \to V$ either projects onto the line generated by the newform $\varphi_0 \in V$ or is 0. 
\end{mydefi}

With future and past applications in mind, we also want to allow the classical PBK formula with level structure at finitely many primes (as in \cite{KLPetersson, knightly_kuznetsovs_2013}, recovering the classical formulae). Let $\nu(n) = [\SL_2(\Z):\Gamma_0(n)] = n \prod_{p \mid n} (1+p^{-1})$.

\begin{spec_assumption}
 We say that a pure tensor $f \in \cH_{\rm fin}$ satisfies the spectral assumption if it admits a representative $\prod_p f_p$ such that for all $p$ the function $f_p$ is either a newform projector or there exists $c \in \Z_{\geq 0}$ so that $f_p=\nu(p^c)1_{ZK_0(p^c)}$. 
\end{spec_assumption}
Note, when $c=0$ the test function $1_{ZK_p}$ is itself a newform projector, but when $c\geq 1$ the test function $\nu(p^c)1_{ZK_0(p^c)}$ is not.

The main purpose of the spectral assumption is to simplify the left hand (spectral) side of \eqref{intro:GeneralizedKuznetsov1} (but see also Section \ref{sec:spec_consequences}). Indeed, writing $\pi =  \bigotimes'_v \pi_v$ and $V = \bigotimes'_v V_v$, the operator $\pi(f)$ is an orthogonal projection onto the subspace 
\begin{equation}\label{pifdef}V_f:=V_\infty \otimes \bigotimes_{p: f_p \text{ newform proj.}} \C \varphi_{0,p} \otimes \bigotimes_{\substack{p : f_p=\nu(p^c)1_{ZK_0(p^{c})} \\ \text{ with } c\geq 1}} V_p^{K_0(p^c)}\end{equation}
of $V_{}^{K(N)}$, where $\varphi_{0,p}$ is an $L^2$-normalized newvector in $V_p$ if $\pi_p(f_p) \neq 0$, and $\varphi_{0,p} =0$ otherwise.  For the implementation of this, see Theorem \ref{theoGeomSpec_FC}.

For our intended applications, we need generalized PBK formulas in terms of Hecke eigenvalues in lieu of Fourier coefficients. These are made possible by the spectral assumption.
 If $f$ is a newform projector, then the space $V_f^{K_\infty}$ is 1-dimensional so that there is essentially only one choice of basis $\cB_f(\pi)$. On the other hand, if $f$ is the classical test function with $c\geq 1$ at some primes, then $\dim V_f^{K_\infty}>1$ and the problem of choosing an orthonormal basis for this space that recovers Hecke eigenvalues from Fourier coefficients has been studied by many authors e.g.\ \cite{ILS, Rouymi, NgBasis,BlomerMilicevic2ndMoment,BBDDM}. Indeed, following e.g.\ \cite[\S 7]{PetrowTraces} there exists an orthonormal basis $\cB_f(\pi)$ of $V_f^{K_\infty}$ and weights $w(\pi, f) \in \C$ such that for all $m_1,m_2 \in \N$ and $(m_1m_2,N)=1$
\begin{equation}\label{Eq:SpectralAssumption}
 \sum_{\varphi \in \cB_f(\pi)} a_{u_{\varphi}}(m_1) \overline{a_{u_\varphi}(m_2)} = w(\pi,f) \lambda_{\pi}(m_1)\overline{\lambda_\pi(m_2)},
\end{equation}
 where $\lambda_\pi(m)$ are the Hecke eigenvalues of $\pi$ normalized so that the Ramanujan conjecture predicts that $|\lambda_\pi(m)|\leq d(m)$. 
 Note that the left hand side of \eqref{Eq:SpectralAssumption} is independent of the choice of orthonormal basis $\cB_f(\pi)$, and therefore so is $w(\pi,f)$.

To continue our discussion, we introduce the ``naive Rankin-Selberg $L$-series''. For $\Pi$ a standard (in the sense of \cite[\S 2.2.1]{MichelVenkateshGL2}) generic automorphic representation of $\overline{G}$, let
\begin{equation}\label{naiveRS}
\mathcal{L}_\Pi(s) = \sum_{n \geq 1} \frac{|\lambda_\Pi(n)|^2}{n^s},
\end{equation} and following a notation of Michel and Venkatesh, write $\mathcal{L}_\Pi^*(1)$ for its leading Laurent series coefficient at $s=1$. For $\pi \in \cF_0(f)$  with $q(\pi)$ the (finite) conductor of $\pi$, write
\begin{equation}
r_\pi(p)^{-1} := \begin{cases} (1+p^{-1}) \sum_{\alpha \geq 0} \frac{\lambda_\pi(p^{2\alpha})}{p^\alpha} & \text{ if } p \nmid q(\pi), \\
(1-p^{-2})^{-1} & \text{ if } p \, \| \, q(\pi), \\
1 & \text{ if } p^2 \mid q(\pi). \end{cases}
\end{equation}
Then, for $f$ of level $N$ the weights $w(\pi,f)$ in \eqref{Eq:SpectralAssumption} are given by 
\begin{equation}\label{weights_explicit}
w(\pi,f) = \frac{1}{2 \xi(2) \mathcal{L}_\pi^*(1)} \prod_{\substack{ p^2 \mid N/q(\pi) \\ p \nmid q(\pi)}} (1-p^{-2})^{-1} \prod_{p \mid N/q(\pi)} r_\pi(p)^{-1}=: \frac{1}{2 \xi(2) \mathcal{L}_\pi^*(1)}\frac{1}{\rho_\pi(N/q(\pi))}.
\end{equation}
In \eqref{weights_explicit}, the weights $\rho_\pi(\ell)$ defined on the right are exactly the same weights $\rho_f(\ell)$ or $\rho_E(\ell)$ defined in \cite[\S 2.4]{PetrowYoungWeyl}, with $f$ being the newform in $\pi$. In particular, we have $w(\pi,f) = ((1+|t_\pi|)N)^{o(1)}$ by \cite{HLAppendix, IwaniecSmallEvals}. Note that the factor $2 \xi(2) =  \vol(\overline{G}(\Q) \backslash \overline{G}(\A))$ appearing in \eqref{weights_explicit} is a global volume factor (see e.g.\ \cite[\S 4.1.2]{MichelVenkateshGL2} for a more general statement).

The spectral assumption also allows us to give a motivated expression for the diagonal term constant in the generalized PBK formula in terms of Plancherel volumes.  \begin{mydefi}[Local family]\label{localfamily}
For $f = \bigotimes_p f_p \in \cH_{\rm fin}$, the subspace
\begin{equation}
\cF_p(f):= \{\pi \in {\overline{G}(\Q_p)}^{\wedge}: \pi(f_p) \neq 0\}
\end{equation}
is called the local family of $f$ at $p$.
\end{mydefi} Let $\alpha$ be the quasicharacter of $\Q_p^\times$ defined by $\alpha: x \mapsto |x|_{p}$. For $\pi$ a smooth irreducible unitary generic representation of $\overline{G}(\Q_p)$ set 
\begin{equation}\label{Appidef}
\mathcal{L}_\pi(1) =  \begin{cases}
\frac{(1-p^{-2})}{ ( 1-e^{2 i \theta}p^{-1})(1-p^{-1})(1-  e^{-2 i \theta}p^{-1} )}    & \text{ if } \pi \simeq \pi(\alpha^{i\theta/\log p}, \alpha^{-i\theta/\log p}),  \\ 
(1+ \frac{1}{p})^{-1} & \text{ if } c(\pi) = 1, \\
  (1-\frac{1}{p})     & \text{ if } c(\pi) \geq 2, 
 \end{cases}
 \end{equation} 
 where in the first line either $\theta \in [0, \pi]$, or $\theta= i \tau \log p$ or $\pi + i \tau \log p$ with $\tau \in (0,1/2)$. 
If $\Pi = \pi_\infty \otimes \bigotimes'_p \pi_p$ is a standard generic automorphic representation of $\overline{G}$, then the leading Laurent series coefficient $\mathcal{L}_\Pi^*(1)$ admits the Euler product factorization
\begin{equation}\label{residue_product_formula}
 \mathcal{L}_{\Pi}^*(1) = \prod_p \mathcal{L}_{\pi_p}(1),
\end{equation}
 in the regularized sense of \cite[\S 4.1.5]{MichelVenkateshGL2}.

Let $f_\infty$ be the bi-$K_\infty$-invariant function on $\GL_2^+(\R)$ defined by \cite[(3.5) and Prop.\ 3.7]{knightly_kuznetsovs_2013} in terms of $h_\infty$. Then, by the Plancherel theorem (see (3.17) of loc.\ cit.) we have
\begin{equation}\label{eq:archPlancherelIntro}f_\infty(1) = \frac{1}{4 \pi} \int_{-\infty}^\infty h_\infty(t) \tanh(\pi t) t \,dt.
\end{equation}
Define
\begin{equation}\label{f_Adef}
 f_\A = f_\infty \cdot f. 
\end{equation}
Suppose that $f$ satisfies the spectral assumption. For each place $v$, define the diagonal weight $\delta_v$ at $v$ as follows. If $v=\infty$, set $\delta_\infty = f_\infty(1)$. If $v=p<\infty$ and $f_p$ is a newform projector, set  
\begin{equation}\label{eq:deltapnewformproj}
\delta_p =  \int_{\cF_p(f)} \frac{1}{\mathcal{L}_\pi(1)} \, d \widehat{\mu}(\pi),
\end{equation}
and if $f_p= \nu(p^c)1_{ZK_0(p^c)}$ for some $c\in \Z_{\geq 0}$, set
\begin{equation}\label{eq:deltapclassical}
\delta_p= \int_{\cF_p(f)} \dim \pi^{K_0(p^c)}\, d \widehat{\mu}(\pi).
\end{equation}
Note that $\delta_p=1$ for all but finitely many $p$. Finally, set $\delta_\fin= \prod_p \delta_p$ and $\delta = \delta_\infty \delta_\fin$.   

\begin{mytheo}\label{theoGeomSpec}
Let $f \in \cH_{\rm fin}$ be a pure tensor satisfing the geometric and spectral assumptions. 
For all $m_1,m_2 \in \Z$ with $m_1m_2 >0$ and $(m_1m_2,N)=1$ we have 
\begin{multline}\label{GeneralizedKuznetsov2}\sum_{ \pi \in \cF_0(f)} h_\infty(t_\pi) w(\pi, f) \lambda_\pi(m_1)\overline{\lambda_\pi(m_2)} + (\text{ cts. })  \\
=  \delta_{m_1=m_2}\delta  +  \sum_{c \equiv 0 \shortmod{k(\cF)} } \frac{H(m_1,m_2;c)}{c}H_\infty\left(\frac{4 \pi \sqrt{ m_1m_2}}{c}\right),
\end{multline}
where $(\text{ cts. })$ is a similar continuous spectrum term that we give explicitly in \eqref{ctsdots}.
\end{mytheo}
\begin{myrema}
The assumption $(m_1 m_2, N) = 1$ appearing in Theorem \ref{theoGeomSpec} is a helpful simplification at this stage of the presentation, but is not crucial.
The intermediate step Theorem \ref{theoGeomSpec_FC} towards Theorem \ref{theoGeomSpec} does not require the condition $(m_1m_2,N)=1$, but leaves the spectral side in terms of Fourier coefficients. Instead of inserting \eqref{Eq:SpectralAssumption} into Theorem \ref{theoGeomSpec_FC} to obtain Theorem \ref{theoGeomSpec}, one can use e.g.\ \cite[(15)]{PetrowYounghybrid} restricted to a single old-class, which requires square-free level but avoids any coprimality condition. The coprimality condition is also used in Section \ref{sec:MTcomputation} to compute the diagonal term, but this section can be easily generalized with some additional computation.  
\end{myrema}

\begin{myrema}\label{rem:other_arch_cases}
Theorem \ref{theoGeomSpec} also holds for other choices of archimedean test functions. For example, let $\kappa\geq 2$ be even and let $\pi_\kappa$ be the discrete series representation of $\GL_2(\R)$ of weight $\kappa$. Define $\cF_\kappa(f)$ as in Definition \ref{F0(f)def} to be the set of cuspidal automorphic representations $\pi$ with $\pi_\infty \simeq \pi_\kappa$ and such that $\pi(f): V_{\pi} \to V_{\pi}$ is not the zero map. Define $V_f^\kappa$ to be the weight $\kappa$ isotypic subspace of $V_f$.  Choose $f_\infty$ to be given by the Bergman kernel \eqref{holo:f_infty}, in particular one has  $f_\infty(1) = \frac{\kappa-1}{4 \pi}$. Then, under the same hypotheses as Theorem \ref{theoGeomSpec} with $\cF_0(f)$ and $V_f^{K_\infty}$ replaced by $\cF_\kappa(f)$ and $V_f^\kappa$, respectively, we have 
\begin{multline}\label{theoGeomSpec_holomorphic}
\sum_{ \pi \in \cF_\kappa(f)}  w(\pi, f) \lambda_\pi(m_1)\overline{\lambda_\pi(m_2)} \\
= \delta_{m_1=m_2} \delta +  \frac{(\kappa-1)}{2} i^{-\kappa} \sum_{c \equiv 0 \shortmod{k(\cF)} } \frac{H(m_1,m_2;c)}{c}J_{\kappa-1}\left(\frac{4 \pi \sqrt{ m_1m_2}}{c}\right),
\end{multline}
with $w(\pi,f) = (\kappa N)^{o(1)}$. 
For $\kappa \geq 4$ the archimedean aspects of the holomorphic forms variation \eqref{theoGeomSpec_holomorphic} were worked out in \cite{KLPetersson} and while a relative trace formula proof of the $\kappa=2$ case strictly speaking has not appeared in the literature, it is expected to follow from a limiting argument and in any case is well-known from the classical Poincar\'e series approach to the Petersson formula. We give more details on the holomorphic/discrete series case in  Sections \ref{sec:holomorphicunrefined} and \ref{theoGeomSpecProof}. 

Similarly, one expects the opposite-sign case of Theorem \ref{theoGeomSpec} in which $m_1m_2<0$ to hold, but at present there is not a relative trace formula proof for this case.  The shape of the formula would be similar but 
with an additional factor of $p(\pi)$ on the spectral side, where
\begin{equation}\label{opposition_sign_parity}
p(\pi) = \text{ parity of } \pi = \begin{cases} \text{ eigenvalue of } u_\varphi, \varphi \in V_\pi \text{ under the involution } x+iy \mapsto -x+iy, \text{ or } \\ (-1)^\epsilon \text{ where } \pi_\infty \simeq \pi(\sgn^\epsilon |\cdot|^s , \sgn^\epsilon |\cdot|^{-s}), \end{cases}
\end{equation}  
and the factor $H_\infty(x)$ on the geometric side of the formula is replaced with  
\begin{equation}\label{Hinfty-}
H_\infty^-(x) = \frac{1}{\pi} \int_{-\infty}^\infty K_{2it}(x) \sinh (\pi t) h(t) t \,dt.
\end{equation}

We posit
that when $\kappa=2$ the formula \eqref{theoGeomSpec_holomorphic} holds, and when $m_1m_2<0$ the formula \eqref{GeneralizedKuznetsov2} with modifications \eqref{opposition_sign_parity} and \eqref{Hinfty-} holds, and assume these to be so in the following discussion. Since this paper concerns non-archimedean aspects, and for the sake of brevity, we do not provide any details for these assertions.
\end{myrema}

\begin{myrema}
Since we have not modified any archimedean aspects of the classical PBK formulas when deducing Theorem \ref{theoGeomSpec} etc., we also get the ``backwards'' Kuznetsov formula as in \cite[\S 16.4]{IK} mutatis mutandis. Indeed, for the rest of this paragraph let $\Phi \in C^2([0,\infty))$ with $$\Phi(0)=0 \quad \text{ and } \quad \Phi^{(a)}(x) \ll_a (1+x)^{-\alpha}$$ 
for $a=0,1,2$ and some $\alpha>2$. Let $\mathcal{M}_\Phi(t)$ be the Hankel transform of $\Phi$ as defined in \cite[16.40]{IK} and $\mathcal{N}_f(k)$ be the Neumann coefficients of $\Phi$ as defined in \cite[16.41]{IK}. Let $f\in \cH_{\rm fin}$ be as in Theorem \ref{theoGeomSpec} with associated generalized Kloosterman sum $H(m,n;c)$. If $m_1m_2>0$, then 
\begin{multline}\label{PBKinverted_even}
\sum_{c \equiv 0 \shortmod{k(\cF)} } \frac{H(m_1,m_2;c)}{c}\Phi\left(\frac{4 \pi \sqrt{ m_1m_2}}{c}\right) \\
= \frac{4}{\pi} \sum_{ \pi \in \cF_0(f)} \mathcal{M}_\Phi(t_\pi) \cosh (\pi t_\pi)  \sum_{\varphi \in \cB_f(\pi)} a_{u_{\varphi}}(m_1) \overline{a_{u_\varphi}(m_2)}  + (\text{ cts. }) \\ 
+ \sum_{\substack{\kappa>0 \\ \kappa \equiv 0 \shortmod 2}} \frac{(4\pi)^\kappa}{\Gamma(\kappa)} \mathcal{N}_\Phi(\kappa) \sum_{ \pi \in \cF_\kappa(f)}   \sum_{\varphi \in \cB_f(\pi)} a_{u_{\varphi}}(m_1) \overline{a_{u_\varphi}(m_2)} , 
\end{multline}
where $(\text{ cts. })$ is a continuous spectrum term given in   
\eqref{GeneralizedKuznetsov2_FC} with $h_\infty(t)$ there replaced by $\frac{4}{\pi} \cosh (\pi t) \mathcal{M}_\Phi(t)$. See Section \ref{sec:holomorphicunrefined} for definitions of $u_\varphi$ and $a_u$ in the holomorphic / discrete series case. Meanwhile, if $m_1m_2<0$, then we set $\mathcal{K}_\Phi(t)$ to be the integral transform of $\Phi$ given in \cite[(16.44)]{IK}. In this case, we have 
\begin{multline}\label{PBKinverted_odd}
\sum_{c \equiv 0 \shortmod{k(\cF)} } \frac{H(m_1,m_2;c)}{c}\Phi\left(\frac{4 \pi \sqrt{ |m_1m_2|}}{c}\right) \\
= \frac{4}{\pi} \sum_{ \pi \in \cF_0(f)}  \mathcal{K}_\Phi(t_\pi) \cosh (\pi t_\pi)  \sum_{\varphi \in \cB_f(\pi)} a_{u_{\varphi}}(m_1) \overline{a_{u_\varphi}(m_2)}  + (\text{ cts. }), 
\end{multline}
where similarly $(\text{ cts. })$ is a continuous spectrum term given in  \eqref{GeneralizedKuznetsov2_FC} with $h_\infty(t)$ there replaced by $\frac{4}{\pi} p(\pi_{\chi,\chi^{-1}}) \cosh (\pi t) \mathcal{K}_\Phi(t)$. 
\end{myrema}

\begin{myrema} To check that the geometric and spectral assumptions hold for a pure tensor $f \in \cH_{\rm fin}$, it suffices to check them for $f_p$ at the finitely many primes $p$ where $f$ is ramified. The following data appearing in Theorem \ref{theoGeomSpec} can also be computed locally:
\begin{itemize}
\item the local families $\cF_p(f)$,
\item the local levels $N_p:= p^{v_p(N)}$,
\item the diagonal weights $\delta_v$, 
\item the generalized Kloosterman sums $H_p(m,n;c)$, and
\item the local geometric conductors $k_p$. 
\end{itemize}
To produce completely explicit cases of Theorem \ref{theoGeomSpec}, it suffices to produce appropriate local test functions $f_p$ and perform purely local computations of the relevant data. We do this for several  key examples in Section \ref{sec:examples} of the paper. 
\end{myrema}

\begin{myexams}
Choose a finite set $S$ of primes. Let $f = \bigotimes_p f_p \in \cH_{\rm fin}$ be a pure tensor such that if $p \not \in S$, then $f_p =  1_{ZK_p}$, and if $p \in S$, then $f_p \in \cH_p$ may be chosen from any of the following:
\begin{itemize}
\item for some integer $c\geq 0$, the classical test function $f_{\leq c}$ defined in \eqref{classical_test_fcn}, 
\item for a non-quadratic character $\chi$ of $\Z_p^\times$, the function $f_\chi$ defined in \eqref{fchi_def}, or 
\item for a quadratic extension $E$ of $\Q_p$ and a character $\xi$ of $E^\times$ satisfying the hypotheses in the first paragraph of Section \ref{sec:examples_supercuspidal}, the function $f_{\xi}$ defined in \eqref{sec7_fxi_def},
\item for $(E/\Q_p,\xi)$ as in the previous point and $1\leq n<c(\xi')$ with $\xi'$ a twist-minimal character underlying $\xi$, the function $f_{\xi,n}$ defined in \eqref{fxin_def}, or 
\item for some integer $c\geq 3$, the test function $f_{=c}$ introduced by Nelson, see \eqref{f=c_def}. 
\end{itemize}
 Then $f$ satisfies the geometric and spectral assumptions (see Sections \ref{classic_gsassumptions}, \ref{PS_gsassumptions}, \ref{SC_gsassumptions}, \ref{sec:neighborhood_of_a_SC}, and \ref{sec:f=c}).

Let $\cF_0(f)$ be the family of cuspidal automorphic representations cut out by $f$ as in Definition \ref{F0(f)def}. In particular, $\cF_0(f)$ consists of $\GL_2/\Q$ cuspidal automorphic representations $\pi$ of trivial central character (spherical at infinity) whose local components $\pi_p$ at finite places are constrained to lie in the local families 
$$\cF_p(f) = \begin{cases} \cF_{\leq c} & \text{ if } f_p = f_{\leq c}, \\ 
\cF_\chi & \text{ if } f_p= f_{\chi}, \\ 
\cF_{\xi} & \text{ if } f_p= f_{\xi}, \\
\cF_{\xi, n} & \text{ if } f_p= f_{\xi,n}, \text{ and }\\
\cF_{=c} & \text{ if } f_p = f_{=c}
\end{cases}
$$
(see Definition \ref{localfamily}), where: 
\begin{itemize}
\item For $c\geq 0$
\begin{equation} 
\cF_{\leq c} = \{ \pi \in {\overline{G}}(\Q_p)^{\wedge}: c(\pi) \leq c\}.
\end{equation}
\item For $\chi \in {\Z_p^{\times}}^{\wedge}$ not quadratic
\begin{equation}
\cF_\chi = \{ \pi(\mu,\mu^{-1}) \in  {\overline{G}}(\Q_p)^{\wedge}: \mu \vert_{\Z_p^\times} = \chi\}.
\end{equation}
\item For $(E/\Q_p,\xi)$ as above, let $\sigma= \sigma(\rho)$ be the (trivial central character) supercuspidal representation of $G(\Q_p)$ corresponding to $\rho = \Ind_{E}^{\Q_p} \xi$ under the Local Langlands Correspondence (LLC). Writing 
 $\eta$ for the unramified quadratic character of $\Q_p^\times$, we have
\begin{equation}\label{intro:locfam:SC}
\cF_{\xi} = \begin{cases} \{\sigma\}  & \text{ if } E/\Q_p \text{ is unramified and } p \neq 2, \\ 
\{\sigma, \sigma \times \eta\} & \text{ if }  E/\Q_p  \text{ is ramified,} \\
\{ \sigma( \Ind_{E}^{\Q_p} \xi_1) : c(\xi_1 \xi^{-1})\leq 1, \xi_1\vert_{\Q_p^\times} = \xi \vert_{\Q_p^\times} \} & \text{ if } E/\Q_p \text{ is unramified and } p = 2.
\end{cases} 
\end{equation}
The set in the last line consists of 3 supercuspidal representations of the same conductor as $\sigma$. For interpretation, it may be helpful to recall that when $p \neq 2$, the extension $E/F$ is unramified if and only if $c(\sigma)$ is even. See Section \ref{sec:parametrization_dihedrals} for a quick overview of the parametrization of dihedral trivial central character supercuspidal representations in terms of pairs $(E/\Q_p,\xi)$. 
\item For $1\leq n<c(\xi')$, we have $$\cF_{\xi,n}= \{ \sigma( \Ind_{E}^{\Q_p} \xi_1) : c(\xi_1 \xi^{-1})\leq n, \xi_1\vert_{\Q_p^\times} = \xi \vert_{\Q_p^\times} \}.$$
\item For $c\geq 3$, we have $$ \cF_{=c} = \{\pi \in {\overline{G}}(\Q_p)^{\wedge} : c(\pi ) = c\}.$$
\end{itemize}
See Sections \ref{sec:locfam-classical}, \ref{sec:locfam-PS}, \ref{sec:locfam-SC}, \ref{sec:neighborhood_of_a_SC}, and \ref{sec:f=c} for more details.

The level $N$ of $f$ satisfies
\begin{equation}\label{intro:level}
v_p(N) = \begin{cases} 
c & \text{ if } f_p= f_{\leq c} \text{ or } f_{=c},\\ 
2c(\chi) & \text{ if } f_p=f_{\chi}, \\ 
c(\sigma) & \text{ if } f_p = f_{\xi} \text{ or } f_{\xi,n}.
\end{cases}
\end{equation}

On the geometric side of the formula, the diagonal weights $\delta_p$ may be given explicitly by 
\begin{equation}\label{diagwt_egs}
\delta_p = \begin{cases} \nu(p^c) & \text{ if } f_p = f_{\leq c}, \\
\frac{\nu(p^{c(\chi)})}{1-p^{-1}} & \text{ if } f_p = f_\chi, \\
\text{ see } \eqref{diagwt:SC} & \text{ if } f_p=f_\xi, \\
\text{ see } \eqref{diagwt:nbhd_of_a_SC} & \text{ if } f_p = f_{\xi,n}, \\ 
p^c(1-p^{-2}) & \text{ if } f_p =f_{=c}.
\end{cases}\end{equation}
for which see \eqref{diagwt:classical}, \eqref{diagwt:PS} and Section \ref{sec:f=c}. In the supercuspidal cases, we write $d=v_p(\disc (E/\Q_p))$. The geometric conductor $k(\cF)= \prod_p p^{k_p}$ and the local geometric conductors for the above test functions are given explicitly by
\begin{equation}\label{kp:3egs}
k_p = \begin{cases} c & \text{ if } f_p = f_{\leq c}, \\
c(\chi) & \text{ if } f_p = f_\chi, \\
c(\xi) & \text{ if } f_{p} = f_{\xi} \text{ with }d = 0, \\
\frac{c(\xi)}{2} +1 & \text{ if } f_p = f_{\xi} \text{ with } d = 1 \text{ or } 2, \\
\frac{c(\xi)}{2} +2 & \text{ if } f_p = f_{\xi} \text{ with } d = 3, \\
\text{ see } \eqref{kp_nbhd_of_a_SC_formula} & \text{ if } f_p = f_{\xi,n}, \\
c-1 & \text{ if } f_p = f_{=c},
\end{cases}\end{equation}
for which see Sections \ref{kp:classical},  \ref{kp:PS}, \ref{kp:SC}, \ref{sec:neighborhood_of_a_SC}, and \ref{sec:f=c}. 

Lastly, for $c \equiv 0 \pmod {k(\cF)}$, the generalized Kloosterman sum is given by $$H(m,n;c) = \prod_{p \mid c} H_p(m,n;c),$$ where each local Kloosterman sum $H_p$ can be explicitly described as follows. For each $p$, let us write $c=c_0p^k$ with $(c_0,p)=1$, and where we \emph{assume that} $k\geq k_p$ (otherwise $H(m,n;c)=0$). Write $\overline{c_0}$ for the inverse of $c_0$ modulo $p^k$. 

If $ f_p = f_{\leq \fc}$ (including the case $\fc=0$), then we have
\begin{equation}
H_p(m,n;c) = \delta_p S(\overline{c_0}m, \overline{c_0}n; p^k).
\end{equation}
If $f_p=f_\chi$ with $\chi$ not quadratic, then we have
\begin{equation}
H_p(m,n;c) = \delta_p \sumstar_{\substack{x,y \shortmod{p^k} \\ xy =mn\overline{c_0}^2}} \overline{\chi(x)}\chi(y)e\left( \frac{x+y}{p^k}\right)
\end{equation}
when $(mn,p)=1$ and $H_p(m,n;c) = 0$ otherwise. 

If $f_p = f_{\xi}$ for a pair $(E/\Q_p,\xi)$ satisfying the hypotheses in the first paragraph of Section \ref{sec:examples_supercuspidal}, then 
\begin{equation}
\label{intro:SCklSum}
H_p(m,n;c)=\delta_p \overline{\gamma} p^{-\frac{d}{2}} \sum_{\substack{u \in (\cO_E/p^k\cO_E)^{\times} \\ \Nm(u) = mn\overline{c_0}^2}} \xi(u) e\left( -\frac{\Tr u}{p^{k}} \right) \end{equation}
when $(mn,p)=1$ and $H_p(m,n;c) = 0$ otherwise. 
In \eqref{intro:SCklSum}, $\Nm, \Tr: E \to \Q_p$ are the field norm and trace, and $\gamma$ is the Langlands constant associated to $E$ and the additive character $\psi_p=e(\{.\}_p)$ of $\Q_p$. See Remark \ref{Remark_following_KlSumThm} following Theorem \ref{KlSumThm} for more detailed information on $\gamma$ and Propositions \ref{statphase_bounds} and \ref{Katz_AG_bounds} for bounds on $H_p(m,n;c)$. 

In \cite[Def.\ 4.6]{Hu} the first author gave an alternative formula for $H_p(m,n;c)$ that at first glance looks quite different from \eqref{intro:SCklSum}. However, these two formulas are in fact equal (up to leading constants) whenever the former formula is valid, as can be seen by computing the Fourier/Mellin transform of both formulas and applying $p$-adic stationary phase analysis.

If $f_p = f_{\xi,n}$, then $H_p(m,n;c)$ is exactly the same as in $\eqref{intro:SCklSum}$, but with $\delta_p$ and $k_p$ given by \eqref{diagwt:nbhd_of_a_SC} and \eqref{kp_nbhd_of_a_SC_formula} in lieu of \eqref{diagwt:SC} and \eqref{kp_SC_formula}. 
\end{myexams}

\subsection{Relations between parameters}
The reader may have already observed that the families of automorphic forms in this paper have several different parameters associated with them. These include:
\begin{itemize}
\item the level $N$ of $f$,
\item the primes $p$ at which $f$ is ramified,
\item the conductors $q(\pi)$ of representations $\pi \in \cF_0(f)$ and the conductor exponents $c(\pi)$ of local representations $\pi \in \cF_p(f)$,
\item the geometric conductor $k(\cF)$ and local geometric conductors $k_p$, 
\item the value $f(1)$ and local values $f_p(1)$, and 
\item the diagonal weight $\delta_\fin$ and local diagonal weights $\delta_p$. 
\end{itemize}
We explicate some of the relations between the above quantities. 

\subsubsection{Level versus ramification}\label{level_vs_ramification}
Under Geometric Assumption \eqref{geo3}, $p \mid N$ if and only if $f$ is ramified at $p$. Indeed, it is clear that $p \mid N$ implies $p$ is ramified for $f$. For the other direction, suppose $p \nmid N$ so that $f_p$ is bi-$ZK_p$-invariant. Then, by the Cartan decomposition, the function $f_p$ is determined by its values on $\sigma_i = \left( \begin{smallmatrix} p^i & \\ & 1 \end{smallmatrix}\right)$ for $i\geq 0$. However, no $\sigma_i$ with $i>0$ lies in a subgroup of the form $a(y)^{-1}ZK_pa(y)$ for any $y \in \Q_+$, since powers of $\sigma_i$ escape any compact modulo center set. Therefore $f_p$ is only supported on $\sigma_0$ and hence is a constant multiple of $1_{ZK_p}$. 
\subsubsection{Level versus conductors of representations}
Suppose that $f$ satisfies Geometric Assumption \eqref{geo2}. Then, any $\pi \in \cF_0(f)$ has $q(\pi) \mid N^2$. Indeed, by Geometric Assumption \eqref{geo2} $f$ is bi-$K_d(N)$-invariant, so any $\pi \in \cF_0(f)$ has a non-zero $K_d(N)$-fixed vector, 
 and hence a non-zero $K_0(N^2)$-fixed vector, since $a(N)^{-1} K_d(N)a(N) = K_0(N^2)$. See Section \ref{notation-groupsandsubgroups} for the definitions of various groups and subgroups. 
 
 If $f$ satisfies the spectral assumption, then $\pi \in \cF_0(\pi)$ satisfies $q(\pi) \mid N$. Indeed, any $\pi \in \cF_0(f)$ has a non-zero $K(N)$-fixed vector that is also a $K_0(M)$-fixed vector for some $M$ by the spectral assumption. Then $\pi$ has a non-zero $K(N)K_0(M) = K_0(g)$-fixed vector where $g=(N,M)$ (see Section \ref{sec:PSlevel}), so in particular $\pi$ has a non-zero $K_0(N)$-fixed vector.  

On the other hand, there is in general no lower bound on the conductors of $\pi$ that appear in $\cF_0(f)$ in terms of the level $N$ of $f$. Indeed, level $1$ forms appear as oldforms in the classical BK formula of level $N$, which is a special case of our framework.
\subsubsection{Level versus geometric conductor}
Suppose that $f$ has level $N$, that $f(1) \neq 0$ and that $f$ satisfies the geometric assumptions.  Then we have $k(\cF) \mid N$ (see Corollary \ref{cor:admmodulus}). 
On the other hand, under the geometric and spectral assumptions we also have that $k_p \geq 0$, see Lemma \ref{specControlonH}(4). 

\subsubsection{Conductors of representations versus $f(1)$}\label{sec:cond_of_reps_vs_f1}
Suppose that $f\neq 0$ satisfies the geometric and spectral assumptions. Let us work locally at $p$. If $f_p=f_{\leq c}$ is the classical test function it is clear that $$\{c(\pi) : \pi \in  \cF_p(f) \} = \{0,\ldots, c\},$$ so we henceforth assume that $f_p$ is a newform projector. 

Since $f\neq 0$, then $\cF_p(f) \neq 0$. If $\cF_p(f)$ contains an irreducible principal series representation $\pi(\chi,\chi^{-1})$ with $\chi\vert_{\Z_p^\times}$ not quadratic, then by Lemma \ref{transport} it contains $\pi(\chi \alpha^{it},\chi^{-1} \alpha^{-it})$ for all $t \in \R$. Suppose that $\cF_p(f)$ only contains $\pi(\chi,\chi^{-1})$ with $\chi\vert_{\Z_p^\times}$ non-trivial quadratic. Then, by Remark \ref{rem_following_lem_transport} it also contains a special representation. Thus, $\cF_p(f)$ either contains a square-integrable representation, or for some $\chi$ with $\chi \vert_{\Z_p^\times}$ not quadratic it contains $\cF_\chi = \{\pi(\chi \alpha^{it},\chi^{-1} \alpha^{-it}):t \in \R\} \subseteq \cF_p(f)$, or it contains $\cF_p(f)= \{\pi : \pi \text{ unramfied }\}$. Thus, since $f_p$ is a newform projector, by the Plancherel formula we have 
\begin{equation}
f_p(1) = \widehat{\mu}(\cF_p(f)) \geq \begin{cases} \widehat{\mu}(\{ \pi \} ) & \text{ if }  \pi \in \cF_p(f) \text{ is square integrable, } \\ 
\widehat{\mu}(\cF_\chi) & \text{ if } \pi(\chi,\chi^{-1}) \in \cF_p(f), \chi \text{ not quadratic, } \\ 
1 & \text{ if } c(\pi)= 0 \text{ for all } \pi \in \cF_p(f).
\end{cases}
\end{equation}
Let $d_\mu(\pi)$ denote the formal degree of $\pi$. If $\pi$ is square integrable, then $\widehat{\mu}(\{\pi\}) = d_\mu(\pi)  \gg p^{c(\pi \times \pi)/2}\gg p^{\lceil c(\pi)/2\rceil}$ by e.g.\ \cite[Thm.\ 2.1]{IchinoLapidMao}, with an absolute implied constant. 
In the principal series case, one has $\widehat{\mu}(\cF_\chi)=\nu(p^{c(\chi)})$. In all cases, if $ \pi \in \cF_p(f)$ with $f$ a newform projector satisfying the geometric assumptions, then $f_p(1) \gg p^{\lceil c(\pi)/2\rceil}$ with an absolute implied constant. 

In the other direction, if we set $c_{\rm max} = \max \{c(\pi) : \pi \in \cF_p(f)\}$, then 
\begin{equation}
f(1) \leq \int_{\cF_p(f)} \dim V_\pi^{K_0(p^{c_{\rm max}})}\,d\widehat{\mu}(\pi) = \nu(p^{c_{\rm max}}) \leq \tfrac{3}{2}p^{c_{\rm max}}.
\end{equation}

\subsubsection{Level versus $f(1)$}\label{sec:level_vs_f1}
Suppose that $f$ satisfies the spectral assumption. Working locally, suppose $c= \max \{c(\pi): \pi \in \cF_p(f)\}$. Then, by Proposition \ref{mainprop}, $f_p$ is bi-$K_0(p^{c})$-invariant and also bi-$K(N)$-invariant, so that $f$ is bi-$K_0(N)$-invariant.  If $f_p$ is a newform projector, then by the Plancherel formula and newform theory
$$f_p(1) = \widehat{\mu}(\cF_p(f)) \leq \widehat{\mu}(\{\pi: c(\pi) \leq v_p(N)\}) = \nu * \mu(N_p) \leq N_p.$$
On the other hand if $f_p$ is the classical test function, then $f_p(1)=\nu(N_p)$. Therefore globally, 
\begin{equation}\label{eq:level_vs_f1}
f(1) \leq \nu(N).
\end{equation} 

In the other direction, we work locally and assume that $f_p$ satisfies the geometric and spectral assumptions. If $f_p$ is the classical test function, then $f_p(1)=\nu(N_p)$, so suppose $f_p$ is a newform projector. Let $\pi \in \cF_p(f)$ be a representation of maximal conductor exponent. Then by Proposition \ref{mainprop}, $f_p$ is bi-$K_0(p^{c(\pi)})$-invariant, thus $N_p \mid p^{c(\pi)}$, and so by Section \ref{sec:cond_of_reps_vs_f1}, we have \begin{equation}\label{eq:N_vs_f1}
N_p^{1/2} \leq p^{c(\pi)/2} \leq p^{\lceil c(\pi)/2\rceil} \ll f_p(1).
\end{equation}
\subsubsection{Diagonal weight versus $f(1)$}\label{diagwt_vs_f1}
Suppose that $f$ satisfies the spectral assumption. In the first case when $f_p$ is a newform projector, we have  by inspecting \eqref{Appidef} that $\delta_p = (1+O(p^{-1})) f_p(1)$ and moreover
$\tfrac{1}{6} f_p(1) \leq \delta_p \leq 2 f_p(1)$, so that $\delta_p$ is non-vanishing. In the second case that $f_p$ is the classical test function, the situation is even simpler as we have $\delta_p = f_p(1)$. Then, by \eqref{eq:level_vs_f1} and \eqref{eq:N_vs_f1} one has \begin{equation}\label{eq:delta_vs_f1}
\delta_\fin = f(1)N^{o(1)} = f(1)^{1+o(1)}.
\end{equation}

\subsection{Weighted Weyl-Selberg Law and equidistribution}\label{sec:introWeylLaw}
In this section and in Section \ref{sec:intro:LSI}, we consider \emph{families} of automorphic representations. That is, we consider  sequences of varying test functions $f$ or $f_\A$ with some parameter, usually $f(1)$ or $f_\A(1)$, going to infinity. Recall by the Plancherel formula (see e.g.\ \eqref{PlDef}), that $f_p(1)$  is equal to an integral over the local family $\cF_p(f)$ of representations with respect to Plancherel measure, and that the spectral assumption implies that $f(1)\geq 0$. 

To streamline the exposition in this section and the next, we only present our results in detail for Maass cusp forms, but similar results hold for the holomorphic / discrete series variation. Accordingly, we now  choose the archimedean test function $h_\infty$ to be one of either \begin{equation}
\label{eq:hdefWindowVersion}
h_\infty(t) = \frac{t^2+\frac14}{T^2} \Big[ \frac{1}{\cosh\left( \frac{t-T}{\Delta} \right)} + \frac{1}{\cosh\left( \frac{t+T}{\Delta} \right)}  \Big],
\end{equation}
where $1 \leq \Delta < T/100$,
to give a smooth approximation to the small window $T -\Delta< \pm t \leq T + \Delta$, or alternatively
\begin{equation}
\label{eq:hdefInitialSegmentVersion}
h_\infty(t) = \frac{t^2+\frac14}{T^2} \exp\Big(-\Big(\frac{t}{T}\Big)^2\Big)
\end{equation}
for a smooth approximation to the large window $|t| \leq T$.  We call $\pm [T-\Delta,T+\Delta]$ the effective support of \eqref{eq:hdefWindowVersion} and $[-T,T]$ the effective support of \eqref{eq:hdefInitialSegmentVersion}.

With these weights, we have the following crude bound. 
\begin{mylemma}\label{lemWeylLaw} 
Let $h_\infty$ be one of the two test functions given by \eqref{eq:hdefWindowVersion} or \eqref{eq:hdefInitialSegmentVersion}. If $f \in \cH_{\rm fin}$ and $w(\pi,f)$ are as in Theorem \ref{theoGeomSpec}, then for all $m_1,m_2 \in \Z$ with $m_1m_2 >0$ and $(m_1m_2,N)=1$ we have 
\begin{equation}\label{eq:lemWeylLaw}
\sum_{ \pi \in \cF_0(f)} h_\infty(t_\pi) w(\pi, f) \lambda_\pi(m_1)\overline{\lambda_\pi(m_2)} + (\text{ cts. })  \\
= \delta_{m_1=m_2} \delta + O\left( \frac{f_\A(1)m_1m_2}{T^2 k(\cF)}\right).
\end{equation}
\end{mylemma}
We emphasize that we made no effort for optimality in Lemma \ref{lemWeylLaw}, including at the archimedean place, but are rather just recording a simple bound.
As an illustration, we only use the trivial bound on the generalized Kloosterman sums in this proof, and any non-trivial bound on the ramified part of the generalized Kloosterman sums would improve the error term in \eqref{eq:lemWeylLaw}.  
Note that by \eqref{eq:delta_vs_f1} the main term in \eqref{eq:lemWeylLaw} is larger than the error term as soon as there exists $\delta>0$ so that $\frac{T^2 k(\cF)}{m_1m_2} \gg f_\A(1)^\delta$. Particularly pleasing is the shape of the main term in the following corollary. 
\begin{mycoro}[Harmonically-weighted Weyl-Selberg Law]\label{thmWeylLaw}
Let $h_\infty$ be one of the two test functions given by \eqref{eq:hdefWindowVersion} or \eqref{eq:hdefInitialSegmentVersion}. If $f$ satisfies Geometric Assumption \eqref{geo3} and is a newform projector,  then we have
\begin{equation}\label{eq:thmWeylLaw}
\sum_{\pi \in \cF_0(f)}  \frac{h_\infty(t_\pi)}{\mathcal{L}_\pi^*(1)}  + (cts.) = \vol( \overline{G}(\Q) \backslash \overline{G}(\A)) f_\infty(1) \prod_p \int_{\cF_p(f)} \frac{1}{\mathcal{L}_{\pi_p}(1)} \,d\widehat{\mu}(\pi_p)  + O\left( \frac{f_\A(1)}{T^2 k(\cF)}\right).
\end{equation}
\end{mycoro}

See Section \ref{sec:WeylLaw} for the proofs of Lemma \ref{lemWeylLaw}  and Corollary \ref{thmWeylLaw}. For the holomorphic/discrete series variation alluded to at the beginning of Section \ref{sec:introWeylLaw}, note that equation \eqref{H_infty_near0} used in the proof of Lemma \ref{lemWeylLaw} may no longer hold due the possible presence of weight $2$ holomorphic forms. In this situation, the trivial bound \eqref{Hbound_proof_of_lemWeylLaw} on Kloosterman sums is no longer sufficient to prove the analogue of Lemma \ref{lemWeylLaw}. Instead, we may factor $H(m,n;c)$ by Theorem \ref{thmKP}(3) and use the classical Weil bound on the unramified factor to conclude the analogues of Lemma \ref{lemWeylLaw}  and Corollary \ref{thmWeylLaw}.

\begin{myrema}
We have called Corollary \ref{thmWeylLaw} a Weyl-Selberg Law (following terminology of Venkov \cite{VenkovWeylSelberg}, e.g.) and not a Weyl Law, as the left side of \eqref{eq:thmWeylLaw} includes continuous as well as cuspidal spectrum. Moreover we emphasize that Corollary \ref{thmWeylLaw} is only a \emph{harmonically-weighted} Weyl-Selberg Law, since we have made no attempt to obtain a sharp cut-off in the archimedean aspect and have retained the weight $\mathcal{L}_\pi^*(1)^{-1}$ in the non-archimedean aspect. Despite these nominal caveats, Corollary \ref{thmWeylLaw} is the statement that turns out to be useful elsewhere in this paper. 
We also mention that there is a well-known method for removing the harmonic weights, as in \cite[Section 3]{KowalskiMichel}.  In addition, the continuous spectrum contribution to \eqref{eq:thmWeylLaw} may often be bounded in a straightforward fashion using explicit information on the Eisenstein series.  A particularly simple case occurs if each $\pi \in \mathcal{F}_0(f)$ is supercuspidal at some prime $p$, since then the continuous spectrum is empty.
\end{myrema}

\begin{myrema}
As mentioned in the introduction, a Weyl law for cusp forms with specified local components was obtained by Palm \cite[Thm.\ 3.2.1]{Palm} in his thesis. We also would like to point out the nice recent work of Knightly \cite{Knightly_counting_supercuspidal_newforms}, who obtained, among other results, dimension formulas for spaces of cusp forms with specified supercuspidal local components using a simple trace formula. In a different direction, Kim, Shin and Templier \cite{KimShinTemplier} gave asymptotics for automorphic representations with specified supercuspidal local components in a very general setting. 
\end{myrema}

\begin{myrema}
Corollary \ref{thmWeylLaw} can be interpreted as an instance of a general equidistribution statement for cusp forms. Let $\mathcal{A}_0(G/k)$ be the set of all unitary cuspidal automorphic representations of $G$ over a number field $k$. Drawing inspiration from the work of Brumley and Mili\'cevi\'c \cite[\S 1.1, \S 2]{BrumleyMilicevic}, who studied the universal family $\mathcal{A}_0(\GL_n/k)$ ordered by analytic conductor, one expects that for any sufficiently well-behaved test function $h$ on $\mathcal{A}_0(G/k)$
\begin{equation}\label{eq:equidistribution}
\sum_{\pi \in \mathcal{A}_0(G/k)} h(\pi) \sim \vol( G(k) \backslash G(\A)^1) \int_{\pi \in {G(\A)^1}^{\wedge}} h(\pi) \, d \widehat{\mu}(\pi),
\end{equation}
 as the average analytic conductor of the effective support of $h$ tends to infinity. 
 Indeed, Brumley and Mili\'cevi\'c (Thm.\ 1.2) prove for $G=\GL_n$ over a number field that if $h$ is the indicator function of forms having analytic conductor $\leq Q$ that \eqref{eq:equidistribution} holds as $Q \to \infty$ with an explicit effective savings of $(\log Q)^{-1}$ over the main term. To see this, follow the proof of their Theorem 1.2, but instead of the final sentence of loc.\ cit.\ Lemma 12.1, use the final displayed equation in loc.\ cit.\ Proof of Proposition 6.1 and Corollary 6.2 to express the main term of loc.\ cit.\ (12.2) summed over all $\mathfrak{q}$ and $\mathfrak{d} \mid \mathfrak{q}$ as the adelic Plancherel volume of a conductor ball. 

Corollary \ref{thmWeylLaw} is also an instance of \eqref{eq:equidistribution} in the case that $G = \PGL_2$ and $k=\Q$ and with $h$ the harmonic weights given by $h(\pi) = h_\infty(t_\pi)/\mathcal{L}_\pi^*(1)$.  
\end{myrema}

\subsection{Large sieve inequality}\label{sec:intro:LSI}  
As  remarked at the beginning of Section \ref{sec:introWeylLaw}, here we consider \emph{families} of automorphic representations, which in practice means that certain implied constants should hold uniformly within a given family. To streamline the exposition we focus on the Maass form / spherical at $\infty$ case, but claim that the results of this section go through in the holomorphic / discrete series variation. See Remark \ref{rem:LSI_holomorphic}.

We now propose a framework for optimal large sieve inequalities. Let $\cF$ be a finite set of cuspidal automorphic representations of $\GL_2$ over $\Q$ with trivial central character all having the same (finite) conductor $q=q(\cF)$.  
Suppose that there exists a pure tensor $f \in \cH_{\rm fin}$ and an $h_\infty$ as in the PBK formula such that 
$\cF \subseteq \cF_0(f)$ and with the effective support (see Section \ref{sec:introWeylLaw} for definition) of $h_\infty$ containing the spectral parameters $\{ t_\pi : \pi \in \cF\}$. We will show in Theorem \ref{thm:abstractversion} that $\cF$ satisfies an optimal large sieve inequality if the test function $f$ satisfies the hypotheses introduced next. 

Let  $T= T(\cF)$ be the infimum of the $T\geq 0$ such that the set of spectral parameters $\{t_\pi : \pi \in \cF\}$ is contained in $[-T,T]$. 
\begin{myhypothesis}[Trace formula (TF)]\label{hypTF}
Suppose that $f \in \cH_{\rm fin}$ satisfies the hypotheses of Theorem \ref{theoGeomSpec}. 
\end{myhypothesis}
We assume that Hypothesis \ref{hypTF} (TF) holds for the remainder of this section. 
The next hypothesis encodes the assumption that $\cF_0(f)$ is not too much larger than $\cF$. 
\begin{myhypothesis}[$\cF_0(f)$ not much larger than $\cF$ (NmL)]\label{hypNmL}
We suppose that 
\begin{itemize}
\item $f \in \cH_{\rm fin}$  has  $N \mid q^{\infty}$ and  $\cF \subseteq \cF_0(f)$,
\item if $T\ll q^\eps$ for all $\eps>0$, then $h_\infty$ is given by  \eqref{eq:hdefInitialSegmentVersion}  with the spectral parameters of $\cF$  in the effective support of $h_\infty$, 
\item if there exists uniform $C, \delta>0$ such that $T> C q^\delta$, then $h_\infty$ is given by \eqref{eq:hdefWindowVersion} with $T^\eps \ll \Delta \ll T^{1-\eps}$ and the spectral parameters of $\cF$ in the effective support of $h_\infty$,
\end{itemize} 
and $f$, $h_\infty$ are such that
\begin{equation}\label{not_too_much_larger}
\sum_{\pi \in \cF_0(f)} h_\infty(t_\pi) w(\pi, f) + (\text{ cts. })= |\cF| (qT)^{o(1)}
\end{equation}
where the weights $w(\pi, f)$ are as in Theorem \ref{theoGeomSpec}.
\end{myhypothesis}
 In practice, the existence of appropriate $h_\infty$ satisfing Hypothesis \eqref{hypNmL} (NmL) can often be achieved after partitioning $\cF$ into subsets according to the sizes of the archimedean spectral parameters $t_\pi$ of $\pi \in \cF$ relative to $q$. See the example at the end of this section.
 
Next, we need a hypothesis asserting some control on the generalized Kloosterman sums $H(m,n;c)$ of $f$. In fact, we do not need a bound on $H(m,n;c)$ itself, but only on its Fourier/Mellin transform for ramified moduli. Note that $c\in \Z$ for any non-vanishing $H(m,n;c)$  by Hypothesis \ref{hypTF} (TF), see Lemma \ref{specControlonH}(4).  For a Dirichlet character $\chi \pmod{c}$, let
 \begin{equation}
 \label{eq:KloostermanFourierTransform}
  \widehat{H}(\chi) = \frac{1}{\varphi(c)} \sumstar_{y \shortmod{c}} H(y,1;c) \overline{\chi}(y),
  \end{equation}
  so that Fourier inversion gives
  \begin{equation}
  \label{eq:FourierInversion}
  H(y,1;c) = \sum_{\chi \shortmod{c}} \widehat{H}(\chi)  \chi(y).
  \end{equation}
   \begin{myhypothesis}[Fourier transform bound (FTB)]\label{hypFTB}
Suppose that for any $c \mid N^\infty$ and $\chi \pmod{c}$ we have
 \begin{equation}\label{eq:hypFTB}
  \|\widehat{H}\|_{\infty} := \max_{\chi \shortmod{c}} |\widehat{H}(\chi)| \ll f(1) c^{\varepsilon}
 \end{equation}  uniformly in $f$ and  for all $\eps>0$. 
\end{myhypothesis}
Hypothesis FTB reduces to checking local statements at each $p\mid N$. Indeed, suppose $\chi$ is a Dirichlet character modulo $c$ with factorization $\chi = \prod_{p \mid c} \chi_p$ and for each $p \mid c$ we write $c=c_0p^{v_p(c)}$. Then, by \eqref{Hfactorization}, Lemma \ref{modulusshift} and \eqref{Hpm,n=Hpmn} we have $$\widehat{H}(\chi) = \prod_{p \mid c}\overline{\chi}(c_0)^2 \widehat{H}_p(\chi_p,v_p(c)),$$
where for $\alpha$ a Dirichlet character with $p$-power conductor (equivalently, a character of $\Z_p^\times$) and $k\geq 0$ we have set
\begin{equation}\label{intro:HhatLocalDef}
\widehat{H}_p(\alpha,k) = \frac{1}{\varphi(p^k)} \sumstar_{y \shortmod{p^k}} H_p(y,1;p^k) \overline{\alpha(y)}= \zeta_p(1)\int_{\Z_p^\times} H_p(y,1;p^k) \overline{\alpha(y)} \,dy
\end{equation}
with $dy$ the additive Haar measure that gives $\Z_p$ volume 1. Thus, to verify Hypothesis \ref{hypFTB} (FTB), it suffices to show that for $p \mid N$ and all $\alpha$ with $p$-power conductor and $k\geq 0$ that \begin{equation}\label{hyp_FTB_local_version}
\widehat{H}_p(\alpha,k) \ll f_p(1),
\end{equation}
with implicit constants independent of $p, \alpha, k, f_p$, but possibly depending on the family in which $f$ varies. 

Note also that if Hypothesis \ref{hypFTB} (FTB) holds (with $c|N^{\infty}$), then the bound \eqref{eq:hypFTB} holds for any character $\chi$ of any modulus $c$ since at primes away from $N$ the generalized Kloosterman sum $H(m,n;c)$ reduces to the classical Kloosterman sum, and we easily derive the required bounds.

Finally, we state our last hypothesis.  
\begin{myhypothesis}[Conductor versus size of family (CvF)]\label{hypCvF}
 We suppose that
 \begin{equation}
  k(\cF) \gg f(1)^{1-\eps}
 \end{equation}
uniformly in $f$ and  for all $\eps>0$. 
\end{myhypothesis} 
Again, note that to verify Hypothesis \ref{hypCvF} (CvF), it suffices (using \eqref{eq:N_vs_f1}) to show for $p \mid N$ that \begin{equation}\label{hypCvF_local_version}
p^{k_p} \gg f_p(1)
\end{equation}
with implicit constants independent of $p, f_p$, but possibly depending on the family in which $f$ varies.  

Here is an example application of Hypothesis \ref{hypCvF} (CvF), which is moreover used in the proof of the following theorem. 
\begin{mylemma}\label{fA(1)_smaller_than_harmonic_family}
Let $h_\infty$ be one of the two test functions given by \eqref{eq:hdefWindowVersion} or \eqref{eq:hdefInitialSegmentVersion}. If $f \in \cH_{\rm fin}$ satisfies Hypotheses \ref{hypTF} (TF) and \ref{hypCvF} (CvF), and $w(\pi,f)$ are as in Theorem \ref{theoGeomSpec}, then 
\begin{equation}\label{eq:fA(1)_smaller_than_harmonic_family}
f_\A(1) \ll_\eps f(1)^\eps\Big(\sum_{\pi \in \cF_0(f)} h_\infty(t_\pi) w(\pi, f) + (\text{ cts. }) \Big).
\end{equation}
\end{mylemma}
\begin{proof}
By Lemma \ref{lemWeylLaw} with $m_1=m_2=1$, \eqref{eq:delta_vs_f1}, and the definition $\delta_\infty = f_\infty(1)$, we have 
$$ f_\A(1)\Big( f(1)^{o(1)} + O \Big(\frac{1}{T^2 k(\cF)}\Big)\Big) =  \sum_{\pi \in \cF_0(f)} h_\infty(t_\pi)w(\pi, f) + ({\rm cts.}).$$ 
By Hypothesis CvF, we have that the sum in parentheses on the left is non-vanishing and $\gg f(1)^{-\eps}$ for $f_\A(1)$ sufficiently large.
\end{proof}

Recall we write $\lambda_\pi(n)$ for the $n$th Hecke eigenvalue of $\pi$, normalized so that the Ramanujan conjecture predicts that $|\lambda_\pi(n)|\leq d(n)$ for all $n \in \N$.

\begin{mytheo}[Optimal Large Sieve Inequality]\label{thm:abstractversion}
Suppose that $\cF$ is a finite set of trivial central character automorphic representations for $\GL_2$ over $\Q$, all with (finite) conductor $q$ and spectral parameters contained in $[-T,T]$. 
Suppose that there exists a pure tensor $f \in \cH_{\rm fin}$ such that hypotheses TF, NmL, FTB and CvF hold for $f,\cF$. Then for any sequence of complex numbers $(a_n)_{n \in \N}$ we have
\begin{equation}\label{eq:LSIgoal} 
 \sum_{\pi \in \cF} \Big| \sum_{n \leq X} a_n \lambda_\pi(n)\Big|^2 \ll_\eps (|\cF|+ X) (XqT)^\eps
\sum_{n \leq X} |a_n|^2
. 
\end{equation}
\end{mytheo}
\begin{myrema}\label{rem:LSI_holomorphic}
Theorem \ref{thm:abstractversion} also holds in the holomorphic / discrete series variation if we interpret the spectral parameter of a weight $\kappa$ holomorphic form to be $\frac{\kappa-1}{2}$ and make the natural modifications to Hypothesis \eqref{hypNmL} (NmL). Indeed following e.g.\ \cite[p.\ 84-86]{IwaniecTopics}, we may consider an average of weights $\kappa$ either in a wide window of width $T$ near the bottom of the spectrum cf.\ \eqref{eq:hdefInitialSegmentVersion}, or in a narrow window of width $\Delta$ at height $T$ in the spectrum cf.\ \eqref{eq:hdefWindowVersion}. It is by now well-known that the function $H_\infty(x)$ resulting from a smooth sum as above of several copies of the Petersson formula \eqref{theoGeomSpec_holomorphic}  enjoys substantially the same properties as the $H_\infty(x)$ appearing in the Maass form case. For instance, see \cite[\S 7]{YoungHybrid} for details and in particular an analogue of Lemma \ref{lemma:archimedeanWindowVersion}.
\end{myrema}

\begin{myrema}Hypotheses TF, FTB and CvF hold for the test functions $f_{\leq c}$, $f_\chi$, $f_\xi$, and $f_{\xi,n}$ presented in the `Examples' of Section \ref{intro:statement_of_results}, for which see Sections \ref{OLSI:classical}, \ref{OLSI:PS}, \ref{hypfrom_intro:SC} and \ref{sec:neighborhood_of_a_SC}. On the other hand, Hypothesis \ref{hypCvF} (CvF) \emph{fails} for the test function $f_{=c}$ in (horizontal) $p$-aspect. Indeed, for $f_p=f_{=c}$, one has $p^{k_p}=p^{c-1}$ while $f_{=c}(1) = 
p^{c}(1+O(p^{-1}))
$.

The features $\cF \subsetneq \cF_0(f)$ and Hypothesis \eqref{hypNmL} (NmL) of our framework for Large Sieve Inequalities serve to patch up the above issue with the test function $f_{=c}$, as explained in the forthcoming example.  
In addition, these conditions are used at the archimedean place, 
since we want $\cF$ to be finite, but only have access to holomorphic spectral weight functions $h_\infty$, which in particular cannot have compact support. \end{myrema}
\begin{myexam}
The classical Spectral Large Sieve Inequality is a special case of Theorem \ref{thm:abstractversion}. Indeed, set $$\mathcal{S}_{p^c,T} = \{\pi \in \mathcal{A}_{0}(\PGL_2/\Q) : c(\pi_p) = c \text{ and } |t_\pi| \leq T\}$$ with $T^2p^c\to \infty$.

To see this, first we partition the space $S:=[0,T] \cup (0,i/2)$ of potential spectral parameters into the disjoint union of $S_0:=([0,q^\eps] \cup (0,i/2))\cap S$ and its complement $(q^\eps,T]\cap S$. We further $(1+T^{-\eps})$-adically partition $(q^\eps,T]\cap S$, i.e. decompose it as a disjoint union $$(q^\eps,T] \cap S = \bigcup_{i>0}  [\frac{T}{(1+T^{-\eps})^{i-1}}, \frac{T}{(1+T^{-\eps})^i})\cap S  =: S_1 \cup S_2 \cup \ldots$$ with $S_i$ empty for $i$ sufficiently large. Let $\mathcal{S}^{(i)}_{p^c,T} := \{\pi \in \mathcal{S}_{p^c,T}: t_\pi \in S_i\}$, so that $\mathcal{S}_{p^c,T} = \cup_{i\geq 0} \mathcal{S}^{(i)}_{p^c,T}$ is a partition of $\mathcal{S}_{p^c,T}$ into $\ll T^{\eps}$ non-empty subsets. Thus, to prove an optimal LSI for $\mathcal{S}_{p^c,T}$ it suffices to do so for each $\mathcal{S}^{(i)}_{p^c,T}$.  For $i=0$ we choose $h_\infty$ to be the test function in \eqref{eq:hdefInitialSegmentVersion}  of width $\min(T,q^\eps)$, while for $i>0$ we choose $h_\infty$ to be the test function in \eqref{eq:hdefWindowVersion} with $\Delta =T^{1-\eps}$.

We take $f$ equal to $f_{\leq c}$ at $p$ and unramified elsewhere. Then $f$ satisfies Hypotheses TF, FTB and CvF (since this choice of $f$ satisfies $k_p = c$ in \eqref{kp:3egs}). 

We check Hypothesis \ref{hypNmL} (NmL). Since $N=p^c$, we have $N \mid p^\infty$ and $\mathcal{S}_{p^c,T}^{(i)}\subset \cF_0(f)$. In either case $i=0$ or $i>0$ the hypothesis on $h_\infty$ in the second or third bullet points of Hypothesis \ref{hypNmL} (NmL) hold by construction. The last statement \eqref{not_too_much_larger} of Hypothesis \ref{hypNmL} (NmL) is given by Lemma \ref{lemWeylLaw}. The Optimal Large Sieve Inequality \eqref{eq:LSIgoal} then holds for $\cF = \mathcal{S}_{p^c,T}^{(i)}$ by Theorem \ref{thm:abstractversion}, and for $\mathcal{S}_{p^c,T}$ by trivial summation over $i\geq 0$. 
\end{myexam}

\subsection{Moments of $L$-functions} 
Let  $\sigma$ be a supercuspidal representation of $\GL_2(\Q_p)$ with trivial central character. Let $\mathcal{S}_\sigma$ be the family of automorphic representations
\begin{equation}\label{Ssigma_def}
\mathcal{S}_\sigma := \{ \pi \in \mathcal{A}_{0}(\PGL_2/\Q) : \pi_p \simeq \sigma, \, \pi_\infty \text{ spherical with } |t_\pi| \leq 1000\}.
\end{equation}
Note that we have $\#\mathcal{S}_{\sigma} \asymp p^{\lceil c(\sigma)/2 \rceil}$ by Corollary \ref{thmWeylLaw}, \eqref{diagwt_egs} and \eqref{c(pi)c0}.
It is well-known that a large sieve inequality may be used to estimate certain moments of $L$-functions; see \cite[Section 7.9]{IK} for the method.
As a simple application of Theorem \ref{thm:abstractversion}, we have the following
Lindel\"of-on-average upper bound. 
\begin{mycoro} 
\label{coro:2ndmoment}
Let $\sigma$ be a supercuspidal representation of $\GL_2(\Q_p)$ with trivial central character. For all $\eps>0$  we have  
\begin{equation}
\sum_{\pi \in \mathcal{S}_\sigma} |L(1/2,\pi)|^2 \ll_\eps (p^{\lceil c(\sigma)/2\rceil})^{1+\eps}.
\end{equation}
\end{mycoro}
Corollary \ref{coro:2ndmoment} also holds if in \eqref{Ssigma_def} we replace the word ``spherical'' with ``discrete series'' and take the spectral parameter $t_\pi$ of a weight $\kappa$ discrete series to be $\frac{\kappa-1}{2}$ cf. Remark \ref{rem:LSI_holomorphic}. 

Let $\chi$ be a character of $\Q_p^\times$ whose restriction to $\Z_p^\times$ is not quadratic. Theorem \ref{thm:abstractversion} also gives a Lindel\"of-on-average upper bound for the 2nd moment of central values of $L$-functions over the family 
\begin{equation}\label{Schi_def}
\mathcal{S}_\chi := \{ \pi \in \mathcal{A}_{0}(\PGL_2/\Q) : \pi_p \simeq \pi(\chi \alpha^{i \theta},\chi^{-1}\alpha^{-i \theta}) \text{ for some } \theta \in \R \text{ and } |t_\pi| \leq 1000\}.
\end{equation}
However, such a second moment estimate already follows easily from previous cubic moment estimates \cite[Thm.\ 1.2]{PetrowYoungCoset} by H\"older's inequality. 

Since Corollary \ref{coro:2ndmoment} follows from a large sieve inequality, it cannot give a subconvex bound by general principles.  However, when $c(\sigma)$ is even, dropping all but one term recovers the convexity bound. In a forthcoming work, we intend to give a Lindel\"of-on-average bound for the cubic moment of central values of $L$-functions over $S_{\sigma}$ and similar families, which will recover strong subconvex bounds for these $L$-functions. 

\subsection{Acknowledgements} We thank Matteo Di Scipio for pointing out several minor mistakes in an earlier version of this paper, and especially for providing the justification for exchanging the summation and integration in Section \ref{sec:secondcellterms} of this paper. We thank Han Wu for a discussion surrounding the adelic pre-trace formula. Finally, we thank the referees for their thorough reading of the paper and their helpful comments. 

Part of this work was carried out on a visit of the 2nd and 3rd authors to the R\'enyi Institute / Erd\H os Center, Budapest. We thank the R\'enyi Institute and the organizers of the conference ``Automorphic Forms in Budapest 2022'' (Gergely Z\'abr\'adi,
P\'eter Maga, 
Gergely Harcos, 
\"Ozlem Imamo\u{g}lu, and 
\'Arp\'ad T\'oth) 
for their support and a pleasant and productive working conditions.

\subsection{Notation, conventions, and measure normalizations}
\subsubsection{Fields}
Sections \ref{sec:Unrefined_KTF} to \ref{sec:spec_assumption} are focused on the relative trace formula set-up for the PBK formula over the rationals $\Q$. Accordingly, in these sections we write $\Q_p$ for the field of $p$-adic numbers with ring of integers $\Z_p$ and absolute value $|\cdot|_p$. Let $\A$ and $\A_{\rm fin}$ denote the adeles and finite adeles of $\Q$, and $\widehat{\Z} = \prod_p \Z_p $ the maximal compact open subgroup of $\A_\fin$. 

On the other hand, the setting of Section \ref{sec:testfunctions} is that of general non-archimedean local fields (even though in Sections \ref{sec:egSupercuspidal_odd} through \ref{sec:intermediatefamily} we restrict to the case that the base field is $\Q_p$). Here, we write $F$ for a non-archimedean local field with ring of integers $\cO$ and absolute value $|\cdot|_F$. We recall the rest of the notation for non-archimedean local fields in Section \ref{sec:tesetfunctions_basic}. 

In any section of the paper, we write $\alpha$ for the quasicharacter of $F^\times$ defined by $\alpha: x \mapsto |x|_F$. 

\subsubsection{Additive characters}
Outside of Section \ref{sec:testfunctions}, we take $\psi$ to be the standard additive character $\psi : \A/ \Q \to \C^\times$, that is, $\psi = \prod_v \psi_v$, where \begin{equation}\label{psidef}\psi_v (x) = \begin{cases} e(-x_\infty) & \text{ if } v = \infty \\ e(\{x_p\}_p) & \text{ if } v=p \end{cases} \qquad (x \in \A),\end{equation} where $\{ \cdot \}_p : \Q_p \to \Q$ is the fractional part function.  

At the outset of Section \ref{sec:testfunctions}, $\psi$ is an arbitrary additive character of the non-archimedean local field $F$. We say $\psi \neq 1$ has conductor $c(\psi) = n$ if $\fp^n$ is the largest fractional ideal of $F$ on which $\psi$ is trivial.  In Section \ref{compactinductionbackground} \emph{only} we take $\psi$ to have conductor $1$ to match a convention in the compact induction theory of Bushnell-Henniart-Kutzko. On the other hand, from Remark \ref{remarkafterQ2lem} until the end of Section \ref{sec:testfunctions}, we assume that $\psi$ has conductor 0 (e.g.\ the one in \eqref{psidef}). 

If $E/F$ is a field extension, we denote by $\psi_E$ the additive character $\psi \circ \Tr_{E/F}$ of $E$. 

\subsubsection{Groups and subgroups}\label{notation-groupsandsubgroups}
Let $G$ be the algebraic group $G= \GL_2$, $Z$ be the subgroup of diagonal matrices of $G$, and $\overline{G}  = Z \backslash G= \PGL_2$. 

Let $N\subset B \subset G$ be the standard upper-triangular unipotent and Borel algebraic subgroups of $G$. Let $A$ be the subgroup of matrices of the form $a(y):= \left(\begin{smallmatrix} y & 0 \\ 0 & 1 \end{smallmatrix}\right)$ for $y$ in any commutative ring $R$. We have $B= ZAN=ZNA$. For any $x,t\in R$ let  $$n(x) = \left( \begin{smallmatrix} 1 & x\\ 0 & 1 \end{smallmatrix} \right) \quad \text{ and } \quad z(t)= \left( \begin{smallmatrix} t & 0\\ 0 & t \end{smallmatrix} \right).$$ 

Let $K_p=G(\Z_p)$ be the standard maximal compact subgroup of $G(\Q_p)$, and $K_\infty = \SO_2(\R)$. We write $Z = Z(\Q_p)$ when the prime $p$ is clear from context, e.g.\ $ZK_p$ denotes $Z(\Q_p)G(\Z_p)$. 

Let $K= \prod_{p} K_p = G(\widehat{\Z})$.  We also use the subgroups $K(N)\subseteq K_d(N)$,  $K_1(N)\subseteq K_0(N)$ of $K$ given by 
$$K(N)= \{\left(\begin{smallmatrix} a & b \\ c & d \end{smallmatrix} \right) \in K : a \equiv d \equiv 1 \shortmod{N}, b \equiv c \equiv 0 \shortmod{N}\},$$
$$K_d(N)= \{\left(\begin{smallmatrix} a & b \\ c & d \end{smallmatrix} \right) \in K : b \equiv c \equiv 0 \shortmod{N}\},$$
$$K_1(N)= \{\left(\begin{smallmatrix} a & b \\ c & d \end{smallmatrix} \right) \in K :  d \equiv 1 \shortmod{N}, c \equiv 0 \shortmod{N}\},$$
$$K_0(N)= \{\left(\begin{smallmatrix} a & b \\ c & d \end{smallmatrix} \right) \in K : c \equiv 0 \shortmod{N}\}.$$
We use the same notation for the corresponding subgroups of $K_p$. For $* = \varnothing, d, 1,$ or $0$, we set as usual $\Gamma_*(N) = K_*(N) \cap \SL_2(\Z)$. 

For $m+n\geq 0$ let us define $K_0(n,m)\subset G(\Q_p)$ to be the compact open subgroup 
\begin{equation}\label{K0mn_def} K_0(m,n) = \begin{cases} \left\{ \zxz{\Z_p^\times}{(p^m)}{(p^n)}{\Z_p^\times}\right\} & \text{ if } m+n>0, \\ a(p^{-m})K_pa(p^{m}) & \text{ if } m+n=0.\end{cases}\end{equation}

For an algebraic group $H$ over $\Q$, let $[H]$ denote the adelic quotient $[H]:= H(\Q) \backslash H(\A)$.  

\subsubsection{Measure normalizations}\label{measurenormalizations}
We choose $dx$ to be Lebesgue measure on $\R$ and $d^\times x=dx/|x|$ on $\R^\times$. For $F$ a non-archimedean local field, we take $dx$ to be the Haar measure on $F$ that gives the maximal compact subgroup $\cO$ measure 1. We set the Haar measure $d^\times x$ on $F^\times$ to be given by $d^\times x = \zeta_F(1)dx/|x|_F$.  Here $\zeta_F(1)=\zeta_\fp(1)=(1-\Nm \fp^{-1})^{-1}.$ 

We let $dk$ be the Haar probability measure on $K_\infty$. Take the measures on $Z(\R)$, $A(\R)$ and $N(\R)$ induced by $dx$ and $d^\times x$. These together determine a Haar measure on $G(\R)$ by the Iwasawa decomposition. 
Let $dg$ be the Haar measure on $G(\Q_p)$ that gives $\vol(K_p)=1$. 

For $H$ one of the algebraic groups in Section \ref{notation-groupsandsubgroups}, we give $H(\A)$ and $H(\A_\fin)$ the associated product measures. We give $\overline{G}(\A)$ and $ \overline{G}(\A_\fin)$ the quotient measure. With these choices we have $\vol(\Q \backslash \A) = 1$ and $\vol([\overline{G}]) =2\xi(2)= \pi/3$.

Each cuspidal automorphic representation $\pi$ (resp.\ global principal series $\pi_{\chi,\chi^{-1}}$ in the induced model) is endowed with the inner product
\begin{equation}\label{aut_inner_prod}\langle \varphi_1, \varphi_2\rangle = \int_{[\overline{G}]} \varphi_1(g) \overline{\varphi_2(g)}\,dg \quad \left( \text{resp.\ } \langle \phi_1, \phi_2\rangle =\int_{K_\infty \times K} \phi_1(k) \overline{\phi_2(k)}\,dk \right).\end{equation}

If $H$ is a unimodular $p$-adic linear algebraic group and $\mu$ is a Haar measure on $H$, then there exists a unique $\sigma$-finite measure $\widehat{\mu}$ called the \emph{Plancherel measure} on the unitary dual $H^\wedge$ such that the \emph{Plancherel formula} \eqref{plancherelformula} holds. In particular, for any locally constant compactly supported function $f$ on $H$, one has
\begin{equation}\label{Plancherel}
f(1) = \int_{\pi \in H^\wedge} \Tr \pi(f)\, d\widehat{\mu}(\pi), 
\end{equation}
which we also refer to as the Plancherel formula. For more details, see Section \ref{sec:localspectral}.

\subsubsection{Test functions and Hecke algebras}
Write $\cH_{\rm fin} =C_c^\infty (\overline{G}(\A_{\rm fin}))$ for the \emph{non-archimedean Hecke algebra} of $\overline{G}= \PGL_2$, that is the space of locally constant functions on $G(\A_{\rm fin})$ that are invariant by and compactly supported modulo center the $Z(\A_{\rm fin})$. Define the \emph{local Hecke algebra} $\cH_p= C_c^\infty (\overline{G}(\Q_p))$ similarly. 

 Throughout this paper (with the exception of in Section \ref{sec:pre-trace}) we will always assume the $f \in \cH_{\rm fin}$ that we use as test functions are non-zero pure tensors, i.e.\ that $f$ admits a representative $\prod_p f_p$ with $f_p \in \cH_p :=C_c^\infty( \overline{G}(\Q_p))$ for each $p < \infty$, which we may moreover assume satisfy $f_p(1) =1$ for all but finitely many $p$.
Note, if $f = \bigotimes_p f_p \in \cH_{\rm fin}$ is a pure tensor, then $f_p$ is a constant multiple of the indicator function $1_{ZK_p}$ for all but finitely many $p$ (see e.g.\ \cite[\S 1.3]{Cartier}).  
We say that such an $f$ is \emph{ramified} at  $p$ if $f_p$ is not a constant multiple of $1_{ZK_p}$. 

Let $N \in \N$ be minimal such that $f \in \cH_{\rm fin}$ is bi-$K(N)$-invariant.  We call $N$ the \emph{level} of $f$ and define similarly the \emph{local level} $N_p$ of $f_p\in \cH_p$. If $f= \bigotimes_pf_p$, then $N_p=p^{v_p(N)}$. 

If $f \in \cH_\fin$ and $(\pi,V)$ with $\pi =  \pi_\infty \otimes \pi_\fin$ and $V = V_\infty \otimes V_{\fin}$ is an irreducible admissible representation of $\overline{G}(\A)$, then define $\pi(f) \in \End(V_{\fin})$ to be given by \begin{equation}\label{pif_def_notn_section} \pi(f): v \mapsto \int_{\overline{G}(\A_{\rm fin})} f(g) \pi_{\rm fin}(g)v\,dg.\end{equation} If $f \in \cH_p$ and $(\pi,V) $ is an irreducible admissible representation of $\overline{G}(\Q_p)$, then define similarly $\pi(f)\in \End(V)$ by \eqref{pif_def_notn_section} with $\Q_p$ in place of $\A_\fin$ and $\pi$ in place of $\pi_\fin$. In Section \ref{sec:localspectral}, where $H$ is a unimodular $p$-adic linear algebraic group with a Haar measure $dg$, we define $\pi(f) \in \End(V)$ for $f \in L^1(H)$ and $(\pi,V)$ an irreducible admissible representation of $H$ by \eqref{pif_def_notn_section} with $H$ in place of $\overline{G}(\A_\fin)$ and $\pi$ in place of $\pi_\fin$.

\subsubsection{Principal series representations}
We use the notations $\pi(\mu_1,\mu_2)$ and $\pi_{\mu_1,\mu_2}$ interchangably for the (normalized) principal series representation induced from the (local or global) characters $\mu_1,\mu_2$. See e.g.\ \cite[\S 3.7, 4.5]{Bump} for definitions. 

\subsubsection{Newform in the Kirillov model}
Let $\psi$ be a non-trivial additive character 
of a non-archimedean local field $F$.  
The following well-known result will be used repeatedly.  \begin{mylemma}\label{NewformKirillov}
Suppose $\pi$ is a smooth irreducible unitary generic representation of $G(F)$ with trivial central character and conductor exponent $c(\pi)\geq 2$. Then a newform $f$ in the Kirillov model $\mathcal{K}(\pi,\psi)$ of $\pi$ is given by $f(x) = 1_{\cO^\times}(x)= 1_{1,0}(x)$ for $x \in F^\times$.
\end{mylemma}
\begin{proof} See e.g.\ \cite[\S 2.4]{Schmidt:02a}.\end{proof}

\subsubsection{Miscellaneous}
Let $\nu(n) = [\SL_2(\Z):\Gamma_0(n)] = n \prod_{p \mid n} (1+p^{-1})$. In this paper we take $\N= \{1, 2, 3, \ldots \}$.

\section{The unrefined relative trace formula}\label{sec:Unrefined_KTF}
The purpose of this section is to prove the following relative  trace formula (cf.\ \eqref{intro:GeneralizedKuznetsov1}) under minimal hypotheses.  \begin{mytheo}[Unrefined generalized BK formula]\label{MT}
 Suppose $f= \bigotimes_p f_p \in \cH_{\rm fin}$ is non-zero and that for each $p$ that $f_p$ is 
 supported inside the subgroup of matrices $g \in G(\Q_p)$ with $v_p(\det g) \in 2\Z$. 
 
For $m_1,m_2 \in \frac{1}{N}\Z$ with $m_1m_2>0$ we have that
\begin{multline}\label{GeneralizedKuznetsov1}\sum_{ \pi \in \cF_0(f)} h_\infty(t_\pi)  \sum_{\varphi \in \mathcal{B}(\pi)}a_{u_{\pi(f)\varphi}}(m_1) \overline{a_{u_\varphi}(m_2) } \\ + \frac{1}{4\pi}\sum_{ \chi \in \cF_E(f)} \sum_{\phi \in \mathcal{B}(\chi, \chi^{-1})} \int_{-\infty}^\infty h_\infty(t) a_{u_{E(\pi_{it}(f)\phi_{it})}}(m_1) \overline{a_{u_{E(\phi_{it})}}(m_2)}\,dt  \\
=  f_\infty(1) \sqrt{\frac{m_2}{m_1}} \int_{\A_{\rm fin}} f\left( \left( \begin{smallmatrix}  m_2/m_1   & t \\ & 1\end{smallmatrix} \right) \right) \psi_{\rm fin}(-m_{1}t)\,dt 
 + \sum_{c \in \cC(\cF)  } \frac{H(m_1,m_2;c)}{c}H_\infty\left(\frac{4 \pi \sqrt{ m_1m_2}}{c}\right),
\end{multline}
as absolutely convergent sums/integrals.
Here:
\begin{itemize}
\item $N$ is the level of $f$,
\item $\mathcal{F}_{E}(f)$ is the Eisenstein series analogue of $\cF_0(f)$; for definition see \eqref{def:cFEf},
\item $\mathcal{B}(\pi)$ (resp.\ $\mathcal{B}(\chi,\chi^{-1})$) is an orthonormal basis consisting of pure tensors for $V_\pi^{K_\infty \times K(N)}$ (resp.\ $V_{\chi,\chi^{-1}}^{K_\infty \times K(N)}$),
\item $u_\varphi$ (resp.\ $u_{E(\phi_{it})}$) defined by $u_\varphi(x+iy) = \varphi( \left( \begin{smallmatrix} y & x \\ & 1\end{smallmatrix}\right) \times 1_{\rm fin})$ is the classical $\Gamma(N)$-Maass form (resp.\ Eisenstein series) corresponding to $\varphi \in \mathcal{B}(\pi)$ (resp.\ $E(\phi_{it})$ for $\phi \in \mathcal{B}(\chi ,\chi^{-1})$),
\item $a_{u_\varphi}(m_i)$ (resp.\ $a_{u_{E(\phi_{it})}}(m_i)$) are the Fourier coefficients of $u_\varphi$ (resp.\ $u_{E(\phi_{it})}$)as defined in Section \ref{sec:FE}, especially \eqref{k=0FCclassical}, 
\item $h_\infty(t) \in PW(\C)$, where $PW(\C)$ is the Paley-Wiener space defined in \cite[\S 3.3]{knightly_kuznetsovs_2013}, 
\item $H_\infty(x)$ is the transform of $h_\infty(t)$ as in \eqref{eq:htoH}, and
\item the $H(m,n;c)$ 
are generalized Kloosterman sums defined in \eqref{GenKloostermanDef}.
 \end{itemize}
\end{mytheo}
\begin{myrema} Theorem \ref{MT} should also extend to the opposite-sign case in which $m_1m_2<0$ with the only modification being that the archimedean factor $H_\infty(x)$ on the geometric side of the formula is replaced with $H_\infty^-(x)$ as defined in \eqref{Hinfty-}. 
Note that the operator $\pi(f)$, being non-archimedean, does not affect the parity of $\varphi$ and  in the opposite-sign case that the diagonal term always vanishes. For the holomorphic forms variation of Theorem \ref{MT}, see Section \ref{sec:holomorphicunrefined}. \end{myrema}

\subsection{Pre-trace formula}\label{sec:pre-trace}
The starting point for Theorem \ref{MT} is an adelic pre-trace formula. While such formulas have appeared in the literature for a long time, we state a recent version of this formula with particularly convenient hypotheses due to Luo, Pi and Wu \cite[Thm.\ 2.2]{LuoPiWu}, which is the special case $F=\Q$ of their more general results. The greater generality of Luo-Pi-Wu's formula is not necessary for the results of this paper, for which Knightly-Li's version would already suffice.
We do not assume that any adelic test function is a pure tensor in this subsection unless explicitly stated otherwise.  

Following \cite[\S 7.1.2]{Wallach1} we define the space of rapidly decreasing functions on $ \overline{G}(\R)$ by
$$\mathcal{S}(\overline{G}(\R)) = \{ f \in C^\infty(\overline{G}(\R)) :  \sup_{g\in \overline{G}(\R)} \|g\|^r|L(X)R(Y)f(g)|< \infty \text{ for all } X,Y \in \mathcal{U}(\fg),\, r \in \Z_{\geq 0}\},$$
 where $\| \cdot \|$ is the norm on $\overline{G}(\R)$ defined in \cite[\S 2.A.2.1]{Wallach1} and $\mathcal{U}(\fg)$ is the universal enveloping algebra of the complexified Lie algebra of $\overline{G}(\R)$ and $L$ and $R$ are the left and right translations. 
 Let $\cH_{\rm fin}$ be the space of locally constant and compactly supported functions on $\overline{G}(\A_{\rm fin})$. 

One defines Schwartz space on $\overline{G}(\A)$ as $$\mathcal{S}(\overline{G}(\A)):=\mathcal{S}(\overline{G}(\R)) \otimes \cH_{\rm fin}.$$
Given $f \in \mathcal{S}(\overline{G}(\A))$ and a cuspidal automorphic representation $(\pi,V)$ of $G$, we denote by $\mathcal{B}(\pi)$ any orthonormal basis of $V$ consisting of $K_{\infty}$-isotypic pure tensors that respect the orthogonal direct sum $V^{K(N)} \oplus (V^{K(N)})^\perp$.  
Similarly, if $\chi_1,\chi_2$ are two Hecke characters, then we denote by $\mathcal{B}(\chi_1,\chi_2)$ any orthonormal basis of the global principal series representations $(\pi_{\chi_1,\chi_2}, V_{\chi_1,\chi_2})$ consisting of $K_{\infty}$-isotypic vectors that respect the orthogonal direct sum $V^{K(N)} \oplus (V^{K(N)})^\perp$.

If $\chi_1,\chi_2$ are finite-order, then we have a Hilbert space isomorphism $V_{\chi_1,\chi_2} \to V_{\chi_1 |\cdot|^s, \chi_2 |\cdot|^{-s}}$ 
 for $s \in \C$ given by $\phi \mapsto \phi_s$, where $\phi_s$ is defined by $\phi_s(g) = |a/d|^s \phi(g)$ and where $g = \left( \begin{smallmatrix} a &  \\  & d\end{smallmatrix}\right)\left( \begin{smallmatrix} 1 & x  \\  & 1\end{smallmatrix}\right)k \in G(\A)$. Similarly, we introduce the shorthand notation $\pi_s:= \pi_{\chi_1 |\cdot|^s, \chi_2 |\cdot|^{-s}}$ when the finite-order characters are clear from context. Lastly, for $\phi \in V_{\chi_1,\chi_2}$ and $g \in G(\A)$ we define the Eisenstein series  $E(\phi_s, g)$ for $\real(s)>1/2$ by
\begin{equation}\label{EisSeriesDef} E(\phi_s, g) = \sum_{\gamma \in B(\Q) \backslash G(\Q) } \phi_s(\gamma g),
\end{equation}
 and for $s \in \mc$ by meromorphic continuation.

If $\phi \in \mathcal{B}(\chi_1,\chi_2)$ is as above and $\real(s)=0$, then it follows that $\| E(\phi_s, \cdot)\|_{\rm Eis} = 1$, where $\| \cdot \|_{\rm Eis}$ is the norm defined in \cite[\S 2.2.1]{MichelVenkateshGL2}, unless $\chi_1=\chi_2$ is quadratic and $s=0$.  Indeed,  for such $\phi_s$ we have  $$\| E(\phi_s, \cdot)\|_{\rm Eis}^2 := \int_{K_{\infty} \times K} |\phi_s(k)|^2\,dk = \int_{K_{\infty} \times K} |\phi(k)|^2\,dk = \|\phi\|^2 = 1.$$  
Later, in the proof of Theorem \ref{theoGeomSpec} (see Section \ref{theoGeomSpecProof}) we will use Michel-Venkatesh's canonical norm $\| \cdot \|_{\it can}$ on the space of Eisenstein series. For a detailed comparison of $\| \cdot \|_{\rm Eis}$ with $\| \cdot \|_{\it can}$ see \cite[Rem.\ 3 of Thm.\ 6.1]{PetrowYoungCoset}.

Finally we alert the reader that in this section only (Section \ref{sec:pre-trace}) the test function $f \in \mathcal{S}(\overline{G}(\A))$ is a function on all places, not only the non-archimedean ones (as it is elsewhere in this paper), and therefore the operators $R_0(f)$ and $\pi_{it}(f)$ are defined by integrals over $\overline{G}(\A)$ (not merely the non-archimedean places, as is the case elsewhere in this paper).

\begin{mytheo}\label{pre-trace}
For any $f \in \mathcal{S}(\overline{G}(\A))$ and $(x,y) \in G(\A)^2$ we have
\begin{equation}\label{pretraceformula}K_{\rm geom}(x,y) = K_{\rm cusp} (x,y) + K_{\rm cont} (x,y) + K_{\rm res} (x,y),\end{equation}
where
$$K_{\rm geom}(x,y) = \sum_{\gamma \in \overline{G}(\Q)} f(x^{-1} \gamma y),$$
$$K_{\rm cusp}(x,y) = \sum_{\pi \text{ cuspidal} }\sum_{\varphi \in \mathcal{B}(\pi)} R_0(f)\varphi(x) \overline{\varphi(y)},$$
where $\pi$ runs through trivial central character cuspidal representations,
$$K_{\rm cont} (x,y) = \frac{1}{4\pi} \sum_{\chi \text{ finite order}} \sum_{\phi \in \mathcal{B}(\chi,\chi^{-1}) } \int_{-\infty}^\infty E(\pi_{it}(f)\phi_{it}, x) \overline{E(\phi_{it}, y)} \,dt,$$
where $\chi$ runs through finite-order Hecke characters, and 
$$K_{\rm res} (x,y) =  \frac{3}{\pi} \sum_{\chi \text{ quadratic}} \chi( \det x) \overline{\chi(\det y)} \int_{\overline{G}(\A) } f(g) \chi(\det(g)) \,dg ,$$
where $\chi$ runs through quadratic Hecke characters.
The right hand side of \eqref{pretraceformula} converges absolutely and uniformly on compacta in $[G]^2$. 
\end{mytheo}

Theorem \ref{pre-trace} generalizes Corollary 6.12 of \cite{knightly_kuznetsovs_2013}.  
\begin{proof} 
The Theorem is implied by e.g.\ \cite[Thm.\ 2.2]{LuoPiWu}, however the statements of these two theorems differ slightly, so we now pin down a few technical points. 
\begin{itemize}
\item The space of test functions $\mathcal{S}(\overline{G}(\A))$ in this paper is defined differently than the space $\mathcal{S}(\PGL_2(\A))$ in \cite[\S 2.2]{LuoPiWu}. Both spaces are defined to be $\mathcal{S}(\overline{G}(\R)) \otimes C_c^\infty(\overline{G}(\A_{\rm fin}))$ with $\mathcal{S}(\overline{G}(\R))$ having exactly the same definition in the two papers. On the other hand, Luo-Pi-Wu define $C_c^\infty(\overline{G}(\A_{\rm fin}))$ to be the space of smooth compactly supported functions, where ``smooth'' in this context is defined in previous work of Wu  \cite[Def.\ 2.6]{Wu2014}. In this paper we define  $C_c^\infty(\overline{G}(\A_{\rm fin})) = \cH_{\rm fin}$ to consist of locally constant compactly supported functions. However, unpacking definitions, it is easily verified that a function $f$ on $\overline{G}(\A_{\rm fin})$ is smooth and compactly supported if and only if it is locally constant and compactly supported.  
\item Our definition of basis vectors $\varphi \in \mathcal{B}(\pi)$ (resp.\ $\phi \in \mathcal{B}(\chi,\chi^{-1})$) above does not match the definition of $\mathcal{B}(\pi)$ (resp.\ $ \mathcal{B}(\chi,\chi^{-1})$) found in \cite[Thm.\ 2.1]{LuoPiWu}. Indeed, the basis vectors in loc.\ cit.\ are required to be $K_{\infty}$-isotypic and $K$-finite pure tensors. However, Theorem \ref{pre-trace} does follow from \cite[Thm.\ 2.2]{LuoPiWu} since basis vectors in $V^{K(N)}$ are necessarily $K$-finite and the orthogonal complement $(V^{K(N)})^\perp$ is annihilated anyway. 
\item The formula for $K_{\rm cont} (x,y)$ in Theorem \ref{pre-trace} has a sum over finite-order Hecke characters, where the corresponding sum in Luo-Pi-Wu runs over the dual of $\R_+ \Q^\times \backslash \A^\times$. These two sets match by strong approximation for ideles, since $\Q$ has class number one and group of units $\{\pm 1\}$. 
\item Our formula for $K_{\rm cont} (x,y)$ has an action of $\pi_{it}(f)$ on the section $\phi_{it}$, whereas the formula \cite[(9)]{LuoPiWu} for $K_{\rm cont} (x,y)$ has an action of $R(f)$ on the corresponding Eisenstein series. However, for any $g\in G(\A)$, $s \in \C$ with $\real(s)$ sufficiently large and $f \in \mathcal{S}(\overline{G}(\A))$, the double sum-integral 
$$ \int_{\overline{G}(\A)} \sum_{\gamma \in B(\Q) \backslash G(\Q)} |f(h)| |\phi_s(\gamma g h)|\,dh$$
converges absolutely, so that by Fubini's Theorem we have
$$ R(f)E(\phi_s)(g) = E(\pi_s \phi_s)(g).$$
The result also holds for $s$ with $\real(s)=0$ by analytic continuation. 
\end{itemize}

\end{proof}

Note that under the assumption that $f$ is bi-$\omega$-isotypic for some character $\omega$ of $K_\infty$, then the bases $\mathcal{B}(\pi)$ and $\mathcal{B}(\chi,\chi^{-1})$ appearing in Theorem \ref{pre-trace} are in fact finite.

\subsection{Proof of the unrefined PBK formula}

In this section we prove Theorem \ref{MT}. We now assume that $f \in \cH_{\rm fin}$ is a pure tensor and that $f_\infty \in C_c^\infty(G^+(\R))$ is bi-$K_\infty$ and $Z(\R)$-invariant, as in \cite[\S 3.1]{knightly_kuznetsovs_2013}. In particular, $f_\infty \in \mathcal{S}(\overline{G}(\R))$, so Theorem \ref{pre-trace} applies to $f_\A = f_\infty  f$.  There is a bijection between the functions $f_\infty$ and $h_\infty(t)$ in appropriate spaces, as explained in Chapter 3 of \cite{knightly_kuznetsovs_2013}. We follow Knightly and Li closely and treat the archimedean aspects exactly as they do.

For any $m \in \Q$, set \begin{equation}\label{thetamdef}\psi_m(x) = \psi(-mx) = \overline{\psi(mx)},\end{equation} where $\psi$ is the additive character of $\A$ chosen in \eqref{psidef}. We let $y_1,y_2 >0$, $m_1,m_2 \neq 0$ and consider 
\begin{equation}\label{eq:Idef}
I := \frac{1}{\sqrt{y_1 y_2}} \iint_{[N]^2} K_{\rm geom}(n_1 \left( \begin{smallmatrix} y_1 & \\ & 1 \end{smallmatrix}\right) , n_2 \left( \begin{smallmatrix} y_2 & \\ & 1 \end{smallmatrix}\right) ) \overline{\psi_{m_1}(n_1)}\psi_{m_2}(n_2)\,dn_1\,dn_2,
\end{equation}
where 
$\psi_m(n) := \psi_m(x)$ for $n=n(x)$ for $x \in \Q \backslash \A$. 

We next apply Theorem \ref{pre-trace} and compute $I$ in two ways. Note that $(N(\Q) \backslash N(\A))^2$ is compact, so that Theorem \ref{pre-trace} permits us to apply Fubini's theorem and exchange the integral over $[N]^2$ with the sums that define each of $K_{\rm cusp}$, $K_{\rm cont}$, and $K_{\rm res}$. The result is a decomposition
\begin{equation}\label{Idecomp}I = I_{\rm cusp} + I_{\rm cont} + I_{\rm res}.\end{equation}

\subsubsection{Fourier Expansion}\label{sec:FE}
We briefly digress to collect some facts that will be useful in the following. 
Let $(\pi,V)$ be a standard generic automorphic representation (see \cite[\S 2.2.1]{MichelVenkateshGL2}) and let $\varphi \in V$. 
Following \cite[\S 4.1.3]{MichelVenkateshGL2} define the constant term $\varphi_N$ and Whittaker function as
\begin{equation}
\varphi_N(g) = \int_{\Q \backslash \A} \varphi(n(x)g)\,dx, \quad \text{ and } \quad W_\varphi(g) = \int_{\Q \backslash \A} \varphi(n(x)g) \overline{\psi(x)} \,dx.
\end{equation}
Then, for almost every $g \in G(\A)$ one has the Fourier-Whittaker expansion 
\begin{equation}
\varphi(g) = \varphi_N(g) + \sum_{ y \in \Q^\times} W_\varphi(a(y) g).
\end{equation}
The function $\varphi$ is called cuspidal if $\varphi_N(g)= 0$ for almost every $g \in G(\A)$.

On the other hand, from $\varphi$ one produces a classical automorphic form $u$ that has a Fourier expansion as follows. Suppose now that $\varphi \in V$ is right-$K_\infty \times K(N)$-invariant.  
Let $u=u_\varphi$ be defined by $u(x+iy) = \varphi( \left( \begin{smallmatrix} y & x \\ & 1\end{smallmatrix}\right) \times 1_{\rm fin}).$ Since $$(G(\R)^+ \times K(N) ) \cap G(\Q) = \Gamma(N),$$ we have that $u(z)= u(\gamma z) $ for all $\gamma \in \Gamma(N)$, $z \in \cH$. Caution: one cannot recover $\varphi$ from $u$ as the group $K(N)$ does not have determinants surjecting onto $\Z_p^\times$ so that strong approximation may fail. 
We may continue nonetheless. 

Since $\varphi$ is right -$K_\infty$-invariant, it follows that $u$ is an eigenfunction of the hyperbolic Laplacian on $\mathcal{H}$ (see e.g.\ \cite[Prop.\ 4.8]{knightly_kuznetsovs_2013}. Thus $u=u_\varphi$ is a weight 0 Maass form / Eisenstein series for $\Gamma(N)$ and so admits a Fourier expansion of the form 
$$
u(x+iy) = \sum_{n \in \Z} a_u(n/N,y) e\Big(\frac{n}{N}x\Big)
$$ with 
$$
a_u(n/N,y) =\frac{1}{N} \int_0^N u(x+iy) e\Big(-\frac{n}{N}x\Big)\,dx.
$$ 
Writing $m=n/N \neq 0$, we define (following \cite[Thm.\ 6.1]{PetrowYoungCoset}) the Fourier coefficient $a_u(m)$ by
\begin{equation}\label{k=0FCclassical} \frac{a_u(m)}{\sqrt{|m|}} W(my)= a_u(m,y),\end{equation}
where $a_u(m)$ does not depend on $y$ and $W$ is a minimal non-negative weight vector in the Kirillov model of $\pi_\infty$ with norm 1. The Whittaker function $W$ is given explicitly by 
\begin{equation}\label{eq:archWhittakerExplicit} W(y) = (\sgn y)^{\epsilon} \left( \frac{\cosh \pi t}{\pi}\right)^{1/2} 2 \sqrt{|y|} K_{it}(2 \pi |y|),\end{equation} with $t$ is the spectral parameter of $u$ and $\epsilon= 0,1$ according to whether $u$ is even or odd.  

The Fourier-Whittaker coefficients above are related to classical Fourier coefficients at the cusp $\infty$ as follows. For any $m \in \Q^\times$ embedded diagonally in $\A^\times$ and $y \in \R^\times$ we have 
\begin{equation}
W_{\varphi}(a(-m)a(y)) = \int_{\Q \backslash \A}\varphi(n(x)a(y)) \overline{\psi_m(x)}\,dx
\end{equation}
 by the left $\overline{G}(\Q)$-invariance of $\varphi$ and a change of variables. Following the same steps as in 
\cite[Lem.\ 3.6]{Gelbart}, 
 e.g.,  the classical Fourier coefficients are related to the Fourier-Whittaker coefficients by
\begin{equation}
W_{\varphi}(a(-my)) = \begin{cases}  a_{u_\varphi}(m,y) &  \text{ if } m = \frac{n}{N} \in \frac{1}{N}\Z, \\ 0 & \text{ otherwise.} \end{cases} 
\end{equation}

\subsubsection{Cuspidal contribution}\label{sec:cuspcontr}
We now return to the computation of $I_{\rm cusp}$. Swapping the order of summation and integration, we have by e.g.\ Propositions 4.7, 4.8 of \cite{knightly_kuznetsovs_2013} that
\begin{multline*}
I_{\rm cusp} = \frac{1}{\sqrt{y_1 y_2}} \iint_{[N]^2} K_{\rm cusp}(n_1 \left( \begin{smallmatrix} y_1 & \\ & 1 \end{smallmatrix}\right) , n_2 \left( \begin{smallmatrix} y_2 & \\ & 1 \end{smallmatrix}\right) ) \overline{\psi_{m_1}(n_1)}\psi_{m_2}(n_2)\,dn_1\,dn_2 \\
= \frac{1}{\sqrt{y_1 y_2}  } \sum_{ \pi \in \cF_0(f)} h_\infty(t_\pi)\sum_{\varphi \in \mathcal{B}(\pi)} W_{\pi(f)\varphi}(a(-m_1y_1))\overline{W_{\varphi}(a(-m_2y_2))},
\end{multline*}
where $t_\pi$ is the spectral parameter of $\pi$. 

Note that for $\pi \in \cF_0(f)$ and $\varphi \in \cB(\pi)$, both $\varphi$ and $\pi(f) \varphi$ are cuspidal, supported on $G^+(\R)$, and bi-$K_\infty$ and $K(N)$-invariant, so $u_{\pi(f)\varphi}$ admits a classical Fourier expansion. Therefore we have if $m_1,m_2 \in \frac{1}{N} \Z$ and $m_1m_2 \neq 0$ that 
\begin{multline}\label{IcuspbeforeIcuspw}
I_{\rm cusp} = \frac{4}{\pi} (\sgn m_1m_2)^\epsilon \sum_{ \pi \in \cF_0(f)} h_\infty(t_\pi) (\cosh \pi t_\pi) K_{ it_\pi } (2\pi |m_1| y_1) K_{ it_\pi } (2\pi |m_2| y_2) \\ \times \sum_{\varphi \in \mathcal{B}(\pi)} a_{u_{\pi(f)\varphi}}(m_1) \overline{a_{u_\varphi}(m_2) } .
\end{multline}
 Assume now that $m_1,m_2$ have the same sign and introduce a new variable $w \in \R_{>0}$. We impose the constraint $w = m_1y_1 = m_2 y_2$ on $y_1,y_2$ on \eqref{IcuspbeforeIcuspw}, writing $I_{\rm cusp}(w)$ for the formula there with this constraint. Then 
 $$\int_{0}^\infty I_{\rm cusp}(w) \,dw = \frac{1}{2}\sum_{ \pi \in \cF_0(f)} h_\infty(t_\pi)  \sum_{\varphi \in \mathcal{B}(\pi)} a_{u_{\pi(f)\varphi}}(m_1) \overline{a_{u_\varphi}(m_2)}$$
by following the proof of \cite[Prop.\ 7.5]{knightly_kuznetsovs_2013} mutatis mutandis.

\begin{myrema} Under the additional hypothesis of Geometric Assumption \eqref{geo2} and the fact that $\pi$ has trivial central character, it would follow that the $u_\varphi$ appearing here are automorphic for the larger $\Gamma_d(N) \supseteq \Gamma(N)$. \end{myrema}

\subsubsection{Continuous contribution}\label{sec:ctscontr} The computation of $I_{\rm cont}$ in this section is in parallel to that of the cuspidal contribution, mutatis mutandis. In similar fashion to $\cF_0(f)$, define
\begin{equation}\label{def:cFEf}
\cF_E(f) := \{ \chi \in (\Q^\times \backslash \A^1)^\wedge : \text{ there exists } t \in \R \text{ with } \pi_{\chi |\cdot|^{it}, \chi^{-1}|\cdot|^{-it}}(f) \neq 0\},
\end{equation}
where for $\mu \in (\Q^\times \backslash \A^\times)^\wedge$ the global principal series representation $\pi_{\mu,\mu^{-1}}$ is as in Section \ref{sec:pre-trace} and $\pi_{\mu, \mu^{-1}}(f)$ is as in \eqref{pif}. Swapping order of summation and integration by the absolute convergence in Theorem \ref{pre-trace}, we have by e.g.\ \cite[Prop.\ 5.2]{knightly_kuznetsovs_2013}
\begin{multline*}
I_{\rm cont} = \frac{1}{\sqrt{y_1 y_2}} \iint_{[N]^2} K_{\rm cont}(n_1 \left( \begin{smallmatrix} y_1 & \\ & 1 \end{smallmatrix}\right) , n_2 \left( \begin{smallmatrix} y_2 & \\ & 1 \end{smallmatrix}\right) ) \overline{\psi_{m_1}(n_1)}\psi_{m_2}(n_2)\,dn_1\,dn_2 \\
= \frac{1}{4 \pi \sqrt{y_1 y_2}  } \sum_{ \chi \in \cF_E(f)} \sum_{\phi \in \mathcal{B}(\chi, \chi^{-1})} \int_{-\infty}^\infty h_\infty(t) W_{E(\pi_{it}(f)\phi_{it})}(a(-m_1y_1))\overline{W_{E(\phi_{it})}(a(-m_2y_2))} dt.
\end{multline*}
Exactly as in Section \ref{sec:cuspcontr} and with conventions on Fourier coefficients as in Section \ref{sec:FE}, we obtain 
 $$\int_{0}^\infty I_{\rm cont}(w) \,dw = \frac{1}{8\pi}\sum_{ \chi \in \cF_E(f)} \sum_{\phi \in \mathcal{B}(\chi, \chi^{-1})} \int_{-\infty}^\infty h_\infty(t) a_{u_{E(\pi_{it}(f)\phi_{it})}}(m_1) \overline{a_{u_{E(\phi_{it})}}(m_2)} dt.$$

\subsubsection{Residual contribution}\label{sec:rescontr}
By Theorem \ref{pre-trace} we have 
\begin{multline*}
I_{\rm res} = \frac{1}{\sqrt{y_1 y_2}} \iint_{[N]^2} K_{\rm res}(n_1 \left( \begin{smallmatrix} y_1 & \\ & 1 \end{smallmatrix}\right) , n_2 \left( \begin{smallmatrix} y_2 & \\ & 1 \end{smallmatrix}\right) ) \overline{\psi_{m_1}(n_1)}\psi_{m_2}(n_2)\,dn_1\,dn_2 \\
= \frac{1}{\sqrt{y_1 y_2}  } \frac{3}{\pi} \sum_{\chi \text{ quad.}} \chi(y_1) \overline{\chi(y_2)} \int_{\overline{G}(\A)} f_\A(g) \chi(\det g)\,dg \int_{\Q \backslash \A} \overline{\psi_{m_1}(n_1)}\,dn_1 \int_{\Q \backslash \A} \psi_{m_2}(n_2)\,dn_2. 
\end{multline*}
Since $m_1m_2 \neq 0$, the last two integrals both vanish identically. Therefore $I_{\rm res}=0$ for all $y_1,y_2$. 

\subsubsection{Geometric side}\label{sec:geomside}
Recall the definition of $I$ from \eqref{eq:Idef} and insert the formula for $K_{\rm geom}$ from Theorem \ref{pre-trace}. We now exchange order of summation and integration and group the geometric terms according to orbits $ \delta \in N(\Q) \backslash \overline{G}(\Q) / N(\Q)$. To that end, define orbital integrals $I_\delta(f_{\A})$ by  
\begin{equation}\label{orbitalintegral} I_{\delta}(f_{\A}) = \int_{H_\delta(\Q) \backslash H(\A)} f_{\A} \left( \left( \begin{smallmatrix} y_1 & x_1 \\ & 1 \end{smallmatrix}\right)^{-1} \delta  \left( \begin{smallmatrix} y_2 & x_2 \\ & 1 \end{smallmatrix}\right) \right) \frac{\overline{\psi_{m_1}(x_1)} \psi_{m_2}(x_2)}{\sqrt{y_1y_2} }\, d(x_1,x_2),\end{equation}
where $H(\A) = N(\A) \times N(\A) \simeq \A \times \A$ and $H_\delta(\Q)$ is the stabilizer in $H(\Q) = N(\Q) \times N(\Q)$ of $\delta$, where $H(\Q)$ acts on $\overline{G}(\Q)$ on the right by $\gamma . (x,y) = x^{-1} \gamma y,$ and $d(x_1,x_2)$ is the quotient measure coming from $dt_1\,dt_2$.
Using the Bruhat decomposition and following Knightly-Li \cite[\S 7.5]{knightly_kuznetsovs_2013}, we have 
\begin{equation}\label{eq:orbitdecomp} I = I_{ \left( \begin{smallmatrix} m_2/m_1 &  \\ & 1 \end{smallmatrix}\right)}(f_{\A}) + \sum_{\mu \in \Q^\times}  I_{ \left( \begin{smallmatrix}  & -\mu \\ 1 &  \end{smallmatrix}\right)}(f_{\A}).\end{equation}
  For the explicit representatives for the orbits $\delta$ that appear in \eqref{eq:orbitdecomp}, we can be more explicit about the shape of $H_\delta(\Q)$. 

The terms $$\delta = \left( \begin{smallmatrix} m_2/m_1 &  \\ & 1 \end{smallmatrix}\right)$$ are called \emph{first cell terms}, and the terms $$\delta = \left( \begin{smallmatrix}  & -\mu \\ 1 &  \end{smallmatrix} \right)$$ are called \emph{second cell terms}. For the first cell terms, we have $$H_{\left( \begin{smallmatrix} m_2/m_1 &  \\ & 1 \end{smallmatrix}\right)}(\Q) = \{ \left( \left( \begin{smallmatrix} 1 & m_2 t / m_1  \\ & 1 \end{smallmatrix}\right), \left( \begin{smallmatrix} 1 &  t  \\ & 1 \end{smallmatrix}\right)  \right) \in H(\Q) : t \in\Q\},$$
while
$$H_{\left( \begin{smallmatrix}  & -\mu \\ 1 &  \end{smallmatrix} \right)}(\Q) = \{(1,1) \in H(\Q)\}.$$
\subsubsection{First cell terms}\label{sec:firstcellterms}
Exactly as in \cite[Section 7.5.1]{knightly_kuznetsovs_2013}  we get (recall the definition of $\psi$ from \eqref{psidef})
$$I_{ \left( \begin{smallmatrix} m_2/m_1 &  \\ & 1 \end{smallmatrix}\right)}(f_{\A})  = \int_{\A} f_{\A} \left( \left( \begin{smallmatrix} m_2 y_2 & t  \\ & m_1 y_1  \end{smallmatrix}\right) \right) \frac{\psi(-t)}{\sqrt{y_1y_2}}\,dt.$$
 This integral factors into archimedean and non-archimedean parts, say $I_\delta(f_{\A}) = I_\delta(f_\infty) I_{\delta}(f)$. The archimedean part is
$$I_{ \left( \begin{smallmatrix} m_2/m_1 &  \\ & 1 \end{smallmatrix}\right)}(f_\infty)  =  \frac{1}{\sqrt{y_1y_2}} \int_{-\infty}^\infty f_\infty \left(  \left( \begin{smallmatrix} m_2 y_2 & t  \\ & m_1 y_1 \end{smallmatrix}\right) \right) e(t)\,dt,$$ and the finite part $I_{\delta}(f)$ does not depend on $y_1,y_2$. Note that since $f_\infty$ is assumed to be supported on $G^+(\R)$, we have $I_{ \left( \begin{smallmatrix} m_2/m_1 &  \\ & 1 \end{smallmatrix}\right)}(f_\infty)=0$ unless $m_1$ and $m_2$ have the same sign. 

 Now choose $y_1$ and $y_2$ so that $w = y_1m_1 = y_2m_2$, and write $ I_{\delta}(f_\infty,w)=I_{\delta}(f_\infty) $ considered as a function of $w \in \R_{>0}$. By following the proof of \cite[Prop.\ 7.9]{knightly_kuznetsovs_2013} mutatis mutandis, we have 
\begin{multline*}\int_{0}^\infty I_{\left( \begin{smallmatrix} m_2/m_1 &  \\ & 1 \end{smallmatrix}\right)}(f_\A,w)\,dw = I_{\left( \begin{smallmatrix} m_2/m_1 &  \\ & 1 \end{smallmatrix}\right)}(f)\int_{0}^\infty I_{\left( \begin{smallmatrix} m_2/m_1 &  \\ & 1 \end{smallmatrix}\right)}(f_\infty,w)\,dw \\
= I_{\left( \begin{smallmatrix} m_2/m_1 &  \\ & 1 \end{smallmatrix}\right)}(f) \frac{\sqrt{m_1m_2}}{2}f_\infty\left( 1\right)\delta_{m_1m_2>0}.
 \end{multline*} 

Thus, to recover the diagonal term in the formula given in Theorem \ref{MT}, it suffices to calculate the finite part $ I_{\delta}(f) $ for $\delta = \left( \begin{smallmatrix} m_2/m_1 &  \\ & 1 \end{smallmatrix}\right)$.  By the $Z(\A_{\rm fin})$-invariance of $f$, we have 
\begin{equation}\label{Ifinfirstcell} I_{\left( \begin{smallmatrix} m_2/m_1 &  \\ & 1 \end{smallmatrix}\right)}(f) = \int_{\A_{\rm fin} } f \left( \left( \begin{smallmatrix} m_2 & t \\ & m_1 \end{smallmatrix} \right)\right) \psi_{\rm fin}(-t)\,dt = \int_{\A_{\rm fin} } f\left( \left( \begin{smallmatrix} m_2/m_1 & t/m_1 \\ & 1 \end{smallmatrix} \right)\right) \psi_{\rm fin}(-t)\,dt. \end{equation}

Changing variables $t \rightarrow m_1 t$ we have
$$I_{\left( \begin{smallmatrix} m_2/m_1 &  \\ & 1 \end{smallmatrix}\right)}(f) = |m_1|_{\rm fin}\int_{\A_{\rm fin}} f\left( \left( \begin{smallmatrix} m_2/m_1 & t \\ & 1 \end{smallmatrix} \right)\right) \psi_{\rm fin}(-m_1t)\,dt.$$
Note also that by a change of variables $t \to t+N$, the integral vanishes unless $m_1$ and $m_2 \in \frac{1}{N} \Z$. 

Putting together the finite and infinite parts, we have that 
\begin{multline}\label{eq:firstcellfinalformula}\int_{0}^\infty I_{\left( \begin{smallmatrix} m_2/m_1 &  \\ & 1 \end{smallmatrix}\right)}(f_{\A},w)\,dw \\ =  \delta_{m_1,m_2  \in \frac{1}{N} \Z} \delta_{m_1m_2>0}\frac{1}{2}\sqrt{\frac{m_2}{m_1}}f_\infty(1) \int_{\A_{\rm fin}} f\left( \left( \begin{smallmatrix} m_2/m_1  & t \\ & 1 \end{smallmatrix} \right)\right) \psi_{\rm fin}(-m_{1}t)\,dt.\end{multline}
A similar formula also holds with the roles of $m_1$ and $m_2$ reversed. The geometric assumptions permit a further simplification of \eqref{eq:firstcellfinalformula}, and the spectral assumption allows further simplification in terms of Plancherel volumes. For these, see Section \ref{sec:MTcomputation}.  

\subsubsection{Second cell terms}\label{sec:secondcellterms}
In the rest of this subsection, we assume that $m_1$ and $m_2$ have the same sign, since we are following the archimedean computations of Knightly and Li. 

Since $f_\infty$ is supported in $G^+(\R)$ and is bi-$K_\infty$-invariant, we may follow \cite[\S 7.5.2]{knightly_kuznetsovs_2013} for $\delta = \left( \begin{smallmatrix}  & - \mu \\ 1 &  \end{smallmatrix} \right)$ with $\mu \in \Q^\times$ to deduce that if $\mu>0$, then $I_{\delta}(f_\A)= I_{\delta}(f) I_{\delta}(f_\infty)$ with 
$$I_\delta(f_{\infty }) = \frac{1}{\sqrt{y_1y_2} }\iint_{\R^2} k(z_1, \frac{-\mu}{z_2}) e(m_2x_2-m_1x_1)\,dx_1dx_2,$$
 where $k(z_1,z_{2}) = f_\infty(g_1^{-1}g_2)$, $z_j =g_j(i)$, and \begin{equation}\label{IdeltaffinIntegral1}I_\delta(f) = \iint_{\A_{\rm fin}^2} f\left( \left( \begin{smallmatrix}1 & -t_1 \\ & 1 \end{smallmatrix}\right) \delta  \left( \begin{smallmatrix}1 & t_2 \\ & 1 \end{smallmatrix}\right) \right) \psi_{\rm fin}(m_1t_1- m_2 t_2)\,dt_1\,dt_2,\end{equation}
 and if $\mu<0$, then $I_\delta(f_\A) =0$.

Since each $f_p$ is supported on matrices with determinant in $\Z_p^\times (\Q_p^\times)^2$, we see that the integral $I_\delta(f)$ is 0 unless $\mu \in  \Z_p^\times (\Q_p^\times)^2$ for all $p$. Since $\mu \in \Q^\times$ and $\mu>0$, we have that  $I_{\left( \begin{smallmatrix} & - \mu \\ 1 & \end{smallmatrix} \right) }(f) = 0$ unless there exists $s \in \Q^\times$ so that $\mu = s^2$. Let us write $c=1/s$.  
 
We have for $f \in \cH_{\rm fin}$, $m,n \in \Q$ and $c \in \Q_+$ definition \eqref{GenKloostermanDef} that
\begin{equation*}I_\delta(f) = I_{\left(\begin{smallmatrix}  & -c^{-2} \\ 1 &  \end{smallmatrix} \right)}(f) = H(m_1,m_2;c)  = \frac{1}{|c|_{\rm fin}^2} \iint_{\A_{\rm fin}^2} f\left( \left(\begin{smallmatrix} t_1 & \frac{-1-t_1t_2}{c} \\ c & t_2 \end{smallmatrix} \right)\right) \psi_{\rm fin}\left( \frac{m_1t_1-m_2t_2}{c}\right)\,dt_1\,dt_2.\end{equation*}
We next prove the ``proposition part'' of Proposition/Definition \ref{intro:admissiblemodulus}. 
\begin{mylemma}\label{lem:existance_of_geom_cond_without_hypotheses}
Suppose $f = \bigotimes_p f_p \in \cH_{\rm fin}$ is a pure tensor. If $\supp f_p$ is contained in $\{g \in G(\Q_p): \det g \in \Z_p^\times(\Q_p^\times)^2\}$ for all $p$, then there exists $q' \in \Q_+$ such that for all $c \in \Q_+\smallsetminus q'\N$ and all $m_1,m_2 \in\Q$ we have $H(m_1,m_2;c)=0 $.  
\end{mylemma}
\begin{proof}
Let $S \subset G(\A_{\rm fin})$ be a fundamental domain for the support of $f$ in $\overline{G}(\A_{\rm fin})$, whose matrices have determinant in $\widehat{\Z}^\times$.  
Since $S$ is compact, and the projection map $\pi : S \to \A_{\rm fin}$ sending a matrix to its bottom left entry is continuous, the set $\pi(S) \subset \prod_p' p^{k_p}\cO_p$, where $k_p \in \Z$ and $k_p=0$ for almost every $p$. Therefore, $H(m_1,m_2;c)$ vanishes identically unless $c$ has a bounded denominator, depending on the support of $f$. 
\end{proof}

Recall that we choose $y_1$ and $y_2$ so that $w = y_1m_1 = y_2m_2$, and write $ I_{\delta}(f_\infty,w)=I_{\delta}(f_\infty) $. We would like to exchange order of summation and integration in 
$$\int_0^\infty \sum_{c \in \cC(\cF)} I_{\delta}(f_\A,w) \,dw = \int_0^\infty \sum_{c \in \cC(\cF)} H(m_1,m_2;c) I_\delta(f_\infty, w)\,dw$$ without any appeal to an estimate of $H(m_1,m_2;c)$. Since the sum over $c$ converges absolutely by the pre-trace formula, to justify the exchange it suffices to show that $I_\delta(f_\infty, w)$ vanishes identically for all sufficiently large or sufficiently small $w$, uniformly in $c$. 

Following the proof of \cite[Prop.\ 7.11]{knightly_kuznetsovs_2013} and making the change of variables $wc/\sqrt{m_1m_2} \to w$ in the second-to-last displayed equation there, it suffices to show that 
\begin{equation}\label{need_abs_conv}
 \iint_{\R^2}\left| k\left( t_1+iw,\frac{-1}{t_2+ iw}\right) \right|\,dt_1\,dt_2
\end{equation} 
vanishes identically for all sufficiently large or small $w$. Recall here that $k$ is a compactly supported point-pair invariant kernel. We begin by noting that the hyperbolic distance $\rho$ satisfies 
$$\rho(z_1,z_2) \geq \rho(iy_1,iy_2) = |\log(y_1/y_2)|.$$
Therefore, since $\rho$ is invariant under $z \mapsto -1/z$ we have
\begin{equation}\label{eq:distintequality}\rho\left( t_1+iw,\frac{-1}{t_2+ iw}\right) =  \rho\left( t_2+iw,\frac{-1}{t_1+ iw}\right)\geq \rho \left( iw, \frac{iw}{t_1^2+ w^2}\right) = |\log \left(t^2_1 + w^2\right)|.\end{equation}
Next, we apply the change of variables $t_1\to wt_1$ in \eqref{need_abs_conv}, and then observe 
$$\rho(w(t_1+i), \frac{-1}{t_2+iw}) = \rho( wt_2+iw^2, \frac{-1}{t_1+i}) \geq |\log \left(w^2(1+t_1^2)\right)|.$$

There exists $k \in C_c^\infty([0,\infty)$ so that $k(z_1,z_2) = k(\rho(z_1,z_2))$ (explicitly, $k(x) = V(2(\cosh x -1))$ where $V$ is as in \cite[Ch.\ 3]{knightly_kuznetsovs_2013}). Let $Y$ be such that $\supp k \subseteq [0,Y]$, and let $X=\exp Y$, so that $\supp k \left( t_1+iw,\frac{-1}{t_2+ iw}\right)$ is contained in 
$$S_{t_1,t_2} := \{w \in [0,\infty): X^{-1} \leq t_1^2+w^2 \leq X\, \text{ and }  X^{-1} \leq w^2 (1+t_1^2) \leq X \}.$$ Let $ S = \cup_{t_1,t_2} S_{t_1,t_2}$, so that the support of the expression in \eqref{need_abs_conv} is contained in $S$. We claim that $S \subseteq [(1+X)^{-1}, \sqrt{X}]$. First, we show that $S_{t_1,t_2} \subseteq [0,\sqrt{X}]$ for all $(t_1,t_2) \in \R^2$: suppose $w>\sqrt{X}$ and  $w \in S_{t_t,t_2}$, then $X<w^2\leq w^2+t_1^2\leq X$, contradiction. Now we show that $S_{t_1,t_2} \subseteq [(1+X)^{-1}, \infty)$ for all $(t_1,t_2) \in \R^2$: suppose $w<(1+X)^{-1}$ and $w\in S_{t_1,t_2}$. Then, $t_1^2 \leq t_1^2+w^2\leq X$, so $1+t_1^2 \leq 1+X$. Also, $w^2<(1+X)^{-2}$, so
$$ (1+X)w^2 < \frac{(1+X)^2}{X} w^2 <  \frac{1}{X} \leq w^2(1+t_1^2) \leq (1+X)w^2.$$ Contradiction.

Having justified the exchange of summation and integration, the calculation in \cite[\S 7.5.2]{knightly_kuznetsovs_2013} goes through and we find that 
$$\int_0^\infty I_{\delta}(f_\infty,w) \,dw =\frac{i \sqrt{\mu}}{4} \int_{-\infty}^\infty J_{2it}(4 \pi \sqrt{\mu m_1m_2}) \frac{h_\infty(t) t}{\cosh (\pi t)} \,dt = \frac{1}{2} \frac{H_\infty(\frac{4 \pi \sqrt{m_1m_2}}{c})}{c}.$$ 

To conclude Theorem \ref{MT}, take \eqref{Idecomp},  integrate it over $w \in \R_{>0}$ as  above, expand the geometric side  in terms of double cosets as in \eqref{eq:orbitdecomp}, and collect the results of sections \ref{sec:cuspcontr}-\ref{sec:secondcellterms}. 

\subsection{Holomorphic/discrete series variation}\label{sec:holomorphicunrefined}
We only need to  modify the archimedean aspects of the above, and these have already been treated in \cite{KLPetersson}. For the holomorphic forms/ discrete series variation, throughout this  paper one should replace instances of $K_\infty$-fixed vectors to $\omega$-isotypic vectors, where $\omega$ is the weight $\kappa$ character of $K_\infty$ defined by $\omega\left( \begin{smallmatrix} \cos \theta & \sin \theta \\ -\sin \theta & \cos \theta \end{smallmatrix} \right) = e^{i\kappa \theta}$. By \cite[Thm.\ 8.1]{Knapp}, the space of $\omega$-isotypic vectors in $V_\pi$ are at most 1-dimensional, just as the $K_\infty$-fixed vectors are $1$-dimensional. 

We give a few brief details of the derivation. Fix $\kappa \geq 2$ even and let 
\begin{equation}
f_\infty = \frac{1}{\| \Phi_{\pi_\kappa,v_0} \|_2^2} \overline{\Phi_{\pi_\kappa,v_0}},
\end{equation}
where $\pi_\kappa$ is the weight $\kappa$ discrete series representation of $\GL_2(\R)$ (see e.g.\ \cite[\S 11.7]{KnightlyLiTracesOfHeckeOperators}), $v_0$ is an $L^2$-normalized lowest weight vector therein, and $\Phi_{\pi_\kappa,v_0}$ is the associated diagonal matrix coefficient. In completely explicit terms, for $g = \left( \begin{smallmatrix} a & b \\ c & d\end{smallmatrix}\right)$
\begin{equation}\label{holo:f_infty}
f_\infty(g) = \begin{cases} \frac{\kappa-1}{4\pi} \frac{ \det(g)^{\kappa/2} (2i)^\kappa }{(-b+c+(a+d)i)^\kappa} & \text{ if } \det g >0, \\ 0 & \text{ else.} \end{cases}
\end{equation}
The operator $\pi(f_\infty):V_\pi \to V_\pi$ projects onto the line of $v_0$ if $\pi \simeq \pi_\kappa$ and is the 0 operator otherwise. The pre-trace formula holds for this choice of test function at the archimedean place, see \cite[\S 15]{KnightlyLiTracesOfHeckeOperators} where $K_{\rm cont}$ and $K_{\rm res}$ are identically equal to 0. 

As in \eqref{eq:Idef} we consider 
$$ I := \iint_{[N]^2} K_{\rm geom} (n_1, n_2) \overline{\psi_{m_1}(n_1)}\psi_{m_2}(n_2)\,dn_1\,dn_2.$$ Applying the pre-trace formula and exchanging order of integration, we have $I =I_{\rm cusp}$. 

To treat $I_{\rm cusp}$, we need the Fourier expansions from Section \ref{sec:FE}. For $\omega$-isotypic vectors $\varphi \in V_\pi, \,\pi \in \cF_\kappa(f)$, one defines $u=u_\varphi$ by 
$$
u(x+iy)= j( \left(\begin{smallmatrix} y & x \\ & 1\end{smallmatrix}\right), i)^\kappa \varphi (\left(\begin{smallmatrix} y & x \\ & 1\end{smallmatrix}\right) \times 1_{\rm fin} ),
$$ 
where $j(g,z) = (cz+d)(\det g)^{-1/2}$ for 
$g = \left(\begin{smallmatrix} a & b \\ c &d\end{smallmatrix}\right) \in \GL_2^+(\R)$. 
 Then, $u$ is a holomorphic modular form of weight $\kappa$ for $\Gamma(N)$, so admits a Fourier expansion of the form 
 $$
 \sum_{n \in \N } a_u(n/N,y) e\Big(\frac{n}{N}x\Big) \quad \text{ with } \quad a_u(n/N,y) = \frac{1}{N} \int_0^N u(x+iy) e\Big(-\frac{n}{N}x\Big)\,dx.
 $$
The normalized Fourier coefficients $a_u(m)$ are given by 
\begin{equation}
\frac{a_u(m)}{\sqrt{m}} W(my) = y^{\kappa/2} a_u(m,y),
\end{equation}
where $W$ is the vector of minimal weight and norm 1 in the archimedean Kirillov model given explicitly by 
\begin{equation}
W(y) = \begin{cases} \left(\frac{(4\pi y)^\kappa}{\Gamma(\kappa)}\right)^{1/2} e^{-2\pi y} & \text{ if } y>0, \\ 0 & \text{ if } y<0. \end{cases}
\end{equation}  

Continuing with the computation of $I$, by the same steps as in Section \ref{sec:cuspcontr} we have when $m_1,m_2>0$ that 
\begin{multline*} I = I_{\rm cusp} = \sum_{\pi \in \cF_\kappa(f)} \sum_{\varphi \in \cB(\pi)} W_{\pi(f)\varphi}(a(-m_1)) \overline{W_{\varphi}(a(-m_2))} \\ = \frac{(4\pi )^\kappa}{\Gamma(\kappa)} (m_1m_2)^{\frac{\kappa-1}{2}} e^{-2\pi(m_1+m_2)} \sum_{\pi \in \cF_\kappa(f)} \sum_{\varphi \in \cB(\pi)}a_{u_{\pi(f)\varphi}}(m_1)\overline{a_{u_{\varphi}}(m_2)},\end{multline*}
and $I=0$ otherwise. 

On the other hand, $I$ has a geometric expansion into first cell terms and second cell terms \eqref{eq:orbitdecomp}, exactly as in Section \ref{sec:geomside}. For the first cell terms, exactly as in Section \ref{sec:firstcellterms} but using \cite[Prop.\ 3.4]{KLPetersson} for the archimedean aspect, if $m_1,m_2>0$, then  
$$I_{ \left( \begin{smallmatrix} m_2/m_1 &  \\ & 1 \end{smallmatrix}\right)}(f_{\A})  = \delta_{m_1=m_2 \in \frac{1}{N}\N} \frac{(4\pi \sqrt{m_1m_2})^{\kappa-1}}{ \Gamma(\kappa -1)}e^{-2\pi(m_1+m_2)} \int_{\A_{\rm fin}} f\left( \left( \begin{smallmatrix} 1 & t \\ & 1 \end{smallmatrix} \right)\right) \psi_{\rm fin}(-mt)\,dt$$ and $I_{ \left( \begin{smallmatrix} m_2/m_1 &  \\ & 1 \end{smallmatrix}\right)}(f_{\A})$ vanishes otherwise, 
where $m$ is the common value of $m_1=m_2$ when they are equal. For the second cell terms, exactly as in Section \ref{sec:secondcellterms} but using \cite[Prop.\ 3.6]{KLPetersson} for the archimedean aspect, if $m_1,m_2>0$, then 
$$I_{\delta}(f_\A) = \frac{(4 \pi i)^\kappa (\sqrt{m_1m_2})^{\kappa-1}e^{-2\pi(m_1+m_2)}}{2 \Gamma(\kappa -1)}\frac{H(m_1,m_2;c)}{c}J_{\kappa-1}\left(\frac{4\pi\sqrt{m_1m_2}}{c}\right).$$
Altogether, with notation and assumptions as in Theorem \ref{MT}, we have for all $m_1,m_2 \in \frac{1}{N}\N$
\begin{multline}
\sum_{\pi \in \cF_\kappa(f)} \sum_{\varphi \in \cB(\pi)}a_{u_{\pi(f)\varphi}}(m_1)\overline{a_{u_{\varphi}}(m_2)} = 
\delta_{m_1=m_2} \frac{\kappa -1}{4\pi} \int_{\A_{\rm fin}} f\left( \left( \begin{smallmatrix} 1 & t \\ & 1 \end{smallmatrix} \right)\right) \psi_{\rm fin}(-mt)\,dt \\ + \frac{(\kappa-1)i^{-\kappa}}{2} \sum_{c \in \cC(\cF)} \frac{H(m_1,m_2;c)}{c}J_{\kappa-1}\left(\frac{4\pi\sqrt{m_1m_2}}{c}\right).
\end{multline}

\section{Generalized Kloosterman sums}\label{sec:KloostermanSums}
Theorem \ref{MT} has only light hypotheses and follows almost immediately from an inspection of the proof found in \cite{knightly_kuznetsovs_2013}. However, without additional information on $f$, one has little control on the set of admissible moduli $\cC(\cF)$ and the properties of the generalized Kloosterman sums $H(m,n;c)$. In this section we assume the geometric assumptions and work out their consequences for the Kloosterman sums.

\subsection{Preliminaries on support of $f$}\label{sec:prelims_on_support}
We begin by working in somewhat more generality than afforded by the geometric assumptions and for the time being assume in lieu of Geometric Assumption \eqref{geo3} that $f$ has support contained in $ZK'$ where $K'$ is some maximal compact open subgroup of $G(\A_{\rm fin})$. Let $K' = \prod_{p} K_p'$ be the factorization of $K'$ into maximal compact open subgroups $K_p'$ of $G(\Q_p)$, where necessarily $K_p'=K_p$ for all but finitely many $p$.

We first observe that the set of pairs $(y,x) \in \Q_+ \times \A_{\rm fin}/ \widehat{\Z}$ parametrizes the maximal compact subgroups $ZK'$ as follows. Define a map $\phi$ by $\phi: (y,x) \mapsto \left( \begin{smallmatrix} y & x \\ & 1 \end{smallmatrix}\right)^{-1} ZK  \left( \begin{smallmatrix} y & x \\ & 1 \end{smallmatrix}\right)$, where $K = G(\widehat{\Z})$. 

\begin{mylemma}\label{yxcontrollingsupportoff}
The map $\phi$ is well-defined and a bijection between $\Q_+ \times \A_{\rm fin}/ \widehat{\Z}$ and groups $ZK'$, where $K'$ is a maximal compact subgroup of $G(\A_{\rm fin})$. 
\end{mylemma} 
\begin{proof}
It is clear that $\left( \begin{smallmatrix} y & x \\ & 1 \end{smallmatrix}\right)^{-1} K  \left( \begin{smallmatrix} y & x \\ & 1 \end{smallmatrix}\right)$ is a maximal compact subgroup of $G(\A_{\rm fin})$.  To see that $\phi$ is well-defined, let $z \in \widehat{\Z}$ and note that 
\begin{equation*}  \left( \begin{smallmatrix} y & x+ z \\ & 1 \end{smallmatrix}\right)^{-1} ZK  \left( \begin{smallmatrix} y & x+ z \\ & 1 \end{smallmatrix}\right) = 
\left( \begin{smallmatrix} y & x \\ & 1 \end{smallmatrix}\right)^{-1} \left( \begin{smallmatrix} 1 & -z \\ & 1 \end{smallmatrix}\right) ZK \left( \begin{smallmatrix} 1 & z \\ & 1 \end{smallmatrix}\right) \left( \begin{smallmatrix} y & x \\ & 1 \end{smallmatrix}\right)  = \left( \begin{smallmatrix} y & x \\ & 1 \end{smallmatrix}\right)^{-1} ZK  \left( \begin{smallmatrix} y & x \\ & 1 \end{smallmatrix}\right).\end{equation*}

We show that $\phi$ is surjective. Any group of the form $ZK'$ with $K'$ a maximal compact subgroup of $G(\A_{\rm fin})$ is equal to $g^{-1} ZKg$ for some $g \in G(\A_{\rm fin})$. 
We may write $g= kb$ by the Iwasawa decomposition and translate by the center to write $g =z k \left( \begin{smallmatrix} y' & x \\ & 1 \end{smallmatrix}\right)  $ for some $y' \in \A_{\rm fin}^\times$, $x \in \A_{\rm fin}$, $k \in K$ and $z \in Z$. Since $\A_{\rm fin}^\times = \Q_+\widehat{\Z}^\times$, let us write $y'=yw$ with $y \in \Q_+$ and $w \in \widehat{\Z}^\times$.  Then $g =z k  \left( \begin{smallmatrix} w &  \\ & 1 \end{smallmatrix}\right)\left( \begin{smallmatrix} y & x/w \\ & 1 \end{smallmatrix}\right)  $, so that $ZK' = \phi((y,x/w))$ with $y \in \Q_+$ and $x/w \in \A_{\rm fin}/\widehat{\Z}$.

To see that $\phi$ is injective, it can be shown by a direct computation that $$b_1^{-1} ZK b_1 = b_2^{-1} ZK b_2$$ for $b_1$ and $b_2$ of the form $\left( \begin{smallmatrix} y_i & x_i \\ & 1 \end{smallmatrix}\right)$ if and only if $|y_1|_p = |y_2|_p$ for all primes $p$ and $y_{2,p}x_{1,p}-x_{2,p}y_{1,p} \in y_{1,p}\Z_p$ for all primes $p$. Since $y \in \Q_+$, its $|y|_p$ determines it uniquely, and plugging this back in,  $x \in \A_{\rm fin}$ is determined modulo $\widehat{\Z}$. 
\end{proof}
Given $f \in \cH_{\rm fin}$ and a maximal compact open subgroup $K'$ such that $\supp f \subseteq ZK'$, we may always pick a representative for $x \pmod{  \widehat{\Z}}$ so that either $x_p=0$ or $v_p(x)<0$ for each prime $p$.  

The next lemma, which was alluded to after the introduction of the geometric assumptions in Section \ref{intro:statement_of_results}, says that Geometric Assumption \eqref{geo3} is only slightly more restrictive than assuming that $\supp f \subseteq ZK'$ for some compact open subgroup $K'$ of $G(\A_{\rm fin})$. 
\begin{mylemma}\label{x=0}
Suppose that $f$ is not identically zero, satisfies Geometric Assumption \eqref{geo2}, and that $\supp f \subseteq ZK'$ for some compact open subgroup $K'$ of $G(\A_{\rm fin})$ with $\phi^{-1}(ZK')=(y,x)$, where $\phi$ is the bijection of Lemma \ref{yxcontrollingsupportoff}. If $v_2(x) \neq -1$, then $f$ satisfies Geometric Assumption \eqref{geo3} and $y$ controls the support of $f$.  
\end{mylemma} 
\begin{proof}
It suffices to work locally at a prime $p$.  We want to show that $x=0$, so for purposes of contradiction we may assume that $v_p(x) < 0$ (see the sentence immediately following the proof of Lemma \ref{yxcontrollingsupportoff}).  Since $f$ is not identically zero and supported in $b^{-1}ZKb$ for $b = \left(\begin{smallmatrix} y & x \\ & 1\end{smallmatrix}\right)$, we have that $f(b^{-1} k b) \neq 0$ for some $k \in K$. Then $f(a b^{-1} k b a') \neq 0$ for all $a, a' \in A(\mz_p)$ by Geometric Assumption \eqref{geo2}.  Hence $a^{}b^{-1} k b a' \in b^{-1} K b$, 
equivalently, $(b a b^{-1}) k (b a' b^{-1}) \in K$ for all $a,a' \in A(\mz_p)$.  

Suppose $a = a(\alpha)$ with $\alpha \in \mz_p^{\times}$.  By direct calculation,
\begin{equation}
b a b^{-1} = \begin{pmatrix} \alpha & -x(\alpha-1) \\ & 1 \end{pmatrix}.
\end{equation}
Suppose $k =  (\begin{smallmatrix} r & t \\ u & v \end{smallmatrix}) $.  
Then taking $a'=1$, we obtain
\begin{equation}
\label{eq:matrixthingy}
(b a b^{-1}) k (b a' b^{-1})
= \begin{pmatrix} \alpha & -x(\alpha-1) \\ & 1 \end{pmatrix}
\begin{pmatrix}  r & t \\ u & v \end{pmatrix}
= 
\begin{pmatrix}
r \alpha - ux(\alpha-1)  & t \alpha - v x (\alpha-1)  \\  u & v 
\end{pmatrix}.
\end{equation}
For $p \neq 2$, we can choose $\alpha \in \mz_p^{\times}$ so that $\alpha - 1 \in \mz_p^{\times}$, and the assumption that $k \in K$ implies that $v_p(u) = 0$ or $v_p(v) = 0$.  This shows that \eqref{eq:matrixthingy} is not in $K$, since $-u(\alpha-1) x \not \in \mz_p$ or $-v(\alpha-1)x \not \in \mz_p$. If $p=2$ then we also have by hypothesis that $v_p(x) < -1$, and so we can choose $\alpha = 3$ so that $-x(\alpha -1) \not \in \mz_2$. 
\end{proof}
In fact, the hypothesis that $v_2(x)\leq -2$ in Lemma \ref{x=0} is necessary. 
The above calculations show that with $K'_2 = b K_2 b^{-1}$ and $x \in 2^{-1} \mz_2^{\times}$, then $1_{ZK'_2}$ is bi-$A(\mz_2)$-invariant.  Take for instance, $y=1/2$ and $x=-1/2$. Then we can check
\begin{equation}
\label{eq:conjugationFormulap=2}
\begin{pmatrix} 1/2 & -1/2 \\ & 1 \end{pmatrix}^{-1}
\begin{pmatrix} r & t \\ u & v \end{pmatrix}
\begin{pmatrix} 1/2 & -1/2 \\ & 1 \end{pmatrix}
= \begin{pmatrix} r + u/2 & 2t - r + v - u/2 \\ u/2 & v - u/2 \end{pmatrix}.
\end{equation}
Similarly,
\begin{equation}
\label{eq:conjugationFormulap=2_2}
\begin{pmatrix} y &  \\ & 1 \end{pmatrix}^{-1}
\begin{pmatrix} \alpha & \beta \\ \gamma & \delta \end{pmatrix}
\begin{pmatrix} y & \\ & 1 \end{pmatrix}
= \begin{pmatrix} \alpha & \beta/ y\\ \gamma y & \delta \end{pmatrix}.
\end{equation}
The upper-left and lower-right corners of the matrix in \eqref{eq:conjugationFormulap=2_2} 
can never leave $\mz_2$, so there does not exist $y \in \mq_2^{\times}$ such that
$K'_2 \subseteq a(y)^{-1} K_2 a(y)$.

{\bf Standing assumptions.} We henceforth assume that Geometric Assumptions \eqref{geo2} and \eqref{geo3} are in force from here until the end of Section \ref{sec:KloostermanSums} and so they may not be explicitly mentioned in the statements of lemmas, propositions, theorems and corollaries.  

Given $y\in \Q_+$ for which $\supp f \subseteq ZK'$ with $K' = a(y)^{-1}Ka(y)$ as afforded by Geometric Assumption \eqref{geo3}, we write 
 \begin{equation}\label{ffin*def}f'(g) = f(a(y)^{-1}ga(y))\end{equation} so that $f'$ is supported in $ZK$.

\begin{mylemma}\label{*defect} If  $f$ is of level $N$ and has support controlled by $y$, then $f'$ is bi-$K(M)$-invariant, where $M= N \xi$ and $\xi = \lcm(y,y^{-1})$. 
\end{mylemma}
\begin{proof} 
By a direct calculation, we see that for all $m \in K(M)$ that there exists $n \in K(N)$ such that $a(y)n = ma(y)$. Then, 
$$ f'(gm) = f(a(y)^{-1}gma(y)) = f( a(y)^{-1} g  a(y)n) = f'(g).$$ The left invariance is similar. 
\end{proof}
 See Section \ref{sec:MTcomputation} for an application of Lemma \ref{*defect} to the computation of the diagonal term in Theorem \ref{MT}.

We conclude this section by giving a lemma that relates the support of $f$ to its level. 

\begin{mylemma}\label{moduli}
Suppose $f$ is not identically $0$ and has level $N$. Any $y$ controlling the support of $f$ satisfies $yN\in \N$ and $N/y \in \N$. 
\end{mylemma}

\begin{proof}
First we show that $K(N) \subseteq a(y)^{-1}K a(y)  =K' $. To do this, we use that $K'$ is a group.  Let $g\in \supp f \subseteq K'$. Then, since $f$ is right $K(N)$-invariant and $\supp f \subseteq K'$, we have $gk \in K'$ for any $k\in K(N)$. Thus, $k\in g^{-1}K'=K'$. 

Now, $\left( \begin{smallmatrix} 1 & N \\  & 1 \end{smallmatrix}\right) \in K(N)$, so $a(y) \left( \begin{smallmatrix} 1 & N \\  & 1 \end{smallmatrix}\right) a(y)^{-1} = \left( \begin{smallmatrix} 1 & Ny \\  & 1 \end{smallmatrix}\right) \in K$, thus $Ny \in \widehat{\Z}$. Similarly, $a(y) \left( \begin{smallmatrix} 1 & \\ N & 1 \end{smallmatrix}\right) a(y)^{-1}=\left( \begin{smallmatrix} 1 &  \\ N/y  & 1 \end{smallmatrix}\right) \in K$, so $N/y \in \widehat{\Z}$. Since $\widehat{\Z} \cap \Q_+ = \N$, this finishes the proof. 
\end{proof}

Note Lemma \ref{moduli} also shows that if $y$ controls the support of $f$ and $f$ has level $N$, then with $\xi$ as in Lemma \ref{*defect}, $\xi \mid N$, so that the level of $f'$ is at most $N^2$.

\subsection{Control on the geometric conductor}\label{sec:control_on_geom_cond}

Recall  Definition \ref{def:GenKloostermanDef} of the generalized Kloosterman sums $H(m_1,m_2;c)$. The sum $H(m_1,m_2;c)$ vanishes unless both $m_1,m_2 \in \frac{1}{N} \Z$. Indeed, by the left-$K(N)$-invariance of $f$, we have  
$$H(m_1,m_2;c) = \psi_{\rm fin}(m_1 N) H(m_1,m_2;c) ,$$ 
so $H(m_1,m_2;c)=0$ unless $m_1 \in \frac{1}{N}\widehat{\Z} \cap \Q = \frac{1}{N}\Z$, and similarly for $m_2$ by the right-$K(N)$-invariance of $f$. As an aside, the fact that $H(m_1,m_2;c)$ vanishes unless $m_1,m_2 \in \frac{1}{N}\Z$ is in perfect accord with the spectral side and first cell terms of Theorem \ref{MT}.
\begin{mylemma}\label{ccondition}
Let $y\in \Q_+$ control the support of $f$. The generalized Kloosterman sum $H(m_1,m_2;c)=0$ unless $c \in y\N$, in which case 
\begin{equation}
\label{eq:orbitalintegralformulaf*}
H(m_1,m_2;c)= \frac{1}{|c|_{\rm fin}^2}\iint_{\A_{\rm fin}^2} f'\left( \left( \begin{smallmatrix} -t_1 & \frac{-y(1+t_1t_2)}{c} \\ \frac{c}{y} & t_2 \end{smallmatrix}\right)  \right) \psi_{\rm fin}\Big(\frac{m_1t_1- m_2 t_2}{c}\Big)\,dt_1\,dt_2,
 \end{equation}
where $f'$ is given by equation \eqref{ffin*def}. The integration may be restricted to $t_1,t_2 \in \widehat{\Z}$ and $t_1t_2 \equiv -1 \pmod {cy^{-1}\widehat{\Z}}$. 
\end{mylemma}
\begin{proof}

Following the notation in Section \ref{sec:secondcellterms}, write $\mu = c^{-2}$ with $c \in \Q_+$. Then 
\begin{equation}
\label{eq:matrixcondition}
\left( \begin{smallmatrix}1 & -t_1 \\ & 1 \end{smallmatrix}\right)\left( \begin{smallmatrix} & -\mu \\ 1&  \end{smallmatrix}\right) \left( \begin{smallmatrix}1 & t_2 \\ & 1 \end{smallmatrix}\right)=  \left( \begin{smallmatrix}-t_1 & -\mu -t_1t_2 \\ 1 & t_2 \end{smallmatrix}\right)  \in ZK_p'
  \Longleftrightarrow \left( \begin{smallmatrix}-t_1c & (-\mu -t_1t_2)c \\ c & t_2c \end{smallmatrix}\right) \in K_p'.
  \end{equation}
 Let $t_1' = t_1c$ and $t_2' = t_2c$. Then \eqref{eq:matrixcondition} holds
  if and only if $( \begin{smallmatrix}-t_1' & (-1 -t_1't_2' )/c \\ c & t_2' \end{smallmatrix}) \in K_p'$.
    Changing variables in \eqref{GenKloostermanDef} accordingly, we find
$$H(m_1,m_2; c) = \frac{1}{|c|_{\rm fin}^2}\iint_{\A_{\rm fin}^2} f\left( \left( \begin{smallmatrix} -t_1 & \frac{-1-t_1t_2}{c} \\ c & t_2 \end{smallmatrix}\right)  \right) \psi_{\rm fin}\left(\frac{m_1t_1- m_2 t_2}{c}\right)\,dt_1\,dt_2.$$ 

Recall the definition of $f'$ from \eqref{ffin*def} and note that $f'$ is supported in $ZK$ by Geometric Assumption \eqref{geo3}. 
For $y \in \Q_+$ controlling the support of $f$ as in \eqref{ffin*def}, we have 
$$ a(y) \left( \begin{smallmatrix} -t_1 & \frac{-1-t_1t_2}{c} \\ c & t_2 \end{smallmatrix}\right) a(y)^{-1}
= 
\left( \begin{smallmatrix} -t_1 & \frac{-y(1+t_1t_2)}{c} \\ \frac{c}{y} & t_2 \end{smallmatrix}\right),
$$
from which \eqref{eq:orbitalintegralformulaf*} follows by substitution.
Now this integral vanishes unless $cy^{-1} \in \widehat{\Z}$. Note also that the integration here may be restricted to $t_1,t_2 \in \widehat{\Z}$ and $t_1t_2 \equiv -1 \pmod {cy^{-1} \widehat{\Z}}$. 
\end{proof}
In terms of the geometric conductor $k(\cF)$, Lemma \ref{ccondition} asserts that  $y \mid k(\cF)$.
Technically, we have not defined $k(\cF)$ if $\cC(\cF) = \varnothing$,
but in fact the next Lemma shows that $\cC(\cF)$ is non-empty and indeed provides an upper bound on $k(\cF)$ if one has information about the possible lower-left entries of matrices on which $f'$ is supported. 

\begin{mylemma}\label{lem:admmodulus} Suppose that $f$ has level $N$ and support controlled by $y \in \Q_+$, and $f'$ as in \eqref{ffin*def} has level $M$. 
Suppose that $c \in \mq_{+}$ and $g = (\begin{smallmatrix} g_1 & g_2 \\ g_3 & g_4 \end{smallmatrix}) \in K$ are such that $cN \equiv 0 \pmod{M}$, $f'(g) \neq 0$, $\det(g) \equiv 1 \pmod{cy^{-1}M}$, and $cy^{-1} \equiv g_3 \pmod{cy^{-1}M}$.
Then $c$ is an admissible modulus.
\end{mylemma}

\begin{proof} The idea is to apply a version of the Plancherel formula to $H(m,n;c)$. 
Note by the second sentence of this section and \eqref{eq:orbitalintegralformulaf*}  that $H(m/N,n/N;c)$ is periodic in $m,n$ modulo $cN$. 
By Lemma \ref{ccondition} 
\begin{multline}\label{approxPlancherel}
\frac{1}{(cN)^2}\sum_{m,n \in \Z/cN\Z} |H(m/N,n/N;c)|^2 = \\ \frac{1}{|c|_{\rm fin}^4} \iint_{{\widehat{\Z}^2}}\iint_{{\widehat{\Z}^2}}  f'\left( \left( \begin{smallmatrix} -t_1 & \frac{-y(1+t_1t_2)}{c} \\ \frac{c}{y} & t_2 \end{smallmatrix}\right) \right)  \overline{ f'\left( \left( \begin{smallmatrix} -u_1 & \frac{-y(1+u_1u_2)}{c} \\ \frac{c}{y} & u_2 \end{smallmatrix}\right) \right) }
\delta_{\substack{ t_1\equiv u_1 \shortmod {cN} \\ t_2\equiv u_2 \shortmod {cN} }} \,dt_1\,dt_2\,du_1\,du_2 \\
=  \frac{1}{|c|_{\rm fin}^4(cN)^2} \iint_{{\widehat{\Z}^2}} \left|  f'\left( \left( \begin{smallmatrix} -t_1 & \frac{-y(1+t_1t_2)}{c} \\ \frac{c}{y} & t_2 \end{smallmatrix}\right) \right) \right|^2 \,dt_1\,dt_2,
\end{multline}
using that $Nc \equiv 0 \pmod{M}$ and that
$f'$ is bi-$K(M)$-invariant.
 The set 
 $$S_{g_1,g_4}:=\{(t_1,t_2) \in \widehat{\Z}^2: t_1 \equiv -g_1 \pmod{cy^{-1}M}, \, t_2 \equiv g_4\pmod{cy^{-1}M}\}$$
 has positive measure in $\A_{\rm fin}^2$. For any $(t_1,t_2) \in S_{g_1,g_4}$, we have 
 $$ -1-t_1t_2 \equiv -1+g_1g_4 \equiv g_2g_3 \equiv g_2cy^{-1} \pmod{cy^{-1}M}$$
 by the hypotheses that $\det(g) \equiv 1 \pmod{cy^{-1}M}$ and $g_3 \equiv cy^{-1} \pmod{cy^{-1}M}$. Therefore $$g \equiv \left( \begin{array}{cc} -t_1 & \frac{-y(1+t_1t_2)}{c} \\ \frac{c}{y} & t_2 \end{array}\right) \pmod{M}.$$
Hence  $|  f'\left( \left( \begin{smallmatrix} -t_1 & \frac{-y(1+t_1t_2)}{c} \\ \frac{c}{y} & t_2 \end{smallmatrix}\right) \right)| = |f'(g)|>0$ for all $(t_1,t_2) \in S_{g_1,g_4}$, so that \eqref{approxPlancherel} is non-vanishing by positivity. 
\end{proof}
Lemma \ref{lem:admmodulus} implies that $\cC(\cF)$ is non-empty and hence that $k(\cF)$ exists. The following corollary makes the upper bound on $k(\cF)$ afforded by Lemma \ref{lem:admmodulus} explicit in a special case. 
\begin{mycoro}\label{cor:admmodulus}
Suppose that $f$ has level $N$ and satisfies the geometric assumptions. If $f(1) \neq 0$, then $k(\cF) \mid N$.
\end{mycoro}
\begin{proof}
Let $y$ control the support of $f$. Since $f(1)\neq 0$ we have $f\left( \left( \begin{smallmatrix} 1 & \\ N & 1 \end{smallmatrix}\right)\right) \neq 0$ and so $f'\left( \left( \begin{smallmatrix} 1 & \\ Ny^{-1} & 1 \end{smallmatrix}\right)\right) \neq 0$, where by definition $f'(g)= f(a(y)^{-1}ga(y))$ (see \eqref{ffin*def}). Writing $M$ for the level of $f'$, we have $M/N \mid \lcm(y ,y^{-1}) \mid N$, by Lemmas \ref{*defect} and \ref{moduli}. Then, Lemma \ref{lem:admmodulus} shows that $N$ is an admissible modulus for $f$, as $N \equiv 0 \pmod{M/N}$, $\det \left( \begin{smallmatrix} 1 & \\ Ny^{-1} & 1 \end{smallmatrix}\right)= 1$, and $N \equiv yNy^{-1}\pmod{MN}$.  
\end{proof}
\begin{myexam}
 Consider the classical case that $f = \nu(N) 1_{ZK_0(N)}$. Then, $\supp f \subseteq ZK'$ with $K' = \left( \begin{smallmatrix} N^{-1} & \\ & 1\end{smallmatrix}\right) K  \left( \begin{smallmatrix} N & \\ & 1\end{smallmatrix}\right)$, so $N$ controls the support of $f$.  Both $f$ and $f'$ have level $N$.

By Lemma \ref{ccondition} applied with $y=N$, we have that $\cC(\cF) \subseteq N\N$. On the other hand, let $g=\left( \begin{smallmatrix} 1 & \\ 1 & 1\end{smallmatrix}\right)$. Then $f'(g) \neq 0$ and $\det g = 1$ with $g_3=1$. Since $c= N \equiv N \pmod {N^2}$, Lemma \ref{lem:admmodulus} shows that $N \in \cC(\cF)$. Thus, $N \mid k(\cF)$, so that $k(\cF) =N$.  
\end{myexam}

\subsection{Kloosterman sum properties}\label{sec:KP}
The main goal of this section is to prove the following.
\begin{mytheo}\label{thmKP} Let $f \in \cH_{\rm fin}$ satisfy the geometric assumptions with level $N$ and support controlled by $y \in \Q_+$ (as defined in Section \ref{sec:prelims_on_support}). 
The generalized Kloosterman sum $H(m,n;c)$ enjoys the following properties:
\begin{enumerate}
\item The sum $H(m,n;c)$ is \`a priori a function of $m,n \in \Q$ and $c \in y\N$, but vanishes unless both $m,n \in \frac{1}{N} \Z$.
\item We have $H(m+ac, n+bc;c) = H(m,n;c)$ for any $a, b \in \Z$.
\item Factoring $c$ as $c=c_0c_{N}$ with $c_0 \in \N$, $(c_0,N)=1$ and $c_{N}$ a product of primes (to positive or negative powers) that divide $N$, we have 
\begin{equation}
\label{KloostermanFactorization}
H(m,n;c) = S(\overline{c_{N}}m, \overline{c_{N}}n;c_0)  H(m \overline{c_0}, n \overline{c_0} ;c_{N}),
\end{equation}
where $\overline{c_{N}}$ is any integer such that $c_{N}\overline{c_{N}} \equiv 1 \pmod {c_0}$,  $\overline{c_0}$ is any integer such that $c_0\overline{c_0} \equiv 1 \pmod {Nc_{N}}$, and $S(m,n;c)$ is the classical Kloosterman sum.

\item If neither the numerator nor the denominator of $n$ is divisible by ramified primes of $f$, then 
\begin{equation}\label{KloostermanFactorization2}
H(m,n;c) = S(\overline{c_{N}}m, \overline{c_{N}}n;c_0) H(mn\overline{c_0}^2, 1;c_{N}).
\end{equation}
\item The sums $H(m,n;c)$ satisfy the trivial bound
\begin{equation}
\label{trivialBoundonK} 
|H(m,n;c)| \leq cy \cdot \| f \|_{L^\infty(G)}.
\end{equation} 
\item Let $k_p \in \Z$ be minimal such that  $H_p(m,n,p^{k})$ is not identically 0, where $H_p$ is the local Kloosterman sum defined in \eqref{eq:KloostermanLocalFormula} below. The geometric conductor factors as
$$k(\cF) = \prod_p p^{k_p}.$$
\end{enumerate}
\end{mytheo}

Recall from Section \ref{level_vs_ramification} that (under Geometric Assumption \eqref{geo3}) the primes of ramification are precisely those that divide $N$.
\begin{proof}[Proof of Theorem \ref{thmKP}] 
\hfill
\begin{enumerate}[wide]
\item See the second sentence of Section \ref{sec:control_on_geom_cond} and the first assertion of Lemma \ref{ccondition}.

\item This follows immediately from Lemma \ref{ccondition}. 

\item 
 As $f\in \cH_{\rm fin}$ is a pure tensor, we have the factorization
\begin{equation}\label{Hfactorization}H(m_1,m_2;c) = \prod_p H_p(m_1,m_2;c),\end{equation}
where
\begin{equation}
\begin{split}
\label{eq:KloostermanLocalFormula}
H_p(m_1,m_2;c) &=    \iint_{\Q_{p}^2} f_{p}\left( \left( \begin{smallmatrix}-t_1 & -c^{-2} -t_1t_2 \\ 1 & t_2 \end{smallmatrix}\right) \right) \psi_{p}(m_1t_1- m_2 t_2)\,dt_1\,dt_2 \\
&= \frac{1}{|c|_p^2} \iint_{\Q_{p}^2} f_{p}'\left( \left( \begin{smallmatrix}-t_1 & -y(1+ t_1t_2)/c\\ \frac{c}{y} & t_2 \end{smallmatrix}\right) \right) \psi_{p}\left(\frac{m_1t_1- m_2 t_2}{c}\right)\,dt_1\,dt_2.
\end{split}
\end{equation}
Let us factor $N$ as $N= N^{(p)}N_p$, where $N_p \mid p^\infty$ and $p \nmid N^{(p)}$. 
We now state and prove a lemma that will be useful for multiple parts of the proof of Theorem \ref{thmKP}.
\begin{mylemma}\label{modulusshift}
Write $c=c_0 p^{v_p(c)}$ where $c_0 \in \Q_+ \cap \Z_p^\times$. Then, for any $m,n \in \frac{1}{N}\Z$, we have 
$$H_p(m,n;c) = H_p(m/c_0, n/c_0 ;p^{v_p(c)})= H_p(m\overline{c_0},n\overline{c_0};p^{v_p(c)}),$$
where $m/c_0, n/c_0 \in \frac{1}{N_p}\Z_p$, and $\overline{c_0}$ is any integer with $c_0 \overline{c_0} \equiv 1 \pmod{N_pp^{v_p(c)}\Z}$. 
\end{mylemma}
\begin{proof}
We have from \eqref{eq:KloostermanLocalFormula}, changing variables $t_i \rightarrow t_i/c_0$ 
and using $Z(\Q_p)$-invariance
$$
H_p(m,n;c) = 
   \iint_{\Q_p^2} f_p\left(  \left( \begin{smallmatrix}-t_1 & (-p^{-2v_p(c)} -t_1t_2 )/c_{0} \\ c_{0} & t_2 \end{smallmatrix}\right)  \right) \psi_p\Big(\frac{mt_1-nt_2}{c_{0}}\Big)\,dt_1\, dt_2.
$$
 We also have
$$\left( \begin{smallmatrix}-t_1 & (-p^{-2v_p(c)} -t_1t_2 )/c_{0} \\ c_{0} & t_2 \end{smallmatrix}\right)  =  \left( \begin{smallmatrix}c_{0} &  \\  &1 \end{smallmatrix}\right)^{-1} \left( \begin{smallmatrix}-t_1 & 
-p^{-2v_p(c)} -t_1t_2  \\ 1 & t_2 \end{smallmatrix}\right) \left( \begin{smallmatrix}c_{0} &  \\  & 1 \end{smallmatrix}\right),$$ so that by Geometric Assumption \eqref{geo2}
$$H_p(m,n;c) = H_p(m/c_{0}, n/c_{0}; p^{v_p(c)}).$$

We claim that an integer $\overline{c_0}$ exists as in the statement of the lemma (despite the fact that $c$ need not be an integer). Indeed, we have by Lemmas \ref{moduli} and \ref{ccondition} that $cN$ is an integer, which we may factor into its $p$-part $p^{v_p(c)}N_p$ and prime-to-$p$-part $c_0N^{(p)}$, both of which are also integers. Then $(p^{v_p(c)}N_p, c_0N^{(p)})=1$, so there exists $a\in \Z$ with $(a,p)=1$ such that $c_0N^{(p)} a \equiv 1 \pmod {p^{v_p(c)}N_p}.$ Setting $\overline{c_0} = a N^{(p)}$, we have that $c_0\overline{c_0} \equiv 1 \pmod  {p^{v_p(c)}N_p}$ with $p \nmid \overline{c_0}$. 

Now, $H_p( \cdot, \cdot; \cdot)$ is a function of $\frac{1}{N_p} \Z_p$ in the first two entries.  Viewing $m\overline{c_0}$ as an element of $\frac{1}{N_p}\Z_p$, we have  $$m\overline{c_0} = \frac{mc_0\overline{c_0}}{c_0} \equiv \frac{m}{c_0} \pmod{ p^{v_p(c)}\Z_p}.$$
By periodicity, we have $H_p(m/c_0, n/c_0;p^{v_p(c)})= H_p(m\overline{c_0},n\overline{c_0};p^{v_p(c)})$, as was to be shown.
\end{proof}

Now we prove (3).
 If $p$ is unramified then following \cite[Prop.\ 3.7]{KLPetersson} we have
\begin{equation}
\label{KpClassical}
H_p(m_1,m_2;c) = 
\sum_{\substack{ t_1,t_2 \in (p^{-v_p(c)} \Z_p / \Z_p)^\times \\ t_1t_2 = c^{-2} \shortmod{\Z_p}}} \psi_p \left(  m_1t_1+m_2 t_2 \right) .
\end{equation}
In particular, writing $c=c_0 p^{v_p(c)}$ we have
\begin{equation}\label{KmncClassical}H_p(m_1,m_2;c) = S(\overline{c_0} m_1 , \overline{c_0}m_2, p^{v_p(c)}).\end{equation}

By Lemma \ref{modulusshift}, we have for any $c \in \Q_+$ that 
\begin{multline*}
H(m,n;c) = \prod_{p \text{ unr}} H_p(m,n;c_0c_{N}) \prod_{p \text{ ram}} H_p(m,n;c_0c_{N}) \\ = \prod_{p \text{ unr}} H_p(m\overline{c_{N}},n\overline{c_{N}};c_0) \prod_{p \text{ ram}} H_p(m\overline{c_0},n\overline{c_0};c_{N}).
\end{multline*}
Let us write $c_0 = c_{00}p^{v_p(c_0)}$, with $(c_{00}, p)=1$. Then by \eqref{KmncClassical} we have for $p$ unramified
$$H_p(m\overline{c_{N}},n\overline{c_{N}};c_0) = S(m\overline{c_{N}c_{00}},n \overline{c_{N}c_{00}},p^{v_p(c_0)}) .$$
Inserting this above and using the twisted multiplicativity of classical Kloosterman sums we get
$$
H(m,n;c) = 
S(m\overline{c_{N}}, n\overline{c_{N}};c_0)  
\prod_{p \text{ ram}} H_p(m\overline{c_0},n\overline{c_0};c_{N}).
$$
For the 2nd factor, note that
\begin{multline*}
 H(m \overline{c_0}, n \overline{c_0} ;c_{N})= \prod_p H_p(m \overline{c_0}, n \overline{c_0} ;c_{N})  
 =  \prod_{p \text{ unr}} H_p(m \overline{c_0}, n \overline{c_0} ;c_{N})  \prod_{p \text{ ram}}  H_p(m \overline{c_0}, n \overline{c_0} ;c_{N})  \\
= \prod_{p \text{ ram}} H_p(m\overline{c_0}, n \overline{c_0} ;c_{N}),
\end{multline*} since for $p$ unramified and $v_p(c) =0$, \eqref{KpClassical} reduces to a single term, so $H_p(m \overline{c_0}, n \overline{c_0} ;c_{N}) = 1$. This concludes the proof of item (3). 

\item  From \eqref{eq:KloostermanLocalFormula} by a change of variables we have for $n \neq 0$
$$H_p(m,n;c) = \iint_{\Q_p^2} f_p\left( \left( \begin{smallmatrix} -nt_1 & -\frac{1}{c^2} - t_1t_2 \\ 1 & t_2/n \end{smallmatrix} \right)\right) \psi_p(mnt_1-t_2)\,dt_1\,dt_2.$$
Observe that 
$$
\begin{pmatrix} -nt_1 & -\frac{1}{c^2} - t_1t_2 \\ 1 & t_2/n \end{pmatrix} =
\begin{pmatrix} n & \\  & n \end{pmatrix}^{-1}
 \begin{pmatrix} n & \\  & 1 \end{pmatrix} 
\begin{pmatrix} -t_1 & -\frac{1}{c^2} - t_1t_2 \\ 1 & t_2 \end{pmatrix} 
\begin{pmatrix} n &  \\  &1 \end{pmatrix}.
$$
Now we suppose that $n \in \Z_p^\times$. Under this additional hypothesis, by $Z(\Q_p)$-invariance and Geometric Assumption \eqref{geo2} we have 
\begin{equation}\label{Hpm,n=Hpmn}
H_p(m,n;c) = H_p(mn,1;c).\end{equation}
Thus, when neither the numerator nor denominator of $n$ is divisible by a ramified prime, 
$$H(m,n;c) = S(m\overline{c_{N}}, n\overline{c_{N}};c_0)  H(mn \overline{c_0}^2, 1 ;c_{N}). $$

\item 
By Lemma \ref{ccondition} we have that 
$$ |H(m_1,m_2;c)| \leq c^2 \| f \|_{L^\infty(G)} \vol\{ (t_1,t_2) \in \widehat{\Z}^2 : t_1t_2=-1\pmod {cy^{-1} \widehat{\Z}}\},$$ 
and that 
\begin{multline*} 
\iint_{\substack{(t_1,t_2) \in \widehat{\Z}^2 \\ t_1t_2 = -1 \shortmod{cy^{-1}\widehat{\Z}}}} 1\,dt_1dt_2 =  \int_{\substack{ t_1 \in \widehat{\Z} \\ t_1 \text{ invertible } \shortmod{cy^{-1}\widehat{\Z}}}} \int_{t_2 \equiv -t_1^{-1} \shortmod{cy^{-1}\widehat{\Z}}} 1 \,dt_1 \,dt_2 \\ =  \frac{1}{cy^{-1}}\int_{ t_1 \in (\widehat{\Z}/cy^{-1}\widehat{\Z})^\times}1 \, dt_1 = \frac{y^2\varphi(cy^{-1})}{c^2},
\end{multline*}
from which \eqref{trivialBoundonK} follows. 

\item  
First we show that $\cC(\cF) \subseteq \left( \prod_p p^{k_p} \right) \Z$. Indeed, let $c \in \cC(\cF)$. Then there exists $m,n$ such that $H(m,n;c) \neq 0$ and thus $H_p(m,n;c) \neq 0$ for all $p$. Using Lemma \ref{modulusshift}, we have $k_p \leq v_p(c)$ for all $p$. Thus, 
$c \in \left( \prod_p p^{k_p} \right)  \Z$, as was to be shown.

Second, we show that $q'=\prod_p p^{k_p}$ is maximal for the property that $\cC(\cF) \subseteq q' \Z$.  Let $S$ denote the set of primes ramified for $f$.

We claim that if $H(\cdot,\cdot;c)$ vanishes identically for some $c$, then there exists a $p \in S$ such that $H_p(\cdot,\cdot;p^{v_p(c)})$ vanishes identically. Indeed, by \eqref{Hfactorization} and the fact that $H_p (m,n; c)=1$ for all $m,n \in \frac{1}{N}\Z/c\Z$ if $p \not \in S$ and $p\nmid c$ (see \eqref{KpClassical}), we have that 
$$
H(m,n;c) = \prod_{p \in S \text{ or } p \mid c} H_p(m,n;c)
$$
for all $m,n \in \frac{1}{N}\Z/c\Z$.  Now, suppose that there exists a prime $\ell \in S$ or $\ell \mid c$ such that $H_\ell(\cdot, \cdot;\ell^{v_p(c)})$ does not vanish identically. Then, there exists $a_\ell,b_\ell \in \frac{1}{\ell^{v_\ell(c)}} \Z/\ell^{v_{\ell}(c)}\Z$ such that $H_\ell( \overline{c_0N_0} a_\ell,  \overline{c_0N_0} b_\ell ; \ell^{v_\ell(c)}) \neq 0$. Then, for all $a_0,b_0 \in \frac{1}{N_0}\Z/c_0\Z$ there exists by the Chinese remainder theorem $m,n \in \frac{1}{N}\Z/c\Z$ such that 
$$ 
\begin{cases}
 N_0 m \equiv a_\ell \pmod {\ell^{v_\ell(c)}}, \\  \ell^{v_\ell(N)} m \equiv a_0 \pmod{c_0}, \end{cases} 
 \quad \text{ and } \quad \begin{cases} N_0 n \equiv b_\ell \pmod {\ell^{v_\ell(c)}}, \\ \ell^{v_\ell(N)} n \equiv b_0 \pmod{c_0}. \end{cases}
 $$
Since $H(\cdot,\cdot;c)$ vanishes identically, we have by Lemma \ref{modulusshift} and the periodicity of the $H_\ell$ (cf. Theorem \ref{thmKP}(2)) that 
$$
0  = H_\ell(\overline{c_0N_0}a_\ell, \overline{c_0N_0}b_\ell; \ell^{v_\ell(c)}) \prod_{\substack{p \in S \text{ or } p \mid c \\ p \neq \ell}} H_p(\overline{\ell}^{v_\ell(cN)} a_0 , \overline{\ell}^{v_\ell(cN)} b_0; c_0).
$$
Since the $H_\ell$ factor is not equal to 0, the second factor must be 0 for all $a_0,b_0\in \frac{1}{N_0}\Z/c_0\Z$ and so vanishes identically. Therefore, $H_p(\cdot,\cdot;p^{v_p(c)})$ vanishes identically for some $p \in S$ or $p\mid c$. The factors at primes $p\mid c$ and $p\not \in S$ are classical Kloosterman sums (see \eqref{KmncClassical}) and by Corollary \ref{cor:admmodulus} with $N=1$ and $f_p=1_{ZK_p}$  these do not vanish identically. Therefore $H_p(\cdot,\cdot;p^{v_p(c)})$ vanishes identically for some $p \in S$.

Now we show that $\prod_p p^{k_p} \in \cC(\cF)$. Suppose not, then $H(\cdot, \cdot ;\prod_p p^{k_p})$ vanishes identically. By the claim, $H_p(\cdot,\cdot; p^{k_p})$ vanishes identically for some $p \in S$. This contradicts the definition of $k_p$. Thus, $\prod_p p^{k_p} \in \cC(\cF)$. If there were a $q'$ such that $\prod_p p^{k_p}$ was a proper divisor of $q'$ and $\cC(\cF) \subseteq q'\Z$, then $\prod_p p^{k_p} \not \in q'\Z$ and yet $\prod_p p^{k_p} \in \cC(\cF)$. Contradiction. So, $q'=\prod_p p^{k_p}$ is maximal for the property that $\cC(\cF) \subseteq q' \Z$.
\end{enumerate}
\end{proof}
Recall that we only claimed the unrefined relative trace formula Theorem \ref{MT} for $h_\infty$ lying in a Paley-Wiener space of functions. However, using Theorem \ref{thmKP}(3)(5), we may follow the same technique as in \cite[Ch.\ 8]{knightly_kuznetsovs_2013} to enlarge the space of admissible archimedean test functions. \begin{mycoro}
If $f$ satisfies the geometric assumptions, then Theorem \ref{MT} holds for all $h_\infty(t)$ that are even, holomorphic in $|\imag t|<A$ for $A$ sufficiently large, and that satisfy $h_\infty(t)\ll (1+|t|)^{-(2+\delta)}$ for some $\delta>0$. 
\end{mycoro}
\begin{proof}
Following the proof of \cite[Thm.\ 8.1]{knightly_kuznetsovs_2013}, the only point to be checked is the convergence of the sum of Kloosterman sums in the proof of Proposition 8.24 of loc.\ cit.  In our case, we factor $H(m,n;c)$ as in Theorem \ref{thmKP}(3), and apply Theorem \ref{thmKP}(5) to the $c_N$-part and the Weil bound to the $c_0$-part to obtain the required absolute convergence. \end{proof}

We end this section with one more consequence of the geometric assumptions that is entirely local in nature. 
\begin{mylemma}\label{transport}
Suppose that $f \in \cH_p$ satisfies Geometric Assumption \eqref{geo3}.  If $\pi(\chi,\chi^{-1}) \in \cF_{p}(f)$ 
 and $s \in \C$ is such that $\pi(\chi \alpha^s, \chi^{-1}\alpha^{-s})$ is irreducible, then $\pi(\chi \alpha^s, \chi^{-1}\alpha^{-s}) \in \cF_{p}(f)$.
\end{mylemma}
\begin{proof}
First, for any $\pi \in {\overline{G}(\Q_p)}^{\wedge}$, we have that $\pi(f) = \pi(a(y))^{-1} \pi(f') \pi(a(y))$, 
where $f'$ is defined as in \eqref{ffin*def}, 
so that $\pi(f) \neq 0$ if and only if $\pi(f') \neq 0$. Therefore, it suffices to show the lemma under the assumption that $f$ has support contained in $ZK_p$. 

Recall that if $\chi$ and $ \tilde{\chi}$ are equal when restricted to $\Z_p^\times$, then $\pi(\chi, \chi^{-1}) \simeq \pi(\tilde{\chi} , \tilde{\chi}^{-1})$ as representations of $K_p$. Indeed, using the induced model we define a map $i:\pi(\chi, \chi^{-1}) \to \pi(\tilde{\chi},  \tilde{\chi}^{-1})$ by
$$i: h \mapsto \tilde{h} \quad \text{ where } \quad \tilde{h}: g= bk \mapsto \delta(b)^{1/2} \tilde{\chi}(b) h(k),$$
and it is easy to check that $i$ is a $ZK_p$-intertwiner.

Now let us write $\pi = \pi(\chi, \chi^{-1})$ and $\tilde{\pi} = \pi(\tilde{\chi} , \tilde{\chi}^{-1})$. We have just shown that $(\pi(k)v)\tilde{\phantom{v}} = \tilde{\pi}(k)\tilde{v}$ for all $k \in K_p$ and since $f$ is supported in $ZK_p$, we have
\begin{equation}\label{eq:transport}
(\pi(f)v)\tilde{\phantom{v}} = \int_{{K_p}} f(k) (\pi(k)v)\tilde{\phantom{v}}\,dk = \int_{{K_p}} f(k) \tilde{\pi}(k)\tilde{v}\,dk= \tilde{\pi}(f)\tilde{v}.
\end{equation}
Therefore $\pi(f) \neq 0$ if and only if $\tilde{\pi}(f) \neq 0$. 
\end{proof}
\begin{myrema}\label{rem_following_lem_transport} Consider the remaining case that $\pi(\chi,\chi^{-1}) \in \cF_{p}(f)$ and $\pi(\chi \alpha^s, \chi^{-1}\alpha^{-s})$ is reducible. Suppose in addition to the assumptions of Lemma \ref{transport} that $f\in \cH_p$ is a newform projector and $\chi \vert_{\Z_p^\times}$ is a non-trivial quadratic character. We claim that if $\pi(\chi,\chi^{-1}) \in \cF_p(f)$, then $\St \times \chi$ and $\St \times \chi \eta$ are in $\cF_p(f)$ as well, where $\eta$ is the unramified quadratic character of $\Q_p^\times$. Indeed, write $\pi =\pi(\chi,\chi^{-1})$ so that $\pi(f)v\neq 0$ spans the $1$-dimensional space $V_\pi^{K_0(p^{2c(\chi)})}$.  Let $\tilde{\pi}= \pi(\chi \alpha^s, \chi^{-1}\alpha^{-s})$ for some $s \in \C$ be the reducible principal series representation with subquotient $\St \times \chi$ or $\St \times \chi \eta$. Then nonetheless $\pi \simeq \tilde{\pi}$ as $ZK_p$-representations, so that $(\pi(f)v)\tilde{\phantom{v}}$ spans the $1$-dimensional space $V_{\tilde{\pi}}^{K_0(p^{2c(\chi)})}$.  Finally, by \eqref{eq:transport}, the vector $\tilde{\pi}(f)\tilde{v}$ is $K_0(p^{2c(\chi)})$-invariant and since $c(\chi)>0$ one can check that the $1$-dimensional subquotient of $\tilde{\pi}$ contains no non-zero $K_0(p^{2c(\chi)})$-invariant vectors, we have that $\sigma(f) \neq 0$ with $\sigma =\St \times \chi$ or $\St \times \chi \eta$. \end{myrema}

\section{Proof of the refined relative trace formula and the spectral assumption}\label{sec:spec_assumption}

\subsection{Local spectral decomposition}\label{sec:localspectral}
In this subsection, we work in much more generality than what is required elsewhere in the paper since it is the natural context dictated by the proof we have in mind. Let $H$ be a unimodular $p$-adic linear algebraic group (i.e.\ the $F$-points of a linear algebraic group, for some non-archimedean local field $F$ of characteristic zero). In particular, $H$ is separable and locally compact.  

Let $H^{\wedge}$ denote the unitary dual of $H$, that is, the space of isomorphism classes of continuous irreducible unitary representations of $H$ on a Hilbert space \cite[\S 13.1.4]{Dixmier} endowed with the Fell topology. The unitary dual $H^{\wedge}$ may be equivalently described as the space of isomorphism classes of smooth irreducible unitary representations of $H$ on a complex vector space (for the equivalence, see e.g.\ \cite{HerzigSmoothReps}). With this definition, a result of Sliman \cite[Thm.\ 1.2.3(i)]{Sliman} building on Duflo \cite{Duflo} asserts that if $H$ is a linear algebraic group over a characteristic zero local field, then $H$ is type 1, or equivalently, is postliminal (see \cite[13.9.4, 9.1]{Dixmier}). 

Let $\Sigma$ be the Borel $\sigma$-algebra of $H^{\wedge}$ (see \cite[\S18.5]{Dixmier}). Let $\mu$ be a Haar measure on $H$. Since $H$ is a postliminal unimodular separable locally
compact group there exists a unique $\sigma$-finite measure $\widehat{\mu}$ on $(H^{\wedge},\Sigma)$ such that 
\begin{equation}\label{plancherelformula} \int_H |f(g)|^2 \,d\mu = \int_{H^{\wedge}} \| \pi(f)\|_{\rm HS}^2 \,d\widehat{\mu}\end{equation}
for all $f \in L^1(H) \cap L^2(H)$ \cite[Thm.\ 18.8.2, B30]{Dixmier}. The measure $\widehat{\mu}$ is called the Plancherel measure. 

\begin{myprop}\label{mainprop}
Let $f \in C_c^\infty(H)$.  If for all $\pi \in H^{\wedge}$ the operator $\pi(f):V_\pi \to V_\pi$ is a projection operator onto a finite dimensional subspace, then we have the spectral expansion
\begin{equation}\label{eq:mainprop}
f(g) = \int_{\pi \in H^{\wedge} } \sum_{v \in \cB_f(\pi)} \overline{\Phi_{\pi, v}(g)} \,d\widehat{\mu}(\pi)
\end{equation}
 where $\cB_f(\pi)$ is any orthonormal basis for $\imag \pi(f)$ and $\Phi_{\pi,v}=\langle \pi(g) v, v \rangle$ is the diagonal matrix coefficient of $\pi$ with respect to $v$. The integrand in \eqref{eq:mainprop} is in $L^1(H^{\wedge},\Sigma)$.
\end{myprop}
\begin{myrema} It follows from the proposition that $f = f^*$ so that the projection operator $\pi(f)$ is self-adjoint and therefore an orthogonal projection. \end{myrema}
\begin{proof}
For a function $f$ on $H$ let $f_g^*$ be the function on $H$ defined by $f^*_g(h) = \overline{f(h^{-1}g)}$. For positive $f \in C^\ast(H)$ (the enveloping $C^\ast$-algebra of $L^1(H)$) the Plancherel theorem (see \cite[\S 18.8.1]{Dixmier}) asserts that
\begin{equation}\label{PlDef}
\overline{f(g)} = \int_{H^{\wedge}} \Tr  \pi(f_g^*)\,d\widehat{\mu}(\pi),
\end{equation}
as traces on $C^\ast(H)$ (see loc.\ cit.\ \S 6 and \S17.2.5). In particular \eqref{PlDef} holds point-wise for positive $f$ that are continuous and compactly supported.   

In particular, the formula \eqref{PlDef} holds when $f$ is the indicator function of a double coset by a compact open subgroup of $H$. Since arbitrary $f \in C_c^\infty(H)$ are a finite linear combinations of indicator functions of double cosets, \eqref{PlDef} extends to $C_c^\infty(H)$ by linearity. We note that for any $g \in H$, the function $\pi \mapsto \Tr  \pi(f_g^*)$ on $(H^{\wedge},\Sigma)$ is measurable and lower semi-continuous, see \cite[Thms.\ 8.8.2(i)c.\ and 18.8.1]{Dixmier}.

Now, for each $\pi \in H^{\wedge}$, choose a representative $(\pi,V)$, an orthonormal basis $\cB_f(\pi)$ for the finite dimensional space $\imag \pi(f)\subseteq V$ (as in the statement of the proposition), and  an orthonormal basis $\cB(\pi)$ for $V$ extending $\cB_f(\pi)$.

We have 
\begin{equation}
\label{mpeq2} 
\Tr  \pi(f_g^*) = \sum_{v \in \cB(\pi)} \langle \pi(f_g^*) v,v\rangle
.
\end{equation}
There are no convergence issues in writing the sum in \eqref{mpeq2}; in fact all the terms with $v \not \in \cB_f(\pi)$ vanish. Indeed $\pi(f_g^*)= \pi(g) \pi(f)^*$, so that $\langle \pi(f_g^*) v,v\rangle =  \langle v, \pi(f) \pi(g^{-1})v\rangle,$ which vanishes if $v$ is not in $\cB_f(\pi)$.  

Pulling the integration outside the inner product, we have 
\begin{equation}\label{traceformula}\Tr \, \pi(f_g^*) =  \sum_{v \in \cB(\pi)} \int_{H} \overline{f(h^{-1}g)} \Phi_{\pi,v}(h)\,d \mu(h) = \sum_{v \in \cB(\pi)} \int_{H} \overline{f(h)} \Phi_{\pi,v}(gh^{-1})\, d \mu(h) .\end{equation}
Now, note that
\begin{align*}
\int_H \overline{f(h)} \Phi_{\pi,v}(gh^{-1})\, d \mu(h) & = \int_H \overline{f(h)} \langle \pi(gh^{-1}) v,v \rangle \,d\mu(h) \\
& = \int_H \overline{f(h)} \langle \pi(h^{-1}) v,  \pi(g^{-1})v \rangle \,d\mu(h) \\
& = \int_H \overline{f(h)} \sum_{w \in \cB'(\pi)} \langle \pi(h^{-1}) v, w \rangle \langle w, \pi(g^{-1})v \rangle \,d\mu(h),
\end{align*}
where $\cB'(\pi)$ is any basis for $(\pi, V)$ extending $\cB_f(\pi)$ and respecting the decomposition $V = \imag \pi(f) \oplus \ker \pi(f)$, which exists because $\pi(f)$ is a projection. Continuing, we have
\begin{multline*}\int_H \overline{f(h)} \Phi_{\pi,v}(gh^{-1})\, d \mu(h) = \int_H \overline{f(h)} \sum_{w \in \cB'(\pi)} \langle  v, \pi(h)w \rangle \langle w, \pi(g^{-1})v \rangle \,d\mu(h) \\ = \sum_{w \in \cB'(\pi)} \langle v, \pi(f) w\rangle \langle w , \pi^{-1}(g)v \rangle .\end{multline*}

Since  $\pi(f)$ is a projection, by definition of $\cB'(\pi)$ we have
$$\int_H \overline{f(h)} \Phi_{\pi,v}(gh^{-1})\, d \mu(h) = \sum_{v_0 \in \cB_f(\pi)} \langle v, \pi(f) v_0\rangle \langle v_0 , \pi^{-1}(g)v \rangle = \sum_{v_0 \in \cB_f(\pi)} \langle v,  v_0\rangle \langle v_0 , \pi^{-1}(g)v \rangle.$$ Inserting this back in \eqref{traceformula} and using that $\cB(\pi)$ is orthonormal, we obtain for each $\pi \in H^{\wedge}$ that
\begin{equation} \label{traceformula2}\Tr \thinspace \pi(f_g^*) =  \sum_{v_0 \in \cB_f(\pi)} \langle v_0,  v_0\rangle \langle v_0 , \pi^{-1}(g)v_0 \rangle = \sum_{v_0 \in \cB_f(\pi)} \Phi_{\pi, v_0}(g).\end{equation}
Inserting \eqref{traceformula2}  back into \eqref{PlDef} and taking conjugates, we obtain the formula in the statement of the proposition. 

Lastly, under the hypothesis that $\pi(f)$ is a projection operator onto a finite-dimensional subspace, it is simple to see that $\Tr \thinspace \pi(f_g^*) \in L^1(H^{\wedge}, \Sigma)$ for all $g\in H$. Indeed, by our hypothesis this function takes only non-negative integer values so that $| \Tr \pi(f_g^*)|\leq \Tr \pi(f_g^*)\pi(f_g^*)^*$, which is integrable over $H^{\wedge}$ by \eqref{plancherelformula} as $f_g^* \in C_c^\infty(H) \subseteq L^1(H) \cap L^2(H)$. 
\end{proof}

\subsection{Computation of the diagonal term}\label{sec:MTcomputation}  
First we only assume that $f = \bigotimes_p f_p\in \cH_{\rm fin}$ satisfies the geometric assumptions. Recall from the unrefined generalized BK formula Theorem \ref{MT} that the diagonal term is given by 
$$\delta_{m_1m_2>0} f_\infty(1)\sqrt{\frac{m_2}{m_1}} \int_{\A_{\rm fin}} f\left( \left( \begin{smallmatrix}  m_2/m_1  & t \\ & 1\end{smallmatrix} \right) \right) \psi_{\rm fin}(-m_1t)\,dt. $$ 
Since the argument of $f$ here has unit determinant, it follows from Geometric Assumption \eqref{geo3} that the integrand vanishes unless $m_2/m_1 \in \Z_p^{\times}$ for all $p$. Therefore, the diagonal term vanishes 
unless $m_1 = m_2  \in \frac{1}{N} \Z$. In that case, we write $m$ for the common value of $m=m_1=m_2$ and we have
\begin{equation}\label{eq:MT_with_geom} \delta_{m_1m_2>0} \sqrt{\frac{m_2}{m_1}} \int_{\A_{\rm fin}} f\left( \left( \begin{smallmatrix}  m_2/m_1  & t \\ & 1\end{smallmatrix} \right) \right) \psi_{\rm fin}(-m_1t)\,dt =  \delta_{m_1=m_2 \in \frac{1}{N} \Z}  \int_{\A_{\rm fin}} f\left( \left( \begin{smallmatrix} 1 & t \\ & 1 \end{smallmatrix} \right)\right) \psi_{\rm fin}(-mt)\,dt.\end{equation}

 As an aside, we can use Lemma \ref{*defect} to give a classical description of the integral in \eqref{eq:MT_with_geom}.  For any $f$ of level $N$ with support controlled by $y$, recall the function $f'$ from \eqref{ffin*def} and set $M= N \lcm(y,y^{-1})$ as in Lemma \ref{*defect}. Then, we have  
\begin{multline}\label{MainTermClassical} \int_{\A_{\rm fin}} f \left( \left( \begin{smallmatrix} 1 & t \\ & 1 \end{smallmatrix} \right)\right) \psi_{\rm fin}(- mt)\,dt 
= y \int_{\widehat{\Z}} f' \left( \left( \begin{smallmatrix} 1 & t \\ & 1 \end{smallmatrix} \right)\right) \psi_{\rm fin}(- y^{-1} mt)\,dt 
\\
= \frac{y}{M} \sum_{t \in \widehat{\Z}/M\widehat{\Z}}  f' \left( \left( \begin{smallmatrix} 1 & t \\ & 1 \end{smallmatrix} \right)\right) e(- y^{-1} mt) \int_{\widehat{\Z}} \psi_{\rm fin} (-y^{-1} mMu )\,du 
\\
= \delta(y^{-1} mM \in \Z) \frac{y}{M} \sum_{t \in \Z/M\Z}  f' \left( \left( \begin{smallmatrix} 1 & t \\ & 1 \end{smallmatrix} \right)\right) e(-y^{-1} mt).
 \end{multline}

\begin{myexam}As a sanity check, let us work this out in the case of the classical Kuznetsov formula for $\Gamma_0(q)$. For this example, $f = \nu(q) 1_{ZK_0(q)}$, where $\nu(q) = [\SL_2(\mz): \Gamma_0(q)]$.
We can take $y=q$, $N=q$, and by direct observation $M=q$ (not using Lemma \ref{*defect}), 
 so that $f'$ is $\nu(q)$ times the indicator function of $ZK_0(q)^\intercal$. All the terms in the above sum vanish except for $t=0$, so the sum reduces to $\nu(q)$ times the indicator of $m \in \Z$. 
Alternately, we can take $y=1$, $N=q$, and $M=q$, in which case $\delta(mM/y \in \Z)=1$ trivially, $f = f'$, and $f'(n(t))$ is $\nu(q)$ times the indicator function of $t\in \widehat{\Z}$. We get that the adelic integral equals $  \frac{\nu(q)}{q} \sum_{t \pmod{q}} e(-mt)$, which is again $\nu(q)$ times the indicator function of $m\in \Z$.\end{myexam}

Now we use the spectral assumption to simplify the left hand side of \eqref{eq:MT_with_geom}.

\begin{myprop}\label{prop:nonarchmt}
If $f_p$ is a newform projector, then for $p \nmid m \in \Z$ we have
\begin{equation}\label{eq:nonarchmt}
\int_{\Q_p} f_{p}\left( \left( \begin{smallmatrix} 1 & t \\ & 1\end{smallmatrix} \right) \right) \psi_{p}(-mt)\,dt =  \int_{\pi \in \cF_p(f)} \frac{1}{\mathcal{L}_{\pi}(1)}\,d \widehat{\mu}(\pi),
\end{equation}
where $\cF_p(f)$ is the local family from Definition \ref{localfamily}, $\mathcal{L}_{\pi}(1)$ is given by formula \eqref{Appidef} and $\widehat{\mu}$ is the Plancherel measure with respect the standard Haar measure $\mu$ on $\overline{G}(\Q_p)$. 
\end{myprop}
On the other hand, if $f_p = \nu(p^c)1_{ZK_0(p^c)}$ for some $c \in \Z_{\geq 0}$, then $\pi(f_p)$ is an orthogonal projection onto $V_{\pi_p}^{K_0(p^c)}$ (containing both old and new forms) and by a direct computation we have that 
$$ \int_{\Q_p} f_{p}\left( \left( \begin{smallmatrix} 1 & t \\ & 1\end{smallmatrix} \right) \right) \psi_{p}(-mt)\,dt = \nu(p^{c}) = \int_{\cF_p(f)} \dim V_\pi^{K_0(p^c)}\, d\widehat{\mu}(\pi) ,$$
where $\dim V_\pi^{K_0(p^c)}= c- c(\pi)+1$ if $c(\pi) \leq c$ and $=0$ otherwise, by newform theory \cite{Casselman_onAL}. 

We begin with a preparatory lemma. Recall that the unramified principal series representation $\pi(\alpha^s,\alpha^{-s})$ is irreducible and unitary if and only if $s$ is either purely imaginary with imaginary part in $[0,\pi/\log p]$, or $s=-\tau$ or $-\tau + \frac{\pi i}{\log p}$ with $\tau$ real and $0<\tau<1/2$. 
\begin{mylemma}\label{Phipiunipotentexplicit}
Let $\pi \in \overline{G}(\Q_p)^\wedge$ be generic and $v_0$ be a unit-length newvector in $V_\pi$. We have 
 \begin{equation}\label{RamifiedSpecial}
\Phi_{\pi, v_0}(\begin{smallmatrix} 1 & t \\ & 1\end{smallmatrix}) = \begin{cases} 
1 & \text{ if }  v(t)\geq  0, \\
\frac{p^{v(t)}}{1+p^{-1}} \frac{p^{-2v(t)s}(p^{s} - p^{-s}p^{-1}) - p^{2v(t)s}(p^{-s}-p^{s}p^{-1})}{p^{s}-p^{-s}} & \text{ if } \pi \simeq \pi(\alpha^s,\alpha^{-s}) \text{ and } v(t)\leq -1, \\
 - p^{1+2v(t)} & \text{ if } c(\pi)=1 \text{ and } v(t) \leq -1, \\
 - \frac{1}{p-1} & \text{ if } c(\pi)\geq 2 \text{ and } v(t) =-1, \\
  0 & \text{ if } c(\pi)\geq 2 \text{ and } v(t) \leq -2. \end{cases}
\end{equation}
\end{mylemma}
\begin{proof}
We use  the Macdonald formula and explicit formulas for the diagonal matrix coefficients of newforms due to the first author \cite[Lem.\ 4.6]{hu_triple_2017} and \cite[Prop.\ 3.1]{Hu:17a}.
Recall that the diagonal newvector matrix coefficient of a trivial central character representation is $K_0(p^c)$-invariant and $Z$-invariant, where $c$ is the conductor exponent of $\pi$.

First, note that if $v(t)\geq 0$, then $n(t) \in K_0(p^c)$ for any generic $\pi \in \overline{G}(\Q_p)^\wedge$. The first line of equation \eqref{RamifiedSpecial} follows. 

We next consider the case that $\pi$ is an unramified principal series representation  and $v(t)\leq -1$. By the Cartan decomposition  
$$\GL_2(\Q_p) = \bigsqcup_{i\geq j} K_p \left( \begin{smallmatrix} p^i & \\ & p^j \end{smallmatrix} \right) K_p,$$
so the matrix coefficient $\Phi_{\pi, v_0}$ is determined by its values on the elements $\sigma_i := \left( \begin{smallmatrix} p^i & \\ & 1 \end{smallmatrix}\right)$ for $i\geq 0$. 
 We have
$$\left( \begin{array}{cc} 1 & t \\ & 1 \end{array} \right) = \left( \begin{array}{cc} t &  \\ & t \end{array} \right) \left( \begin{array}{cc} 0 & 1 \\ -1 & 1/t \end{array} \right) \left( \begin{array}{cc} 1/t^2 &  \\ & 1 \end{array} \right) \left( \begin{array}{cc} 1 & 0 \\ 1/t & 1 \end{array} \right),$$
so that when $v(t)<0$, up to a scalar,  $\left( \begin{smallmatrix} 1 & t \\ & 1 \end{smallmatrix}\right) $ lies in $K_p\sigma_{-2v(t)}K_p$. Since $\pi$ has trivial central character, $\Phi_{\pi, v_0}$ also transforms trivially by scalars, and then by the Macdonald formula (see e.g.\ \cite[Thm.\ 4.6.6]{Bump})
one may read off the claimed formula in the second line of equation \ref{RamifiedSpecial}. 

Next, suppose that $\pi$ is either the Steinberg representation or its unramified quadratic twist. Letting $\omega = \left( \begin{smallmatrix} & 1 \\ -1 & \end{smallmatrix} \right)$, we have
\begin{equation}
\left( \begin{array}{cc} 1 & t \\ & 1 \end{array} \right) = \left( \begin{array}{cc} t &  \\ & t \end{array} \right) \left( \begin{array}{cc} 1 & 0 \\ 1/t & 1 \end{array} \right)  \omega \left( \begin{array}{cc} 1/t^2 &  \\  & 1 \end{array} \right) \left( \begin{array}{cc} 1 & 0 \\ 1/t & 1 \end{array} \right) .
\end{equation}
 If $v(t) \geq 0$, then $\left( \begin{smallmatrix} 1 & t \\ & 1\end{smallmatrix}\right) \in K_1(p)$, and if $v(t)<0$, then by the above
\begin{equation} \left( \begin{array}{cc} t &  \\ & t \end{array} \right)^{-1}\left( \begin{array}{cc} 1 & t \\ & 1 \end{array} \right)  \in K_1(p) \omega \sigma_{-2v(t)} K_1(p).\end{equation} 
The claimed formula in the third line of equation \eqref{RamifiedSpecial} may now be read off from \cite[Lem.\ 4.6]{hu_triple_2017}.

If $\pi$ is a trivial central character supercuspidal or ramified principal series representation then the result in the lemma is \cite[Prop.\ 3.1(i)]{Hu:17a}. If $\pi$ is a ramified twist of the Steinberg representation, then the result is not stated in \cite[Prop.\ 3.1(i)]{Hu:17a}, but follows by identical arguments. We reproduce Hu's proof for sake of completeness. 
  
We compute the matrix coefficient in the Kirillov model. Let $d^\times \alpha$ be the Haar measure on $\Q_p^\times$ that gives $\Z_p^\times$ volume $1$. By Lemma \ref{NewformKirillov} we have
\begin{multline}
\Phi_{\pi, v_0}(\begin{smallmatrix} 1 & t \\ & 1\end{smallmatrix}) =  \int_{\Q_p^\times} \pi (\begin{smallmatrix} 1 & t \\ & 1\end{smallmatrix}) W_0\left( \begin{smallmatrix} \alpha & \\ & 1 \end{smallmatrix}\right) \overline{ W_0\left( \begin{smallmatrix} \alpha & \\ & 1 \end{smallmatrix}\right) }\,d^\times \alpha = \int_{\Q_p^\times}  W_0\left( \begin{smallmatrix} \alpha & \\ & 1 \end{smallmatrix}\right) \overline{ W_0\left( \begin{smallmatrix} \alpha & \\ & 1 \end{smallmatrix}\right) } \psi_p(t \alpha)\,d^\times \alpha \\
= \int_{\Z_p^\times} \psi_p(t \alpha)\,d^\times \alpha,
\end{multline}
where $W_0$ is the vector in the Whittaker model corresponding to $v_0$. Note since $\psi_p$ has conductor 0, we have for $n \in \Z_p$ that 
\begin{equation}\label{RamanujanSum}
\int_{ p^{-i}\Z_p^\times}  \psi_p(n\alpha)\, d\alpha = p^i \int_{\Z_p^\times} \psi_p(\frac{n\alpha}{p^i})\,d\alpha = \sum_{\alpha \in (\Z/p^i\Z)^\times}  e(\frac{n \alpha}{p^i}) =: R_{p^i}(n)
\end{equation}
   is the classical Ramanujan sum.
If $v(t)<0$, then 
\begin{multline}
 \int_{\Z_p^\times} \psi_p(t \alpha)\,d^\times \alpha = \int_{p^{v(t)}\Z_p^\times} \psi_p\left(\frac{t \alpha}{p^{v(t)}}\right)\,d^\times \alpha = \zeta_p(1) \int_{p^{v(t)}\Z_p^\times} \psi_p\left(\frac{t \alpha}{p^{v(t)}}\right)\,\frac{d\alpha}{|\alpha|} \\  \\
 = \frac{1}{\phi(p^{-v(t)})} R_{p^{-v(t)}}(tp^{-v(t)}) 
 = \begin{cases} -\frac{1}{p-1} & \text{ if } v(t)=-1, \\ 0 & \text{ if } v(t) \leq -2\end{cases} 
\end{multline}
by changing multiplicative to additive measure (see Section \ref{measurenormalizations}) and applying \eqref{RamanujanSum}.
\end{proof}
\begin{proof}[Proof of Prop.\ \ref{prop:nonarchmt}]
Since $\overline{G}(\Q_p)$ is a unimodular $p$-adic linear algebraic group, the results of Section \ref{sec:localspectral} apply. Let us define 
\begin{equation}\label{fhatm_def}
\widehat{f}(m) = \int_{\Q_p} f_{p}\left( \left( \begin{smallmatrix} 1 & t \\ & 1\end{smallmatrix} \right) \right) \psi_{p}(-mt)\,dt.
\end{equation}
Since $f_p$ is assumed to satisfy the spectral assumption, by Proposition \ref{mainprop} we have 
\begin{equation}\label{hfm2}
\widehat{f}(m) = \int_{\Q_p} \int_{\cF_p(f)}   \overline{\Phi_{\pi, v_0}(\begin{smallmatrix} 1 & t \\ & 1\end{smallmatrix})} \,d\widehat{\mu}(\pi) \psi_{p}(-mt) \, dt.
\end{equation}

For any generic $\pi \in {\overline{G}(\Q_p)}^{\wedge}$, note that $\Phi_{\pi, v_0}\left( \begin{smallmatrix} 1 & t \\ & 1\end{smallmatrix} \right)$ only depends on $t$ via $v(t)$ and is constant $=1$ if $v(t)\geq 0$. So, from \eqref{hfm2} and \eqref{RamanujanSum} we have
\begin{multline*}\widehat{f}(m) = \int_{\Z_p} \int_{\cF_p(f)} \,d\widehat{\mu}(\pi) \,dt + \sum_{i \geq 1}  \int_{t \in p^{-i}\Z_p^\times}  \psi_p(-mt)\, dt \int_{\cF_p(f)} \Phi_{\pi, v_0}\left( \begin{smallmatrix} 1 & p^{-i} \\ & 1\end{smallmatrix} \right)\, d \widehat{\mu}(\pi) \\
= \int_{\cF_p(f)} \,d\widehat{\mu}(\pi)  + \sum_{i \geq 1} R_{p^i}(m)  \int_{\cF_p(f)} \Phi_{\pi, v_0}\left( \begin{smallmatrix} 1 & p^{-i} \\ & 1\end{smallmatrix} \right)\, d \widehat{\mu}(\pi).
\end{multline*} 
Since $v(m)=0$ by assumption we have $R_p(m)=-1$ and $R_{p^i}(m)=0$ for $i \geq 2$, so 
\begin{equation}\label{fhatm2}  \widehat{f}(m) =  \int_{\cF_p(f)} \,d\widehat{\mu}(\pi)  - \int_{\cF_p(f)} \Phi_{\pi, v_0}\left( \begin{smallmatrix} 1 & p^{-1} \\ & 1\end{smallmatrix} \right) \, d \widehat{\mu}(\pi)= \int_{\cF_p(f)}\left( 1-  \Phi_{\pi, v_0}\left( \begin{smallmatrix} 1 & p^{-1} \\ & 1\end{smallmatrix} \right)\right) \,d\widehat{\mu}(\pi) .\end{equation}
By Lemma \ref{Phipiunipotentexplicit}, we immediately recognize the integrand of \eqref{fhatm2} as $\mathcal{L}_\pi(1)^{-1}$ in the cases that $c(\pi)\geq 1$ (see \eqref{Appidef}).

It remains to treat the case that $\pi$ is unramified. Suppose that $\pi \simeq \pi(\alpha^s, \alpha^{-s})$, and define $\theta$ by $i\theta = s\log p$ so that either $\theta \in [0,\pi]$ is real, or $\theta=i\tau\log p$ or $\theta = \pi +i \tau \log p$ with $0<\tau<1/2$ real. Inserting the second line of equation \eqref{RamifiedSpecial}, we find that the integrand of \eqref{fhatm2} is 
\begin{multline*}= 1- \frac{p^{-1}}{1+p^{-1}} \frac{ e^{2i \theta} (e^{i \theta}-e^{-i \theta}/p) - p^{-2i \theta} (e^{-i \theta} - p^{i \theta}/p)}{e^{i \theta} - e^{-i \theta}} \\
= \frac{1}{1+p^{-1}} \left( 1+ p^{-1} -p^{-1} \frac{ e^{3i\theta}-e^{-3i\theta}}{e^{i \theta}-e^{-i\theta}} + p^{-2} \right) \\
= \frac{ 1-p^{-1} (e^{2 i \theta} + e^{-2 i \theta}) +p^{-2}}{1+p^{-1}},
\end{multline*}  
which matches the definition of $\mathcal{L}_\pi(1)^{-1}$ from \eqref{Appidef}. 
\end{proof}
\begin{myrema}It is possible to generalize Proposition \ref{prop:nonarchmt} to drop the condition that $p \nmid m$, but the resulting formula becomes more complicated, so we have omitted this case.\end{myrema}

\subsection{The spectral assumption}\label{sec:spec_consequences}
Here we record a few consequences of the spectral assumption. 
We begin with a motivational remark. Under the spectral and geometric assumptions, an open subset $\cF_p$ of the local unitary dual ${\overline{G}(\Q_p)}^{\wedge}$ that occurs as the local family $\cF_p(f)$ of some $f_p \in \cH_p$ determines the geometric test function $f_p$ completely (cf.\ Section \ref{sec:motivation}). 
Indeed, by the spectral assumption $f_p$ is either a classical test function, or a newform projector onto $\cF_p$. 
Suppose that $\cF_p$ contains all generic representations of conductor exponents $\leq c$ with $c>0$, and is also a newform projector. Then $f_p$ cannot be compactly supported, since the sum of diagonal newform matrix coefficients of the Steinberg representation and its unramified twist is not compactly supported (see \cite[Lem.\ 4.6]{hu_triple_2017}), yet all generic representations of larger conductor sit in a connected component whose newform projector is compactly supported (see Section \ref{sec:examples}). 

By Geometric Assumption \eqref{geo3}, if $\cF_p$ contains all generic representations of conductor exponents $\leq c$ with $c>0$, then the function $f_p$ must be the classical test function. On the other hand, if $\cF_p$ does not contain all generic representations of conductor exponents $\leq c$ or $c=0$, then it is a newform projector. In either case, Proposition \ref{mainprop} determines $f_p$ uniquely. 

In particular, the notation $\cC(\cF)$ for the set of admissible moduli and $k(\cF)$ for the geometric conductor are justified under the spectral assumption when we interpret $\cF$ as $\prod_p \cF_p$.

Recall the diagonal, unipotent and Borel subgroups $A,N\subset B \subset G$ from Section \ref{notation-groupsandsubgroups}.

\begin{mylemma}\label{specControlonH}
Suppose that $f \in \cH_{\rm fin}$ satisfies the spectral assumption. Then:
\begin{enumerate}
\item $f$ is bi-$B(\widehat{\Z})$-invariant, 
\item $f$ satisfies Geometric Assumption \eqref{geo2},
\item $H(m,n;c)=0$ if $m \not \in \mz$ or $n \not \in \mz$.
\item if in addition $f$ satisfies Geometric Assumption \eqref{geo3}, then $H(m,n;c)=0$ unless $c\in \Z$, i.e.\ $k(\cF) \in \N$. 
\end{enumerate} 
\end{mylemma}
\begin{proof}
\begin{enumerate}
\item Since $f$ is a pure tensor, it suffices to check that $f_p$ is $B(\Z_p)$-invariant for each $p$. If $f_p = \nu(p^c) 1_{ZK_0(p^c)}$ for some $c \in \Z_{\geq 0}$, then $f_p$ is clearly $B(\Z_p)$-invariant, so we focus on the case that $f_p$ is a newform projector. In this case, Proposition \ref{mainprop} applies and thus it suffices to check that diagonal matrix coefficients of newforms $\Phi_{\pi, \varphi_0}$ are bi-$B(\Z_p)$-invariant for each $\pi \in \cF_p(f)$. However, if $\pi$ has conductor $p^c$, then $\Phi_{\pi,\varphi_0}$ is clearly bi-$K_0(p^c)$-invariant, and since $B(\Z_p) \subseteq K_0(p^c)$ for all $c$ we are done. 
\item Clear, since $A(\Z_p) \subseteq B(\Z_p)$. 
\item Suppose that $m \not \in \Z$. Then let $p$ be a prime dividing the denominator of $m$ and make a change of variables $t_1 \to t_1+1$ in the definition \eqref{eq:KloostermanLocalFormula} of $H_p(m,n;c)$:
$$H_p(m,n;c) =    \iint_{\Q_{p}^2} f_{p}\left( \left( \begin{smallmatrix}-t_1-1 & -c^{-2} -t_1t_2-t_2 \\ 1 & t_2 \end{smallmatrix}\right) \right) \psi_{p}(mt_1 + m- n t_2)\,dt_1\,dt_2.$$
By the left invariance of $f_p$ by $n(1)$, we obtain
$$ H_p(m,n;c) = \psi_p(m) H_p(m,n;c),$$ so we must have that $H_p(m,n;c)= 0$, thus $H(m, n;c)=0$. If $n \not \in \Z$, then a similar argument works using the right $N(\Z_p)$-invariance of $f_p$.
\item If $c \not \in \Z$, then for any $m,n\in \Z$ we would have $m+c \not \in \Z$ and $$H(m,n;c)= H(m+c,n;c)=0$$ by the $c$-periodicity of $H$ (Theorem \ref{thmKP}(2)) and the previous fact. This shows that $\cC(\cF) \subseteq \N$ and thus the geometric conductor must be an integer. 
\end{enumerate}
\end{proof}

The following lemma will be useful later. 
\begin{mylemma}\label{projection_upper_bound}
Suppose that $H$ is a linear algebraic group over a local field, $f \in C_c^\infty(H)$, and that $\pi(f)$ is a projection operator for all $\pi \in H^{\wedge}$. Then, $f$ attains its maximum at $1\in H$ and $f(1) = \| f\|_{L^2}^2$. 
\end{mylemma}
\begin{proof}
Since $\pi(f)$ is a projection operator for every $\pi \in H^{\wedge}$, we have that $\pi(f)^2 = \pi(f)$ for all $\pi \in H^{\wedge}$. Since the Fourier transform is injective \cite[\S 18.2.3]{Dixmier}, we have that $f * f = f$. Since $f \in C_c^\infty(H)$, 
$$ f(x) = \int_H f(y) f(xy^{-1})dy$$ for all $x \in H$ (not merely almost every), so we may evaluate this at $x=1$ to obtain $f(1) = \int_H f(y) f(y^{-1})dy.$ Since $\pi(f)$ is a projection, it is self-adjoint, so $f$ is self-adjoint as well, that is $f(y^{-1}) = \overline{f(y)}$. Thus, $f(1) = \|f\|_{L^2}^2$. 

Finally, by Cauchy-Schwarz, 
$$|f(x)|= 
\Big| \int_H f(y) f(xy^{-1})dy \Big|
 \leq \Big(\int_H |f(y)|^2 dy \Big)^{1/2} \Big(\int_H |f(xy^{-1})|^2 dy \Big)^{1/2} = \|f\|_{L^2}^2. \qedhere
$$ 
\end{proof}

\subsection{Proof of Theorem \ref{theoGeomSpec}}\label{theoGeomSpecProof}
We begin by deducing the following result from Theorem \ref{MT} and the theory built up in the meantime, assuming the geometric and spectral assumptions. 
\begin{mytheo}\label{theoGeomSpec_FC}
Let $f \in \cH_{\rm fin}$ be a pure tensor satisfing the geometric and spectral assumptions. 
Then, for all $m_1,m_2 \in \Z$ with $m_1m_2 >0$ we have 
\begin{multline}\label{GeneralizedKuznetsov2_FC}\sum_{ \pi \in \cF_0(f)} h_\infty(t_\pi) \sum_{\varphi \in \cB_f(\pi)} a_{u_{\varphi}}(m_1) \overline{a_{u_\varphi}(m_2)}  \\ 
+ \frac{1}{4\pi}\sum_{ \chi \in \cF_E(f)} \sum_{\phi \in \mathcal{B}_f(\chi, \chi^{-1})} \int_{-\infty}^\infty h_\infty(t) a_{u_{E(\phi_{it})}}(m_1) \overline{a_{u_{E(\phi_{it})}}(m_2)}\,dt   \\
= \delta_{m_1=m_2} \delta_\infty  \int_{\A_{\rm fin}} f\left( \left( \begin{smallmatrix} 1 & t \\ & 1\end{smallmatrix} \right) \right) \psi_{\rm fin}(-mt)\,dt +  \sum_{c \equiv 0 \shortmod{k(\cF)} } \frac{H(m_1,m_2;c)}{c}H_\infty\left(\frac{4 \pi \sqrt{ m_1m_2}}{c}\right)
\end{multline}
with notation as in Theorem \ref{MT}, $\cB_f(\pi)$ any orthonormal basis for $V_f^{K_\infty}$ (see \eqref{pifdef}), and $\cB_f(\chi,\chi^{-1})$ any orthonormal basis for the $K_\infty$-fixed space of the image of $\pi(f):V_\pi \to V_\pi$ with $\pi= \pi_{\chi,\chi^{-1}}$ the global principal series representation. 
\end{mytheo}

\begin{proof}
Recall that Geometric Assumption \eqref{geo3} implies that the hypothesis of Theorem \ref{MT} is satisfied, and so the unrefined Petersson/Kuznetsov formula \eqref{GeneralizedKuznetsov1} holds. We next record how \eqref{GeneralizedKuznetsov1} simplifies to \eqref{GeneralizedKuznetsov2_FC} in the presence of the geometric and spectral assumptions. 

Using the spectral assumption, for each $\pi \in \cF_0(f)$ (resp.\ $\cF_E(f)$, $\cF_{\kappa}(f)$), recall that the image $V_f$ of $\pi(f): V_\pi \to V_\pi$ was explicated in \eqref{pifdef}. Then, we have that $\pi(f)\varphi = \varphi$ if $\varphi  \in V_f$, and $\pi(f)\varphi = 0$ if $\varphi \in V_f^\perp$. We choose the orthonormal bases $\cB(\pi)$ (resp.\ $\cB(\chi,\chi^{-1})$) to respect the direct sum decomposition $\pi = V_f \oplus V_f^\perp$, so that all basis vectors are killed by $\pi(f)$ except in the finite-dimensional subspace $V_f^{K_\infty}$ (resp.\ $V_f^\kappa$, the weight $\kappa$ isotypic subspace of $V_f$). The result of these reductions is that 
$$  \sum_{\varphi \in \mathcal{B}(\pi)}a_{u_{\pi(f)\varphi}}(m_1) \overline{a_{u_\varphi}(m_2)} =  \sum_{\varphi \in \mathcal{B}(V_f^{K_\infty})}a_{u_{\varphi}}(m_1) \overline{a_{u_\varphi}(m_2)},$$
or $V_f^\kappa$ in place of $V_f^{K_\infty}$ in the holomorphic / discrete series case, respectively.  For the Eisenstein series contribution we have similarly
$$ \sum_{\phi \in \mathcal{B}(\chi, \chi^{-1})} a_{u_{E(\pi_{it}(f)\phi_{it})}}(m_1) \overline{a_{u_{E(\phi_{it})}}(m_2)} = \sum_{\phi \in \mathcal{B}_f(\chi, \chi^{-1})} a_{u_{E(\phi_{it})}}(m_1) \overline{a_{u_{E(\phi_{it})}}(m_2)}.$$
Since the spectral assumption guarantees that $f$ is bi-$N(\widehat{\Z})$-invariant (Lemma \ref{specControlonH}(1)), the above sums of Fourier coefficients vanish unless both $m_1,m_2 \in \Z$. 

To derive the geometric side of the formula in \eqref{GeneralizedKuznetsov2_FC} from that of \eqref{GeneralizedKuznetsov1}, we insert the formula \eqref{eq:MT_with_geom} for the diagonal term, and for the off-diagonal we  note that the geometric assumptions imply $\cC(\cF) \subseteq k(\cF)\N$ via Lemmas \ref{ccondition} and \ref{lem:admmodulus}. Note by the spectral assumption that the generalized Kloosterman sums $H(m_1,m_2;c)$ also vanish unless both $m_1,m_2 \in \Z$, by Lemma \ref{specControlonH}(3). 
\end{proof}
\begin{myrema} In Section \ref{sec:MTcomputation} we moreover computed the non-archimedean diagonal term contribution of \eqref{GeneralizedKuznetsov2_FC} in terms of Plancherel volumes, but only under the hypothesis that $(m_1m_2,N)=1$ (otherwise we would have included the result in Theorem \ref{theoGeomSpec_FC}). In fact, one can carry through the computation in Proposition \ref{prop:nonarchmt} without the assumption $p \nmid m$ (i.e.\ $(m_1m_2,N)=1$), but the resulting formula for the diagonal term becomes more complicated and in particular (unlike the factor $\delta_\fin$  from \eqref{eq:deltapnewformproj} and \eqref{eq:deltapclassical}) depends on $m_1,m_2$. We therefore leave this case aside. \end{myrema}

\begin{proof}[Proof of Theorem \ref{theoGeomSpec}.]
As just remarked, in Section \ref{sec:MTcomputation} we deduced the form of the diagonal term in Theorem \ref{theoGeomSpec} from that of Theorem \ref{theoGeomSpec_FC} under the spectral hypothesis using the assumption that $(m_1m_2,N)=1$. We thus obtain the geometric side of Theorem \ref{theoGeomSpec}.

To finish the proof of Theorem \ref{theoGeomSpec}, it remains to express the Fourier coefficients on the spectral side of \eqref{GeneralizedKuznetsov2_FC} in terms of Hecke eigenvalues by appealing to \eqref{Eq:SpectralAssumption} and its analogous statement for Eisenstein series. To state the Eisenstein series version, we abbreviate $\pi_{it}(\chi)= \pi_{\chi_1 \alpha^{it}, \chi_2 \alpha^{-it}}, \, V_{it}(\chi) = V_{\chi_1 \alpha^{it}, \chi_2 \alpha^{-it}}$ and set \begin{equation}
\lambda_{\pi_{it}(\chi)}(n) = \sum_{ab=n} \chi_1(a)\chi_2(b) \left(b/a\right)^{it} \quad (n \in \N).
\end{equation}
Then the Eisenstein series analogue of \eqref{Eq:SpectralAssumption} is that there exists an orthonormal basis $\mathcal{B}_f(\chi, \chi^{-1})$ and weights $w(\pi_{it}(\chi), f) \in \C$ such that for all $m_1,m_2 \in \N$ and $(m_1m_2,N)=1$ we have 
\begin{equation}\label{Eq:SpectralAssumption_Eis}
\sum_{\phi \in \mathcal{B}(\chi, \chi^{-1})_f^{K_\infty}} a_{u_{E(\phi_{it})}}(m_1) \overline{a_{u_{E(\phi_{it})}}(m_2)} = w(\pi_{it}(\chi), f)\lambda_{\pi_{it}(\chi)}(m_1) \overline{\lambda_{\pi_{it}(\chi)}(m_2)}.
\end{equation}
Applying \eqref{Eq:SpectralAssumption} and \eqref{Eq:SpectralAssumption_Eis}, the spectral side of Theorem \ref{theoGeomSpec_FC} becomes the spectral side of Theorem \ref{theoGeomSpec} (note that when $m_1,m_2<0$ we have $a_{u}(m_1)\overline{a_u(m_2)} =  a_{u}(|m_1|)\overline{a_u(|m_2|)}$ for either parity of $u$). Under the assumption in \eqref{Eq:SpectralAssumption_Eis}, the suppressed $(\text{cts.})$ in Theorem \ref{theoGeomSpec} in detail is
\begin{equation}\label{ctsdots}
(\text{ cts. }):= \frac{1}{4\pi} \sum_{\chi \in \cF_E(f)} \int_{-\infty}^\infty h_\infty(t) w(\pi_{it}(\chi), f)\lambda_{\pi_{it}(\chi)}(m_1) \overline{\lambda_{\pi_{it}(\chi)}(m_2)}\,dt.
\end{equation}

We next give a proof of the sentence containing \eqref{Eq:SpectralAssumption} in the introduction, namely, that there exists some orthonormal basis $\cB_f$ of $V_f^{K_\infty}$ (resp.\ $V_f^\kappa$) such that \eqref{Eq:SpectralAssumption} holds. Such a basis $GS_f$ was explicitly constructed by the Gram-Schmidt process in the several works mentioned just before \eqref{Eq:SpectralAssumption}. Indeed, by the spectral assumption (Lemma \ref{specControlonH}) $\varphi \in V_f^{K_\infty}$ (resp.\ $V_f^\kappa)$ is $K_0(N)$-invariant, thus $u_\varphi$ is modular with respect to $\Gamma_0(N)$ (at least). Then, defining the Petersson inner product by
\begin{equation}\label{Petersson_Normalization}\langle u,v \rangle = \frac{1}{[\SL_2(\Z): \Gamma_0(N)]} \iint_{\Gamma_0(N) \backslash \cH} u(z) \overline{v(z)}\, \frac{dx \,dy }{y^2},\end{equation}
 (similarly, for holomorphic forms) one has for any $\varphi_1,\varphi_2 \in V^{K_\infty}_f$ (resp.\ $V_f^\kappa$) that
$$\langle \varphi_1,\varphi_2\rangle = \langle u_{\varphi_1}, u_{\varphi_2}\rangle$$ by strong approximation (see e.g.\ \cite[\S 7.11, (12.20)]{KnightlyLiTracesOfHeckeOperators}). Note for future reference that we also have $\|\varphi\|_{can}^2 = \|\varphi\|_{L^2([\overline{G}])}^2$ by \cite[Rem.\ 3 of Thm.\ 6.1]{PetrowYoungCoset}. 

Next, for $\varphi \in V_f^{\kappa}$ (resp.\ $V_f^{K_\infty}$) let us define $\varphi^{(\delta{})}$ 
 by  $$\varphi^{(\delta)}(g) = \sum_{d\mid \delta} \xi_\delta(d) \varphi\left( a(d) g\right),$$
 in which $a(d) \in G(\R)$ is understood as an adelic matrix via the embedding $G(\R) \hookrightarrow G(\A)$ in the first position, and $\xi_\delta(d)$ is a jointly multiplicative function in $\delta, d$, supported on $d \mid \delta$ and given on primes by 
 \begin{equation}
 \xi_1(1) = 1, \quad \xi_p(p)  =  \left(1- \frac{|\lambda_\pi(p)|^2}{p(1+\frac{1_{p \nmid q(\pi)}}{p})^2}\right)^{-\frac{1}{2}}, \quad \xi_p(1)  = \frac{-\overline{\lambda_\pi(p)}}{\sqrt{p}(1+ 1_{p \nmid q(\pi)}/p)}\xi_p(p),
 \end{equation}
 and for $\nu \geq 2$ on prime powers by 
 \begin{multline}
 \xi_{p^\nu}(p^\nu)   = \left( 1- \frac{|\lambda_\pi(p)|^2}{p(1+\frac{1_{p \nmid q(\pi)}}{p})^2}\right)^{-\frac{1}{2}}\left(1-\frac{1_{p \nmid q(\pi)}}{p^2}\right)^{-\frac{1}{2}}, \quad \xi_{p^\nu}(p^{\nu-1})  = \frac{-\overline{\lambda_\pi(p)}}{\sqrt{p}}\xi_{p^\nu}(p^{\nu}), \\ \xi_{p^{\nu}}(p^{\nu-2})  = \frac{1_{p \nmid q(\pi)}}{p}\xi_{p^\nu}(p^{\nu}),
 \end{multline}
 and $\xi_{p^a}(p^b)= 0$ in all other cases.  
  Let $q(\pi)$ be the (finite) conductor of $\pi$ and $\varphi_0$ be an $L^2$-normalized newform in $V_\pi$. Then by the above discussion on inner products and \cite[Prop.\ 7.1]{PetrowTraces} , the set \begin{equation}
GS_f := \{ \varphi^{(\delta)}_0: \delta \mid N / q(\pi)\}
\end{equation}
 is an orthonormal basis for $V^{\kappa}_f$ (resp.\ $V_f^{K_\infty}$). 
 
 We claim that \eqref{Eq:SpectralAssumption} holds with $w(\pi, f)$ given by the formula \eqref{weights_explicit} for $V_f^\kappa$. The $V_f^{K_\infty}$ case is similar. To check this formula, let us temporarily and for this paragraph only let $M$ and $N$ with $M \mid N$ be as in \cite[\S 7]{PetrowTraces}. 
 Now, the Fourier coefficients $b_g(n)$ of the holomorphic modular forms $g$ that appear in \cite[(7-1)]{PetrowTraces} are related to the Fourier coefficients $a_u(n)$ in \eqref{Eq:SpectralAssumption} by $\nu(N)^{1/2}n^{-\frac{k-1}{2}}b_g(n) =a_g(n)$ due to the different choice of inner products. Then, the sum on the left hand side of \eqref{Eq:SpectralAssumption} is equal to the restriction of the sum $\Delta_{k,N, \epsilon_{0,N}}$ to the single oldclass corresponding to $\pi$ times $\nu(N)/c_k$ (see the first line of \cite[(7-1)]{PetrowTraces} and the first paragraph of loc.\ cit.\ \S 7 for definitions). Note that $\frac{1}{\nu(M)}\|g\|_M^2 = \| \varphi_g\|_{can}^2$ where $\varphi_g \in V_\pi$ is the vector corresponding to $g$ and $\|.\|_{can}$ is as in \cite[\S 2.2.2]{MichelVenkateshGL2}. If we restrict the expression for $\Delta_{k,N, \epsilon_{0,N}}$ in \cite[(7-2)]{PetrowTraces} to a single oldclass and multiply it by $\nu(N)/c_k$, we obtain \eqref{weights_explicit} from \cite[(6.4)]{PetrowYoungCoset}. 

We next give a proof of the Eisenstein series analogue, that is we check that there exists an orthonormal basis $\cB_f(\chi,\chi^{-1})$ such that the sentence containing \eqref{Eq:SpectralAssumption_Eis} holds. In this context, the orthonormal basis analogous to $GS_f$ for the space of Eisenstein series was constructed in \cite[\S 8.5]{YoungExplicit}.  For $\phi_{it} \in V_{it}(\chi)^{K_\infty}$ one has when $\pi_{it}(\chi)$ is non-singular (see \cite[\S2.2.1]{MichelVenkateshGL2}) that \begin{equation}\label{eisnorms} \| \phi_{it}\|^2 = \| E(\phi_{it})\|_{\rm Eis}^2 = \frac{1}{2}\| E(\phi_{it})\|_{can}^2\end{equation} by \cite[Rem.\ 3 of Thm.\ 6.1]{PetrowYoungCoset}, where $\|\cdot\|^2$ on the global principal series was defined in \eqref{aut_inner_prod} and $\|\cdot \|_{can}$ and $\|\cdot\|_{\rm Eis}$ are as in \cite[\S2.2]{MichelVenkateshGL2} (see also Section \ref{sec:pre-trace} of this paper). If $\phi \in V_{it}(\chi)_f^{K_\infty}$, then in addition we have 
\begin{equation}\label{eisnorm_to_formal_norm}{\| \phi_{it}\|^2 =} \frac{1}{4\pi \nu(N)} \langle u_{E(\phi_{it})},u_{E(\phi_{it})}\rangle_{N},\end{equation} 
where $\langle \cdot,\cdot\rangle_N$ is the formal inner product on the space of Eisenstein series of level $N$ defined in \cite[(8.1)]{YoungExplicit}.

When $\pi_{it}(\chi)$ is non-singular, we claim that \eqref{Eq:SpectralAssumption_Eis} holds with 
\begin{equation}\label{weights_explicit_Eis}
w(\pi_{it}(\chi),f) =  \frac{1}{ \xi(2) \mathcal{L}_{\pi_{it}(\chi)}^*(1)}\frac{1}{\rho_\pi(N/q(\pi(\chi)))}.
\end{equation}
Indeed, the left hand side of \eqref{Eq:SpectralAssumption_Eis} is equal to the restriction of \cite[(2.11)]{PetrowYoungWeyl} (which plays the role of \cite[(7-2)]{PetrowTraces} in the Eisenstein case) to a single oldclass times $4 \pi \nu(N)$. Converting to the canonical norm by \eqref{eisnorms} (whence the missing factor of 2 compared to \eqref{weights_explicit}) and \eqref{eisnorm_to_formal_norm}, and using \cite[(6.4)]{PetrowYoungCoset}, we obtain \eqref{weights_explicit_Eis}. 
\end{proof}
Warning: Unlike the cuspidal case, it is \emph{not} generally true that $w(\pi_{it}(\chi),f)  = ((1+|t|)N)^{o(1)}$. Indeed, near singular $\pi_{it}(\chi)$, the weight $w(\pi_{it}(\chi),f)$ may approach zero, as $ \mathcal{L}_{\pi_{it}(\chi)}^*(1) \sim |L(1+2it,\chi^2)|^2$, which blows up to order 2 when $\chi$ is quadratic and $t \to 0$. If $\chi$ is not quadratic, however, it is true that $w(\pi_{it}(\chi),f)  = ((1+|t|)N)^{o(1)}$ by explicit computation. 

\section{Applications}\label{sec:applications}
\subsection{Proof of harmonically-weighted Weyl-Selberg Law}\label{sec:WeylLaw}

\begin{proof}[Proof of Lemma \ref{lemWeylLaw}]
We apply trivial bounds to the sum of generalized Kloosterman sums in Theorem \ref{theoGeomSpec}. 
By Theorem \ref{thmKP}(5), we have
\begin{equation}
\label{Hbound_proof_of_lemWeylLaw}
|H(m,n;c)| 
\leq c y
\|f \|_{L^\infty(G)}.
\end{equation}
Recall $f$ satisfies the Spectral Assumption, so that by Lemma \ref{projection_upper_bound} we have  $\| f \|_{L^\infty(G)} \leq f(1)$. 

Next we need a bound on $H_\infty(x)$ for $x$ small. 
Recall \eqref{eq:archPlancherelIntro}, i.e., the Plancherel formula for the archimedean place:  
$$f_\infty(1) =  \frac{1}{4\pi} \int_{\R} h_\infty(t) \tanh(\pi t) t \,dt.$$
If $h_\infty(t)$ is given by \eqref{eq:hdefWindowVersion} or \eqref{eq:hdefInitialSegmentVersion}, we have by trivial estimates 
\begin{equation}\label{eq:easyVolumeLemma}
f_\infty(1) \asymp \Delta T \quad \text{ or } \quad f_\infty(1) \asymp T^2,
\end{equation}
respectively. In any case we note that
\begin{equation}\label{log_finfty1_log_T}
\log f_\infty (1) \asymp \log T.
\end{equation}

\begin{mylemma}\label{Hboundprop}
For $h_\infty$ as in either \eqref{eq:hdefWindowVersion} or \eqref{eq:hdefInitialSegmentVersion}, we have
\begin{equation}\label{H_infty_near0}
H_\infty(x)  = \frac{i}{2} \int_{-\infty}^\infty J_{2it}(x) \frac{t h_\infty(t)}{\cosh \pi t}\,dt \ll f_\infty (1) \left( \frac{x}{T}\right)^2 .
\end{equation}
\end{mylemma}
\begin{proof}
Here we follow  \cite[(3.10)]{JutilaMotohashi}. Indeed, the key property of $h_\infty$ is that its zeros at $\pm i/2$ permit us to shift the contour in \eqref{H_infty_near0} to the line $\imag(t)=-1$ without crossing any poles.  We have $$H_\infty(x) \ll \int_{-\infty}^\infty \frac{(|t|+1)h_\infty(t)}{\cosh \pi t} |J_{2+2it}(x)|\,dt,$$ which by Poisson's integral formula $$ J_\nu(x) = \frac{(\frac{1}{2}x)^\nu}{\sqrt{\pi} \Gamma( \nu + \frac{1}{2})} \int_{-1}^1 (1-y^2)^{\nu -1/2} \cos(xy)\,dy$$ is $\ll f_\infty(1)x^2/T^2$ as claimed. \end{proof}

Now we apply \eqref{Hbound_proof_of_lemWeylLaw} and Lemma \ref{Hboundprop} to the sum of Kloosterman sums in Theorem \ref{theoGeomSpec} to obtain
\begin{equation}
 \sum_{c \equiv 0 \shortmod{k(\cF)} } \frac{H(m,n;c)}{c}H_\infty\left(\frac{4 \pi\sqrt{mn} }{c}\right) \ll \frac{f_\A(1)mny}{T^{2}}  \sum_{c \equiv 0 \shortmod{k(\cF)} } 
\frac{1}{c^2} \ll \frac{f_\A(1)mn}{T^{2}k(\cF)},
\end{equation}
using that $y \leq k(\cF)$ (see Lemma \ref{ccondition}). 
\end{proof}
We deduce Corollary \ref{thmWeylLaw} from Lemma \ref{lemWeylLaw} by taking $m_1=m_2=1$ and observing that if $f$ is a newform projector, then all $\varphi \in V_f^{K_\infty}$ for $\pi \in \cF_0(f)$ are newforms, so by \cite[(6.4)]{PetrowYoungCoset} we deduce \eqref{Eq:SpectralAssumption} with
\begin{equation}
w(\pi, f) = |a_{u_\varphi}(1)|^2 = \frac{1}{2\xi(2) \mathcal{L}^*_\pi(1)}.
\end{equation}

\subsection{The $\GL_1$ large sieve inequalities}
We present some preliminary results that will be useful in the proof of Theorem \ref{thm:abstractversion}.
We first recall a classical large sieve inequality:
\begin{mylemma}
\label{lemma:classicallargesieve}
 Let $\alpha_r \in \mr$ be a set of points with $\mathop{dist}(\alpha_r - \alpha_s, \mz) \geq \delta > 0$ for $r \neq s$.  Then for any complex numbers ${\bf a} = (a_n)$, we have
 \begin{equation}
  \sum_{r} \Big| \sum_{M \leq n <M+N} a_n e(\alpha_r n) \Big|^2 \ll (\delta^{-1}+N) 
  \| {\bf a} \|^2.
 \end{equation}
\end{mylemma}
We also need a hybrid version, which is essentially due to Gallagher.
\begin{mylemma}
\label{lemma:Gallagher}
 Let conditions be as in Lemma \ref{lemma:classicallargesieve}, and let $T \geq 1$.  Then
 \begin{equation}
\int_{-T}^{T}  \sum_{r} \Big| \sum_{1 \leq n \leq N} a_n e(\alpha_r n) n^{-it} \Big|^2 \ll (T\delta^{-1} + N) 
\| {\bf a} \|^2.
 \end{equation}
\end{mylemma}
Strictly speaking, Lemma \ref{lemma:Gallagher} does not appear in \cite{GallagherLargeSieve}, but its proof is virtually identical to \cite[Theorem 3]{GallagherLargeSieve}.
We will need the following special case.
\begin{mylemma}
\label{lemma:largesieveclassicalvariant2}
Suppose that $(r,s) = 1$.
We have
\begin{equation}
\label{eq:largesieveclassicalvariant2plain}
\sum_{\substack{c \leq C \\ (c,r) = 1 \\ c \equiv 0 \shortmod{s}}}
\sumstar_{y \shortmod{c}} 
\sum_{u \shortmod{r}}
\Big|
\sum_{n \leq N} a_n e_{r}(nu) e_c(ny) \Big|^2
\ll \Big(\frac{C^2 r}{s} + N\Big) \| {\bf a } \|^2.
\end{equation}
Likewise, for $T \geq 1$ we have
\begin{equation}
\label{eq:largesieveclassicalvariant2hybrid}
\int_{-T}^{T}
\sum_{\substack{c \leq C \\ (c,r) = 1 \\ c \equiv 0 \shortmod{s}}}
\sumstar_{y \shortmod{c}} 
\sum_{u \shortmod{r}}
\Big|
\sum_{n \leq N} a_n n^{-it} e_{r}(nu) e_c(ny)  \Big|^2 dt
\ll \Big(T \frac{C^2 r}{s} + N \Big) \| {\bf a } \|^2.
\end{equation}
\end{mylemma}
\begin{proof}
We will derive \eqref{eq:largesieveclassicalvariant2plain} from Lemma \ref{lemma:classicallargesieve}, whereby \eqref{eq:largesieveclassicalvariant2hybrid} will follow immediately from Lemma \ref{lemma:Gallagher}.
For the proof, we only need to understand the spacings of some rational numbers as follows.
We have
\begin{equation}
\Big|\frac{y_1}{c_1} + \frac{u_1}{r} - \frac{y_2}{c_2} - \frac{u_2}{r} \Big| =
\Big|
\frac{r(y_1 c_2 - y_2 c_1) + c_1 c_2(u_1 - u_2)}{c_1 c_2 r}
\Big|.
\end{equation}
Provided that not both $y_1/c_1 = y_2/c_2$ and $u_1 = u_2$, then the numerator is a non-zero integer (since $(c_1 c_2, r) = 1$).  Moreover, the numerator is divisible by $s$ since $s|c_1$ and $s|c_2$.  Therefore the spacing of distinct points is at least $\frac{s}{c_1 c_2 r} \geq \frac{s}{C^2 r}$.
\end{proof}
\subsection{Archimedean analysis -- separation of variables}
In the archimedean aspect, our method of proving the spectral large sieve essentially follows Jutila's refinement \cite{JutilaSpectralLargeSieve} of Deshouillers-Iwaniec \cite{DI}. Jutila's work only considers level $1$ but nicely handles narrow spectral windows in lieu of the full $|t_j| \leq T$ range considered by Deshouillers-Iwaniec.

In this section we record some further properties of the integral transform $H_{\infty}(x)$  when the spectral weight function $h_\infty$ is given by \eqref{eq:hdefWindowVersion} or \eqref{eq:hdefInitialSegmentVersion}.

\begin{mylemma}
\label{lemma:archimedeanWindowVersion}
 Suppose that $h_\infty$ is given as in \eqref{eq:hdefWindowVersion} with $T^{\varepsilon} \ll \Delta \ll T^{1-\varepsilon}$.  If $x \leq \Delta T^{1-\varepsilon}$ then $H_{\infty}(x) \ll_A T^{-A}$, for $A > 0$ arbitrarily large.  
Suppose that $P \gg T^{\varepsilon}$ and $w$ is a fixed smooth weight function on $(0, \infty)$, supported 
on $[1,2]$. 
 If $x \geq \Delta T^{1-\varepsilon}$ then
 \begin{equation}
 \label{eq:MellinTransformOfHinf}
 w(x/P) H_{\infty}(x) = \frac{\Delta T}{P} \int_{|t| \asymp  P} W(t) x^{it} dt + O(T^{-A}),
 \end{equation}
where $W(t) \ll 1$.
\end{mylemma}
\begin{proof}
Most of these properties were derived in \cite[Lem.\ 7.1]{YoungHybrid} (cf.\ \cite[(3.19)]{JutilaMotohashi} which treats a different $h_\infty(t)$). In particular, applying the stationary phase method to \cite[(7.1)]{YoungHybrid}, one derives an asymptotic expansion of $H_{\infty}(x)$ with leading term roughly of the form $\Delta T x^{-1/2} e^{ix}$ when $x \gg \Delta T^{1-\eps}$.
The representation \eqref{eq:MellinTransformOfHinf} then follows by Mellin inversion, using stationary phase to bound $W(t)$, cf.\ \cite[p.\ 256]{DI}. For details see \cite[Lemma 4.4]{KhanYoung}.
\end{proof}

\begin{mylemma}
\label{lemma:archimedeanInitialSegmentVersion}
 Suppose that $h_\infty$ is given by \eqref{eq:hdefInitialSegmentVersion}.
If $x \asymp P \gg T^{2+\varepsilon}$, 
and $w(y)$ is a fixed smooth weight function on $(0,\infty)$ supported on $[1,2]$
then  we have
\begin{equation}
  w(x/P) H_{\infty}(x) = \frac{T^2}{P} \int_{|t|  \asymp   P} W(t) x^{it} dt + O(T^{-A}),
\end{equation}
 for some function $W$ with $W(t) \ll T^{\varepsilon}$.   
 In addition, if $w(y)$ is a fixed smooth weight function on $(0, \infty)$ vanishing for $y \geq 2$, then we have
 \begin{equation}
 \label{eq:MellinTransformOfHinfInitialSegmentVersion}
w(x/T^{2+\varepsilon}) H_{\infty}(x) = x^{} \int_{|t| \ll T^{10}} W(t) x^{it} dt + O(T^{-A}),
 \end{equation}
where $W(t) \ll T^4$. 
\end{mylemma}
\begin{myrema}  The $T$-dependence in the integral in \eqref{eq:MellinTransformOfHinfInitialSegmentVersion} is quite bad, but we will only use this when $T$ is small so there is no significant harm in doing so. \end{myrema}

\begin{proof}
The first statement is similar to that in Lemma \ref{lemma:archimedeanWindowVersion}, but using \cite[Lemma 10.3]{PetrowYoungWeyl} in place of \cite[Lem.\ 7.1]{YoungHybrid} and stationary phase.  For the second statement, we use \cite[Lemma 10.2]{PetrowYoungWeyl}, which gives the derivative bound $x^k H_{\infty}^{(k)}(x) \ll x (1+x^{2k})T^{k+1} \ll 
{(T^{1+\varepsilon})^{5k+3}}$.  Now by standard Mellin inversion, we obtain
\begin{equation}\label{mellin_inversion}w(x/T^{2+\varepsilon}) H_{\infty}(x) = \frac{1}{2 \pi i} \int_{(\sigma)} F(s) x^{-s} ds,\end{equation} where $F(s) = \int_0^{\infty} w(x/T^{2+\varepsilon}) H_{\infty}(x) x^s \frac{dx}{x}$.  Integration by parts shows that $F(s)$ equals
$$\frac{(-1)^k}{s(s+1) \dots (s+k-1)} \int_0^{\infty} 
\frac{\partial^k}{\partial x^k} \Big[ w\Big(\frac{x}{T^{2+\varepsilon}}\Big) H_{\infty}(x) \Big] x^{k+s} \frac{dx}{x}
\ll \frac{(T^{1+\varepsilon})^{5k+3}}{|s(s+1) \dots (s+k-1)|}.
$$
Therefore if $|\text{Im}(s)| \gg T^{6}$, say, then $F(s)$ is very small.  Finally, we take the Mellin formula \eqref{mellin_inversion} and shift the contour to $\text{Re}(s) = -1$ (without crossing a pole, by e.g.\ Lemma \ref{Hboundprop}).  We can then truncate the integral at $|t| \ll T^{10}$ leading to the error term in \eqref{eq:MellinTransformOfHinfInitialSegmentVersion}.
\end{proof}

\subsection{Proof of Theorem \ref{thm:abstractversion}}
\label{section:proofoflargesievebound}
Now we have the tools in place to prove Theorem \ref{thm:abstractversion}.  It suffices to suppose that $a_n$ is supported on $X/2 < n \leq X$, say. We also wish to assume that $a_n=0$ if $(n,q) \neq 1$.  To accomplish this, we note that $|\lambda_{\pi}(p)| \leq 1$ for $p | q$ and $\pi \in \cF$.  
Then we can apply Cauchy's inequality as follows:
\begin{equation}\label{proofofOLSI_cauchystep}
\sum_{\pi \in \mathcal{F}} \Big|\sum_{m | q^{\infty}} \sum_{(n,q) = 1} a_{mn} \lambda_{\pi}(m) \lambda_{\pi}(n) \Big|^2
\ll (qN)^{\varepsilon}
\sum_{m | q^{\infty}}
\sum_{\pi \in \mathcal{F}}
\Big|\sum_{(n,q) = 1} b_n  \lambda_{\pi}(n) \Big|^2
\end{equation}
where $b_n = a_{mn}$.  Applying \eqref{eq:LSIgoal}  with coefficients $a_n$ supported on $(n,q) = 1$ to the interior two sums of \eqref{proofofOLSI_cauchystep} we conclude that \eqref{eq:LSIgoal} holds without the coprime condition after moving the sum over $m\mid q^\infty$ back inside.

Let $f \in \cH_{\rm fin}$ be a test function afforded by the hypotheses for the Large Sieve Inequality as in Section \ref{sec:intro:LSI} and let $h_\infty$ be as in Hypothesis \ref{hypNmL} (NmL) of that section. 
Hypotheses TF and NmL relate the quantities $q$ and $T$ (which pertain to $\cF$) to $f_\infty(1)$ and $f(1)$ (which pertain to $\cF_0(f)$) as follows. 

\begin{mylemma}\label{qandf(1)}
For a finite family of cusp forms $\cF$ all having conductor $q$, spectral parameters contained in $[-T,T]$ and satisfying Hypotheses TF, NmL and CvF of Section \ref{sec:intro:LSI}, we have for the $f$ and $h_\infty$ given by these hypotheses that
\begin{equation}
\label{eq:fAversusqT}
f_\A(1) \ll   |\cF| (qT)^{o(1)} \ll qT^2 (qT)^{o(1)}.
\end{equation}
\end{mylemma}
\begin{myrema} It is also true that $\log q \ll \log f(1)$ (see Section \ref{sec:cond_of_reps_vs_f1}), but we do not need this for the proof of the Large Sieve Inequality. \end{myrema}
\begin{proof}
By Lemma \ref{fA(1)_smaller_than_harmonic_family} we have 
$$f_\A(1) \ll_\eps f(1)^\eps\Big(\sum_{\pi \in \cF_0(f)} h_\infty(t_\pi) w(\pi, f) + (\text{ cts. }) \Big),$$
which  
is $ \ll  |\cF| (qT)^{o(1)}$ by Hypothesis \ref{hypNmL} (NmL).
Finally, \eqref{eq:fAversusqT} follows from bounding $|\cF|$ by the total number of cuspidal automorphic forms of conductor $q$ and spectral parameters bounded by $T$ (see \cite{MR664496} for a trace formula-free proof of such a ``weak Weyl law''). 
\end{proof}
Let $\mathcal{M} = \sum_{\pi \in \mathcal{F}} | \sum_n a_n \lambda_{\pi}(n)|^2$. By Hypothesis \ref{hypTF} (TF),  the first part of Hypothesis \ref{hypNmL} (NmL) and the positivity of $h_\infty$, we have 
$$\mathcal{M} \ll (qT)^{o(1)} \sum_{\pi \in \cF_0(f)} h_\infty(t_\pi) w(\pi,f)  | \sum_n a_n \lambda_{\pi}(n)|^2 + (\text{cts.}).$$
Opening the square, and applying Theorem \ref{theoGeomSpec}, we encounter
$$\mathcal{D} = \|\mathbf a \|^2    \delta$$
and
$$\mathcal{S} = \sum_m \sum_n a_m \overline{a_n}  \sum_{c \equiv 0 \shortmod{k(\cF)} } \frac{H(m,n;c)}{c}H_\infty\left(\frac{4 \pi \sqrt{ mn}}{c}\right)$$ 
so that $\mathcal{M}$ satisfies 
$$\mathcal{M} \ll (qT)^{o(1)} \left( \mathcal{D} +  \mathcal{S}\right).$$

By \eqref{eq:delta_vs_f1} and Lemma \ref{qandf(1)}, we have that $\mathcal{D} \ll \|\mathbf a \|^2 (qT)^{o(1)} |\cF|$, which is of acceptable size.

Next we focus on the non-diagonal term $\mathcal{S}$.
We apply a dyadic partition of unity to the $c$-sum, and consider the portion with $c \asymp C$, writing $\mathcal{S} = \sum_{C} \mathcal{S}_C$.
 If $C$ is very large, say $C \gg (X|\mathcal{F}|)^{100}$, then we factor $H(m,n;c)$ into ramified and unramified parts as in Theorem \ref{thmKP}(3). Applying the classical Weil bound to the unramified part, and the trivial bound Theorem \ref{thmKP}(5) to the ramified part, we  obtain an acceptable result.  
 
 By the first phrase of Hypothesis \ref{hypNmL} (NmL) and the assumption that $(n,q)=1$, we have $(n,c_{N})=1$, where $c_N$ is the factor of $c$ dividing the level $N$ as in Theorem \ref{thmKP}(3). Applying Theorem \ref{thmKP}(4) to the Kloosterman sum in $\mathcal{S}$, we  
  obtain
\begin{equation}
\mathcal{S}_C = \frac{1}{C}
\sum_{m,n} a_m \overline{a_n}
\sum_{\substack{c_{N} | {N}^{\infty} \\ c_{N} \equiv 0 \shortmod{k(\cF)}}}
\sum_{(c_0, {N}) = 1} \eta\Big(\frac{c_{N} c_0}{C}\Big)
S(m \overline{c_{N}}, n \overline{c_{N}};c_0)
 H(mn \overline{c_0}^2, 1;c_{N}) H_{\infty}\Big(\frac{4 \pi \sqrt{mn}}{c_{N} c_0}\Big),
\end{equation}
where $\eta$ is some fixed dyadically-supported smooth weight function. 

Recall that $H(u, 1; c_{N})$ is periodic in $u$ modulo $c_{N}$ and vanishes if $c_{N} \not \in \N$ by Lemma \ref{specControlonH}(4). 
Next we apply \eqref{eq:FourierInversion}, giving 
\begin{multline}
\mathcal{S}_C = \frac{1}{C}
\sum_{m,n} a_m \overline{a_n}
\sum_{\substack{c_{N} | {N}^{\infty} \\ c_{N} \equiv 0  \shortmod{k(\cF)} }}
\sum_{(c_0, {N}) = 1} \eta\Big(\frac{c_{N} c_0}{C}\Big)  
S(m \overline{c_{N}}, n \overline{c_{N}};c_0)
\\
\times
\sum_{\chi \shortmod{c_{N}}} \widehat{H}(\chi)  \chi(mn \overline{c_0}^2)
  H_{\infty}\Big(\frac{4 \pi \sqrt{mn}}{c_{N} c_0}\Big).
\end{multline}

The analogous step on the archimedean side is to use the Mellin inversion formula from Lemmas \ref{lemma:archimedeanWindowVersion} and \ref{lemma:archimedeanInitialSegmentVersion}. 
If there exists $\delta, C>0$ such that $T \geq Cq^\delta$, then  we have a paramter $\Delta$ satisfying $T^\eps \ll \Delta \ll T^{1-\eps}$ and a corresponding  $h_\infty$ afforded by Hypothesis \ref{hypNmL} (NmL) that is  of the form \eqref{eq:hdefWindowVersion},   so that we may use Lemma \ref{lemma:archimedeanWindowVersion}. 
If $T \ll q^\eps$ then  the  $h_\infty$ afforded by Hypothesis \ref{hypNmL} (NmL) is  of the form \eqref{eq:hdefInitialSegmentVersion} and we may apply Lemma \ref{lemma:archimedeanInitialSegmentVersion}.
In the remainder of the proof of Theorem \ref{thm:abstractversion} below, we focus on the first case that $T \gg q^\delta$. In the second case, $T$ is small compared to $q$ and the large powers of $T$ occurring in Lemma \ref{lemma:archimedeanInitialSegmentVersion} cause no problems and are absorbed by the $q^{o(1)}$ factor. The proof in the range $T = q^{o(1)}$ follows the same steps as the case $T \gg q^\delta$ with minor changes, so we omit the details.

We henceforth assume that there exists $\delta>0$ and an implicit constant such that $T \gg q^\delta$.
Since $\sqrt{mn} \asymp N$ and $c = c_{N} c_0 \asymp C$, we can freely apply a redundant weight function $w(x/P)$ to $H_{\infty}(x)$, where $P = N/C$. 
After this we apply Lemma \ref{lemma:archimedeanWindowVersion}.  Since $T \gg q^{\delta}$, the error term of size $O(T^{-A})$ in \eqref{eq:MellinTransformOfHinf} is satisfactory.  By the first assertion of Lemma \ref{lemma:archimedeanWindowVersion} we may assume $P \gg \Delta T^{1-\varepsilon}$, equivalently, $C \ll \frac{N}{\Delta T^{1-\varepsilon}}$. 
We thus obtain
\begin{multline}
\mathcal{S}_C = \frac{\Delta T}{CP}
\sum_{m,n} a_m \overline{a_n}
\int_{|t| \asymp P} W(t)
\sum_{\substack{c_{N} | {N}^{\infty}, c_{N} \ll C \\ c_{N} \equiv 0 \shortmod{k(\cF)}}}
\sum_{\substack{c_0 \asymp C/c_{N}  \\ (c_0, {N}) = 1}}
S(m \overline{c_{N}}, n \overline{c_{N}};c_0)
\\
\sum_{\chi \shortmod{c_{N}}} \widehat{H}(\chi)  \chi(mn \overline{c_0}^2)
\Big(\frac{\sqrt{mn}}{c_{N} c_0}\Big)^{it} dt  + O\Big(\frac{N \|{\bf a}\|^2}{(qT)^{A}}\Big).
\end{multline}
Opening the definition of the standard Kloosterman sum and reordering the sums, we obtain 
\begin{multline}
\mathcal{S}_C \ll
\frac{f_\infty(1)}{CP}
\int_{|t| \asymp P} |W(t)|
\sum_{\substack{c_{N} | {N}^{\infty}, c_{N} \ll C \\ c_{N} \equiv 0 \shortmod{k(\cF)}}}
\sum_{\substack{c_0 \asymp C/c_{N}  \\ (c_0, {N}) = 1}}
\thinspace
\sumstar_{y \shortmod{c_0}}
\\
\sum_{\chi \shortmod{c_{N}}} |\widehat{H}(\chi)|
\Big|
\sum_{m,n} a_m \overline{a_n} e_{c_0}(my \overline{c_{N}} + n \overline{y c_{N}})
\chi(mn)
(mn)^{\frac{it}{2}}
\Big| dt  + O\Big(\frac{N \|{\bf a}\|^2}{(qT)^{A}}\Big).
\end{multline}
We then apply $|\sum_{m} |\cdot |\sum_n| \leq 2 |\sum_m|^2 + 2 |\sum_n|^2$
and simplify using Hypothesis \ref{hypFTB} (FTB) and Lemma \ref{qandf(1)}, giving 
\begin{multline}
\mathcal{S}_C \ll  \frac{f_\A(1) (NqT)^{\varepsilon}}{CP}
\int_{|t| \asymp P}
\sum_{\substack{c_{N} | {N}^{\infty}, c_{N} \ll C \\ c_{N} \equiv 0 \shortmod{k(\cF)}}}
\sum_{\substack{c_0 \asymp C/c_{N}  \\ (c_0, {N}) = 1}}
\thinspace
\sumstar_{y \shortmod{c_0}}
\sum_{\chi \shortmod{c_{N}}} 
\Big| \sum_n a_n e_{c_0}(ny) \chi(n) n^{it} \Big|^2
 dt
 \\
   + O\Big(\frac{N \|{\bf a}\|^2}{(qT)^{A}}\Big).
\end{multline}
Note by orthogonality of characters that
\begin{equation}
\sum_{\chi \shortmod{d}} \Big| \sum_{n} b_n \chi(n) \Big|^2 = \frac{\phi(d)}{d} \sum_{u \shortmod{d}} \Big| \sum_{(n,d)=1} b_n e_d(un) \Big|^2.
\end{equation}
  This gives
\begin{multline}
\mathcal{S}_C \ll 
\frac{f_\A(1) (NqT)^{\varepsilon}}{CP}
\int_{|t| \asymp P}
\sum_{\substack{c_{N} | {N}^{\infty}, c_{N} \ll C \\ c_{N} \equiv 0 \shortmod{k(\cF)}}}
\sum_{\substack{c_0 \asymp C/c_{N}  \\ (c_0, {N}) = 1}}
\thinspace
\sumstar_{y \shortmod{c_0}}
\sum_{u \shortmod{c_{N}}} 
\Big| \sum_n a_n e_{c_0}(ny) e_{c_{N}}(nu) n^{it} \Big|^2
 dt
 \\
  + O\Big(\frac{N \|{\bf a}\|^2}{(qT)^{A}}\Big).
\end{multline}
Applying Lemma \ref{lemma:largesieveclassicalvariant2} (the $\GL_1$ large sieve), we derive 
\begin{equation}
\label{eq:SCboundLSIproof}
\mathcal{S}_C \ll \frac{f_\A(1) (NqT)^{\varepsilon}}{CP}
\sum_{\substack{c_{N} | {N}^{\infty}, c_{N} \ll C \\ c_{N} \equiv 0 \shortmod{k(\cF)}}}
\Big(\frac{ C^2}{c_{N}} P + N\Big) \| {\bf a } \|_2^2 .
\end{equation}
The bound above breaks into two parts, corresponding to the two terms $\frac{C^2}{c_{N}} P$ and $N$, respectively.  Using $P=\frac{N}{C}$ bounds the latter term as $f_\A(1)(NqT)^{\varepsilon} \| {\bf a} \|_2^2$. By Lemma \ref{qandf(1)} again this is $\ll |\cF|(NqT)^\eps \|{\bf a} \|^2$, which matches the size of the diagonal term.  
Since we are considering the range $C \ll \frac{N}{\Delta T^{1-\varepsilon}}\asymp \frac{N T^\eps}{f_\infty(1)}$,
the former term reduces to 
\begin{equation}
\label{eq:LSIboundfirstTerm}
f(1) N (NqT)^{\varepsilon} \| {\bf a } \|_2^2
\sum_{\substack{c_{N} | {N}^{\infty}, c_{N} \ll C \\ c_{N} \equiv 0 \shortmod{k(\cF)}}}
\frac{1}{c_{N}}
.
\end{equation}
Hypothesis \ref{hypCvF} (CvF) implies this is bounded by $N (NqT)^{\varepsilon} \| {\bf a} \|_2^2$, as needed for Theorem \ref{thm:abstractversion}.

\subsection{Exceptional spectrum} It can be very useful in practice to have a generalization of Theorem \ref{thm:abstractversion} for the exceptional spectrum, with weights taking into account
the size of potential violations of the Ramanujan conjecture.

Compare with \cite[Thm.\ 5]{DI}.
\begin{myprop}\label{prop:exceptional_spec_LSI}
 Let $\mathcal{F}$, $q$, $f$ be as in Theorem \ref{thm:abstractversion}, and suppose that Hypotheses TF, NmL, FTB, and CvF hold for $f$ and $\mathcal{F}$.  
 Suppose that for each $\pi \in \mathcal{F}$, we have $i t_{\pi} \in (0, 1/4)$.  Let $Y \geq 1$.  Then for any sequence of complex numbers 
 $(a_n)_{n \in \N}$ we have
\begin{equation}
\label{eq:LSIgoalExceptional} 
 \sum_{\pi \in \cF} 
 Y^{2 i t_{\pi}}
 \Big| \sum_{n \leq N} a_n \lambda_\pi(n)\Big|^2 \ll_\eps (|\cF|+ NY) (NqY)^\eps
\| {\bf a } \|_2^2
. 
\end{equation}
 \end{myprop}
 Note that by assumption, every $\pi \in \mathcal{F}$ in Proposition \ref{prop:exceptional_spec_LSI} violates the Ramanujan conjecture (Selberg eigenvalue conjecture) at the archimedean place.  For the forms satisfying Ramanujan, then we may take $Y=1$ and obtain a stronger bound from Theorem \ref{thm:abstractversion}.
\begin{proof}
The structure of the proof is the same as that of Theorem \ref{thm:abstractversion}, but the archimedean analysis will be different.  Let
\begin{equation}
 h_{\infty}(t) = (3 + Y^{2it} + Y^{-2it}) \cosh( \pi t) \exp(-t^2),
\end{equation}
which satisfies the required conditions in \eqref{hinfty_conditions}, 
is nonnegative on the spectrum (both tempered and exceptional), and satisfies $h_{\infty}(t) \gg Y^{2i t}$ for $it \in (0, 1/4)$.
We need to understand the integral transform $H_{\infty}(x)$, which we write as
$H_{\infty}(x) = 3 H_1(x) + H_{Y}(x) + H_{1/Y}(x)$, with
\begin{equation}
 H_{Z}(x) = \frac{i}{2} \intR J_{2it}(x) Z^{2it} \exp(-t^2) t dt.
\end{equation}
Shifting contours to the right shows that $H_{Z}(x) \ll_A (xZ)^{A}$ for $A > 0$ arbitrarily large.  
Hence $H_Z(x)$ is effectively supported on $x \gg \frac{Z^{-1}}{(qNY)^{\varepsilon}}$.  

By \cite[17.43.16]{GR}, we have
\begin{equation}
\label{eq:HZMellinFormula}
 H_{Z}(x) = 
 \frac{1}{4 \pi i} 
 \intR
 \int_{(\sigma)} 
 Z^{2it} \exp(-t^2)
 \frac{2^{s-1} \Gamma(\frac{s+ 2it}{2})}{\Gamma(1 + \frac{s-2it}{2})} x^{-s} ds dt,
\end{equation}
valid for $0 < \mathrm{Re}(s + 2it) < 1$, and $\mathrm{Re}(2it) > -1/2$.
We shift contours to $\mathrm{Re}(2it) = - \varepsilon$ and $\mathrm{Re}(s) = 2 \varepsilon$, which is enough to secure absolute convergence of the double integral in \eqref{eq:HZMellinFormula}.

We now follow the same proof as in Section \ref{section:proofoflargesievebound}, using \eqref{eq:HZMellinFormula} as a substitute for \eqref{eq:MellinTransformOfHinfInitialSegmentVersion}.  The only significant change is that the maximal size of $C$ is now $(NY)^{1+\varepsilon}$ instead of $N^{1+\varepsilon} T^{O(1)}$.  This has the effect that the former term in \eqref{eq:SCboundLSIproof} is of size $\ll NY (NqY)^{\varepsilon} \|{\bf a} \|^2$.
\end{proof}

\section{Test functions for supercuspidal representations}\label{sec:testfunctions}
\subsection{Supercuspidal families, background}\label{sec:egSupercuspidalbackground} 

Let $F$ be a $p$-adic field, $(\sigma,V)$ be a supercuspidal representation of $G(F)$, and $\langle, \rangle$ be a unitary pairing on $V$. Let $\varphi_0$ be an $L^2$-normalized newform in $V$ and define the matrix coefficient 
\begin{equation}
\Phi(g) = \langle \sigma(g)\varphi_0, \varphi_0 \rangle.
\end{equation}
 It is 
well-known (see e.g.\ \cite[Cor.\ 10.26]{KnightlyLiTracesOfHeckeOperators})
that the function $$ f=\frac{1}{ \| \Phi \|_2^2} \overline{\Phi}$$ has the property that $\pi(f)$ is a non-zero newform projector supported on the specified $\{\sigma\} \subseteq {G(F)^{\wedge}}$, i.e.\ $\pi(f)$ has the smallest possible (non-zero) support as a function on $G(F)^{\wedge}$.

Although $f$ has compact support modulo center, this control on the support of $f$ on $G(F)$ is insufficient for the purposes of this paper -- we need test functions with support in a compact open \emph{subgroup} of $G(F)$. Instead, we will choose our test functions to be restrictions of the diagonal newform matrix coefficients to appropriate compact open subgroups, and show in Sections \ref{sec:egSupercuspidalbackground}, \ref{sec:egSupercuspidal_odd} and \ref{sec:egSupercuspidal_even} that these retain the property of being newform projectors, and only slightly enlarge the support of $\pi(f)$.

\subsubsection{Basics}\label{sec:tesetfunctions_basic}
Given $F$, let $\cO$ be its ring of integers, $\fp$ its prime ideal, $k_F= \cO/\fp\simeq \F_q$ its residue field and choose a uniformizer $\varpi\in \fp$. We write $$U(i)= \begin{cases} \cO^\times & \text{ if } i=0, \\ 1+ \fp^i & \text{ if } i>0\end{cases}$$ for the standard multiplicative filtration of $\cO^\times$. We will sometimes decorate these notations with a subscript $F$ if we want to emphasize the field of definition.

Let $\psi$ be an additive character of $F$ of conductor exponent $c(\psi)$. Let $E/F$ be a finite extension with residue field extension degree $f=f(E/F)$, ramification exponent $e=e(E/F)$ and valuation of the discriminant $d=d(E/F)$. One extends $\psi$ to an additive character $\psi_E$ of $E$ by $\psi_E = \psi \circ \Tr$. The conductor exponent of $\psi_E$ is then given by 
\begin{equation}\label{add_char_cond}
c(\psi_E) = e c(\psi) - df^{-1},
\end{equation}
see e.g.\ \cite[Lem.\ 2.3.1]{Schmidt:02a}.

For $\chi$ a multiplicative character of $F$, let $c(\chi)$ be its conductor exponent with respect to the filtration $U(i)$. 
\begin{mylemma}[Postnikov]\label{postnikov}
For any integer $i>e_{F/\Q_p}/(p-1)$ the $p$-adic logarithm $\log : U(i) \to \fp^i$ is an isomorphism of topological groups defined by 
\begin{equation*}
	\log \left( 1+u \right) =u-\frac{u^2}{2}+\frac{u^3}{3}+\cdots.
	\end{equation*}
For any character $\chi$ of $F^\times$ and integer $i>e_{F/\Q_p}/(p-1)$ satisfying $c(\chi)\geq \max(i,2)$, there exists a unique $\alpha_\chi \in \varpi^{-c(\chi)+c(\psi_F)}\left( \cO/ \fp^{c(\chi)-i} \cO \right)^\times$ 
such that 
\begin{equation}
\chi(1+u) =\psi_F\left(\alpha_\chi \log (1+u) \right) \quad \text{ for all } u \in \fp^i.
\end{equation}
If $1\leq i $ and $c(\chi)\leq 2i$, then there exists $\alpha_\chi \in F$ with $v(\alpha_\chi)= -c(\chi)+c(\psi_F)$ such that 
\begin{equation}\label{postnikov_v2}
\chi(1+u) = \psi_F(\alpha_\chi u) \quad \text{ for all } u \in \fp^i.
\end{equation} 
If $i< c(\chi)$ then $\alpha_\chi \in \varpi^{-c(\chi)+c(\psi_F)}\left( \cO/ \fp^{c(\chi)-i} \cO \right)^\times$ is uniquely determined by $\chi$. If $i\geq c(\chi)$ then any $\alpha_\chi$ with $v(\alpha_\chi)\geq -i+c(\psi_F)$ satisfies \eqref{postnikov_v2}. 
\end{mylemma}
\begin{proof}
See e.g.\ \cite[\S 1.7, 1.8]{BushnellHenniart:06a} and \cite[Lem.\ 2.1]{PetrowYoungCoset}, the proof of which generalizes in a straightforward way. 
\end{proof}

Now let $E/F$ be a quadratic extension of non-archimedean local fields. Recall the notion of a \emph{minimal element} from \cite[\S 13.4]{BushnellHenniart:06a}. We explicate a special case of this notion.
\begin{mydefi}\label{normminlelt}
An element $\alpha_0 \in E$ is called a normalized minimal element if 
\begin{enumerate}
\item $E=F[\alpha_0]$, 
\item $v_E(\alpha_0)= e(E/F)-1$, and 
\item if $E/F$ is unramified, then $\alpha_0 \pmod {\fp_E}$ generates the residue field extension $k_E/k_F$.
\end{enumerate}
\end{mydefi}
If $\alpha$ a minimal element, then one obtains a normalized minimal element $\alpha_0$ by scaling by an appropriate power of $\varpi$, see \cite[(13.4.1)]{BushnellHenniart:06a}. On then other hand,  if $\alpha_0$ is a normalized minimal element, then  $\alpha_0 \varpi^n \in E\smallsetminus F$ is a minimal element for any $n\in \Z$.  A normalized minimal element $\alpha_0$ for $E/F$ moreover satisfies 
\begin{enumerate}
\item $\cO_E = \cO_F[\alpha_0]$, and 
\item if $E/F$ is unramified, then the minimal polynomial $g(x) = x^2 +Ax+B$ of $\alpha_0$ satisfies $v_F(A)\geq 0$ and $v_F(B) =0$. 
\end{enumerate}

Given a character $\chi$ of $E^\times$, we re-normalize the factor $\alpha_\chi$ from Lemma \ref{postnikov} by defining
\begin{equation}\label{ellchidef}
\ell_\chi :=\alpha_\chi \varpi_E^{c(\chi)},
\end{equation}
so that when $c(\psi_F)=0$, the factor $\ell_{\chi}$ lies in the inverse different $\mathfrak{D}_{E/F}^{-1}$ of $E/F$:
$$\ell_\chi \in  \fp_E^{-d} = \{ x \in E : \Tr (xy) \in \cO_F, \, \forall y \in \cO_E\} = g'(\alpha_0)^{-1} \cO_E,$$
see e.g.\ \cite[41.2 Prop.\ (1)]{BushnellHenniart:06a} and \cite[Ch.\ III (2.4) Prop.]{Neukirch}.

\begin{mylemma}[cf.\ Cor.\ 2.12 of \cite{HN18}]\label{normminlelt_lemma}
Suppose $\alpha_0 \in \cO_E$ is a normalized minimal element. We have $v_E(a+b\alpha_0) = \min (v_E(a), v_E(b\alpha_0))$ for any $a,b \in F$. 
\end{mylemma}
\begin{proof} {The statement is clear if $E/F$ is ramified, so suppose otherwise.} 
We have
\begin{align*}
v_E(a+b\alpha_0)&=\frac{1}{2}v(\Nm(a+b\alpha_0))=\frac{1}{2}v(a^2+ab\Tr(\alpha_0) +b^2\Nm(\alpha_0))\\
&=v(b)+\frac{1}{2}v\left(g\left(-\frac{a}{b}\right)\right).
\end{align*}
When $v(a)>v(b)$ it is easy to see that $v\left(g\left(-\frac{a}{b}\right)\right)=v(B)=0$; 
When $v(a)<v(b)$, $v\left(g\left(-\frac{a}{b}\right)\right)=2v(a)-2v(b)$;
When $v(a)=v(b)$, 
$v\left(g\left(-\frac{a}{b}\right)\right)=0$ as the congruence class of $g(x)$ is also an irreducible polynomial over $k_F$, thus will not have a solution $-\frac{a}{b}$. 
\end{proof}
Given a normalized minimal element $\alpha_0$ for $E/F$ with minimal polynomial $g(x)=x^2+Ax+B$, we fix the embedding 
\begin{align}\label{Oembedding}
E^\times & \hookrightarrow G(F) \nonumber \\
x+y \alpha_0 & \mapsto \left( \begin{smallmatrix} x & y \\ -By & x-Ay \end{smallmatrix}\right).
\end{align}

We say that a character $\theta$ of $E^\times$ is twist-minimal if $c(\theta) = \min_{\chi} c(\theta \chi_E)$, where $\chi$ runs over characters of $F^\times$ and $\chi_E$ denotes the character $\chi \circ \Nm$ of $E^\times$. 
\begin{mylemma}[cf.\ Lem.\ 2.13 of \cite{HN18}]\label{minimalimpliesminimal}
Suppose $E/F$ is ramified. If a character $\theta$ of $E^\times$ is twist-minimal, then $\alpha_\theta \in E^\times$ is a minimal element for $E/F$. 
\end{mylemma}
\begin{proof}
It suffices to show that $v_E(\alpha_\theta)$ is odd, as $E/F$ is ramified. 
For any character $\chi$ of $F^\times$ we have $\alpha_{\theta\chi_E} = \alpha_{\theta}+ \alpha_\chi$, where $\alpha_\chi \in F$, and so by minimality of $\theta$ we have $v_E(\alpha_{\theta}) = \max_\chi v_E(\alpha_{\theta}+ \alpha_\chi)$. Now let $\beta_0$ be some other normalized minimal element in $E/F$. Then, we write  $\alpha_{\theta} = a+b\beta_0$. By Lemma \ref{normminlelt_lemma}, we get that 
$$ v_E(\alpha_{\theta}) = \max_\chi v_E(\alpha_{\theta}+ \alpha_\chi) = \max_\chi \min( v_E(\alpha_\chi + a), v_E(b\beta_0)),$$
and the maximum is attained when $\chi$ is chosen so that $\alpha_\chi = -a$. Therefore, we have shown that $v_E(\alpha_{\theta}) = v_E(b \beta_0) = 2v_F(b)+ v_E(\beta_0),$ which is odd since $\beta_0$ is a minimal element of $E/F$.
\end{proof}

\subsubsection{Parametrization of dihedral supercuspidals}\label{sec:parametrization_dihedrals}
In this paper we are only interested in projections $\pi(f)$ to \emph{dihedral} supercuspidal representations. 
We next recall some of the dihedral Local Langlands Correspondence (LLC) following Bushnell and Henniart \cite{BushnellHenniart:06a}.

Let $E/F$ be a quadratic extension of non-archimedean local fields and recall that a character $\xi$ of $E^\times$ is called \emph{regular} if $\xi$ does not factor through the norm map $\Nm:E^\times \to F^\times$ (equivalently, if $\xi \neq \xi^\sigma$ for the non-trivial $\sigma \in \Gal(E/F)$). Two pairs $(E/F,\xi), (E'/F,\xi')$ are said to be $F$-isomorphic $\sim_F$ if there is an $F$-isomorphism $j: E \to E'$ such that $\xi = \xi' \circ j$. In the case $E=E'$, this amounts to $\xi = \xi^\sigma$ for some $\sigma \in \Gal(E/F)$.

To each pair $(E/F, \xi)$ consisting of a quadratic extension $E/F$ and a regular character $\xi$, the Weil group $W_F$ representation $\rho= \Ind_{E}^F \xi$ is irreducible. The LLC then associates to $\rho$ an irreducible supercuspidal representation $\pi=\pi(\rho)$ of $G(F)$. The central character of $\pi$ is equal to $\eta_{E/F}\xi \vert_{F^\times}$, where we write $\eta_{E/F}$ for the character of $F^\times$ corresponding to $E/F$ by class field theory, i.e.\ the unique quadratic character of $F^\times$ that is trivial on $\Nm E^\times$. The conductor exponent of $\pi$ satisfies \cite[Thm.\ 2.3.2]{Schmidt:02a}
\begin{equation}\label{c(pi)c0}
c(\pi) = \frac{2}{e}c(\xi) + d.
\end{equation}

Denote by $\mathcal{A}_2^0(F)$ the set of equivalence classes of irreducible supercuspidal representations of $G(F)$.  Let $$\widetilde{\P}_2(F) = \{ (E/F, \xi) : \xi \text{ regular }\}/\sim_F,$$ and define the map
$$ i: \widetilde{\P}_2(F) \to \mathcal{A}_2^0(F)$$
$$ (E/F, \xi) \mapsto \rho= \Ind_{E}^F \xi \mapsto \pi(\rho).$$
In general, the map $i$ is neither injective nor surjective. However, the restriction of $i$ to some special subsets of $\widetilde{\P}_2(F)$ will be injective and one can determine its image as follows.

First suppose that $E/F$ is at most tamely ramified.  Recall \cite[\S 18.2 Def.]{BushnellHenniart:06a} the following
\begin{mydefi}\label{def:admissible_pair}
A pair $(E/F, \xi) \in \widetilde{\P}_2(F)$ is called \emph{admissible} if 
\begin{enumerate}
\item $E/F$ is at most tamely ramified, and
\item if $\xi\vert_{U_E(1)}$ factors through  $\Nm_{E/F}$, then $E/F$ is unramified. 
\end{enumerate} Write $\P_2(F)$ for the set of admissible pairs:
$$ \P_2(F) = \{ (E/F ,\xi) \in \widetilde{\P}_2(F) : (E/F ,\xi) \text{ is admissible } \}/ \sim_F.$$
\end{mydefi}
Let us say that $\pi \in \mathcal{A}_2^0(F)$ is \emph{non-ramified} if there exists an unramified character $\phi \neq 1$ of $F^\times$ such that $\pi \times \phi \simeq \pi$, and denote the set of non-ramified representations by $\mathcal{A}_2^{\rm nr}(F) \subset  \mathcal{A}_2^0(F)$. 
\begin{mytheo}[Tame Parametrization Theorem]\label{TameParametrizationTheorem}
The map $i$ is a bijection of the sets 
\begin{equation}\label{tamebijection}  i:\P_2(F) \rightarrow \mathcal{A}_2^{\rm tame}(F):= \begin{cases}  \mathcal{A}_2^0(F) & \text{ if } p \neq 2, \text{ or } \\ \mathcal{A}_2^{\rm nr}(F) & \text{ if } p=2.\end{cases}\end{equation}
\end{mytheo}
\begin{proof} Compose the Tame Parametrization Theorem \cite[\S 20.2]{BushnellHenniart:06a} with loc.\ cit.\ 34.4 Lemma (2). 
\end{proof}
Recall that when $F=\Q_p$, $p\neq 2$ and $\pi$ has trivial central character, we have $c(\pi)$ even if and only if $E/F$ is unramified.

 Now consider the case $F=\Q_2$. 
Let $\P_2(\Q_2)^1_{\geq 9}$ be given by 
$$\P_2(\Q_2)^1_{\geq 9} = \{ (E/\Q_2,\xi) \in \widetilde{\P}_2(\Q_2): \xi \vert_{\Q_2^\times} = \eta_{E/\Q_2} \text{ and } \frac{2}{e}c(\xi) + d \geq 9\}/ \sim_{\Q_2}.$$ 

\begin{mytheo}\label{thm_p2bijection}
The map $(E/\Q_2,\xi) \mapsto \rho=\Ind_E^{\Q_2} \xi$ is a bijection between  $\P_2(\Q_2)^1_{\geq 9}$ and the set of irreducible smooth 2-dimensional representations of $W_{\Q_2}$ with $\det(\rho)=1$ and $c(\rho)\geq 9$. 
\end{mytheo}
\begin{proof}[Proof sketch]
On the one hand all 2-dimensional smooth irreducible representations $\rho$ of $W_{\Q_2}$ with $\det(\rho)=1$ and Artin conductor $\geq 8$ are induced representations by e.g.
 \cite[\S6]{Rio}  or \cite[Prop.\ 3.9]{CesnaviciusNeururerSaha}, 
	and on the other hand, one can use the theory in \cite[\S 41]{BushnellHenniart:06a} to show that there are no triply-imprimitive representations $\rho$ with $\det(\rho)=1$ and $c(\rho)\geq 9${, so the map is injective}. 
	We omit the details.
\end{proof} 
\begin{mycoro}\label{cor_p2bijection}
The map $i$ from $\P_2(\Q_2)^1_{\geq 9}$ to the set of trivial central character supercuspidal representations $\pi$ of $G(\Q_2)$ with $c(\pi)\geq 9$ is a bijection. 
\end{mycoro}
 For any character $\xi$ of $E^\times$, we set
\begin{equation}\label{c0def}
c_0 := c(\xi)/e.
\end{equation}
This quantity was introduced by the first author \cite[(1.3)]{Hu}, where it was denoted $i_0$ and defined in terms of a character $\theta$ giving rise to a representation $\sigma$ of $\overline{G}(\Q_p)$ by compact or parabolic induction cf.\ Section \ref{compactinductionbackground}. On the other hand, in this paper the quantity $c_0$  will generally be defined in terms of the pair $(E/F,\xi)$ corresponding to $\sigma$ via the Local Langlands Correspondence. We observe, however, that $c(\theta)= c(\xi)$ whenever $E/F$ is at most tamely ramified \cite[\S 34.4]{BushnellHenniart:06a}, in particular, whenever $p \neq 2$. Since the first author's paper assumed $p\neq2$, the $c_0$ in this paper may be viewed as an extension of $i_0$ to the wildly ramified case. 
\begin{mylemma}
Suppose that either $(E/F,\xi)$ is an admissible pair with $E/F$ ramified, or that $(E/\Q_2,\xi) \in \P_2(\Q_2)^1_{\geq 9}$ with $E/\Q_2$ ramified. If $\xi \vert_{F^\times} = \eta_{E/F}$, then $c(\xi)$ is even, and in particular the corresponding $c_0$ is an integer. 
\end{mylemma}
\begin{proof}
Suppose that $c(\xi)$ is odd, and let $n=c(\xi)-1$. 
Consider the map induced by inclusion
\begin{equation}\label{eq:newlem1}I: U_F(\frac{n}{2}) / U_F(\frac{n}{2}+1) \to U_E(n)/U_E(n+1).\end{equation}
It is simple to check that $I$ is injective. On the other hand, 
since $E/F$ is totally ramified, the residue fields $k_F \simeq k_E$ of $F$ and $E$ are isomorphic. Recall that for any $m\geq 1$ and any non-archimedean local field $L$ that $U_L(m)/U_L(m+1) \simeq k_L$. Therefore, the map $I$ is in fact an isomorphism. 

Recall \cite[Ch.\ V \S3 Prop.\ 4]{SerreLocalFields} that for a totally ramified cyclic extension $E/F$ of non-archimedean local fields and each integer $m\geq 0$ there is a group homomorphism
\begin{equation}
N_m: U_E(\psi(m))/U_E(\psi(m)+1) \to U_F(m)/U_F(m+1)
\end{equation}
induced by the norm map, where $\psi$ is the Hasse-Herbrand function (see loc.\ cit.\ Chapter IV \S3). By loc.\ cit.\ Chapter V, \S3, Lemma 3 and Corollary 2, the map $N_m$ is surjective whenever $m>d-1$. We apply this with $m=n/2$ to find that whenever $n/2 >d-1$ every element of the domain of $I$ is a norm from $E^\times$ (precisely, it is a norm from $U_E(\psi(\frac{n}{2}))$).

Finally, since $\xi$ has conductor $c(\xi) = n+1$, it restricts to a non-trivial character of $U_E(n)/U_E(n+1)$. But since the inclusion map $I$ is an isomorphism, every element of this group in fact has a representative from $U_F(\frac{n}{2})$, all of which have a representative that is a norm from $E^\times$. But by assumption, $\xi$ is trivial on norms from $E^\times$. Contradiction, at least when $n/2>d-1$. 

Now we proceed by cases to check that the result for $n/2>d-1$ implies the statement in the lemma. First assume that $E/F$ is tamely ramified. Then, $d=1$, so that $n/2>d-1$ if and only if $c(\xi) \geq 3$. Thus, there are no characters $\xi$ as in the statement of the Lemma with odd conductor exponents $\geq 3$. It remains to consider the case $c(\xi)=1$. In this case $\xi \vert_{U_E(1)}$ is trivial, and since $\xi \vert_{\Nm(E^\times)}=1$ by hypothesis, this means that $\xi \vert_{U_E(1)}$ factors through the norm map. Since $E/F$ is ramified, this contradicts the hypothesis that $(E/F,\xi)$ is an admissible pair (Definition \ref{def:admissible_pair}(2)). 

Lastly, consider the case that $\xi$ is a character of $E^\times$, where $E/\Q_2$ is ramified and has $\frac{2}{e}c(\xi)+d\geq 9$. In particular, $c(\xi)\geq 7$, which  
implies $n/2 \geq 3>d-1$. 
Thus, there do not exist pairs $(E/\Q_2,\xi) \in  \P_2(\Q_2)^1_{\geq 9}$ with $E/\Q_2$ ramified and $c(\xi)$ odd. 
\end{proof}

Finally, given a pair $(E/F,\xi)$ and $0\leq n \leq c(\xi)$, define the neighborhood $\xi[n]$ around $\xi$ of radius $n$ by 
\begin{equation}\label{xi_n_def}
\xi[n]=\{\xi_1\in (E^\times)^\wedge: c(\xi_1 \xi^{-1})\leq n, \, \xi_1 \vert_{F^\times} = \xi \vert_{F^\times} \}.
\end{equation}
For $0\leq i \leq n$ define the equivalence relation $\sim_i$ on $\xi[n]$ by $\xi_1\sim_i \xi_1'$ if and only if $c(\xi_1^{-1}\xi_1')\leq i$.
\begin{myrema}\label{size_thetaell_rem} When $0\leq \ell <c_0$ we have (cf.\ \cite[Lem.\ 3.5]{Hu}) that 
$$
\# \theta[\ell e_{E/F}] = 
\begin{cases} 
q^\ell(1+q^{-1}) & \text{ if } e_{E/F} =1, \\ 2q^\ell & \text{ if } e_{E/F} =2.
\end{cases}
$$\end{myrema}

\subsubsection{Compact Induction}\label{compactinductionbackground}
\textbf{Case: $E/F$ at most tamely ramified.} To each $F$-isomorphism class of admissible pairs $(E/F,\theta)$ one associates a supercuspidal representation $\pi_\theta$ by compact induction:
\begin{align}\label{TameCompactIndMap} \P_2(F) & \to \mathcal{A}_2^0(F) \nonumber \\ (E/F, \theta) & \mapsto \pi_{\theta}, \end{align} specifically by the process described in \cite[\S 19]{BushnellHenniart:06a}, culminating in (19.6.3) of loc.\ cit.. 

The map \eqref{TameCompactIndMap} does \emph{not} match the map $i$ in \eqref{tamebijection} that is defined via the LLC. However, in the tame case one patches up this discrepancy by defining, for each $(E/F, \xi) \in \P_2(F)$, an auxiliary character $\Delta_\xi$ of $E^\times$ as in \cite[\S 34.2-34.4]{BushnellHenniart:06a} for which the following lemma holds. 
\begin{mylemma}\label{lem:tamediagram}
The following diagram commutes and all arrows are bijections. 
$$\xymatrixcolsep{8pc}\xymatrix{ \P_2(F) \ar[r]^{(E/F, \xi) \mapsto (E/F, \Delta_\xi \xi)} \ar[rd]_{i} & \P_2(F) \ar[d]^{(E/F,\theta) \mapsto \pi_\theta} \\ &  \mathcal{A}_2^{\rm tame}(F).} 
$$
\end{mylemma}
\begin{proof}
The diagram commutes by \cite[34.4 Tame Langlands Correspondence]{BushnellHenniart:06a}. 
The horizontal map is a bijection by \cite[34.4 Lem.(2)]{BushnellHenniart:06a}, and the other two are as well by the Tame Parametrization Theorem (Theorem \ref{TameParametrizationTheorem}). 
\end{proof}
One of the properties of the character $\Delta_\xi$ that can be found in \cite[\S 34.4]{BushnellHenniart:06a} is that $\Delta_\xi \vert_{F^\times}= \eta_{E/F}$, so that $\pi$ has trivial central character if and only if $\Delta_\xi \xi \vert_{F^\times} =1$. For later use, note that if $(E/F, \xi)\in \P_2(F)$ satisfies $\xi \vert_{F^\times}=\eta_{E/F}$, then $x^\sigma = -x$ for any of $x= \alpha_\xi, \alpha_{\Delta_\xi\xi}, \ell_{\xi},$ or $\ell_{\Delta_\xi \xi}$ and $\sigma \in \Gal(E/F)$, $\sigma \neq 1$.

For later use in the $p=2$ non-ramified case, we very briefly describe the construction of the tame compact induction $(E/F,\theta) \mapsto \pi_\theta$ of \eqref{TameCompactIndMap}, referring the reader to \cite[\S 3.2.1]{Hu} and \cite{BushnellHenniart:06a} for more details. 

Given $(E/F, \theta) \in \P_2(F)$, let $\alpha_0$ be a normalized minimal element for $E/F$ and $E^\times \hookrightarrow G(F)$ be the corresponding embedding \eqref{Oembedding}. Then $\theta$ naturally extends to a character $\widetilde{\theta}$ of a subgroup $ZB^1$ of $G(F)$, see \cite[(3.8), Def.\ 3.11]{Hu}. We can further induce and extend $\widetilde{\theta}$ to an irreducible finite-dimensional representation $\Lambda$ of the subgroup $J \subset G(F)$ defined between (3.7) and (3.8) of loc.\ cit. 
If $c(\theta)\geq 2$, then $$ \dim \Lambda = \begin{cases} 1 & \text{ if } c(\theta) \text{ is even, } \\ q & \text{ if } c(\theta) \text{ is odd.}\end{cases}$$ If $\theta$ is also twist-minimal, then $\pi_\theta :=\cInd_J^G \Lambda$ is irreducible and supercuspidal and realizes the map \eqref{tamebijection}, see \cite[\S 19.2-19.4]{BushnellHenniart:06a}. For later use in the proof of Proposition \ref{minlvec_projection_prop}, we note in particular that \cite[Prop.\ 3.14]{Hu} holds in the case $p=2$ and $E/F$ unramified if we add the additional hypothesis that $\pi$ is twist-minimal. 

\textbf{Case: $E/F$ wildly ramified.} We describe the compact induction theory in more detail, following closely \cite{BushnellHenniart:06a}. We would like (for later purposes) a diagram similar to the one appearing in Lemma \ref{lem:tamediagram} by which we can relate the characters $\theta$ and $\xi$ that lead to the same supercuspidal representation by compact induction and the LLC, respectively. Such a relation is given by \cite[44.3 Thm.]{BushnellHenniart:06a}, but to describe it precisely and in a form useful for our purposes, we need to recall the notions of cuspidal types and simple strata. For the next two paragraphs we follow closely \cite[\S12, \S13]{BushnellHenniart:06a}. 

Let $A =M_2(F)$ and consider the $\cO$-orders in $A$ given by $\mathfrak{A}_1 = \left( \begin{smallmatrix} \cO & \cO \\ \cO & \cO \end{smallmatrix}\right)$ and $\mathfrak{A}_2 = \left( \begin{smallmatrix} \cO & \cO \\ \fp & \cO \end{smallmatrix}\right)$. In general, an $\cO$-order $\mathfrak{A}\subset A$ is called a \emph{chain order} in $A$ if it $G(F)$-conjugate to either $\mathfrak{A}_1$ or $\mathfrak{A}_2$. Given a chain order $\mathfrak{A}$, let $\mathfrak{P}= {\rm rad} (\mathfrak{A})$ be its Jacobson radical. In particular, we have ${\rm rad} (\mathfrak{A}_1) = \varpi \mathfrak{A}_1$ and ${\rm rad} (\mathfrak{A}_2) = \left( \begin{smallmatrix}  & 1 \\ \varpi &  \end{smallmatrix}\right)\mathfrak{A}_2.$ Define the filtration on $\mathfrak{A}^\times$ $$ U_{\mathfrak{A}}^k =\begin{cases} \mathfrak{A}^\times & \text{ if } k=0,\\ 1+ \mathfrak{P}^k & \text{ if } k\geq 1.\end{cases}$$ 
For any $F$-subalgebra $E \subset A$ such that $E/F$ is a quadratic field extension, {there is a unique chain order $\mathfrak{A}$ such that $E^\times$ is a subgroup of the normalizer $\mathcal{K}_{\mathfrak{A}}$ {of $\mathfrak{A}^\times$} 
by	\cite[12.4 Prop.(2)]{BushnellHenniart:06a}. For such $\mathfrak{A}$,} we have 
\begin{align}\mathfrak{P}^k\cap E & = \fp_E^k, \notag \\
	 U_{\mathfrak{A}}^k\cap E^\times & =U_E(k). \label{BH_s_12point4}
\end{align}
For an example of these objects with $F=\Q_2$ and $E\hookrightarrow A$ by \eqref{Oembedding}, see Lemma \ref{Q2chainorders}.

A \emph{stratum} is a triple $(\mathfrak{A}, n, a)$ consisting of a chain order $\mathfrak{A}$, an integer $n$ and an element $a \in \mathfrak{P}^{-n}$. We leave the reader to recall the notions of fundamental, simple and ordinary strata from \cite[\S 12.8, 13.1, 45]{BushnellHenniart:06a}, but let us recall explicitly that those fundamental strata $(\mathfrak{A}, n, a)$ for which $\mathfrak{A}$ is conjugate to $\mathfrak{A}_2$ and $n$ is odd are called \emph{ramified simple strata}. 

For use until the end of Section \ref{compactinductionbackground}, let $\psi'$ be an additive character of $F$ of conductor $1$ so that $\psi'$ matches the convention on additive characters in Bushnell and Henniart's book.  Given a stratum $(\mathfrak{A}, n, a)$, let $\psi'_a$ be the character of $U_{\mathfrak{A}}^n$ defined by $x \mapsto \psi' ( \Tr( a (x-1)))$ for $x \in U_{\mathfrak{A}}^n$. If $(\mathfrak{A}, n, a)$ is moreover a simple stratum, then let us take the subring $F[a] \subset A$ with $F$ embedded diagonally in $A$, so that $F[a]/F$ is a quadratic field extension \cite[\S 13.4]{BushnellHenniart:06a}. Define the subgroup \begin{equation}\label{Jalphadef}
J_{a} = F[a]^\times U_{\mathfrak{A}}^{\lfloor \frac{n+1}{2} \rfloor} \subset G(F).
\end{equation}

Recall the following. 
\begin{mydefi}A \emph{cuspidal type of the second kind} in $G(F)$ is a triple $(\mathfrak{A}, J, \Lambda)$ where, $\mathfrak{A}$ is a chain order, $J$ is a subgroup of $\mathcal{K}_{\mathfrak{A}}$ , and $\Lambda$ is an irreducible smooth representation of $J$ such that there exists a simple stratum $(\mathfrak{A}, n, \alpha)$ with $n\geq 1$, $J=J_\alpha$, and for which $\Lambda \vert_{U_{\mathfrak{A}}^{\lfloor n/2 \rfloor +1}}$ is a multiple of $\psi'_\alpha$. 
\end{mydefi}
Let $T(F)$ denote the set of $G(F)$-conjugacy classes of cuspidal types of the second kind. The following is \cite[15.5 Classification Theorem]{BushnellHenniart:06a}. 
\begin{mytheo}[Classification Theorem]
The map \begin{equation}\label{classificationbijection}(\mathfrak{A}, J, \Lambda) \mapsto \cInd_J^G \Lambda\end{equation}
is a bijection $$T(F) \to \{\pi \in \mathcal{A}_2^0(F): \pi \text{ is twist-minimal}, c(\pi)\geq 3\}.$$
\end{mytheo}

We are interested in the ramified dihedral subset of the above classification bijection \eqref{classificationbijection}. 
\begin{mylemma}\label{lem:wilddiagram}
The map $(\mathfrak{A}, J, \Lambda) \mapsto \cInd_J^G \Lambda$  
is a bijection from $$\{ (\mathfrak{A}, J, \widetilde{\theta}) \in T(F): \exists \text{ an ordinary  ramified simple } (\mathfrak{A}, n, \alpha) \text{ with }n\geq 1, J=J_\alpha, \widetilde{\theta} \vert_{U_{\mathfrak{A}}^{\frac{n+1}{2}}} \simeq  \psi'_\alpha\}$$ 
to 
$$\{ \pi \in \mathcal{A}_2^0(F):  \pi \text{ twist-minimal}, c(\pi)\geq 3
, \exists E/F \text{ ramified with } \pi \simeq \pi( \Ind_E^F \xi) \}.$$

In this bijection, we have that $n=c(\pi)-2$ and that $n,c(\pi)$ are necessarily odd. 
\end{mylemma}
\begin{proof}
The compact induction map has image in the latter set of supercuspidal representations by \cite[44.3 Thm.]{BushnellHenniart:06a}, and is surjective by the discussion in \cite[\S 44.4]{BushnellHenniart:06a}. Note that every irreducible representation $\Lambda$ of $J_\alpha$ for which $\Lambda \vert_{U_{\mathfrak{A}}^{\lfloor n/2 \rfloor +1}}$ is a multiple of $\psi'_\alpha$ is necessarily $1$-dimensional when $n$ is odd \cite[15.6 Prop.\ 1]{BushnellHenniart:06a}, as is the case for a ramified simple strata. Therefore, the restriction to $\Lambda = \widetilde{\theta}$ a character in the set of cuspidal types is no restriction at all. The fact that $n = c(\pi)-2$ follows from \cite[\S 44.4]{BushnellHenniart:06a}, recalling that $n=n(\pi,\psi')$ in Bushnell-Henniart is with respect to $\psi'$ having level 1, whereas our definition of conductor exponent $c(\pi)$ is with respect to an additive character of conductor 0. 
\end{proof}

Suppose that $\pi, (E/F, \xi)$ are as in the image set of Lemma \ref{lem:wilddiagram}.  Recall the element $\alpha_\xi$ associated to the character $\xi$ of $E^\times$ and the additive character $\psi' \circ \Tr$ of $E$ by Lemma \ref{postnikov}.  Suppose $(\mathfrak{A},n, \alpha)$ is an ordinary ramified simple stratum giving rise to a cuspidal type $(\mathfrak{A}, J_\alpha, \widetilde{\theta})$ in the conjugacy class corresponding to $\pi$ in the domain of Lemma \ref{lem:wilddiagram}.

\begin{mylemma}\label{alphathetaalphaxi}
Write $n=2m+1$ and suppose that 
$$ 2 \min(v_E(2)+1,2 \lfloor \frac{d+1}{2}\rfloor)<m+3 \quad \text{ and } \quad d\leq \lfloor \frac{m}{2} \rfloor +1.$$ There exists an isomorphism $E \simeq F[\alpha]$ sending $F$ to $Z$ with respect to which $\alpha_{\xi} \equiv \alpha \mod{ \fp_E^{-\frac{n-3}{2}-\min(v_E(2)+1,2 \lfloor \frac{d+1}{2}\rfloor)}}$. 
\end{mylemma}
\begin{proof}
The elements $\alpha$ and $\alpha_\xi$ are minimal by \cite[13.4 Prop.\ (1)]{BushnellHenniart:06a} and Lemma \ref{minimalimpliesminimal}, respectively. 
Let $\alpha_0$ and $\alpha_{\xi,0}$ be the corresponding normalized minimal elements as in Definition \ref{normminlelt}. Let $$g(x) = x^2- (\Tr \alpha_0)  x+ \det \alpha_0$$ be the minimal polynomial of $\alpha_0$. 
By \cite[44.3 Thm.]{BushnellHenniart:06a}, we have that
\begin{equation*} \Tr \alpha_0 \equiv \varpi_F^{\frac{n+1}{2}}  \delta_{E/F}+  \Tr_{E/F} \alpha_{\xi,0}\pmod{\fp_{F}^{\lfloor \frac{m+3}{2}\rfloor}}, \end{equation*}
where $\delta_{E/F} \in \fp_F^{-(d-1)}$ is such that $\eta_{E/F}(1+x) = \psi'(\delta_{E/F}x)$ for all $x \in \fp_F^{1+\lfloor \frac{d-1}{2}\rfloor}$. 
By hypothesis, we have $ \varpi_F^{\frac{n+1}{2}}  \delta_{E/F} \in \fp_{F}^{\lfloor \frac{m+3}{2}\rfloor}$, so 
$$ \Tr \alpha_0 \equiv  \Tr_{E/F} \alpha_{\xi,0}\pmod{\fp_{F}^{\lfloor \frac{m+3}{2}\rfloor}}.$$ Meanwhile, Section 44.4 of loc.\ cit.\ also gives us that $$\frac{\det \alpha_0 }{ \Nm \alpha_{\xi,0}} \in U_F(\lfloor \frac{m}{2} \rfloor +1).$$ Setting $f$ to be the minimal polynomial of $\alpha_{\xi,0}$ we obtain $$0=f(\alpha_{\xi,0}) \in  g(\alpha_{\xi,0}) + \alpha_{\xi,0} \fp_{F}^{\lfloor \frac{m+3}{2}\rfloor}  +  \Nm  \alpha_{\xi,0}\fp_F^{\lfloor \frac{m}{2} \rfloor + 1}.$$ Therefore \begin{equation}\label{vEg}v_E(g(\alpha_{\xi,0}))\geq m+3.\end{equation}

We want to apply Hensel's lemma, so we also need an upper bound on $v_E( g'(\alpha_{\xi,0}))$. We have $$g'(\alpha_{\xi,0}) \equiv  2 \alpha_{\xi,0}- \Tr_{E/F} \alpha_{\xi,0}\pmod{\fp_{F}^{\lfloor \frac{m+3}{2}\rfloor}}.$$ By Lemma \ref{normminlelt_lemma}, we have
$$ v_E(g'(\alpha_{\xi,0}) ) = \min( v_E(2\alpha_{\xi,0}), v_E(-\Tr_{E/F} \alpha_{\xi,0})).$$
Recall that $v_E( \alpha_{\xi,0})=1$, so that by \cite[41.2 Prop.\ (1)]{BushnellHenniart:06a}, we have that $v_F(\Tr_{E/F} \alpha_{\xi,0})= \lfloor \frac{d+1}{2}\rfloor$. Therefore
\begin{equation}\label{vEgprime}v_E(g'(\alpha_{\xi,0}) ) = \min(v_E(2) +1, 2  \lfloor \frac{d+1}{2}\rfloor).\end{equation}
Combining \eqref{vEg} and \eqref{vEgprime} along the hypothesis that $m$ is sufficiently large, we obtain that $v_E( g(\alpha_{\xi',0}))>2v_E(g'(\alpha_{\xi',0}) ) .$ 
 By Hensel's lemma, we get that there exists a unique $y_0 \in \cO_E$ such that $g(y_0)=0$ and $y_0 \equiv \alpha_{\xi,0} \mod \fp_E^{m+3-\min(v_E(2) +1, 2  \lfloor \frac{d+1}{2}\rfloor)}$.

 Letting $y= \varpi_F^{- \lfloor \frac{n+1}{2}\rfloor } y_0$, we get that  $E \simeq F[\alpha]$ by sending $y$ to $\alpha$. Identifying $\alpha$ with $y \in E$, we get that $\alpha \equiv \alpha_{\xi} \mod \fp_E^{-(n+1)+m+3-\min(v_E(2) +1, 2  \lfloor \frac{d+1}{2}\rfloor)}.$
\end{proof}
When $E\simeq F[\alpha]$, the embedding $F[\alpha]^\times\hookrightarrow G(F)$ is conjugate to a standard $E^\times \hookrightarrow G(F)$, so we may work with standard choices of cuspidal types, precisely, the following. 
\begin{mylemma}\label{choice_of_cusp_type}
Suppose that $\pi, (E/F, \xi)$ are as in the image set of Lemma \ref{lem:wilddiagram}. Let $\beta$ be a normalized minimal element for $E/F$ and fix the corresponding embedding $E^\times \hookrightarrow G(F)$ \eqref{Oembedding}. 
If $n$ is sufficiently large in the sense of Lemma \ref{alphathetaalphaxi}, then there exists a representative $(\mathfrak{A}, J, \widetilde{\theta})$ for the conjugacy class of cuspidal types corresponding to $\pi$ with $\mathfrak{A}$ the unique chain order such that $\beta \in  \mathcal{K}_{\mathfrak{A}} $, \begin{equation}\label{Jramdef}
J= E^\times U_{\mathfrak{A}}^{\frac{n+1}{2}}
\end{equation}
 and $\alpha \in E^\times \subset G(F)$.
\end{mylemma}

For a cuspidal type  $(\mathfrak{A}, J, \widetilde{\theta})$ as in Lemma \ref{choice_of_cusp_type}, let $\theta = \widetilde{\theta} \vert_{E^\times}$. For future reference, note that we can always recover $\widetilde{\theta}$ from $\theta$ by the extension \begin{equation}\label{theta_to_thetatilde}
\widetilde{\theta}(\ell (1+x)) = \theta(\ell) \psi'( \Tr(\alpha_{\theta} x)) \quad \ell \in E^\times, 1+x \in  U_{\mathfrak{A}}^{ \frac{n+1}{2}}.
\end{equation}

\begin{mycoro}\label{alphaxialphathetacor}
Suppose $(\mathfrak{A}, J, \widetilde{\theta})$ and  $(\mathfrak{A},n,\alpha)$ are as in Lemma \ref{choice_of_cusp_type}, with $n$ sufficiently large in the sense of Lemma \ref{alphathetaalphaxi}. Then, we have $v_E(\alpha_\xi) = v_E(\alpha)=v_E(\alpha_\theta) = -n$ and $\alpha_\xi \equiv \alpha_\theta \mod{ \fp_E^{-\frac{n-3}{2}-\min(v_E(2) +1, 2  \lfloor \frac{d+1}{2}\rfloor)}}$. 
\end{mycoro}
\begin{proof}
By \eqref{BH_s_12point4} we have $U_{E}(i) = U_{\mathfrak{A}}^i \cap E^\times$, so that 
\begin{equation}\label{thetatopsi}\theta \vert_{U_{E}(\frac{n+1}{2})} = \psi'_{\alpha}\end{equation} 
and $c(\theta) = n+c(\psi'_{E})$. Then, $$ \alpha_\theta \in \varpi_{E}^{-n} \left( \cO_E/\fp_{E}^{c(\theta)-\lceil c(\theta)/2 \rceil}\right)^\times$$ by Lemma \ref{postnikov} and $\alpha = \alpha_\theta \pmod {\fp_{E}^{c(\psi'_{E})- \lfloor n/2 \rfloor -1}}$ by \eqref{thetatopsi}. Now apply Lemma \ref{alphathetaalphaxi}.
\end{proof}

By Schur's lemma and Frobenius reciprocity for compact inductions, there is a 1-dimensional space of $\varphi \in V_\pi$ such that 
\begin{equation}\label{minlvecdef}
\pi(u) \varphi = \widetilde{\theta}(u) \varphi \quad \text{ for all } u \in J.
\end{equation}
If $\pi$ is as in Lemma \ref{lem:wilddiagram}, then a vector $\varphi$ as in \eqref{minlvecdef} is called a \emph{minimal vector} for $\pi$ (following \cite{HuNelsonSaha:17a, HN18}). Note: a conjugacy class of cuspidal types $(\mathfrak{A}, J, \widetilde{\theta})$ corresponding to such $\pi$ by Lemma \ref{lem:wilddiagram} necessarily has $\dim  \widetilde{\theta}=1$. Therefore, in the case at hand there is no distinction between minimal vectors of type 1 and type 2 cf.\ \cite[Def.\ 3.21]{HN18}.
\begin{mylemma}\label{mx_coeff_ofminlvec}
For $\varphi$ a minimal vector, we have
$$ \frac{\langle \pi(g)\varphi , \varphi\rangle}{\langle \varphi , \varphi\rangle} =\begin{cases} \widetilde{\theta}(g) & \text{ if } g \in J, \\ 0 & \text{ otherwise}. \end{cases}$$
\end{mylemma}
\begin{proof}
The lemma follows from  \cite[Lem.\ 3.1]{HuNelsonSaha:17a} along the lines of loc.\ cit.\ Proposition 3.2.
\end{proof}

\subsection{Supercuspidal families, $p\neq 2$}\label{sec:egSupercuspidal_odd}
We assume that $F=\Q_p$ and $p \neq 2$ until the end of this section (although some intermediate results hold more generally). In this case, if $D$ is either a non-square unit or has $v(D)=1$, then $\alpha_0=\sqrt{D}$ is a normalized minimal element for $F(\sqrt{D})/F$. If $(E/F,\xi)$ is an admissible pair with $\xi\vert_{F^\times } = \eta_{E/F}$, then $\ell_{\Delta_\xi \xi} \alpha_0 \in \cO^\times$, see the discussion just before Lemma \ref{normminlelt_lemma} and just after the proof of Lemma \ref{lem:tamediagram}. 

Given a trivial central character supercuspidal $\sigma$ of $\GL_2(F)$ corresponding by the Tame LLC to an admissible pair $(E/F,\xi)$ (up to $F$-equivalence), we define the test function
\begin{equation}\label{fandKprimedef}f_{\xi} := \frac{1}{\|\Phi \vert_{ZK'}\|_2^2}\overline{\Phi} \vert_{ZK'} \quad \text{ with } \quad K' = a(p^{-c_0})K a(p^{c_0}),\end{equation}
where $K$ is the standard maximal compact in $G$, the normalized conductor $c_0$ was defined in \eqref{c0def}, and $\Phi(g)= \langle \sigma(g)\varphi_0,\varphi_0\rangle$ where $\varphi_0$ is an $L^2$-normalized newform in $\sigma$.

We begin by reviewing the previous work of the first author \cite{Hu}.  Let $\phi_{\xi,0}$ be the function on $G(F)$ given by $\phi_{\xi,0}(g) = \langle \pi_{\Delta_\xi \xi}(g)\varphi,\varphi\rangle\vert_{ZB^1}$, where $\varphi$ is an $L^2$-normalized minimal vector in $ \pi_{\Delta_\xi \xi}$ and $Z B^1 \subset J \subset G$ is the subgroup described in Section \ref{compactinductionbackground}. Note that the function $\phi_{\xi,0}(g)$ is equal to the function $\widetilde{\Phi}_{0,0}$ found in  \cite[Def.\ 3.18]{Hu}.

In explicit terms, if $g \in G(F)$ can be written as $g=u \left(\begin{smallmatrix} 1+x & m \\ 0 & 1 \end{smallmatrix}\right)$ or $\left(\begin{smallmatrix} 1+x & m \\ 0 & 1 \end{smallmatrix}\right)u$ with $v(x) \geq \lceil c_0/2 \rceil$, $v(m)  \geq \lfloor c_0/2 \rfloor$ and $u \in Z U_E(1)$ embedded in $ZK$ by the map given in \eqref{Oembedding}, then we set
\begin{equation}\label{fsupercuspidaldef}
\phi_{\xi,0}(g) = \xi(u) \psi(p^{-c_0} \sqrt{D} \ell_{\Delta_\xi \xi} m).
\end{equation} 
If $g$ cannot be written in this way, then we set $\phi_{\xi,0}(g) =0$. See \cite[Cor.\ 3.19]{Hu}. 
In particular, note that $\phi_{\xi,0}$ has support contained in $ZK_0(p)$.

The function $\phi_{\xi,0}$ is a projector but not onto newforms. In order to recover a newform projector, following \cite[Def.\ 3.20]{Hu}, let
\begin{equation}\label{sctestfcn}
\phi_{\xi}(g) := \nu(p^{\lfloor c_0/2 \rfloor+1}) \sumstar_{\alpha,\alpha' \shortmod{p^{\lceil c_0/2\rceil}}} \overline{\phi_{\xi, 0} (a(p^{c_0}\alpha') g a(p^{c_0}\alpha)^{-1})}.
\end{equation}

\begin{mylemma}[Hu]\label{FEtheta_satisfies_geom_and_spec} Suppose $p\neq 2$. 
The function $\phi_{\xi}$ satisfies the spectral and geometric assumptions with support controlled by $y=p^{c_0+1}$.
\end{mylemma}
In particular, Theorem \ref{theoGeomSpec} applies with test function $\phi_{\xi}$ at a prime $p$, and with this choice Theorem \ref{theoGeomSpec} recovers the main theorems of \cite{Hu}. 
\begin{proof}
It was shown in \cite[Prop.\ 3.21]{Hu} that $\phi_{\xi}$ is a newform projector, so that $\phi_{\xi}$ satisfies the spectral assumption.  Since $\phi_{\xi,0}$ has support contained in $ZK_0(p)$, it follows that $\phi_{\xi}$ satisfies Geometric Assumption \eqref{geo3} with $y=p^{c_0+1}$. \end{proof}
The following is the main result of this section and extends Lemma \ref{FEtheta_satisfies_geom_and_spec}.
\begin{mytheo}\label{cor:SpecAssumption_supercuspidal} Suppose $p \neq 2$ and $F=\Q_p$. Let $(E/F,\xi) \in \P_2(F)$ have $\xi \vert_{F^\times } = \eta_{E/F}$. 
The test function $f_\xi$ is a newform projector in the sense of Definition \ref{def:newformproj}. 

Write $\sigma$ for the supercuspidal corresponding to $(E/F,\xi)$ by the Tame Parametrization Theorem (Theorem \ref{TameParametrizationTheorem}).
\begin{itemize}
\item If $c(\sigma)$ is even, then \begin{enumerate} \item $\supp \Phi \subseteq ZK'$, i.e. $$f_\xi 
=  \frac{1}{\|\Phi \|_2^2}\overline{\Phi},
$$
\item the operator $\pi(f_\xi)$, $\pi \in {\overline{G}}^{\wedge}$ is non-zero if and only if $\pi \simeq \sigma$, and 
\item the  function $f_\xi$ satisfies Geometric Assumption \eqref{geo3} with $y=p^{c_0}$.
\end{enumerate}
\item If $c(\sigma)$ is odd, then 
\begin{enumerate} 
\item $$f_\xi 
= \phi_{\xi},$$
\item the operator $\pi(f_\xi)$, $\pi \in {\overline{G}}^{\wedge}$ is non-zero if and only if $\pi \simeq \sigma$ or $\pi \simeq \sigma \times \eta$, where $\eta$ is the unramified quadratic character of $F^\times$, and
\item the  function satisfies Geometric Assumption \eqref{geo3} with $y=p^{c_0+1}$.
\end{enumerate}
\end{itemize}
\end{mytheo}
\begin{myrema} We have
\begin{equation}\label{sec:7.3fp(1)}
f_\xi(1) = \frac{1}{\|\Phi \vert_{ZK'}\|_2^2} = \begin{cases} (1-p^{-2}) p^{c_0+1} & \text{ if } c(\sigma) \equiv 1 \pmod 2, \\ (1-p^{-1})p^{c_0} & \text{ if } c(\sigma) \equiv 0 \pmod 2 . \end{cases}
\end{equation}
Indeed, the case that $c(\sigma)$ is odd follows from point (1) of Theorem \ref{cor:SpecAssumption_supercuspidal} and Proposition \ref{MxCoeffSpectralAssumption}, below.  In the case that $c(\sigma)$ is even, we have that $\|\Phi \vert_{ZK'}\|_2^2 = \|\Phi \|_2^2$ equals the inverse formal degree $d_\sigma^{-1}$ by point (1) of Theorem \ref{cor:SpecAssumption_supercuspidal}. In particular, it does not depend on the choice of unit vectors used to define $\Phi$, and can e.g.\ be computed as in formulas (A.15) and (A.16) of \cite{HN18} using minimal vectors (see Section 3.1 of loc.\ cit.\ for the relevant background).
\end{myrema}
The first author's test function $\phi_{\xi}$ is a newform projector regardless of the parity of $c(\sigma)$. However, if $c(\sigma)$ is even, then the projection operator-valued function $\pi(\phi_{\xi})$ is supported on the neighborhood $i(\xi[1]/\sim_0)$ around $\sigma$ (which has cardinality $\asymp p$)
, whereas  $\pi(f_\xi)$ is supported on the single point $\sigma$. In this sense, Theorem \ref{cor:SpecAssumption_supercuspidal} is a refinement of \cite{Hu}.

We need several preliminary results before proving Theorem \ref{cor:SpecAssumption_supercuspidal}. 
Let $\chi$ be a character of $\cO^\times$ and define the function $1_{\chi,n}$ on $F^\times$ by 
\begin{equation}\label{1chinx_def}1_{\chi,n} (x) = \begin{cases} \chi(u) & \text{ if } x = u p^{n} \text{ with } u \in \cO^\times, \\ 0 & \text{ otherwise.} \end{cases}\end{equation}
We will use an explicit description of the diagonal matrix coefficient $\Phi$ of the newform due to the first author. To state this, we need the following variant of the Iwasawa decomposition.
\begin{mylemma} \label{Lem:Iwasawadecomp}
For every positive integer $c$,
$$G(F)=\bigsqcup\limits_{0\leq i\leq c} B\zxz{1}{0}{p^i}{1}K_1(p^c).$$
\end{mylemma}
\begin{proof} See e.g.\ \cite[Lem.\ 2.1.1]{Schmidt:02a}. \end{proof}
Let $\pi$ now be either a supercuspidal representation of $G(F)$ of conductor exponent $c$ or a principal series representation $\pi(\mu_1,\mu_2)$ with $c(\mu_1)=c(\mu_2) = c/2$ for some $c \geq 2$. By the right $K_1(p^c)$-invariance {of the newform $\varphi_0$} 
of $\pi$ and Lemma \ref{Lem:Iwasawadecomp}, to give a complete description of the diagonal matrix coefficient $\Phi$ of the newform in $\pi$, it suffices to explicate the values of $$\phi_i(a,m):=\Phi(\big( \begin{smallmatrix} a & m \\  0 & 1 \end{smallmatrix}\big)\left( \begin{smallmatrix} 1 & 0 \\  p^i & 1 \end{smallmatrix}\right)).$$
\begin{mylemma}\label{Prop:MC}
Suppose $\pi$ is either a supercuspidal representation of $G(F)$ of conductor exponent $c$ or a principal series representation $\pi(\mu_1,\mu_2)$ with $c(\mu_1)=c(\mu_2) = c/2$ for some $c \geq 2$. 
\begin{enumerate}
\item[(i)] For $c-1\leq i\leq c$, $\phi_i(a,m)$ is supported on $v(a)=0$ and $v(m)\geq -1$. 
\item[(ii)] For $0\leq i<c-1$, $i\neq c/2$, $\phi_i(a,m)$ is supported on $v(a)=\min\{0,2i-c\}$ and $v(m)=i-c$. 
\item[(iii)]
\begin{enumerate} 
\item[(a)]When $c$ is even and $i=c/2>1$, $\phi_i(a,m)$ is supported on $v(a)\geq 0$ and $v(m)=-c/2$. 
\item[(b)] When $i=c/2=1$, $\phi_i(a,m)$ is supported on $v(a)\geq 0$, $v(m)\geq -1$. 
\end{enumerate}
\end{enumerate}
\end{mylemma}
\begin{proof}
This is a weak version of Proposition 3.1 of \cite{Hu:17a}. 
\end{proof}
\begin{myrema} When $p\neq 2$ and $\pi$ is a trivial central character supercuspidal, the proof of Lemma \ref{MxCoeffFurther} provides a full proof of (a refinement of) Lemma \ref{Prop:MC}.\end{myrema}

Recall that the unique unitary pairing on the Whittaker model of a smooth irreducible (pre-)unitary generic representation of $G(F)$ is given by \cite[Ch.\ 1 Thm.\ 12]{Godement_Notes_on_JL}
\begin{equation}\label{GinvtInnerProd}
\langle W_1,W_2\rangle =\int_{F^\times} W_1(a(y))\overline{W_2(a(y))}d^\times y.
\end{equation}
\begin{mylemma}\label{MxCoeffFurther}
If $\pi$ is a twist-minimal supercuspidal representation of trivial central character and conductor $c$, then in the case that $i=c/2$, the matrix coefficient $\phi_i(a,m)$ has support contained in $v(a)=0$. 
\end{mylemma}
\begin{myrema}\label{twminRem} An exercise with \cite[Prop.\ 3.4]{TunnellLLC} shows that any trivial central character supercuspidal representation with $p\neq 2$ is necessarily twist-minimal, see e.g.\ \cite[Lem.\ 2.1]{HuNelsonSaha:17a}. If $p=2$ then a trivial central character supercuspidal representation is twist-minimal if and only if $c(\pi) =2$ or $c(\pi)$ is odd, see e.g.\ Proposition \ref{prop:twistminimalQ2}. \end{myrema}

\begin{proof}
We work in the Whittaker model. Let $W$ be the newform in the Whittaker model of $\pi$. We have from \eqref{GinvtInnerProd} that 
\begin{equation}\label{mxCoeffInWhittakerModel}
 \Phi(\big( \begin{smallmatrix} a & m \\  0 & 1 \end{smallmatrix}\big)\left( \begin{smallmatrix} 1 & 0 \\  p^i & 1 \end{smallmatrix}\right)) = \int_{F^\times} W\left( a(y) \big( \begin{smallmatrix} a & m \\  0 & 1 \end{smallmatrix}\big)\big( \begin{smallmatrix} 1 & 0 \\  p^i & 1 \end{smallmatrix}\big) \right) \overline{W(a(y))} \, d^\times y.
\end{equation}
Thus, the support of the matrix coefficient is directly related to the support of $W(  \big( \begin{smallmatrix} a & 0 \\  0 & 1 \end{smallmatrix}\big)\big( \begin{smallmatrix} 1 & 0 \\  p^i & 1 \end{smallmatrix}\big) )$, which we can study directly in the Kirillov model. 
First of all, note that 
$$ W\left( a(y) \big( \begin{smallmatrix} a & m \\  0 & 1 \end{smallmatrix}\big)\big( \begin{smallmatrix} 1 & 0 \\  p^i & 1 \end{smallmatrix}\big) \right)  = \psi(my)  W\left( a(y) \big( \begin{smallmatrix} a & 0 \\  0 & 1 \end{smallmatrix}\big)\big( \begin{smallmatrix} 1 & 0 \\  p^i & 1 \end{smallmatrix}\big) \right) $$ 
by the defining property of the Whittaker model. 
Next, note that 
$$\big( \begin{smallmatrix} 1 & 0 \\  p^i & 1 \end{smallmatrix}\big) = \big( \begin{smallmatrix} & 1 \\  1 &  \end{smallmatrix}\big) \big( \begin{smallmatrix} 1 & p^i \\  0 & 1 \end{smallmatrix}\big) 
\big( \begin{smallmatrix} & 1 \\  1 &  \end{smallmatrix}\big) .$$ 
 Now we use the well-known explicit form of the newform in the Kirillov model (Lemma \ref{NewformKirillov}). We want to compute 
$$
\pi(\big( \begin{smallmatrix} & 1 \\  1 &  \end{smallmatrix}\big) 
 \big( \begin{smallmatrix} 1 & p^i \\  0 & 1 \end{smallmatrix}\big) 
 \big( \begin{smallmatrix} & 1 \\  1 &  \end{smallmatrix}\big) ) 1_{1, 0}.
$$
Let $w = \left( \begin{smallmatrix} 0 & -1 \\ 1 & 0 \end{smallmatrix}\right)$. 
By \cite[(9)]{YoshidaOnExtraordinaryReps} (see also \cite[Lem.\ 2.1]{SaitoOnTunnelsFormula}) we have for $\pi$ supercuspidal, $\chi$ a quasicharacter of $F^\times$ and $\psi$ an additive character of conductor 0 that
\begin{equation}\label{YoshidaFormula}
\pi(w)1_{\chi_0,n} = \eps(1/2, \pi \times \chi^{-1}, \psi) 1_{\omega_{\pi 0} \chi^{-1}_0, -c(\pi \times \chi^{-1})-n},
\end{equation}
where $\chi_0 = \chi \vert_{\cO^\times}$, $\omega_{\pi 0}$ is the central character of $\pi$ restricted to $\cO^\times$ (which is trivial here by hypothesis), and $c(\pi\times \chi^{-1})$ is the conductor exponent of $\pi\times \chi^{-1}$. Let us now denote $\eps(\pi) = \eps(1/2, \pi, \psi)$ for simplicity. We get that $\pi(\big( \begin{smallmatrix} & 1 \\  1 &  \end{smallmatrix}\big) )1_{1,0} = \eps(\pi)1_{1,-c}$ and $\pi(n(p^{i}))1_{1,-c} = \psi_{p^{-i}}1_{1,-c}$, where $\psi_{q}$ is the additive character defined by $\psi_q(x)= \psi(x/q)$. Now we convert this additive character to multiplicative characters to use \eqref{YoshidaFormula} again. Since the argument of $\psi_{p^{-i}}$ is restricted to valuation $-c$, for characters $\chi$ we are interested in 
$$\int_{\cO^\times} \psi_{p^{c-i}}(u) \chi(u)\,du = \sumstar_{u_0 \shortmod{p^{c-i}}}  \psi_{p^{c-i}}(u_0) \chi(u_0) \int_{p^{c-i}\cO} \chi(1+\Delta)\,d\Delta,$$
where we have set $u=u_0(1+\Delta)$ with $v(\Delta)\geq c-i$. The interior integral vanishes if $c(\chi)>c-i$ and otherwise equals $p^{i-c}$ (see e.g.\ \cite[(3.9)]{IK}). So, we get 
$$ \int_{\cO^\times} \psi_{p^{c-i}}(u) \chi(u)\,du = p^{-(c-i)} \delta_{c(\chi)\leq c-i} \sumstar_{u_0 \shortmod{p^{c-i}}}  \psi_{p^{c-i}}(u_0) \chi(u_0).$$ 
By e.g.\ \cite[Lem.\ 7.1]{PetrowYoungCoset}, we have that the above Gauss sum vanishes if $c(\chi)< c-i$ and $c-i>1$. If $c-i=1$ and $c(\chi)=0$, then the sum equals $-1$. In summary, 
$$\int_{\cO^\times} \psi_{p^{c-i}}(u) \chi(u)\,du = p^{-(c-i)} \mu(p^{c-i-c(\chi)}) \tau(\chi) \begin{cases}  \delta_{c(\chi)= c-i} & \text{ if } c-i \neq 1, \\  \delta_{c(\chi)\leq c-i} & \text{ if } c-i=1,\end{cases}$$
where {$\mu$ is the M\"obius function,} $\tau(\chi)$ is the Gauss sum of the primitive Dirichlet character corresponding to $\chi$. Therefore, if $c-i \neq 1$ we have
$$\pi(n(p^i))1_{1,-c} = \psi_{p^{-i}} 1_{1,-c} = p^{-(c-i)} \sum_{\chi: c(\chi) = c-i} \tau(\chi) 1_{\chi^{-1},-c},$$
and if $c-i=1$ we have
$$ \pi(n(p^i))1_{1,-c} = \psi_{p^{-i}} 1_{1,-c} = p^{-(c-i)} \sum_{\chi: c(\chi) \leq c-i} \mu(p^{c-i-c(\chi)})\tau(\chi) 1_{\chi^{-1},-c}.$$
Appealing to \eqref{YoshidaFormula} one more time, we finally have if $c-i \neq 1$ that
$$\pi\left( ( \begin{smallmatrix} 1 & 0 \\ p^{i} & 1 \end{smallmatrix})\right)1_{1,0} =  \eps(\pi) p^{-(c-i)} \sum_{\chi: c(\chi) = c-i} \tau(\chi) \eps(\pi \times \chi) \chi(-1)1_{\chi,c-c(\pi \times \chi)}$$
 and if $c-i=1$ that
$$= \eps(\pi) p^{-(c-i)} \sum_{\chi: c(\chi) \leq c-i} \mu(p^{c-i-c(\chi)})\tau(\chi) \eps(\pi \times \chi)  \chi(-1) 1_{\chi,c-c(\pi \times \chi)}.$$
Since $\pi$ is twist-minimal, we have by e.g.\ \cite[Lem.\ 6.2]{PetrowYoungCoset} that if $i\geq c/2$, then $c-c(\pi \times \chi)=0$ and if $i<c/2$, then $c-c(\pi \times \chi)=2i-c$.
Thus, if $c-i \neq 1$ we have 
$$ W\left( a(y) \big( \begin{smallmatrix} a & m \\  0 & 1 \end{smallmatrix}\big)\big( \begin{smallmatrix} 1 & 0 \\  p^i & 1 \end{smallmatrix}\big) \right)  = \psi(my)  \eps(\pi) p^{-(c-i)} \sum_{\chi: c(\chi) = c-i} \tau(\chi) \eps(\pi \times \chi)  \chi(-1) 1_{\chi,\min(0, 2i-c)}(ay),$$
and if $c-i = 1$, then 
\begin{multline*} W\left( a(y) \big( \begin{smallmatrix} a & m \\  0 & 1 \end{smallmatrix}\big)\big( \begin{smallmatrix} 1 & 0 \\  p^i & 1 \end{smallmatrix}\big) \right)  \\ = \psi(my)  \eps(\pi) p^{-(c-i)} \sum_{\chi: c(\chi) \leq c-i} \mu(p^{c-i-c(\chi)}) \tau(\chi) \eps(\pi \times \chi)  \chi(-1) 1_{\chi,\min(0, 2i-c)}(ay). \end{multline*}
Re-inserting these in \eqref{mxCoeffInWhittakerModel}, we only get new information on the support of the matrix coefficients in case (iii) of Lemma \ref{Prop:MC}, and in these cases the support of the matrix coefficient is further restricted to $v(a)=0$. 
\end{proof}
The Atkin-Lehner operator 
can be used to obtain further information when $i\leq c/2$.
\begin{mylemma}\label{Lem:refineSuppPhi}
Suppose that $\pi$ is as in Lemma \ref{Prop:MC} and moreover has trivial central character. If $i\leq c(\pi)/2$, 
then 
\begin{equation}\label{eq:refineSuppPhi}
\Phi(\big( \begin{smallmatrix} a & m \\  0 & 1 \end{smallmatrix}\big)\left( \begin{smallmatrix} 1 & 0 \\  p^i & 1 \end{smallmatrix}\right))
\end{equation} vanishes unless
\begin{itemize}
\item $v(a+mp^i)=0$ if $i>1$, or
\item $v(a+mp^i) \geq i-1$ if $i\leq 1$.
\end{itemize}
\end{mylemma}
\begin{proof}
Let us write $c=c(\pi)$. 
Since the Atkin-Lehner operator $\left(\begin{smallmatrix}  & 1 \\ -p^{c} \end{smallmatrix}\right)$ acts on the newform by a scalar, if the matrix coefficient \eqref{eq:refineSuppPhi} does not vanish, then $\Phi$ does not vanish on 
\begin{equation}
\label{eq:ALlem1}
\zxz{p^{-i}}{}{}{p^{-i}}\zxz{a}{m}{}{1}\zxz{1}{}{p^i}{1}
\zxz{}{1}{-p^c}{}=
\zxz{-ap^{c-2i}}{p^{-i}(a+mp^i)}{}{1}\zxz{1}{}{p^{c-i}}{1}\zxz{-1}{}{}{1}.
\end{equation}
Appealing to Lemma \ref{Prop:MC}, one concludes the proof of the Lemma.
\end{proof}
Note that Lemma \ref{Lem:refineSuppPhi} gives non-trivial information in cases (ii) and (iii)(a) of Lemma \ref{Prop:MC} since $v(m)= i-c(\pi)$ in those cases, but we do not obtain anything new if $i=c/2=1$.  

\begin{myprop}\label{MxCoeffSpectralAssumption}
Let $p, (E/F,\xi)$ and $\sigma$ be as in Theorem \ref{cor:SpecAssumption_supercuspidal}. 
\begin{enumerate}
\item If $c(\sigma)$ is even, then $\Phi|_{ZK'}=\Phi$.
\item If $c(\sigma)$ is odd, then $\overline{\Phi}|_{ZK'}=\frac{1}{(1-p^{-2})p^{c_0+1}} \phi_{\xi}$.
\end{enumerate}
\end{myprop}
\begin{proof}
First let us suppose that $c(\sigma)$ is even. From the discussion following Lemma \ref{Lem:Iwasawadecomp} it suffices to consider the matrix coefficients $\phi_i(a,m)$ for $0\leq i \leq 2c_0$. 

When $i\geq c_0$, we have from Lemma \ref{Prop:MC} that the matrix coefficient $\phi_i(a,m)$ is supported in $v(m)\geq-c_0$. Since $p \neq 2$ and $\sigma$ is a trivial central character supercuspidal, it follows  \cite[Prop.\ 3.4]{TunnellLLC} that $\sigma$ is twist-minimal (see Remark \ref{twminRem}). So, by Lemma \ref{MxCoeffFurther}, we have that the matrix coefficient is supported in $v(a)=0$. Then we have
$$g=\zxz{a}{m}{}{1}\zxz{1}{}{p^i}{1}=\zxz{a+mp^i}{m}{p^i}{1}\in K'$$ whenever $g$ is in the support of $\Phi$, 
as its determinant is $a\in \cO^\times$ and its four entries satisfy $v(m)\geq -c_0$, $v(p^i)=i\geq c_0$, $v(a+mp^i)\geq 0$, $v(1)=0$.

When $i<c_0$, we get from Lemma \ref{Prop:MC} that $v(\det(g))=v(a)=2i-c(\sigma)=2i-2c_0$.
So it suffices to check that 
$p^{c_0-i}g \in K'$.
Indeed from Lemma \ref{Lem:refineSuppPhi}, we get $v(p^{c_0-i}(a+mp^i))\geq 0$. From Lemma \ref{Prop:MC} part (ii) we get $v(p^{c_0-i}m)=c_0-i+i-c(\sigma)=-c_0$, 
$v(p^{c_0-i}p^i)=c_0$,
$v(p^{c_0-i})\geq 0$. 

Next suppose that $c(\sigma)$ is odd, so necessarily $E/F$ is ramified and the character $\theta$ of $E^\times$ for which $\sigma\simeq \pi_\theta$ by compact induction (see Section \ref{compactinductionbackground}) has $c(\theta)$ even. In this case the proof of the proposition is an extension of \cite[Lem.\ 5.2]{Hu}. Indeed, begin by recalling the definition of the matrix coefficient $\Phi_{0,0}$ from loc.\ cit.\ Definition 3.18. Then, following the notation and proof  of loc.\ cit.\ Lemma 5.2, it suffices to show that $$\Phi_{0,0} \vert_{ZB^1} = \Phi_{0,0}\vert_{ZK}$$ for all $g \in G$, not just the elements $g_{a,a'}$. 

To see this assertion, recall from \cite[Cor.\ 3.13]{Hu} that $\Phi_{0,0}$ is supported in the subgroup $J = E^\times K_{\mathfrak{A}_2}(c(\theta)/2)$. Now, by the structure of the unit group of a $p$-adic field (and using that $E/F$ is ramified), the group $$ E^\times / F^\times U_E(1)$$ has cardinality 2, its two cosets being represented by $1$ and $\varpi_E$.  Therefore, we have for $B^1 = U_E(1) K_{\mathfrak{A}_2}(c(\theta)/2)$ that 
$$J = ZB^1 \sqcup \left(\begin{smallmatrix} 0 & 1 \\ p & 0 \end{smallmatrix}\right)ZB^1.$$ Finally, note that $ZB^1 \subseteq ZK$ but the other coset is disjoint from $ZK$. This proves the 2nd assertion of the proposition.
\end{proof}
\begin{proof}[Proof of Theorem \ref{cor:SpecAssumption_supercuspidal}]
Given Proposition \ref{MxCoeffSpectralAssumption}, essentially all that remains to prove the Theorem is to compute $\|\Phi|_{ZK'}\|_2^2$.

First suppose that $c(\sigma)$ is even. By the definition \eqref{fandKprimedef} of $f_\xi$ and Proposition \ref{MxCoeffSpectralAssumption} we have that
\begin{equation}\label{fxi_csigmaeven_formula_in_proof} f_\xi=\frac{1}{\|\Phi \vert_{ZK'}\|_2^2}\overline{\Phi} \vert_{ZK'} = \frac{1}{\|\Phi\|_2^{2}}\overline{\Phi}.\end{equation}
The fact that $f_\xi$ is a newform projector as well as points (1), (2), and (3) of the Theorem are clear from equation \eqref{fxi_csigmaeven_formula_in_proof} and orthogonality of matrix coefficients. 

Suppose $c(\sigma)$ is odd. Again, the definition \eqref{fandKprimedef} of $f_\xi$ and Proposition \ref{MxCoeffSpectralAssumption} give us that
\begin{equation}\label{fxi_csigmaodd_formula_in_proof} f_\xi =\frac{1}{\|\Phi \vert_{ZK'}\|_2^2}\overline{\Phi} \vert_{ZK'} = \frac{(1-p^{-2})p^{c_0+1}}{\| \phi_\xi\|_2^2} \phi_\xi.\end{equation}
By Lemmas \ref{FEtheta_satisfies_geom_and_spec}, \ref{projection_upper_bound}, and formula  \eqref{sctestfcn}, we have 
\begin{equation}\label{F_Etheta1}
\| \phi_\xi\|_2^2 = \phi_{\xi}(1)= \nu(p^{\lfloor c_0/2 \rfloor+1}) \sumstar_{a,a' \shortmod{p^{\lceil c_0/2\rceil }}} 1_{a\equiv a' \shortmod{p^{\lceil c_0/2\rceil }}} = (1-p^{-1})\nu(p^{c_0+1}).
\end{equation}
Putting together \eqref{fxi_csigmaodd_formula_in_proof} and \eqref{F_Etheta1}, point (1) of the Theorem holds, namely $f_\xi = \phi_\xi$. The fact that $f_\xi$ is a newform projector and point (3) of the Theorem follow from Lemma \ref{FEtheta_satisfies_geom_and_spec}.
 Point (2) of the Theorem is \cite[Prop.\ 3.21]{Hu}.
\end{proof}

\subsection{Supercuspidal families, $p= 2$}\label{sec:egSupercuspidal_even}
Throughout this section we set $F=\Q_2$. 

Let $\sigma$ be a trivial central character supercuspidal representation of $G(F)$ with $c(\sigma)\geq 9$ and $(E/F,\xi) \in \P_2(F)_{\geq 9}^1$ be the corresponding pair by the bijection $i$ of Corollary \ref{cor_p2bijection}. Throughout this section, we write $e=e_{E/F}$ and  $d=v(\disc E/F)$. Note that $d$  can only take the values $0,2,3$ because of our assumption on $F$. 
Set $c_0= c(\xi)/e_{}$ and recall the neighborhood $\xi[n]$ of $\xi$ from \eqref{xi_n_def}.

\begin{mytheo}\label{cor:SpecAssumption_supercuspidal_p2}
Suppose $\sigma, E, \xi$ are as above, and let $\Phi$ be the diagonal matrix coefficient of an $L^2$-normalized newform in $\sigma$. 
\begin{itemize}
\item If $d =0$,  then $$f = \frac{1}{\| \Phi \vert_{ZK_0(c_0,-c_0)}\|_2^2} \overline{\Phi} \vert_{ZK_0(c_0,-c_0)}$$ is a newform projector. The operator $\pi(f)$ is non-zero if and only if $\pi$ is isomorphic to one of the three representations $i(\xi[1])$. 
\item If $d =2$, then $$f = \frac{1}{\| \Phi \vert_{ZK_0(c_0+1,-c_0-1)}\|_2^2} \overline{\Phi} \vert_{ZK_0(c_0+1,-c_0-1)}$$ is a newform projector. The operator $\pi(f)$ is non-zero if and only if $\pi \simeq \sigma$ or $\sigma \times \eta$ where $\eta$ is the unramified quadratic character of $F^\times$.  
\item If $d =3$ and $c(\sigma)\geq 11$, then $$f = \frac{1}{\| \Phi \vert_{ZK_0(c_0+2,-c_0-1)}\|_2^2} \overline{\Phi} \vert_{ZK_0(c_0+2,-c_0-1)}$$ is a newform projector. The operator $\pi(f)$ is non-zero if and only if $\pi \simeq \sigma$ or $\sigma \times \eta$ where $\eta$ is the unramified quadratic character of $F^\times$.  
\end{itemize}
\end{mytheo}
\begin{myrema} We have 
\begin{equation}\label{f1p=2_formula}
f(1) 
= \begin{cases}  (1-p^{-2}) p^{c_0+1} & \text{ if } d=0 \text{ or } 2, \\   (1-p^{-2}) p^{c_0+2} & \text{ if } d=3,\end{cases}
\end{equation}
see Proposition \ref{Cor:twotestfunsame} and Lemma \ref{lem:Vexplicit}.
\end{myrema}
Recall the normalized minimal elements $\alpha_0$ from Definition \ref{normminlelt}. 
\begin{mylemma}\label{minl_elts_in_Q2extns}
Any quadratic extension of $F$ is one of the following types and has a normalized minimal element $\alpha_0$ with minimal polynomial $g(x)=x^2+Ax+B$ of the following form.
\begin{enumerate}
\item The unique unramified quadratic extension with $d=0$ and $g(x) = x^2+x+1$. 
\item A ramified quadratic extension with $d=2$ and $v(A)=v(B)=1$.
\item A ramified quadratic extension with $d=3$ and $A=0$ and $v(B)=1$. 
\end{enumerate}
\end{mylemma}
\begin{proof}
It suffices to consider the case that $E/F$ is ramified. In this case, any uniformizer $\varpi_E$ for $E$ is a normalized minimal element. 
The ramified $d=2$ case follows from \cite[41.1 Lem.\ (1)(2)]{BushnellHenniart:06a}, where we caution that in Bushnell-Henniart the symbol $d$ has a different meaning than in this paper. When $d=3$, we again use loc.\ cit., and then complete the square to find a uniformizer for $E$ of the prescribed shape.  
\end{proof} 
Recall the notion of a chain order $\mathfrak{A} \subset M_2(F)$, its normalizer $\mathcal{K}_{\mathfrak{A}}$, and the standard chain orders $\mathfrak{A}_e$, $e=1,2$ from Section \ref{compactinductionbackground}. 
\begin{mylemma}\label{Q2chainorders}
Suppose $\alpha_0 \in E$ is as in Lemma \ref{minl_elts_in_Q2extns}. Using $\alpha_0$ to embed $E^\times \hookrightarrow G(F)$ by \eqref{Oembedding}, the group $E^\times$ normalizes the chain order $\mathfrak{A}_e$. 
The standard order $\mathfrak{A}_e$ is the unique chain order in $A=M_2(F)$ such that $E^\times \subseteq \mathcal{K}_{\mathfrak{A}_e}.$
\end{mylemma}
\begin{proof}
Since $F^\times$ embeds as the center in $G(F)$ under \eqref{Oembedding}, it suffices to check that $\alpha_0 \mathfrak{A}_e \alpha_0^{-1}= \mathfrak{A}_e $. Let $\mathfrak{P}=\mathfrak{P}_e = {\rm rad}( \mathfrak{A}_e)$ be the Jacobson radical of $\mathfrak{A}_e$, explicitly, 
$$\mathfrak{P}_1= \zxz{\fp}{\fp}{\fp}{\fp}, \quad \mathfrak{P}_2=\zxz{\fp}{ \cO}{\fp}{\fp}.$$ Then, it is simple to check from the information in Lemma \ref{minl_elts_in_Q2extns} that $\alpha_0 \in \mathfrak{P}^{e-1}$ and $\alpha_0^{-1} \in \mathfrak{P}^{1-e}$. Since $\mathfrak{P}^i \mathfrak{P}^j = \mathfrak{P}^{i+j}$ for any $i,j \in \Z$ (see \cite[\S 12.2]{BushnellHenniart:06a}), the first assertion of the lemma follows. The second assertion follows from the first by \cite[12.4 Prop.\ (2)]{BushnellHenniart:06a}.
\end{proof}
For $k \geq \ell\geq 0$ set (cf.\ \eqref{Jramdef}) \begin{equation}\label{H_group_def}
H = ZU_E(\ell)U_{\mathfrak{A}_e}^k.
\end{equation}
\begin{mylemma}\label{Hvol}
For $E$ and $\mathfrak{A}_e$ as in Lemma \ref{Q2chainorders}, and $k \geq \ell \geq 1$, we have
$$\vol(Z\backslash H)=\begin{cases}
	\frac{1}{(p^2-1)p^{2k+\ell-2}} &\text{ if }e=1,\\
	\frac{1}{(p^2-1)p^{k+\lfloor \frac{\ell}{2}\rfloor -1}} &\text{ if }e=2.
	\end{cases}$$
\end{mylemma}
\begin{proof}
The volume is the same as $\vol(\cO_F^\times U_E(\ell)U_{\mathfrak{A}_e}^k)$ by quotient measure.
		By \eqref{BH_s_12point4} and the group isomorphism theorems, we have
		$$
		[\cO_F^\times U_E(\ell)U_{\mathfrak{A}_e}^k:U_{\mathfrak{A}_e}^k]=[\cO_F^\times U_E(\ell):U_E(k)]=(1-p^{-1}) p^{\lceil \frac{\ell}{e} \rceil +\frac{2}{e}(k-\ell)}.
		$$
		On the other hand we have
		$$[G(\cO_F):U_{\mathfrak{A}_e}^k]=\begin{cases}
			(p^2-1)(p^2-p)p^{4(k-1)} &\text{ if }e=1,\\
		(p-1)(p^2-1)p^{2(k-1)} &\text{ if }e=2.
		\end{cases}$$
		Since we take $\vol(G(\cO_F))=1$, the lemma follows by combining these two computations.
\end{proof}

For a multiplicative character $\xi$ of a $p$-adic field $E$, recall the linearization $\alpha_\xi \in E$ from Lemma \ref{postnikov}, which depends on a choice of an additive character $\psi_E$ of $E$ which we have taken to be $\psi_E = \psi \circ \Tr_{E/F}$ for a choice of additive character $\psi$ of $F$. 
\begin{mylemma}\label{Q2lem2}
Suppose $\xi$ is a character of $E^\times$ such that $\xi \vert_{F^\times} = \eta_{E/F}$ and $c(\xi)\geq 2$. Let $\alpha_0$ be one of the normalized minimal elements for $E$ from  Lemma \ref{minl_elts_in_Q2extns} with minimal polynomial $x^2+Ax+B$. There exists $z\in F^\times$ for which 
\begin{equation}\label{yalphaxidef}
\alpha_\xi = z (\frac{A}{2} + \alpha_0).
\end{equation}
Such $z$ satisfies $v(z) = -\frac{c(\xi)}{e} +c(\psi)+1-d$.
\end{mylemma} 
\begin{proof}
First note that for any of the three cases in Lemma \ref{minl_elts_in_Q2extns} we have $A/2 + \overline{\alpha_0} = -(A/2 + \alpha_0)$, 
	using that $F=\Q_2$ has characteristic 0. It follows from $\xi \vert_{F^\times} = \eta_{E/F}$ that $\overline{\alpha_\xi} \equiv - \alpha_\xi \pmod{\fp_E^{c(\psi_E)-\lceil c(\xi)/2 \rceil}}$ cf.\ the discussion in the second paragraph of Section \ref{compactinductionbackground}. So, $\alpha_\xi \equiv z (A/2 + \alpha_0)\pmod{\fp_E^{c(\psi_E)-\lceil c(\xi)/2 \rceil}}$ for some $z \in F^\times$. 

It suffices to calculate the valuation of $z$. We have 
$$ v_E(\alpha_\xi)  = v_E(z) + \min ( v_E(A/2) , v_E(\alpha_0)) = ev(z) + \min ( ev(A)-ev(2), e-1) $$
by Lemma \ref{normminlelt_lemma} and Definition \ref{normminlelt}. At the same time, by Lemma \ref{postnikov} we have $$v_E(\alpha_\xi) = c(\psi_E)-c(\xi).$$ In the three cases of Lemma \ref{minl_elts_in_Q2extns}, we have $v(A) =0,1, \infty$, respectively, so that $$ ev(z) + \min ( ev(A)-ev(2), e-1)= \begin{cases}  v(z) + \min(-v(2),0) = v(z)-1 &  \text{ if } d=0, \\ 2v(z) + \min(2-2v(2),1) = 2v(z) & \text{ if } d=2, \\ 2v(z) +\min(\infty, 1) = 2v(z) +1 &  \text{ if } d=3. \end{cases}$$ Combining these formulas with \eqref{add_char_cond}, we obtain the formula in the Lemma.
\end{proof}
Now fix an additive character $\psi$ of $F$ of conductor 0. 
\begin{myrema}\label{remarkafterQ2lem}
For later purposes we introduce a new parameter $j$, which in the totally ramified case matches the ``thickness'' of the group $J$ in \eqref{Jramdef}, but is merely ad hoc in the unramified case.  Table \ref{table:xi} gives a dictionary between $j$ and the other parameters associated to dihedral supercuspidal $\sigma$ corresponding to $\Ind_{E}^F \xi$ under the LLC, where $z\in F^\times$ is as in \eqref{yalphaxidef} assuming $c(\psi)=0$.

\begin{table}[h!]
\centering
\begin{tabular}{ |c|c|c|c|c| } 
 \hline
 $d$ & $c(\xi)$ & $c_0$ & $v(z)$ & $c(\sigma)$ \\ 
 \hline\hline
$0$ & $j+1$ & $j+1$ & $-j$ & $2j+2$\\ 
\hline
 $2$ & $2j$ & $j$ & $-j-1$ & $2j+2$\\ 
 \hline
 $3$ & $2j-2$ & $j-1$ & $-j-1$ & $2j+1$\\ 
 \hline
\end{tabular}
\caption{}
\label{table:xi}
\end{table}
\end{myrema}

\begin{mydefi}\label{def:twistminimalQ2} For $E/F, \alpha_0$ and $\xi$ as above, write $\chi_m$ for any choice of character of $F^\times$ satisfying \begin{enumerate}  \item $\chi_m(2)=1$, and \item when $d=0$ or $2$,  $\chi_m(1+x) = \psi( z\frac{A}{2} x)$ for all $x$ with $v(x)\geq \lceil \frac{-v(zA/2)}{2}\rceil$. \end{enumerate}
When $d=3$ set $\chi_m=1$. 
\end{mydefi}
 
\begin{myprop}\label{prop:twistminimalQ2}
Suppose that $\sigma$ is a trivial central character supercuspidal representation that corresponds by the LLC to $\Ind_E^F \xi$ with $c(\xi)\geq 2$. Suppose $\alpha_0$ is a normalized minimal element for $E/F$ with minimal polynomial as in Lemma \ref{minl_elts_in_Q2extns}. Then, $\sigma \times \chi_m^{-1}$ is twist-minimal and moreover we have
$$ \min_{\chi} c(\sigma \times \chi) = \begin{cases} c(\sigma)-2=2j & \text{ if } d=0, \\ c(\sigma)-1=2j+1 & \text{ if } d=2,\\
c(\sigma)= 2j+1 & \text{ if } d=3.\end{cases}$$
\end{myprop}
\begin{proof}
Let $\rho = \Ind_E^F \xi$ be the Galois representation corresponding to $\sigma$ by the LLC. 
We have $\rho \otimes \chi = \Ind_{E}^{F} \xi \chi_E$, where $\chi_E= \chi \circ \Nm$. By \eqref{c(pi)c0}, the formula $c(\xi)= -v_E(\alpha_\xi)+c(\psi_E)$ of Lemma \ref{postnikov}, the Artin conductor $c(\rho \otimes \chi)$ is minimized when the valuation of  $$\alpha_{\xi \chi_E} = \alpha_\xi + \alpha_\chi = z\frac{A}{2} +z\alpha_0 + \alpha_\chi$$ is maximized. Since $\alpha_\chi \in \Q_2$, we have by Lemma \ref{normminlelt_lemma}
\begin{equation}\label{propdef_formula} v_E(\alpha_{\xi \chi_E}) = \min( v_E(z\frac{A}{2} + \alpha_\chi),v_E(z\alpha_0)),\end{equation} which can be maximized by taking $\alpha_\chi = -z\frac{A}{2}$, matching Definition \ref{def:twistminimalQ2} of $\chi_m^{-1}$. This proves the first assertion of the proposition. 

Computing the conductor of $c(\sigma \times \chi_m^{-1})$ in cases using  \eqref{c(pi)c0}, $c(\xi)= -v_E(\alpha_\xi)+c(\psi_E)$, \eqref{propdef_formula}, Lemma \ref{Q2lem2}, and  \eqref{add_char_cond}, we conclude the formula for $\min_{\chi} c(\sigma \times \chi)$ in the second assertion. 
\end{proof}

Given $\sigma, E/F, \alpha_0$ as in Proposition \ref{prop:twistminimalQ2}, let $\sigma'=\sigma \times \chi_m^{-1}$ where $\chi_m$ is as in Definition \ref{def:twistminimalQ2}. In particular, the representation $\sigma'$ is twist-minimal.  Let $\rho' = \Ind_E^F \xi'$ be the corresponding Weil group representation under the LLC.  We have $\alpha_{\xi'} = z \alpha_0$ for $z$ as in Lemma \ref{Q2lem2} and Remark \ref{remarkafterQ2lem}.   
If $E/F$ is ramified let $\varphi' \in \sigma'$ be a minimal vector defined by \eqref{minlvecdef}, and if $E/F$ is unramified let $\varphi' \in \sigma'$ be defined by \cite[Def.\ 3.12]{Hu}.

If $E/F$ is ramified and $c(\sigma')$ is sufficiently large in the sense of Lemma \ref{alphathetaalphaxi}, then we explained in Section \ref{compactinductionbackground} the construction of a character $\theta'$ of $E^\times$ that leads to $\sigma'$ by compact induction. Precisely, if \begin{equation}\label{csigmaprime_suff_large}c(\sigma')\geq \begin{cases} 
7 &  \text{ if } d=2,  \\ 
11 &  \text{ if } d =3,
\end{cases}\end{equation} then Corollary \ref{alphaxialphathetacor} applies, so that in particular $c(\theta') = c(\xi')$.

If $E/F$ is unramified, then let $\theta'$ be the character of $E^\times$ corresponding to $\sigma'$ across the compact induction bijection \eqref{TameCompactIndMap}. Recall that $\theta'= \xi' \Delta_{\xi'}$ with $\Delta_{\xi'}$ unramified in this case \cite[\S34.4]{BushnellHenniart:06a}. 

Table \ref{table:xiprime} gives the conductors of the twist-minimal $\xi'$ and $\sigma'$ in terms of the parameter $j$ introduced in Remark \ref{remarkafterQ2lem}.
\begin{table}[h!]
\centering
\begin{tabular}{ |c|c|c|c| } 
 \hline
 $d$ & $c(\xi')$ & $c(\sigma')$ \\ 
 \hline\hline
$0$ & $j$ & $2j$\\ 
\hline
 $2$ & $2j-1$  & $2j+1$\\ 
 \hline
 $3$ & $2j-2$  & $2j+1$\\ 
 \hline
\end{tabular}
\caption{}
\label{table:xiprime}
\end{table}

We take $\sigma'$ to be given by the compact-induced model $\sigma' = \cInd_J^G \Lambda'$, where, following Section \ref{compactinductionbackground}, the representation $\Lambda'$ is constructed from the above $\theta'$ on $E^\times$ embedded in $G(F)$ by \eqref{Oembedding} with respect to $\alpha_0$. Set $\theta= \theta' .\chi_{m,E}$ and $\Lambda = \Lambda' . \chi_m \circ \det$. Let $\varphi(g) = \chi_m(\det g)\varphi'(g)$, so that $\varphi \in \cInd_J^G \Lambda\simeq \sigma$. We continue to call such a $\varphi$ a minimal vector despite the fact  that $\sigma$ is not necessarily twist-minimal.

Let $\langle,\rangle$ be a unitary pairing on the space of $\cInd_J^G \Lambda \simeq \sigma$, and let $\Phi_\varphi$ be the diagonal matrix coefficient $\Phi_\varphi(g) = \langle \sigma(g)\varphi, \varphi\rangle$ of $\varphi \in \cInd_J^G \Lambda$. For $\ell \geq {1}$ 
 and $j$ as in Remark \ref{remarkafterQ2lem}, the group $H$ as in \eqref{H_group_def}, and the group $ZB^1$ as in Section \ref{compactinductionbackground} (defined in \cite[Def.\ 3.11]{Hu}), set 
\begin{equation}\label{tildePhiphidef}
\widetilde{\Phi}_\varphi =\begin{cases} \Phi_\varphi \vert_{H} & \text{ if } E/F \text{ is ramified,} \\ \Phi_\varphi \vert_{ZB^1} & \text{ if } E/F \text{ is unramified}\end{cases} \text{ and } V=\begin{cases} \vol(Z\backslash H)& \text{ if } E/F \text{ is ramified, } \\ \vol(Z \backslash ZB^1) & \text{ if } E/F \text{ is unramified}.\end{cases}
\end{equation}
\begin{mylemma}\label{lem:Vexplicit}We have that 
\begin{equation}\label{Vexplicit}
V= \begin{cases}  \frac{1}{(1-p^{-2}) p^{c_0+1}} & \text{ if } d=0 \text{ or } 2, \\   \frac{1}{(1-p^{-2}) p^{c_0+2}} & \text{ if } d=3.\end{cases} 
\end{equation}
\end{mylemma}
\begin{proof} First consider the case that $d=0$. If $c_0$ is even we have that $ZB^1 =  ZU_E(1)U_{\mathfrak{A}_1}^{c_0/2}$, and if $c_0$ is odd we have that $ZB^1$ is a proper intermediate subgroup between $ZH^1=ZU_E(1)U_{\mathfrak{A}_1}^{\lceil c_0/2\rceil}$ and $ZJ^1=ZU_E(1)U_{\mathfrak{A}_1}^{\lfloor c_0/2\rfloor}$. Lemma \ref{Hvol} computes the volumes of these groups, giving \eqref{Vexplicit} in both $d=0$ cases. Next consider the case that $E/F$ is ramified, in which we have $j=c_0$ if $d=2$ and $j=c_0+1$ when $d=3$ by Remark \ref{remarkafterQ2lem}.  The formula for $V$ then follows directly from Lemma \ref{Hvol}. \end{proof}
Recall the neighborhood $\theta[n]$ of characters around $\theta$ from \eqref{xi_n_def}, and the slightly generalized notion of minimal vectors introduced between the proof of Proposition \ref{prop:twistminimalQ2} and the statement of Lemma \ref{lem:Vexplicit}. 
\begin{myprop}\label{minlvec_projection_prop}
Suppose $\sigma, E/F, \alpha_0, \ell, j$ are as above and satisfy $2-e\leq \ell \leq \min(j, c(\theta')-1)$. The operator $\pi (\overline{\widetilde{\Phi}}_\varphi)$ vanishes unless there exists $\theta_1 \in \theta[\ell] 
$
such that $\pi \simeq \cInd_J^G \Lambda_1$, where $\theta_1$ and $\Lambda_1$ are related as in Section \ref{compactinductionbackground}.  The operator $V^{-1}\pi (\overline{\widetilde{\Phi}}_\varphi)$ is a projection onto the line of the minimal vector in $\pi$ whenever $\pi (\overline{\widetilde{\Phi}}_\varphi)$ is non-vanishing. 
\end{myprop} 
\begin{proof}
First assume that $E/F$ is ramified. 
We will show that if $\pi(\overline{\widetilde{\Phi}}_\varphi)$ is non-trivial, then $\pi \simeq \cInd_J^G \widetilde{\theta}_1$ for some $\theta_1\in \theta[\ell]$.   

Taking $\sigma = \cInd_J^G \tilde{\theta}$, we have $\sigma(u)\varphi = \tilde{\theta}(u) \varphi$ for $u \in J$. Let $v \in V_\pi$ be any vector such that $\pi(\overline{\tilde{\Phi}}_\varphi)v \neq 0$. Then, define the non-zero linear map $\tilde{\theta}'\vert_{H} \to \pi \times \chi_m^{-1} \vert_{H}$ by $z \mapsto z \pi(\overline{\tilde{\Phi}}_\varphi)v.$ We check that it is $H$-equivariant: Any $g\in H$ acts by
\begin{align}\label{nonzerooperator_implies_sigma_acts_by_char}
(\pi \times \chi_m^{-1})(g) \left[z \pi(\overline{\tilde{\Phi}}_\varphi)v\right] & = z \int_H \overline{\langle \sigma(h) \varphi,\varphi\rangle} (\pi \times \chi_m^{-1})(g) \pi(h)v\,dh \nonumber \\
& = \chi_m^{-1}(\det g) z \int_H \overline{\langle \sigma(h) \varphi,\varphi\rangle} \pi(gh)v\,dh \nonumber \\
& = \chi_m^{-1}(\det g) z \int_H \overline{\langle \sigma(h) \varphi,\sigma(g) \varphi\rangle} \pi(h)v\,dh \nonumber \\
& = \tilde{\theta}'(g) z \pi(\overline{\tilde{\Phi}}_\varphi)v. 
\end{align}
Therefore, we have
\begin{equation}\label{sigmachiHomequation}
  0 \neq \Hom_H(\tilde{\theta}'|_{H},\pi\times\chi_m^{-1}|_H) = \Hom_G( \cInd_H^G(\tilde{\theta}'|_{H}) ,\pi\times \chi_m^{-1}).
\end{equation}

Next, we claim that
\begin{equation}\label{IndDirectSum}\Ind_H^J(\tilde{\theta}'|_{H})=\bigoplus\limits_{\theta_1'\in \theta'[l]}\tilde{\theta_1'}.\end{equation}
By Postnikov (Lemma \ref{postnikov}), $\alpha_{\theta'} - \alpha_{\theta_1'} \equiv 0 \pmod {\fp_E^{c(\psi_E) - \ell}}$. Then, for any $x \in U_{\mathfrak{A}_e}^{j}$, we have that $\psi ( \Tr((\alpha_{\theta'} - \alpha_{\theta_1'} ) x)) = 1$ by the hypothesis that $j\geq \ell$. Since $\theta_1 \in \theta[\ell]$, we have $c(\theta_1 \theta^{-1}) \leq \ell$, so that $\tilde{\theta}_1' = \tilde{\theta}'$ on $H$. Then, $$\C = \Hom_H(\tilde{\theta}' \vert_{H}, \tilde{\theta}'_1\vert_{H}) = \Hom_J(\Ind_H^J \tilde{\theta}' \vert_{H} ,  \tilde{\theta}'_1).$$ Therefore, each $\tilde{\theta}'_1$ is a sub-representation of $\Ind_H^J \tilde{\theta}' \vert_{H}$, and occurs only once as a sub-representation of it. Moreover, $\theta'[\ell]$ is simply a translate of $\widehat{J/H}$, so the dimension of both sides of \eqref{IndDirectSum} are equal, and therefore the sum of these 1-dimensional sub-representations exhausts $\Ind_H^J \tilde{\theta}'$. 

Since $\cInd$ is additive and transitive, it follows from \eqref{IndDirectSum} that
\begin{equation}\label{IndDirectSum2}\cInd_H^G(\tilde{\theta}'|_{H})=\bigoplus\limits_{\theta_1'\in \theta'[l]}\cInd_J^G\tilde{\theta_1'}.\end{equation}

Next we claim that if $\theta'$ is minimal and $c(\theta')\geq \ell+1$, $\ell\geq 0$, then all $\theta_1'\in \theta'[\ell]$ are minimal. Indeed, since $c(\theta'_1\theta'^{-1})\leq \ell$, we have that $\theta'_1 \vert_{U(\ell)} = \theta' \vert_{U(\ell)}$. Since $c(\theta') \geq \ell+1$, we have $\theta'_1 \vert_{U(c(\theta')-1)} = \theta' \vert_{U(c(\theta')-1)} $, so that $c(\theta')=c(\theta'_1)$ and also $\chi_E\theta'_1 \vert_{U(c(\theta')-1)} = \chi_E\theta' \vert_{U(c(\theta')-1)} $ for any $\chi$.  
Since $\theta'$ is minimal, $\chi_E \theta'$ is non-trivial on $U(c(\theta')-1)$. Thus, $\chi_E \theta'_1$ is non-trivial on $U(c(\theta')-1)= U(c(\theta'_1)-1)$. That is to say, $c(\chi_E \theta'_1)\geq c(\theta'_1)$ for all $\chi$.

The representations $\cInd_J^G\tilde{\theta_1'}$ are irreducible and supercuspidal by \cite[15.3 Thm.]{BushnellHenniart:06a} provided we check hypotheses as follows. Since $\theta_1'$ is minimal, the element $\alpha_{\theta_1'}\in E^\times$ is minimal by Lemma \ref{minimalimpliesminimal} and so by 13.5 Prop.\ of loc.\ cit.\ we have that there exists a chain order $\mathfrak{A}$ such that $(\mathfrak{A}, -v_E(\alpha_{\theta'_1}),\alpha_{\theta'_1})$ is a simple stratum. Since $c(\theta_1') =c(\theta')$, we have that $n:=-v_E(\theta_1')= -v_E(\theta')$ so that $J_{\alpha}$ with $\alpha=\alpha_{\theta'_1}$ as defined in loc.\ cit.\ (15.3.1) matches $J$ as defined in \eqref{Jramdef} with respect to the character $\theta'$. Moreover, $\widetilde{\theta}_1'$ is a $1$-dimensional representation of $J$ that by definition contains the character $\psi_{\alpha}$ of $U_{\mathfrak{A}_e}^{\lfloor \frac{n}{2} \rfloor+1}$. 

Since the $\cInd_J^G\tilde{\theta_1'}$ on the right of \eqref{IndDirectSum2} are irreducible, we conclude from \eqref{sigmachiHomequation} the claim in the second sentence of the proof. 

Choose an orthonormal basis $\mathcal{B}_\pi$ for $V_\pi$ that contains the minimal vector, say, $\phi$. Then, by definition $H$ acts through $\pi$ on $\phi$ by $\pi(h)\phi = \theta_1(h)\phi$ for $h \in H$. Now, we also have $\overline{\tilde{\Phi}}_{\varphi} (h) = \overline{\theta}(h)$ for $h \in H$ by Lemma \ref{mx_coeff_ofminlvec}. So, for any $v \in \mathcal{B}_\pi$ we have $$ \langle \pi (\overline{\tilde{\Phi}}_{\varphi}) v, \phi \rangle = \langle v, \int_{H} \overline{\tilde{\Phi}}_{\varphi} (h^{-1}) \pi(h)\phi dh\rangle =  \vol(Z\backslash H) \langle v, \phi \rangle,$$ since  $\theta = \theta_1$ on $U(l)$. So, for $v \in \mathcal{B}_\pi$, we have $ \pi (\overline{\tilde{\Phi}}_{\varphi}) v = 0 $ unless $\varphi = \phi$. Thus, $\frac{1}{\vol(Z\backslash H)} \pi (\overline{\tilde{\Phi}}_{\varphi})$ is a projection onto the line of the minimal vector in $V_\pi$. 

Now suppose that $E/F$ is unramified. If $\pi(\overline{\widetilde{\Phi}}_\varphi)$ is non-trivial then the same calculation as \eqref{nonzerooperator_implies_sigma_acts_by_char} with $H$ replaced by $ZB^1$ shows that there exists a vector $v \in V_\pi$ such that $(\pi \times \chi_m^{-1})(b)v = \widetilde{\theta}'(b)v$ for all $b \in ZB^1$. Then, the assertions of the proposition follow from \cite[Props.\ 3.14 and 3.21]{Hu} applied to $\pi \times \chi_m^{-1}$. 
\end{proof}
For a moment let us consider the more general situation that $\sigma$ is a supercuspidal  representation of $G(F)$ endowed with a unitary pairing $\langle \cdot, \cdot \rangle_\sigma$. Let $V$ be the space of functions $f$ on $G(F)$ satisfying $f(ng) = \psi(n)f(g)$ for all $n \in N(F)$ and $g \in G(F)$. For any $g_0 \in G(F)$ and $v'\in \sigma$, let $W: \sigma \to V$ be defined by \begin{equation}\label{Whittakerdef}W: v \mapsto W_v(g)=\int_{N(F)} \langle \sigma(g_0 n g)v,v'\rangle_\sigma \psi(-n)\,dn.\end{equation}
One can directly check that 
$$W_v(ng)=\psi(n)W_v(g) \quad \text{ and } \quad W_{\sigma(h)v}(g) = W_v({gh}).$$ 
Thus, if $g_0$ and $v'$ are be chosen so that the map $W$ is non-zero, then it follows that $W$ is an isomorphism onto the Whittaker model of $\sigma$. 

Let us return now to the situation at hand  introduced just before Proposition \ref{minlvec_projection_prop}. As in \cite[\S 3.2.2]{Hu} we can compute the Whittaker function $W_{\varphi'}$ of the minimal vector $\varphi'$ in the twist-minimal representation $\sigma'$ using \eqref{Whittakerdef}. 
\begin{mylemma}
Let $\varphi'$ be a minimal vector in a twist-minimal dihedral supercuspidal representation $\sigma'$ of sufficiently large conductor in the sense \eqref{csigmaprime_suff_large}.  
Its Whittaker function along the diagonal $W_{\varphi'}(a(x))$ is, up to a scalar, equal to $1_{-zBU_F(\lceil j/2\rceil)}$, where $z$ and $B$ are as in Lemma \ref{Q2lem2} with respect to an additive character $\psi$ with $c(\psi)=0$, and $j$ is as in Remark \ref{remarkafterQ2lem}. 
\end{mylemma}
\begin{proof}
The case that $E/F$ is unramified is given by \cite[Lem.\ 3.15]{Hu}, whose proof goes through with the extra assumption that $\sigma'$ is twist-minimal to permit the use of \cite[15.3 Thm.]{BushnellHenniart:06a} in the final step. 

We therefore assume that $E/F$ is ramified for the rest of the proof. 
We compute the integral in \eqref{Whittakerdef} along $A(F)$ using Lemma \ref{mx_coeff_ofminlvec}. To do so, we need to explicate the $J,\theta'$ used to construct $\sigma'$ by compact induction.

Recall the character $\theta'$ of $E^\times$ defined just above Corollary \ref{alphaxialphathetacor} that gives rise to $\sigma'$ by compact induction. Since $F=\Q_2$, the hypothesized lower bound on $c(\sigma')$ implies the condition of Lemma \ref{alphathetaalphaxi} is satisfied, so that by Corollary \ref{alphaxialphathetacor} we have 
$\alpha_{\theta'} \equiv \alpha_{\xi'} \pmod {\fp_E^{-j+c(\psi_E)}}$. 

Now let $z\in F^\times$ be as in Lemma \ref{Q2lem2} and Remark \ref{remarkafterQ2lem}, so that $\alpha_{\xi'} = z \alpha_0$. 
Therefore, choosing the embedding $E^\times \hookrightarrow G(F)$ of \eqref{Oembedding} in terms of $\alpha_0$, we have that $\alpha_{\theta'}$ is given in matrix form by
$$ \alpha_{\theta'} = z \zxz{}{1}{-B}{-A} \pmod {\mathfrak{P}_2^{-j+c(\psi_E)}}.$$

We choose $v'$ and $g_0$ in \eqref{Whittakerdef} to be given by $v'=\varphi'$ and $g_0=a(-1/zB)$, and $c(\psi)=0$.  By Lemma \ref{Q2chainorders} and since $$ U_{\mathfrak{A}_2}^j = 1+\mathfrak{P}_2^j = 1+ \zxz{p^{\lceil j/2\rceil} \cO}{p^{\lfloor j/2\rfloor} \cO}{p^{\lfloor j/2\rfloor+1} \cO}{p^{\lceil j/2\rceil} \cO}$$
 we have
$$g_0na(x) \in J \quad \text{ if and only if } \quad \begin{cases}x \in -zB U_F(\lceil j/2\rceil) \text{ and } \\ v(n) \geq - \lceil j/2\rceil. \end{cases}$$
Thus by Lemma \ref{mx_coeff_ofminlvec} we have
\begin{align}
W_{\varphi'}(a(x))&=\langle \varphi',\varphi' \rangle \int\limits_{v(n)\geq -\lceil j/2 \rceil}\psi\circ \Tr\left( z \zxz{0}{1}{-B}{-A} \zxz{-\frac{x}{zB} -1}{-\frac{1}{zB}n}{0}{0}  \right)    \psi(-n)dn\\
&=\langle \varphi',\varphi' \rangle \int\limits_{v(n)\geq -\lceil j/2 \rceil}dn. \notag
\end{align}
\end{proof}
By translating the minimal vector $\varphi'$ back to the minimal vector $\varphi \in \cInd_J^G \Lambda\simeq \sigma$, we have that the $L^2$-normalized \eqref{GinvtInnerProd} Whittaker function of $\varphi$ satisfies
\begin{equation}\label{minimal_vec_in_Kmodel}
W_\varphi(a(x)) = \vol(U_F(\lceil j/2\rceil))^{-1/2} \chi_m(x)  1_{-zBU_F(\lceil j/2\rceil)}(x).
\end{equation}

Next we express new vectors as a sum of translates of the minimal vector. Recall the Gauss sum 
$$\tau(\chi,\psi) = \sum_{(\cO/ p^{c(\chi)}\cO)^\times} \chi(x) \psi(x/p^{c(\chi)}).$$ 
\begin{mylemma}\label{Lem:minimaltonew}
Let $\varphi\in\sigma$ be the $\chi_m$-translate of the minimal vector $\varphi' \in \sigma'$ and $\varphi_0\in\sigma $ be a newvector. Then
\begin{enumerate}
\item when $d=0$ or $2$, 
\begin{multline}\label{eq:minimaltonew02}
\varphi_0= \frac{1}{\vol(U_F(\lceil j/2\rceil))^{-1/2}\tau(\chi_m,\psi)}   \sum\limits_{b\in  \cO^\times/U(j+1)}\chi_m(b)\sigma \Big(\zxz{1}{p^{-j-1}b}{}{1}\Big)\\ \times \sum\limits_{a\in (\cO/\fp^{\lceil j/2 \rceil})^\times }\chi_m^{-1}(a)\sigma\Big(\zxz{p^{-j}a}{}{}{1}\Big) \varphi,
\end{multline}
\item when $d=3$
\begin{equation}\label{eq:minimaltonew3}
\varphi_0 =\frac{1}{\vol(U_F(\lceil j/2\rceil))^{-1/2}} \sum\limits_{a\in \cO^\times/U(\lceil j/2\rceil) }\sigma\Big(\zxz{p^{-j}a}{}{}{1}\Big) \varphi.
\end{equation}
\end{enumerate}
\end{mylemma}
\begin{proof}
We compute in the Kirillov model $\mathcal{K}(\sigma,\psi)$ using \eqref{minimal_vec_in_Kmodel}, following \cite[\S 3.2.2, 3.3.2]{Hu}. We focus on the $d=0,2$ case as the $d=3$ case is strictly simpler.  

Consider the interior sum. Since $\chi_m(p)=1$ by definition, we have by \eqref{minimal_vec_in_Kmodel} that  
$$\sum\limits_{a\in \cO^\times/U(\lceil j/2\rceil) }\chi_m^{-1}(a)\sigma\Big(\zxz{p^{-j}a}{}{}{1}\Big) W_{\varphi}(a(x)) = \vol(U_F(\lceil j/2\rceil))^{-1/2}  \chi_m(x)1_{\cO^\times}(x) \in \mathcal{K}(\sigma,\psi).$$
By the defining property of the Whittaker model, the sum over $b$  on the right hand side of \eqref{eq:minimaltonew02} in the Kirillov model is equal  to 
$$\vol(U_F(\lceil j/2\rceil))^{-1/2} \sum\limits_{b\in  \cO^\times/U(j+1)}\chi_m(b) \psi \left( \frac{bx}{p^{j+1}}\right) \chi_m(x)1_{\cO^\times}(x).$$
The character $\chi_m$ has conductor $j+1$ in either case $d=0$ or $d=2$ (see Definition \ref{def:twistminimalQ2}), so that the Gauss sum above satisfies 
$$ \sum\limits_{b\in  \cO^\times/U(j+1)}\chi_m(b) \psi \left( \frac{bx}{p^{j+1}}\right) =  \chi_m(x)^{-1}\tau(\chi_m,\psi)$$
for $x \in \cO^\times$ (cf.\ \cite[Lem.\ 3.25]{Hu}).  
\end{proof}
Now we give a preliminary definition of the newform projector in the $p=2$ case. For  $\varphi \in \sigma$  the minimal vector and $\ell \geq 2-e$ recall the restricted matrix coefficient $\widetilde{\Phi}_{\varphi}$ and volume $V$ from \eqref{tildePhiphidef}. Let $c= \frac{1}{\vol(U_F(\lceil j/2\rceil))^{-1/2}}$ be the leading constant in   \eqref{eq:minimaltonew02} and \eqref{eq:minimaltonew3}. 
\begin{mydefi}\label{p=2_testfcn_def}
Set $f\in \cH_2$ to be the function satisfying 
\begin{enumerate}
\item when $d=0$ or $2$
$$ \overline{f(g)} =|c|^2V^{-1} \sum\limits_{b,b' \in \cO^\times/U(j+1)}\sum\limits_{a,a' \in \cO^\times/U(\lceil j/2\rceil) }\chi_m\left(\frac{ba'}{ab'}\right) \widetilde{\Phi}_{\varphi}\left(\zxz{p^{-j}a'}{p^{-j-1}b'}{}{1}^{-1} g \zxz{p^{-j}a}{p^{-j-1}b}{}{1}\right),$$
\item when $d=3$
$$ \overline{f(g)} =|c|^2V^{-1} \sum\limits_{a,a' \in \cO^\times/U(\lceil j/2\rceil)} \widetilde{\Phi}_{\varphi}\left(\zxz{p^{-j}a'}{}{}{1}^{-1} g \zxz{p^{-j}a}{}{}{1}\right).$$
\end{enumerate}
\end{mydefi}
\begin{mycoro}[of Proposition \ref{minlvec_projection_prop} and Lemma \ref{Lem:minimaltonew}]\label{p=2_newform_projector_cor}
For $\sigma$ a trivial central character supercuspidal representation of $G(F)$ with $$c(\sigma)\geq \begin{cases}5 & \text{ if } d=0, \\ 8 & \text{ if } d=2,\\ 11 & \text{ if } d=3\end{cases}$$ and $\ell$ as in Proposition \ref{minlvec_projection_prop}, 
\begin{enumerate}
\item the $f\in \cH_2$ constructed from these as in Definition \ref{p=2_testfcn_def} is a newform projector in the sense of Definition \ref{def:newformproj}, and 
\item the operator $\pi (f)$ is 0 unless there exists $\theta_1 \in \theta[\ell] 
$ such that $\pi \simeq \cInd_J^G \Lambda_1$. 
\end{enumerate}
\end{mycoro}
Here, recall that if $E/F$ is ramified, we have $\Lambda_1 = \widetilde{\theta}_1$ where $\widetilde{\theta}$ is the extension of $\theta$ from $E^\times$ to $J$ in \eqref{theta_to_thetatilde}. If $E/F$ is unramified, then $\Lambda_1$ is constructed from $\theta_1$ as in \cite[\S 3.2.1]{Hu}. See also Section \ref{compactinductionbackground} for more detailed references.

We have shown that $f$ satisfies the spectral assumption for an appropriate choice of scalar. Now, we give an alternate description of $f$ from which  Geometric Assumption \eqref{geo3} is obvious.

\begin{myprop}\label{Cor:twotestfunsame}
Suppose $F$ and $c(\sigma)$ are as in Corollary \ref{p=2_newform_projector_cor}, and choose $\ell=1$.  
\begin{enumerate}
\item If $d=0$, $$f=V^{-1}\overline{\Phi}_{\varphi_0}|_{ZK_0(j+1,-j-1)}= \frac{1}{\| \Phi_{\varphi_0}|_{ZK_0(j+1,-j-1)}\|_2^2} \overline{\Phi}_{\varphi_0}|_{ZK_0(j+1,-j-1)}.$$
\item If $d=2$, we have
$$f=V^{-1}\overline{\Phi}_{\varphi_0}|_{ZK_0(j+1,-j-1)}= \frac{1}{\| \Phi_{\varphi_0}|_{ZK_0(j+1,-j-1)}\|_2^2} \overline{\Phi}_{\varphi_0}|_{ZK_0(j+1,-j-1)}.$$
\item If $d=3$, we have
$$f=V^{-1}\overline{\Phi}_{\varphi_0}|_{ZK_0(j+1,-j)}= \frac{1}{\| \Phi_{\varphi_0}|_{K_0(j+1,-j)}\|_2^2} \overline{\Phi}_{\varphi_0}|_{K_0(j+1,-j)}.$$
\end{enumerate}
\end{myprop}
\begin{proof}
We assume that $d=2$, as the $d=0,3$ cases are simpler. 
By Lemma \ref{Lem:minimaltonew}, 
$$\Phi_{\varphi_0}(g) = |c|^2 \sum\limits_{b,b',a,a'}\chi\left(\frac{ba'}{ab'}\right) {\Phi}_{\varphi}\left(\zxz{p^{-j}a'}{p^{-j-1}b'}{}{1}^{-1} g \zxz{p^{-j}a}{p^{-j-1}b}{}{1}\right).$$

We claim that for any $h\in H=ZU_E(\ell)U_{\mathfrak{A}_e}^j$
\begin{equation}\label{Eq:suppcompare1}
\zxz{p^{-j}a'}{p^{-j-1}b'}{}{1} h\zxz{p^{-j}a}{p^{-j-1}b}{}{1}^{-1} \in ZK_0(j+1,-j-1 ), 
\end{equation}
and for any $h\in E^\times U_{\mathfrak{A}_e}^j \smallsetminus H$,
\begin{equation}\label{Eq:suppcompare2}
\zxz{p^{-j}a'}{p^{-j-1}b'}{}{1} h\zxz{p^{-j}a}{p^{-j-1}b}{}{1}^{-1} \not\in ZK_0(j+1,-j-1 ).
\end{equation}

When $h\in H$ there exists $s \in F^\times$, $y,z\in F$ such that 
\begin{equation}\label{Eq:expressionh}
h=s(y+z\alpha_0)(1+x)=s\zxz{y}{z}{-zB}{y-Az}(1+x)
\end{equation} 
with $1+x\in U_{\mathfrak{A}_e}^j$, $v(y-1)\geq 1$ and $v(z)\geq 0$.  Then, using Lemma \ref{minl_elts_in_Q2extns} we can check directly by computing the valuation of each entry that \eqref{Eq:suppcompare1} is true. 

Now suppose that $h \in E^\times U_{\mathfrak{A}_e}^j \smallsetminus Z U_E(1) U_{\mathfrak{A}_e}^j$. If we can write $h$ as $h=s(y+z\alpha_0)(1+x)$ with $y \neq 0$, then $h = sy(1+\frac{z}{y} \alpha_0)(1+x)$, so that since $h \not \in H$ we must have $v_E(z\alpha_0/y)\leq 0$, i.e.\ $v(z)+1 \leq v(y)$. Now we look at the valuation of the determinant of $h$, which is $v(y^2-Ayz+z^2 B)+2v(s)$, but we have (by Lemma \ref{minl_elts_in_Q2extns}) that 
$$v(z^2 B) = 2v(z)+1 \leq v(zy)< v(Azy)\leq v(y^2).$$
Thus, $v(\det(h)) = v(z^2 B)+2v(s)$. If on the other hand $y=0$, then we have directly that $v(\det(h)) = v(z^2 B)+2v(s)$. In either case, $v(\det(h))$ is odd, which proves \eqref{Eq:suppcompare2}.



We thus have that
$$\Phi_{\varphi_0} \vert_{K_0(j+1,-j-1)} = |c|^2 \sum\limits_{b,b',a,a'}\chi\left(\frac{ba'}{ab'}\right) \widetilde{\Phi}_{\varphi}\left(\zxz{p^{-j}a'}{p^{-j-1}b'}{}{1}^{-1} g \zxz{p^{-j}a}{p^{-j-1}b}{}{1}\right).$$
This establishes the first equality of the Proposition.

For the second equality, $f = V^{-1}\overline{\Phi}_{\varphi_0}|_{K_0(j+1,-j-1)}$ is a newform projector, so by Lemma \ref{projection_upper_bound} we have $$  \| \Phi_{\varphi_0}|_{K_0(j+1,-j-1)}\|_2^2=V\overline{\Phi}_{\varphi_0}|_{K_0(j+1,-j-1)}(1) =V \overline{\Phi}_{\varphi_0}(1) = V.$$ 
\end{proof}

\begin{proof}[Proof of Theorem \ref{cor:SpecAssumption_supercuspidal_p2}]
Combine Corollary \ref{p=2_newform_projector_cor}(1) and Proposition \ref{Cor:twotestfunsame} to see that relevant test functions are newform projectors. Combine Remark \ref{size_thetaell_rem} and Table \ref{table:xi} with Corollary \ref{p=2_newform_projector_cor}(2) for the assertions on the size of the support of $\pi(f)$. 
\end{proof}

\subsection{Local generalized Kloosterman sums for supercuspidal representations}\label{sec:SCKloostermanSum}
Let $\sigma$ be a trivial central character dihedral supercuspidal representation of $\GL_2(F)$ and $(E/F,\xi)$ be a pair such that $\sigma$ correspond to $\Ind_{E}^{F} \xi$ by the LLC. 
Let $H_p(m,n;c)$ be the local Kloosterman sum (see \eqref{eq:KloostermanLocalFormula}) associated with the test function $f=f_\xi$ defined in either Theorem \ref{cor:SpecAssumption_supercuspidal} ($p$ odd) or Theorem \ref{cor:SpecAssumption_supercuspidal_p2} ($p$ even).  If $p=2$ assume further that $c(\sigma)$ is sufficiently large in the sense of Theorem \ref{cor:SpecAssumption_supercuspidal_p2}.
\begin{mytheo}\label{KlSumThm}
 If $k<\lceil c(\sigma)/2 \rceil$, then $H_p(m,n;p^k)$ vanishes identically. If $k \geq \lceil c(\sigma)/2 \rceil$, then $H_p(m , n; p^k)$ is given by 
\begin{equation}\label{eq:KlSumThm} H_p(m , n; p^k)= \overline{\gamma}(1-p^{-1})^{-1}f_{\xi}(1) p^{-\frac{d}{2}}\sum_{\substack{u \in (\cO_E/p^k \cO_E)^\times \\ \Nm(u) \equiv mn \shortmod{p^k}}} \xi(u) \psi\left( - \frac{\Tr( u)}{p^k}\right),\end{equation}
where $\gamma \in S^1$ depends only on the isomorphism class of $E$ and the choice of additive character $\psi$. In particular, $H_p(m,n;p^k)=0$ if $(mn,p) \neq 1$. 
\end{mytheo}
\begin{myrema}\label{Remark_following_KlSumThm} More precisely, $\gamma=\lambda(E,\psi)$ is the Langlands constant as in \cite[Lem.\ 1.2(iv)]{JacquetLanglands}. For an explicit description of $\gamma$, see \cite[\S 34.3]{BushnellHenniart:06a}. 
Note that explicit formulas for $f_{\xi}(1)$ were given in \eqref{sec:7.3fp(1)} for $p$ odd and in \eqref{f1p=2_formula} for $p=2$. 
\end{myrema}

\begin{myrema}\label{KlSumThm_rem}
By Theorem \ref{cor:SpecAssumption_supercuspidal_p2}, when $p=2$ and $E/F$ is the unramified extension the sum on the right hand side of \eqref{eq:KlSumThm} may be restricted to $U_E(1)/U_E(k)$ without changing the validity of the equation. This assertion can also be (sanity) checked by decomposing $u=u_0+u'$ with $v_E(u')\geq 1$, using Lemma \ref{postnikov}, and noting that $U_{\Q_2}(0)= U_{\Q_2}(1)$. 
\end{myrema}

\begin{proof}
Recalling the definition of $H_p(m,n;c)$ from \eqref{eq:KloostermanLocalFormula} we choose the test function $f=f_p$ in the $p$ odd case from Theorem \ref{cor:SpecAssumption_supercuspidal} and in the $p=2$ case from Theorem \ref{cor:SpecAssumption_supercuspidal_p2}. Such an $f$ has support contained in $a(y)^{-1} ZKa(y)$ with \begin{equation}\label{yc0cases}y= \begin{cases} 
p^{c_0} & \text{ if }  d=0,\\
p^{c_0+1} & \text{ if } d=1 \text{ or } 2,\\
p^{c_0+2} & \text{ if }  d=3.
\end{cases}\end{equation}
We can unify these (see \eqref{c(pi)c0}) as 
$v_p(y) 
= \lceil c(\sigma)/2 \rceil .$ 

If $k<\lceil c(\sigma)/2 \rceil$, then $H_p(m,n,p^k)=0$ by Lemma \ref{ccondition}. This proves the first assertion of the Theorem. 
We now assume for the rest of the proof that $k \geq v_p(y)$, equivalently $2k\geq c(\sigma)$. 

\begin{mylemma}\label{drop_supp_cond_in_Klsum_computation}
For the above choice of $f$, when $2k\geq c(\sigma)$ we have \begin{equation}\label{eq:KlSumThm_step1}
H_p(m,n;p^k) =f_{\xi}(1) \iint\limits_{\substack{v(t_2) \geq -k \\  v(t_1) \geq -k}} \overline{\Phi}\left( n(t_1)^{-1} \left( \begin{smallmatrix} & -p^{-2k} \\ 1 & \end{smallmatrix}\right) n(t_2)\right) \psi(-mt_1 + nt_2)\,dt_1\,dt_2.
\end{equation}
If moreover $2k>c(\sigma)$, then one or both conditions $v(t_i)\geq -k$ on the integration may be replaced by $v(t_i)=-k$. 
\end{mylemma}
\begin{proof}
By looking at the determinant (cf.\ \eqref{eq:matrixcondition}) we have that $v(t_1)$ and $v(t_2) \geq -k$ whenever $n(t_1)^{-1} \left( \begin{smallmatrix} & -p^{-2k} \\ 1 & \end{smallmatrix}\right) n(t_2)$ is in a diagonal conjugate of $ZK$. If $p \neq 2$ and $d=1$ or $p=2$ and $d=3$, then $\supp f$ is contained in a group $ZK_0(a,b)$ with $a+b>0$, so that we must in fact have $v(t_1) = v(t_2) = -k$. Suppose now that $d=0$ or ($p=2$ and $d=2$), and that $2k>c(\sigma)$. Then, $\supp f \subseteq Za(p^{c_0+d/2})^{-1}Ka(p^{c_0+d/2})$, so that $v(p^{-2k}+t_1t_2)\geq -k-(c_0+d/2)$. But if either $v(t_1)$ or $v(t_2)$ were $>-k$, we would have that $v(p^{-2k}+t_1t_2)=-2k$, which contradicts the assumption that $2k>c(\sigma)$. 

Now, to show \eqref{eq:KlSumThm_step1} it suffices to show that we can drop the restriction on the support of $\Phi$ that appears in Theorems \ref{cor:SpecAssumption_supercuspidal} and \ref{cor:SpecAssumption_supercuspidal_p2} for matrices of the form $ n(t_1)^{-1} \left( \begin{smallmatrix} & -p^{-2k} \\ 1 & \end{smallmatrix}\right) n(t_2) \in  \supp \Phi$. 

Lemma \ref{Prop:MC} describes the support of $\Phi$ in terms of the Iwasawa decomposition in Lemma \ref{Lem:Iwasawadecomp}. To implement this, we write 
\begin{equation}\label{nt1deltant2_decomposition} n(t_1)^{-1} \left( \begin{smallmatrix} & -p^{-2k} \\ 1 & \end{smallmatrix}\right) n(t_2) = \zxz{p^{-v(t_2)-2k}}{-p^{-2k} -t_1t_2}{}{t_2}\zxz{1}{}{p^{-v(t_2)}}{1}\zxz{p^{v(t_2)}t_2^{-1}}{}{}{1}.\end{equation}
 Lemma \ref{Prop:MC} breaks into cases depending on the size of $k$ and $c(\sigma)$. 
\begin{itemize}
\item If $k\geq c(\sigma)-1$ and $n(t_1)^{-1} \left( \begin{smallmatrix} & -p^{-2k} \\ 1 & \end{smallmatrix}\right) n(t_2) \in \supp \Phi$, then $v(p^{-2k}+t_1t_2)\geq v(t_2)-1$ by Lemma \ref{Prop:MC}(i). Since $c_0\geq 1$ and $v(t_2)\geq-k$ it follows that $v(p^{-2k}+t_1t_2)\geq -k-c_0$. 
\item If $k\leq c(\sigma)-2$ and $2k \neq c(\sigma)$, then $v(t_2)=-k$. So, when $n(t_1)^{-1} \left( \begin{smallmatrix} & -p^{-2k} \\ 1 & \end{smallmatrix}\right) n(t_2) \in \supp \Phi$ we have $v(p^{-2k}+t_1t_2) = -c(\sigma)$  by Lemma \ref{Prop:MC}(ii).   
\item If $2k=c(\sigma)$  and $-v(t_2)<c(\sigma)/2$, then Lemma \ref{Prop:MC}(ii) applies, so we get that $v(p^{-2k}+t_1t_2) = -c(\sigma)$. 
\item If $2k=c(\sigma)$ and $-v(t_2)=c(\sigma)/2$, then $v(p^{-2k}+t_1t_2) = v(t_2)-c(\sigma)/2$ by Lemma \ref{Prop:MC}(iii), which is $\geq -c(\sigma)$. 
\end{itemize}

We proceed by cases. If $p\neq 2$ and $d=0$ then $\Phi \vert_{ZK'} = \Phi$ by Theorem \ref{cor:SpecAssumption_supercuspidal} so there is nothing to show. 

Assume that $p\neq 2$ and $d=1$. According to Theorem \ref{cor:SpecAssumption_supercuspidal} we need to show that $n(t_1)^{-1} \left( \begin{smallmatrix} & -p^{-2k} \\ 1 & \end{smallmatrix}\right) n(t_2) \in   Za(p^{c_0})^{-1}Ka(p^{c_0})$. Multiplying out and scaling by the square root of the determinant, it suffices to show that $v(p^{-2k}+t_1t_2)\geq -c_0-k$. 
\begin{itemize} 
\item If $k\geq c(\sigma)-1$, then we already showed $v(p^{-2k}+t_1t_2)\geq -k-c_0$ without casework.  
\item If $k\leq c(\sigma)-2$ and $2k \neq c(\sigma)$, then we showed $v(p^{-2k}+t_1t_2) = -c(\sigma)$, which is $=-2c_0-1 \geq -c_0-k$ by \eqref{c(pi)c0} and since $p\neq 2$ and $d=1$.  
\end{itemize}

We move on to the $p=2$ cases. Multiplying out $n(t_1)^{-1} \left( \begin{smallmatrix} & -p^{-2k} \\ 1 & \end{smallmatrix}\right) n(t_2)$ and scaling by the square root of the determinant, according to Theorem \ref{cor:SpecAssumption_supercuspidal_p2} it suffices for the claim to show that $v(p^{-2k}+t_1t_2)\geq -c_0-k-e+1$.
\begin{itemize} 
\item If $k\geq c(\sigma)-1$, then we already showed $v(p^{-2k}+t_1t_2)\geq -k-c_0$ without casework.  
\item If $k\leq c(\sigma)-2$, then we showed $v(p^{-2k}+t_1t_2) \geq -c(\sigma)$, which is $=-2c_0-d $ by \eqref{c(pi)c0}. We have by assumption that $k\geq v_p(y)$, with $v_p(y)$ given by \eqref{yc0cases}, 
so that indeed  $v(p^{-2k}+t_1t_2)\geq -c_0-k-e+1$.
\end{itemize}
Note that when $2k>c(\sigma)$, the integral is restricted to $v(t_1)=-k$ and $v(t_2)=-k$, but either integration or both may be trivially extended to $v(t_i) \geq -k$ and \eqref{eq:KlSumThm_step1} remains valid. 
\end{proof}

Now we open the matrix coefficient in \eqref{eq:KlSumThm_step1} in the Whittaker model \eqref{GinvtInnerProd} to obtain 
\begin{multline} H_p(m,n;p^k) \\ =f_{\xi}(1) \iint\limits_{\substack{v(t_2)\geq -k \\ v(t_1)\geq -k}}\int_{F^\times}\overline{W}\left( a(y) \left( \begin{smallmatrix}  & -p^{-2k} \\ 1 &  \end{smallmatrix}\right) n(t_2) \right) W\left( a(y) n(t_1) \right)  d^\times y\, \psi(-mt_1+nt_2)\,dt_1\,dt_2,
\end{multline}
where $W$ is a $L^2$-normalized newform in the Whittaker model of $\sigma$. Now we swap order of integration and evaluate the $t_1$ integral. By the defining property of the Whittaker model we have $ W\left( a(y) n(t_1) \right) = \psi(yt_1)W\left(a(y)\right)$, so 
$$ \int\limits_{v(t_1)\geq -k} W\left( a(y) n(t_1) \right) \psi(-mt_1)\,dt_1 = p^k W\left(a(y)\right) \delta_{y \equiv m \shortmod{p^k}}.$$ Since $W(a(y))= 1_{\Z_p^\times}(y)$ (see Lemma \ref{NewformKirillov}), we already see that $H_p(m,n;p^k)=0$ unless $(mn,p)=1$, which we may freely assume for the rest of the proof. To compute $H_p(m,n;p^k)$ it suffices by Theorem \ref{thmKP}(4) to compute $H_p(m,1;p^k)$. Collecting the above, we have
\begin{equation}\label{KlSumThmStartingPoint2}  H_p(m,1;p^k) = f_{\xi}(1)p^k\int\limits_{v(t_2)\geq -k}\int\limits_{\substack{y\in\Z_p^\times \\ y \equiv m \shortmod{p^k}}}\overline{W}\left( a(y) \left( \begin{smallmatrix}  & -p^{-2k} \\ 1 &  \end{smallmatrix}\right) n(t_2) \right)    \psi(t_2) \,d^\times y\, \,dt_2.\end{equation}
Note that if $2k>c(\sigma)$ then the condition $v(t_2)\geq -k$ in \eqref{KlSumThmStartingPoint2} may be replaced by $v(t_2)= -k$. 

The computation now breaks into cases depending on whether $2k>c(\sigma)$ or $2k=c(\sigma)$. 

\textbf{Case $2k>c(\sigma)$:} We claim that the integrand in $y$ is a constant function on $y \equiv m \pmod {p^k}$ for fixed $t_2$. We would like to use \cite[Prop.\ 2.12]{Hu:17a} to accomplish this, so we need to decompose the argument of $\overline{W}$ according to Lemma \ref{Lem:Iwasawadecomp}. Precisely, we have
$$\zxz{y}{}{}{1}\zxz{0}{-p^{-2k}}{1}{0}\zxz{1}{t_2}{0}{1}=\zxz{t_2}{}{}{t_2}\zxz{\frac{y}{t_2p^k}}{\frac{-y}{t_2p^{2k}}}{}{1}\zxz{1}{}{p^k}{1}\zxz{t_2^{-1}p^{-k}}{}{}{1},$$
thus
\begin{equation*}
W \left( \zxz{y}{0}{0}{1}\zxz{0}{-p^{-2k}}{1}{0}\zxz{1}{t_2}{0}{1}\right)=\psi\left(-\frac{y}{t_2p^{2k}}\right) W\left(\zxz{\frac{y}{t_2p^k}}{}{}{1}\zxz{1}{}{p^k}{1}\right).
\end{equation*}
By \cite[Prop.\ 2.12]{Hu:17a} this is indeed a constant function of $y$ once we restrict to $y \equiv m \pmod{p^k}$. Collecting these calculations, we have proven the following. 
\begin{mylemma}
Suppose that $k>c(\sigma)/2$. If $(m,p) \neq 1$ then $H_p(m,1;p^k)=0$ and if $(m,p) =1$ then 
\begin{equation}\label{Eq:Ip}
 H_p(m,1;p^k) = (1-p^{-1})^{-1} f_{\xi}(1)\int\limits_{v(t)=-k} \overline{W}\left(a(mp^{-2k})wn(t)\right)\psi(t) dt,
\end{equation}
where $W$ is an $L^2$-normalized newform in the Whittaker model of $\sigma$ and $w = \left( \begin{smallmatrix} & -1 \\ 1 & \end{smallmatrix}\right)$. 
\end{mylemma}
To continue the evaluation of the Kloosterman sum, we need to substitute in an expression for the Whittaker function. There are at least two different choices. One is to use minimal vectors as in \cite{Hu}.  Another choice is to use results of Assing, namely \cite[Lem.\ 3.1]{Assing_padic_Whittaker}, which is the path that we pursue in this paper. 

We use explicit expressions for the newform in the Whittaker model due to Assing \cite{Assing_padic_Whittaker}. Following the notation after (1.4) in loc.\ cit.\, let
\begin{equation}\label{Assing_After_14}
g_{t,l,v} = \zxz{0}{p^t}{-1}{-vp^{-l}} 
\end{equation} and note from the paragraph following (1.3) that Assing normalizes the additive Haar measure so that the total volume of $\cO_E$ is $p^{-\frac{d}{2}}$ whereas we have taken the volume of $\cO_E$ to be 1. Now, Assing's Lemma 3.1 asserts that (the $\Omega^{t/f}$ there equals $\varpi_E^{-ke_{E/\Q_p}}$ in our situation)
\begin{equation}\label{Eq:W}
W\left(a(mp^{-2k})wn(t)\right) =W \lb g_{-2k,k,\frac{tp^{k}}{m}}\rb 
=\gamma p^{k-\frac{d}{2}}\int_{\cO_E^\times}\xi^{-1}(x)\psi(\Tr(x)p^{-k}+\frac{t}{m}\Nm(x))\,dx,	
\end{equation}
where $\gamma$ is as in \cite[Lem.\ 1.2(iv)]{JacquetLanglands}. In particular, $|\gamma|=1$ and its value only depends on $E$.  

 Using \eqref{Eq:W} in \eqref{Eq:Ip}, we have
\begin{equation}
 H_p(m,1;p^k) =  \frac{\overline{\gamma}f_{\xi}(1)}{1-p^{-1}}  p^{k-\frac{d}{2}} \int\limits_{v(t)=-k} \int_{\cO_E^\times}\xi(x)\psi(-\Tr(x)p^{-k}-\frac{t}{m}\Nm(x))\,dx\, \psi(t) dt.
\end{equation}
Now we swap order of integration and execute the integral in $t$:
\begin{multline}
\label{penultimate_k>c0_case} 
H_p(m,1;p^k) = \frac{\overline{\gamma}f_{\xi}(1)}{1-p^{-1}} p^{k-\frac{d}{2}}  \int_{\cO_E^\times}\xi(x)\psi(-\Tr(x)p^{-k})
 \\ 
 \Big(  \int\limits_{v(t)\geq-k} -  \int\limits_{v(t)\geq-k+1}\Big) \psi\Big(\Big(1-\frac{\Nm(x)}{m} \Big)t\Big)\,  dt \, dx.\end{multline}
 Then for the $t$ integral on the smaller domain we have
\begin{multline}\label{vanishes_if_2nmidp}
\int_{\cO_E^\times}\xi(x)\psi(-\Tr(x)p^{-k})   \int\limits_{v(t)\geq-k+1} \psi\Big(\Big(1-\frac{\Nm(x)}{m}\Big)t\Big)\,  dt \, dx \\
= p^{k-1} \int_{ \cO_E^\times}\xi(x)\psi(-\Tr(x)p^{-k}) \delta_{\Nm(x)\equiv m \shortmod{p^{k-1}}}\,dx.
\end{multline}
Our goal is to show that \eqref{vanishes_if_2nmidp} vanishes. 

 Let $k'= k-1-\lfloor d/2 \rfloor.$ Set $x=x_0+\Delta$ with $v_E(\Delta)\geq ek'$. Note by \cite[41.2 Prop.(1)]{BushnellHenniart:06a} that $\Nm(x) \equiv m \pmod{p^{k-1}}$ if and only if $\Nm(x_0) \equiv m \pmod{p^{k-1}}$.  Also, since $2k>c(\sigma)$, we have $k\geq c_0+\lceil \frac{d+1}{2}\rceil$, so that $k'\geq c_0$. Therefore, $\xi(x)=\xi(x_0)$. Collecting these facts, we have 
 we have that the integral in \eqref{vanishes_if_2nmidp} equals
 \begin{equation}\label{vanishes_if_2=p}\sum_{x_0 \in (\cO_E / p^{k'}\cO_E)^\times} \xi(x_0)\psi(-\Tr(x_0)p^{-k}) \delta_{\Nm(x_0)\equiv m \shortmod{p^{k-1}}} \int_{\Delta \in p^{k'}\cO_E} \psi(-\Tr(\Delta)p^{-k})\,d\Delta.\end{equation} The additive character $\psi \circ \Tr$ of $E$ has conductor $-d$ by \eqref{add_char_cond}, and $v_E(\Delta p^{-k}) = e(k'-k) < -d$, so that the interior integral in \eqref{vanishes_if_2=p} vanishes. 
The result is that from \eqref{penultimate_k>c0_case} we obtain
\begin{equation}\label{penultimate2_k>c0_case}
H_p(m,1;p^k) = \frac{\overline{\gamma}f_{\xi}(1)}{1-p^{-1}} p^{k-\frac{d}{2}}  \int_{\cO_E^\times}\xi(x)\psi(-\Tr(x)p^{-k})
 \int\limits_{v(t)\geq-k}  \psi\Big(\Big(1-\frac{\Nm(x)}{m} \Big)t\Big)\,  dt \, dx.
\end{equation}
By orthogonality of additive characters, the interior integral in \eqref{penultimate2_k>c0_case} equals $p^k \delta( \Nm(x) \equiv m \pmod {p^k})$. Changing integrals to sums, 
 we conclude the  statement of the theorem in the case that $2k>c(\sigma)$.  

\textbf{Case $2k=c(\sigma)$:} This case can only occur when $E/F$ is unramified, or when $p=2$ and $d=2$. 
In this case, the condition $2k = c(\sigma)$ is equivalent to $k=c_0+d/2$. 

We pick up the calculation at \eqref{KlSumThmStartingPoint2}, and use Atkin-Lehner theory to continue. 
Indeed, let $\delta_\pi$ denote the eigenvalue of the newform $\varphi_0 \in V_\pi$ for  $\pi \in \overline{G}(F)^\wedge$ under the Atkin-Lehner operator:
\begin{equation}\label{ALdef}
\pi \left( \begin{smallmatrix} & 1 \\ p^{c(\pi)} & \end{smallmatrix} \right)\varphi_0 = \delta_\pi \varphi_0.
\end{equation}
If the central character of $\pi$ is trivial, one has $\delta_\pi = \pm 1$. Applying this in the Whittaker model of $\sigma$ (in which $W$ is an $L^2$-normalized newform), we have
$$ W\lb a(y) \zxz{0}{-p^{-2k}}{1}{0} n(t) \rb=\delta_\sigma W\lb  a(y) \zxz{0}{-p^{-2k}}{1}{0}n(t) \zxz{}{1}{p^{c(\sigma)}}{}\rb.$$
Write $i:=-v(t)\leq k$. Now one can verify that
\begin{multline*}a(y) \zxz{0}{-p^{-2k}}{1}{0}n(t) \zxz{}{1}{p^{c(\sigma)}}{}=a( -y/tp^i)\zxz{1}{}{p^{{c(\sigma)}-i}}{1}a(tp^i) \\ \in B\zxz{1}{}{p^{{c(\sigma)}-i}}{1}K_1(p^{c(\sigma)}).
\end{multline*}
We thus get that for $i=-v(t) \leq k$
\begin{equation}\label{eq:Lem3.16} W\lb a(y)\zxz{0}{-p^{-2k}}{1}{0} n(t) \rb=\delta_\sigma W\lb a( -y/tp^i)\zxz{1}{}{p^{{c(\sigma)}-i}}{1}\rb. \end{equation}
By \cite[Prop.\ 2.12]{hu_triple_2017}, this is $U(i)$-invariant in $y$. 
Inserting \eqref{eq:Lem3.16} in \eqref{KlSumThmStartingPoint2} we get
\begin{multline}\label{eq:beforeLemWi}  H_p(m,1;p^k) =f_{\xi}(1) p^{c_0} \delta_\sigma \int\limits_{v(t)=-i\geq -k}\int\limits_{\substack{y\in\Z_p^\times \\ y \equiv m \shortmod{p^k}}} \overline{W}\lb a( -y/tp^i)\zxz{1}{}{p^{{c(\sigma)}-i}}{1}\rb    \psi(t) \,d^\times y\, \,dt \\
= (1-p^{-1})^{-1}f_{\xi}(1) \delta_\sigma \int\limits_{v(t)= -i\geq -k}\overline{W}\lb a( -m/tp^i)\zxz{1}{}{p^{{c(\sigma)}-i}}{1}\rb    \psi(t) \, dt,
\end{multline}
where the 2nd line follows from the $U(i)$-invariance, since $i\leq k$.

With $W$ an $L^2$-normalized newform in the Whittaker model as before, set 
$$W^{(i)}(y):=W\left(\zxz{ y }{0}{0}{1}\zxz{1}{0}{p^i}{1}\right).$$
The following Lemma is a mild extension of \cite[Lem.\ 5.7]{HuSa:19}. 
\begin{mylemma}\label{Lem:Wi}
Let $\pi$ be a dihedral supercuspidal representation corresponding to $\Ind_E^{F} \xi$ by the LLC.  Let $W$ be an $L^2$-normalized newform in its Whittaker model. 
When $i\geq c(\pi)/2$ and $v(y)=0$, 
\begin{equation}
W^{(i)}(y)=\frac{\delta_\pi \gamma }{\zeta_\fp(1)}p^{\frac{c(\pi)-d}{2}}\int_{\cO_E^\times} \xi^{-1}(x) \psi\left(-\frac{\Nm(x)}{yp^{c(\pi)-i}} + \frac{\Tr(x)}{p^{c(\pi)/2}}\right)\,d^\times x,
\end{equation}
where $\zeta_\fp(1)= (1-q_E^{-1})^{-1}$ and $q_E$ is the cardinality of the residue field of $E$. 
\end{mylemma}
\begin{proof}
Let us write $c=c(\pi)$ within this proof. 
We have \begin{align*}
W^{(i)}(y)&=W\Big(\zxz{y}{}{}{1}\zxz{}{1}{1}{} \zxz{1}{p^i}{}{1}\zxz{}{1}{1}{}\Big)\\
&=W\Big(\zxz{yp^{-c}}{}{}{1}\zxz{}{1}{1}{} \zxz{1}{p^{i-c}}{}{1}\zxz{}{1}{p^c}{}\Big)\\
&=\delta_\pi W\Big(\zxz{yp^{-c}}{}{}{1}\zxz{}{1}{1}{} \zxz{1}{p^{i-c}}{}{1}\Big)=\delta_\pi W(g_{-c, c-i,-y^{-1}} ),
\end{align*}
where $\delta_\pi$ is the eigenvalue of the Atkin-Lehner involution and $g_{t, l, v}$ is as in \eqref{Assing_After_14}. Note that for $W$ the newform in the Whittaker model of a supercuspidal representation, we have $\|W\|_2= W(1) = 1$, so that Assing's normalization matches the normalization here. Now we apply the middle case of  \cite[Lem.\ 3.1]{Assing_padic_Whittaker}, noting that Assing normalizes the measure on $\cO_E$ to have total volume $p^{-d/2}$, with $n=c$ in both the $i>c/2$ and $i=c/2$ cases to get 
\begin{equation*}
W^{(i)}(y)=\delta_\pi \gamma p^{\frac{c-d}{2}}\int_{\cO_E^\times}\xi^{-1}(x)\psi_E(\varpi_E^{-ce/2} x)\psi\left(- \frac{p^{i-c}}{y}\Nm(x)\right)dx. 
\end{equation*}

Converting additive to multiplicative measure yields the result.
\end{proof}
Applying Lemma \ref{Lem:Wi} to \eqref{eq:beforeLemWi} we get  
\begin{equation}
   \frac{f_{\xi}(1)}{(1-p^{-1})}  \frac{\delta_\sigma^2 \overline{\gamma} p^{\frac{c(\sigma)-d}{2}}}{\zeta_\fp(1)}\int\limits_{v(t)= -i\geq -k}   \psi(t) \int_{\cO_E^\times}  \xi(x) \psi\left( - \frac{t}{m} \Nm(x) -\Tr(xp^{-c(\sigma)/2})\right) \,d^\times x\, dt.
\end{equation}
Using orthogonality of additive characters to execute the $t$-integral, we obtain 
\begin{multline}
 H_p(m,1;p^k) = (1-p^{-1})^{-1} f_{\xi}(1)\frac{\overline{\gamma} }{\zeta_\fp(1)}p^{c_0+k} \int\limits_{\substack{x \in \cO_E^\times \\ \Nm(x) = m \shortmod{p^k}}}  \xi(x) \psi(-\Tr(x)p^{-k}) \,d^\times x \\
= \overline{\gamma} (1-p^{-1})^{-1} f_{\xi}(1)p^{-\frac{d}{2}}\sum_{\substack{x \in (\cO_E/p^{k} \cO_E)^\times \\  \Nm(x)=m \shortmod{p^{k}}}} \xi(x) \psi(-p^{-k}\Tr(x))
\end{multline}
after converting from multiplicative to additive measure. 
\end{proof}
Under the same hypotheses as Theorem \ref{KlSumThm}, we have that the supercuspidal Kloosterman sums degenerate into classical Kloosterman sums for $k\geq c(\sigma)$.  
 \begin{myprop}\label{Hpmnpk_is_classicalKl}
For $k\geq c(\sigma)$ and $(m,n,p)=1$ we have $$H_p(m,n,p^k) = f_{\xi}(1) \zeta_p(1) S(m,n,p^k).$$ 
\end{myprop}
\begin{proof}
We use the expression in Lemma \ref{drop_supp_cond_in_Klsum_computation} for $H(m,n;p^k)$, that is
$$H_p(m,n;p^k) =f_{\xi}(1) \iint\limits_{\substack{v(t_2) = -k \\  v(t_1) = -k}} \overline{\Phi}\left( n(t_1)^{-1} \left( \begin{smallmatrix} & -p^{-2k} \\ 1 & \end{smallmatrix}\right) n(t_2)\right) \psi(-mt_1 + nt_2)\,dt_1\,dt_2.$$
Now, using the matrix decomposition in \eqref{nt1deltant2_decomposition}, we get that this is equal to
$$ f_{\xi}(1) \iint\limits_{\substack{v(t_2) = -k \\  v(t_1) = -k}} \overline{\Phi}\left( \left( \begin{smallmatrix} (p^kt_2)^{-1} & \frac{-p^{-2k}-t_1t_2}{t_2} \\  & 1 \end{smallmatrix}\right) \right) \psi(-mt_1 + nt_2)\,dt_1\,dt_2.$$
By Proposition 3.1 of \cite{Hu:17a} this is 
\begin{equation}\label{ScKlsumdegenerates_eq1} f_{\xi}(1) \iint\limits_{\substack{v(t_2) =v(t_1)= -k \\  v\left(  \frac{-p^{-2k}-t_1t_2}{t_2}\right)  \geq 0 }} \psi(-mt_1 + nt_2)\,dt_1\,dt_2 - \frac{f_{\xi}(1)}{p-1} \iint\limits_{\substack{v(t_2) =v(t_1)= -k \\  v\left(  \frac{-p^{-2k}-t_1t_2}{t_2}\right)  =-1 }} \psi(-mt_1 + nt_2)\,dt_1\,dt_2.\end{equation}
The first of the two terms in \eqref{ScKlsumdegenerates_eq1} equals 
$$ f_{\xi}(1) \sum_{\substack{t_1,t_2 \in (\Z/p^k\Z)^\times \\ v(-1-t_1t_2)\geq k}} \psi\left(\frac{-mt_1+n t_2}{p^k}\right) =  f_{\xi}(1)S(m,n,p^k).$$ The second term (including the minus sign) in \eqref{ScKlsumdegenerates_eq1} equals
$$ - \frac{f_{\xi}(1)}{p-1} \iint\limits_{\substack{v(t_2) =v(t_1)= -k \\  v\left(  \frac{-p^{-2k}-t_1t_2}{t_2}\right)  \geq-1 }} \psi(-mt_1 + nt_2)\,dt_1\,dt_2 +  \frac{f_{\xi}(1)}{p-1} \iint\limits_{\substack{v(t_2) =v(t_1)= -k \\  v\left(  \frac{-p^{-2k}-t_1t_2}{t_2}\right)  \geq 0 }} \psi(-mt_1 + nt_2)\,dt_1\,dt_2.$$
The second of these again equals $ \frac{f_{\xi}(1)}{p-1}S(m,n;p^k)$, while the first equals  
 $$- \frac{f_{\xi}(1)}{p-1} \sum_{\substack{t_1,t_2 \in  (\Z/p^k\Z)^\times \\ v(-1-t_1t_2)\geq k-1}} \psi\left(\frac{-mt_1+n t_2}{p^k}\right)= - \frac{f_{\xi}(1)}{p-1}\sum_{\substack{t_1,t_2 \in  (\Z/p^k\Z)^\times \\ t_1t_2 \equiv 1 \shortmod{p^{k-1}}}} \psi\left(\frac{mt_1+n t_2}{p^k}\right).$$ 
 Since $k\geq 2$, writing $t_i=t_{i,0}+p^{k-1}t_{i,1}$, this is 
 $$ - \frac{f_{\xi}(1)}{p-1} \sum_{\substack{t_{1,0},t_{2,0} \in  (\Z/p^{k-1}\Z)^\times \\ t_1t_2 \equiv 1 \shortmod{p^{k-1}}}} \psi\left(\frac{mt_{1,0}+n t_{2,0}}{p^{k-1}}\right)\sum_{t_{1,1},t_{2,1} \in \Z/p\Z }  \psi\left(\frac{mt_{1,1}+n t_{2,1}}{p}\right)=0,$$
 since $(m,n,p)=1$. 
 \end{proof}

\subsection{$p$-adic stationary phase}\label{sec:intermediatefamily}
 Let $\alpha_0$ be a normalized minimal element as in Definition \ref{normminlelt}, which we moreover assume to have $\Tr (\alpha_0) = 0$ when $p$ is odd and to be given by Lemma \ref{minl_elts_in_Q2extns} when $p=2$. Let $D = (\Tr (\alpha_0))^2 -4 \Nm(\alpha_0)$ and $d=v(D)$. Recall $c_0$ from \eqref{c0def}.   
\begin{mylemma}\label{statphase_sc_Klsums_nohyponp}
Suppose $E/\Q_p$ is a quadratic extension and $\xi$ is a character of $E^\times$ such that $\xi \neq \overline{\xi}$ and $\xi \vert_{\Q_p^\times} = \eta_{E/\Q_p}$ is the unique non-trivial quadratic character of $\Q_p^\times$ that is trivial on $\Nm(E^\times) \subset \Q_p^\times$. If $p=2$ suppose moreover that $c_0\geq 2$.  Suppose $k\geq \max(c_0+\lceil d/2\rceil ,2)$. Let $u_0 \in \cO_E^\times$ with $\Nm(u_0) =m \pmod{p^k}$ and write $u_0=a+b\alpha_0$ with $a,b \in \cO$. 
The integral $R_{k,\xi}(u_0)$ given by $$ R_{k,\xi}(u_0)= \int_{\substack{v_E(u') \geq ek/2 \\ \Nm(u_0+u')\equiv m\shortmod{p^k}}} \xi(1+\frac{u'}{u_0}) \psi_E(-\frac{u'}{p^k})\,du'$$ 
vanishes if $p=2$, $d=0$ and $v(a)>0$. Suppose now that $v(a)=0$ if $p=2$ and $d=0$. 
\begin{enumerate} 
\item  If $v(2b\Nm(\alpha_0) + a \Tr(\alpha_0))<\lfloor\frac{k+(e-1)}{2}\rfloor$, then $$R_{k,\xi}(u_0)= p^{-\lceil \frac{3k-d}{2}\rceil } \delta\left( b D\equiv 2 \Tr \alpha_0 \alpha_\xi p^k \pmod{p^{\lfloor \frac{k+d}{2}\rfloor}}\right),$$ and
\item if $v(2b\Nm(\alpha_0) + a \Tr(\alpha_0))\geq \lfloor\frac{k+(e-1)}{2}\rfloor$, then  $$R_{k,\xi}(u_0)= p^{-\lceil \frac{3k-d}{2}\rceil } \delta\left(\lceil \frac{k-(e-1)}{2}\rceil \geq c_0 \right).$$ 
\end{enumerate}
\end{mylemma}
\begin{proof}
Write $u'= a' + b' \alpha_0$, $a', b' \in \cO$. Since $v_E(u')\geq ek/2$, we have (Lemma \ref{normminlelt_lemma}) that $$\min(ev(a'),ev(b')+e-1)\geq ek/2, \quad \text{ i.e. } \quad v(a')\geq k/2, \text{ and } v(b') \geq \frac{k-(e-1)}{2}.$$
Thus,  
\begin{align*} \Nm(u) & = (a+a')^2 + \Tr(\alpha_0)(a+a')(b+b') + \Nm(\alpha_0) (b+b')^2 \\
& \equiv m+ 2 aa' + 2 \Nm(\alpha_0) bb' + \Tr(\alpha_0) (ab'+ba') \pmod{p^k}. \end{align*}
So, the condition $\Nm(u_0+u')=m\pmod{p^k}$ on the integration is equivalent to $$(2a + \Tr (\alpha_0) b) a' + (2b\Nm(\alpha_0)+\Tr(\alpha_0) a)b' \equiv 0 \pmod{p^k}.$$ Set $\widetilde{a}=2a + \Tr (\alpha_0) b$ and $\widetilde{b} = 2b\Nm(\alpha_0)+\Tr(\alpha_0) a$. 
Since $$v_E(\frac{u'}{u_0}) \geq \frac{ek}{2} \geq \frac{e}{2} (c_0+\lceil \frac{d}{2}\rceil) \geq \frac{c(\xi)}{2},$$
we have for any $\alpha_\xi \in E$ with $v_E(\alpha_\xi) = -c(\xi)+c(\psi_E)$ corresponding to $\xi$ by the Postnikov Lemma \ref{postnikov} that 
\begin{equation}\label{firstdisplay_statphaselem}
R_{k,\xi}(u_0)=\int_{v(a')\geq k/2} \int_{\substack{v(b') \geq (k-(e-1))/2 \\ \widetilde{a} a' + \widetilde{b}b' \equiv 0 \shortmod{p^k}}} \psi_E\Big(\Big( \frac{\alpha_\xi}{u_0} - \frac{1}{p^k}\Big) (a'+b' \alpha_0)\Big) \,da'\,db'.
\end{equation}
We have $$\psi_E(-(a'+b' \alpha_0)p^{-k})= \psi(-2p^{-k}a')\psi(-\Tr(\alpha_0)p^{-k}b')$$ and 
$$ \frac{\alpha_\xi}{u_0} (a'+b' \alpha_0) = \alpha_\xi \frac{aa' + \alpha_0 a b' + \overline{\alpha_0}ba' + \Nm(\alpha_0)bb'}{\Nm(u_0)}.$$  Since  $\xi(x)$ is trivial on norms from $E^\times$  we have $\xi(\overline{x})=\xi(x)^{-1}$ and thus $\Tr (\alpha_\xi)=0$. So,
\begin{equation}\label{statphaselem_tr_computation}\Tr \left( \frac{\alpha_\xi}{u_0}  (a'+b' \alpha_0) \right) = \Tr(\alpha_\xi \alpha_0) \frac{ab' -ba'}{\Nm (u_0)}.\end{equation} 
Thus, the integral in \eqref{firstdisplay_statphaselem} is equal to
\begin{multline}\label{intermediate_family_statphase_beforecases_p=2}
\int_{v(a')\geq k/2}\psi\Big(-\frac{2a'}{p^{k}}\Big) \psi\Big(\Tr(\alpha_\xi \alpha_0) \frac{ -ba' }{\Nm(u_0)}\Big)\\ \times \int_{\substack{v(b') \geq (k-(e-1))/2 \\ \widetilde{a} a' + \widetilde{b}b' \equiv 0 \shortmod{p^k}}}
\psi\Big(-\Tr(\alpha_0)\frac{b'}{p^k}\Big)
\psi(\Tr\Big(\alpha_\xi \alpha_0) 
\frac{a b' }{\Nm(u_0)}\Big) \,da'\,db'.
\end{multline}
  Note that ${v(}\Tr(\alpha_0\alpha_\xi){)} = -c_0$  
  by a case check using e.g.\ \cite[41.2 Prop.]{BushnellHenniart:06a} when $p=2$. 
  
  We now restrict to case (1), i.e.\ we have the hypothesis that $v(\widetilde{b})<\lfloor \frac{k-(e-1)}{2}\rfloor$.  We split the $a'$ integral into two ranges: $v(a') \geq k-v(\widetilde{a})$ and $k/2 \leq v(a')< k-v(\widetilde{a})$. 
  Consider the first one:
 \begin{multline}\label{intermediate_family_statphase_beforecases_p=2-2}
\int_{v(a')\geq \max(k-v(\widetilde{a}),k/2)}
\psi\Big(-\frac{2a'}{p^{k}}\Big) 
\psi\Big(\Tr(\alpha_\xi \alpha_0) \frac{ -ba' }{\Nm(u_0)}\Big)\\ \times \int_{\substack{v(b') \geq (k-(e-1))/2 \\ \widetilde{a} a' + \widetilde{b}b' \equiv 0 \shortmod{p^k}}}
\psi\Big(-\Tr(\alpha_0)\frac{b'}{p^k}\Big)\psi\Big(\Tr(\alpha_\xi \alpha_0) \frac{a b' }{\Nm(u_0)}\Big) \,da'\,db'.
\end{multline} 
In this case, the congruence $\widetilde{a}a' + \widetilde{b}b' \equiv 0 \pmod{p^k}$ is equivalent to $v(b')\geq k-v(\widetilde{b})$. 
The integral becomes
\begin{multline}\label{intermediate_family_statphase_beforecases_p=2-3}
\int_{v(a')\geq \max(k-v(\widetilde{a}),k/2)}\psi\Big(-\frac{2a'}{p^{k}}\Big) \psi\Big(\Tr(\alpha_\xi \alpha_0) \frac{ -ba' }{\Nm(u_0)}\Big) \,da'\\ \times \int_{v(b') \geq k-v(\widetilde{b})}
\psi\Big(-\Tr(\alpha_0)\frac{b'}{p^k})\psi(\Tr(\alpha_\xi \alpha_0) \frac{a b' }{\Nm(u_0)}\Big) \,db'.
\end{multline} 
 The integral in $b'$ is $$p^{-(k-v(\widetilde{b}))}\delta\left(\frac{a \Tr(\alpha_0\alpha_\xi)}{\Nm(u_0)} \equiv \frac{\Tr \alpha_0}{p^k} \pmod{p^{-(k-v(\widetilde{b}))}}\right).$$ 
 
Now consider the other part of the $a'$ integral, i.e.\ 
 \begin{multline}\label{intermediate_family_statphase_beforecases_p=2-4}
\int_{k/2 \leq v(a')< k-v(\widetilde{a})}\psi(-\frac{2a'}{p^{k}}) \psi(\Tr(\alpha_\xi \alpha_0) \frac{ -ba' }{\Nm(u_0)})\\ \times 
\int_{\substack{v(b') \geq (k-(e-1))/2 \\ \widetilde{a} a' + \widetilde{b}b' \equiv 0 \shortmod{p^k}}}
\psi\Big(-\Tr(\alpha_0)\frac{b'}{p^k}\Big)
\psi\Big(\Tr(\alpha_\xi \alpha_0) \frac{a b' }{\Nm(u_0)}\Big) \,da'\,db'.
\end{multline} 
 In this case, the congruence $ \widetilde{a} a' + \widetilde{b}b' \equiv 0 \pmod{p^k}$ implies the condition $v(\widetilde{a} a') = v(\widetilde{b}b')$, in the presence of which the condition $v(b') \geq (k-(e-1))/2$ is equivalent to $$ v(a') \geq \frac{k}{2} + \max(v(\widetilde{b})-v(\widetilde{a})-\frac{e-1}{2},0).$$  
We consider $a'$ to be a fixed variable and write the congruence condition for $b'$ as $$b'= -\frac{\widetilde{a} a'}{\widetilde{b}} + p^{k-v(\widetilde{b})}x',$$ where $x' \in \cO_F$.  The result of these transformations is that the integral in \eqref{intermediate_family_statphase_beforecases_p=2-4} is
\begin{multline}\label{intermediate_family_statphase_beforecases_p=2-5}
p^{-(k-v(\widetilde{b}))}
\int_{\substack{v(a')< k-v(\widetilde{a}) \\ v(a')\geq \frac{k}{2} + \max(v(\widetilde{b})-v(\widetilde{a})-\frac{e-1}{2},0)}}
\psi\Big(-\frac{2a'}{p^{k}}\Big) 
\psi\Big(\Tr(\alpha_\xi \alpha_0) \frac{ -ba' }{\Nm(u_0)}\Big)
\psi\Big(-\Tr(\alpha_0)\frac{-\widetilde{a} a' }{\widetilde{b}p^k}\Big)\\ \times 
\psi\Big(\Tr(\alpha_\xi \alpha_0) \frac{a (-\widetilde{a} a' ) }{\widetilde{b}\Nm(u_0)}\Big)
\,da' 
\int_{\cO}
\psi\Big(-\Tr(\alpha_0)\frac{ p^{k-v(\widetilde{b})}x'}{p^k}\Big)
\psi\Big(\Tr(\alpha_\xi \alpha_0) 
\frac{a ( p^{k-v(\widetilde{b})}x') }{\Nm(u_0)}\Big) \,dx'.
\end{multline} 
The integral in $x'$ is $$\delta\left(\frac{a \Tr(\alpha_0\alpha_\xi)}{\Nm(u_0)} \equiv \frac{\Tr \alpha_0}{p^k} \pmod{p^{-(k-v(\widetilde{b}))}}\right).$$
Putting the cases back together, we have that $R_{k,\xi}(u_0)$ is equal to \begin{equation}\label{da_int_in_2_pieces-2-condition}p^{-(k-v(\widetilde{b}))}\delta\left(\frac{a \Tr(\alpha_0\alpha_\xi)}{\Nm(u_0)} \equiv \frac{\Tr \alpha_0}{p^k} \pmod{p^{-(k-v(\widetilde{b}))}}\right)\end{equation} times
\begin{multline}\label{da_int_in_2_pieces-2} 
\int_{v(a')\geq \max(k-v(\widetilde{a}),k/2)}
\psi\Big(-\frac{2a'}{p^{k}}\Big) 
\psi\Big(\Tr(\alpha_\xi \alpha_0) \frac{ -ba' }{\Nm(u_0)}\Big)
 \\ +  
 \int_{\substack{v(a')< k-v(\widetilde{a}) \\ v(a')\geq \frac{k}{2} + \max(v(\widetilde{b})-v(\widetilde{a})-\frac{e-1}{2},0)}}
\psi\Big(-\frac{2a'}{p^{k}}\Big) 
\psi\Big(\Tr(\alpha_\xi \alpha_0) \frac{ -ba' }{\Nm(u_0)}\Big)
\psi\Big(-\Tr(\alpha_0)\frac{-\widetilde{a} a' }{\widetilde{b}p^k}\Big)\\ \times 
\psi\Big(\Tr(\alpha_\xi \alpha_0) 
\frac{a (-\widetilde{a} a' ) }{\widetilde{b}\Nm(u_0)}\Big)
\,da'.
\end{multline}
 Under the condition in \eqref{da_int_in_2_pieces-2-condition} we can combine the integrals in \eqref{da_int_in_2_pieces-2} as
 \begin{multline*} 
 \int_{ v(a')\geq \frac{k}{2} + \max(v(\widetilde{b})-v(\widetilde{a})-\frac{e-1}{2},0)}
 \psi\Big(-\frac{2a'}{p^{k}}\Big) 
 \psi\Big(\Tr(\alpha_\xi \alpha_0) \frac{ -ba' }{\Nm(u_0)}\Big) \,da'
 \\ 
 \times 
 \psi\Big(-\Tr(\alpha_0)\frac{-\widetilde{a} a' }{\widetilde{b}p^k}\Big) 
 \psi\Big(\Tr(\alpha_\xi \alpha_0) \frac{a (-\widetilde{a} a' ) }{\widetilde{b}\Nm(u_0)}\Big)
 \,da'.\end{multline*}
Note that 
$$b+ \frac{a\widetilde{a}}{\widetilde{b}} = \frac{2\Nm(u_0)}{\widetilde{b}},$$
so $R_{k,\xi}(u_0)$ is equal to the expression in \eqref{da_int_in_2_pieces-2-condition} times
\begin{equation}
\label{intermediate_family_statphase_beforecases_p=2-5.5} 
\int_{ v(a')\geq \frac{k}{2} + \max(v(\widetilde{b})-v(\widetilde{a})-\frac{e-1}{2},0)}
\psi\Big(-\frac{2a'}{p^{k}}\Big)
\psi\Big(\Tr(\alpha_0)\frac{\widetilde{a} a' }{\widetilde{b}p^k}\Big) 
\psi\Big(-2\Tr(\alpha_\xi \alpha_0) \frac{a' }{\widetilde{b}}\Big) 
\,da'.
\end{equation}
Note that 
$$ \frac{-2}{p^k} + \frac{\Tr (\alpha_0) \widetilde{a}}{\widetilde{b} p^k} = \frac{1}{\widetilde{b}p^k} ((\Tr \alpha_0)^2 - 4 \Nm(\alpha_0))b = \frac{bD}{\widetilde{b}p^k},$$
so that the integral in \eqref{intermediate_family_statphase_beforecases_p=2-5.5} equals
$$p^{-\lceil \frac{k}{2} + \max(v(\widetilde{b})-v(\widetilde{a})-\frac{e-1}{2},0)\rceil} \delta\left( \frac{1}{\widetilde{b}p^k} (b D - 2 \Tr \alpha_0 \alpha_\xi p^k)\equiv 0 \pmod{ p^{-\lceil\frac{k}{2} + \max(v(\widetilde{b})-v(\widetilde{a})-\frac{e-1}{2},0)\rceil}}\right).$$
We have the following table of cases.
\begin{table}[h!]
\centering
\begin{tabular}{|c|c|c|c|}
\hline
Case & $v(\widetilde{a})$ & $v(\widetilde{b})$ & $v(\widetilde{b})-v(\widetilde{a})-\frac{e-1}{2}$ \\
\hline \hline
$p = 2,$ $d=0, v(a)=0$ & $\geq 0$ & $=0$ & $\leq 0$ \\
\hline
$p=2$, $ d=0, v(a)>0$ & $=0$ & $\geq 1$ & $v(\widetilde{b})$ \\
\hline
$d=3$ & $=1$ & $=v(b)+2$ & $v(\widetilde{b})-3/2$ \\
\hline
$d=2$ & $\geq 1$ & $=1$ & $\leq 0$\\
\hline
$p \neq 2$, $d=0$, $v(a)=0$ & $=0$ & $ = v(b)$ & $v(\widetilde{b})$\\
\hline
$p \neq 2$, $d=0$, $v(a)>0$ & $\geq 1$ & $=0$ & $\leq 0$\\
\hline
$d=1$ & $=0$ & $=v(b)+1$ & $v(\widetilde{b})-1/2$\\
\hline
\end{tabular}
\caption{}
\label{table:a'b'cases}
\end{table}

Note that we have uniformly that
$$ - \max(v(\widetilde{b})-v(\widetilde{a})-\frac{e-1}{2},0) +v(\widetilde{b}) = \frac{d}{2}.$$
Collecting these computations,

\begin{equation}\label{intermediate_family_statphase_beforecases_p=2-6} R_{k,\xi}(u_0)= p^{-\lceil \frac{3k-d}{2}\rceil } \delta( b D\equiv 2 \Tr \alpha_0 \alpha_\xi p^k \pmod{p^{\lfloor \frac{k+d}{2}\rfloor}}) \delta(\frac{a \Tr( \alpha_0 \alpha_\xi)}{\Nm u_0 } \equiv \frac{\Tr \alpha_0}{p^k}\pmod {p^{-k+v(\widetilde{b})}}).\end{equation}
Note that $\frac{\Tr \alpha_0}{p^k} \equiv 0 \pmod {p^{-k+v(\widetilde{b})}})$ in every case except $p=2$,$d=0$, $v(a)>0$. However, in that exceptional case $v(\frac{a \Tr( \alpha_0 \alpha_\xi)}{\Nm u_0 })>v(\frac{\Tr \alpha_0}{p^k})$, so that the latter congruence of \eqref{intermediate_family_statphase_beforecases_p=2-6} can never be satisfied. Thus $R_{k,\xi}(u_0)$ is identically 0 if $p=2$, $d=0$ and $v(a)>0$. 

Excluding now the case $p=2$,$d=0$, $v(a)>0$, the expression in \eqref{intermediate_family_statphase_beforecases_p=2-6} simplifies to  
$$ p^{-\lceil \frac{3k-d}{2}\rceil } \delta( b D\equiv 2 \Tr \alpha_0 \alpha_\xi p^k \pmod{p^{\lfloor \frac{k+d}{2}\rfloor}} ) \delta(v(a) - c_0 \geq -k+v(\widetilde{b})).$$

If $p=2$, $d=0$, $v(a)=0$, the condition $\delta(v(a) - c_0 \geq -k+v(\widetilde{b}))$ is trivially satisfied by the hypothesis $k \geq c_0  + \lceil d/2 \rceil $ of the proposition. 

Now excluding the unramified $p=2$ case, we have that the congruence condition $b D\equiv 2 \Tr \alpha_0 \alpha_\xi p^k \pmod{p^{\lfloor \frac{k+d}{2}\rfloor}}$ implies that $v(a) - c_0 \geq -k+v(\widetilde{b})$, so in fact the latter condition can be omitted. The result is: if $v(\widetilde{b})<\lfloor \frac{k+(e-1)}{2}\rfloor$ then 
$$R_{k,\xi}(u_0)= \begin{cases} 0 & \text{ if } p=2, d=0, v(a)>0, \\  p^{-\lceil \frac{3k-d}{2}\rceil } \delta( b D\equiv 2 \Tr \alpha_0 \alpha_\xi p^k \pmod{p^{\lfloor \frac{k+d}{2}\rfloor}}) & \text{ otherwise. }\end{cases}$$

Now consider case (2), i.e.\ that $v(\widetilde{b})\geq \lfloor \frac{k+(e-1)}{2} \rfloor$. We pick up the calculation at \eqref{intermediate_family_statphase_beforecases_p=2}. In this case, the congruence condition $\widetilde{a}a' + \widetilde{b}b' \equiv 0 \pmod{p^k}$ becomes just $v(a')\geq k-v(\widetilde{a})$, so the two integrals separate, i.e.\ we have that 
\begin{multline*} 
R_{k,\xi}(u_0)= \int_{v(a')\geq k-v(\widetilde{a})}
\psi\Big(-\frac{2a'}{p^{k}}\Big) 
\psi\Big(\Tr(\alpha_\xi \alpha_0) \frac{ -ba' }{\Nm(u_0)}\Big) 
\,da' 
\\ \times  
\int_{v(b') \geq (k-(e-1))/2 }
\psi\Big(-\Tr(\alpha_0)\frac{b'}{p^k}\Big)
\psi\Big(\Tr(\alpha_\xi \alpha_0) \frac{a b' }{\Nm(u_0)}\Big)
\,db'.
\end{multline*}

First assume that $p=2$, $d=0$, and $v(a)>0$. In this case we have  $v(\frac{a \Tr( \alpha_0 \alpha_\xi)}{\Nm u_0 })>v(\frac{\Tr \alpha_0}{p^k})$, so that the $b'$ integral vanishes for all $k, \xi, u_0$. 

Now, excluding this case, it remains to consider only the cases $d=3$, ($p \neq2$, $d=0$, $v(a)=0$), and ($p \neq 2$, $d=1$), since only these cases may have $v(\widetilde{b}) \geq 1$, and by assumption $v(\widetilde{b}) \geq \lfloor \frac{k+(e-1)}{2} \rfloor$. (Note, the case $d=2$, $c_0=1$, $k=2$ is excluded by the hypothesis that $c_0\geq 2$  when $p=2$.)  
All of these cases conveniently have $\Tr \alpha_0=0$ and $v(\widetilde{a})=v(2)$, 
so
\begin{multline*} 
R_{k,\xi}(u_0)= \int_{v(a')\geq k- v(2)}
\psi\Big(-\frac{2a'}{p^{k}}\Big) 
\psi\Big(\Tr(\alpha_\xi \alpha_0) \frac{ -ba' }{\Nm(u_0)}\Big) 
\,da'
\\ 
\times  \int_{v(b') \geq (k-(e-1))/2 }
\psi\Big(\Tr(\alpha_\xi \alpha_0) \frac{a b' }{\Nm(u_0)}\Big)
\,db'.
\end{multline*}
Note that $\psi(-\frac{2a'}{p^{k}}) =1$, and since $$k-c_0 \geq \lceil \frac{d}{2} \rceil \geq v(2) \geq v(2)-v(b),$$ we have $\psi(\Tr(\alpha_\xi \alpha_0) \frac{ -ba' }{\Nm(u_0)})=1$ as well. Thus the $a'$ integral is equal to $p^{-k+v(2)}$. Meanwhile, the $b'$ integral equals 
$$ p^{-\lceil\frac{k-(e-1)}{2}\rceil}\delta( \lceil\frac{k-(e-1)}{2}\rceil \geq c_0).$$  So, the whole integral, under the hypothesis $v(\widetilde{b})\geq \lfloor \frac{k+(e-1)}{2} \rfloor$ is 
$$\begin{cases} 0 & \text{ if } p=2, d=0, v(a)>0, \\ p^{-\lceil\frac{3k-d}{2}\rceil} \delta\left(\lceil \frac{ k-(e-1)}{2}\rceil \geq c_0 \right) & \text{ otherwise. }\end{cases}$$
\end{proof}
   \begin{mylemma}\label{OFxU_En_lem}
Let $n\geq 1$. Then $y+z\alpha_0 \in \cO_F^\times U_E(n)$ if and only if $v(y)=0$ and $v_E(z\alpha_0)\geq n$. 
 \end{mylemma}
 \begin{proof}
 If: Write $y+z\alpha_0 = y(1+\frac{z}{y}\alpha_0)$, which we are allowed since $v(y)=0$. Then $v_E(\frac{z}{y}\alpha_0) = v_E(z\alpha_0)\geq n$ by assumption. So $y+z\alpha_0 \in \cO_F^\times U_E(n)$. Only if: by hypothesis there exists $s,a,b \in \cO_F$ with $v(s)=0$, $v_E(a-1+b\alpha_0)\geq n$ and $y+z\alpha_0 = s(a+b\alpha_0)$. By Lemma \ref{normminlelt_lemma}, we have $\min(v_E(a-1),v_E(b\alpha_0))\geq n$. From these it follows that $v(a-1)\geq n/e$, $v_E(b\alpha_0)\geq n$, $y=sa$, and $z\alpha_0= sb\alpha_0$. Since $n\geq 1$, we have $v(a)=0$. Thus, $v(y)=0$ and $v_E(z\alpha_0)=v_E(b\alpha_0)\geq n$.
\end{proof}

Set \begin{equation}\label{I_def_intermediate_section} I_{\xi}(m,p^k) =  \sum_{\substack{u \in (\cO_E / p^k \cO_E)^\times \\ \Nm(u) \equiv m \shortmod{p^k}}}\xi(u) \psi(-\Tr(u)p^{-k}),\end{equation}
so that the supercuspidal Kloosterman sum $H(m,1,p^k)$ associated to $\Ind_E^F \xi$ is equal to $\delta_p \overline{\gamma}p^{-d/2} I_\xi(m,p^k)$ if $(m,p)=1$ and $k\geq c_0+\lceil d/2\rceil$ and $0$ otherwise, see Theorem \ref{KlSumThm}. 

Let $\xi$ be a character of $E^\times$, and for $0\leq n \leq c(\xi)$, recall \eqref{xi_n_def} the neighborhood $\xi[n]$ of characters around $\xi$, and for $0\leq i \leq n$ the equivalence relation $\sim_i$ on $\xi[n]$. 
\begin{myprop}\label{intermediate_family_mainprop}
 Set $i=1$ if the $E$ on which $\xi$ is defined is the unramified quadratic extension of $\Q_2$ and $i=0$ otherwise. Suppose $i\leq n < c(\xi)$ and $k\geq 2$.  If $k\geq c_0+\lceil d/2 \rceil -i + \lfloor \frac{n}{e}\rfloor$, then
\begin{equation}\label{intermediate_family_goal_pgeneral}\frac{1}{[\xi[n]:\xi[i]]} \sum_{\xi_1 \in \xi[n]/\sim_i}  I_{\xi_1}(m,p^k) 
= 
I_{\xi}(m,p^k) .  
\end{equation}
\end{myprop}
\begin{proof}
We have for any $0 \leq i \leq n\leq c(\xi)$
  $$\xi[n]/\sim_i = \{ \xi_1 \in (U_E(i))^\wedge : c(\xi_1 \xi^{-1})\leq n, \, \xi_1 \vert_{\cO_F^\times} = \xi \vert_{\cO_F^\times}\}=\xi \{\theta \in  (U_E(i))^\wedge : c(\theta) \leq n,\, \theta \vert_{\cO_F^\times}=1\}.$$ 
So, for $u\in U_E(i)$, we have
 \begin{equation}\label{intermediate_family_ortho_rln}  \frac{1}{[\xi[n]:\xi[i]]} \sum_{\xi_1 \in \xi[n]/\sim_i} \xi_1(u) = \frac{1}{[1[n]:1[i]]} \sum_{\theta \in 1[n]/1[i]} \xi(u) \theta(u) = \xi(u) \delta_{u \in \cO_F^\times U_E(n)}.\end{equation}
    We get for $i$ as in the statement of the proposition (using Remark \ref{KlSumThm_rem} when $i=1$) that 
  \begin{equation}\label{intermediate_family_goal0_pgeneral}\frac{1}{[\xi[n]:\xi[i]]} \sum_{\xi_1 \in \xi[n]/\sim_i}  I_{\xi_1}(m,p^k) = \sum_{\substack{u \in (\cO_E / p^k \cO_E)^\times \\ \Nm(u) \equiv m \shortmod{p^k}\\ u \in \cO_F^\times U_E(n)}} \xi(u)\psi(-\Tr(u)p^{-k}).\end{equation} So, to prove the proposition, it suffices to show that the right hand side of \eqref{intermediate_family_goal0_pgeneral} is equal to the right hand side of \eqref{intermediate_family_goal_pgeneral}. 
 Note that these clearly match if $n=i$ by Theorem \ref{KlSumThm} and Remark \ref{KlSumThm_rem}, so we may freely assume that $1 \leq n < c(\xi)$ for the remainder of the proof.

First suppose that $ k\geq c_0+\lceil d/2 \rceil -i + \lfloor \frac{n}{e}\rfloor$ and work from the right hand side of \eqref{intermediate_family_goal_pgeneral}. Note that $c_0+\lceil d/2 \rceil -i + \lfloor \frac{n}{e}\rfloor= c_0+d -v(2)+ \lfloor \frac{n}{e}\rfloor.$ 
Writing $u=u_0+u'$ with $v_E(u')\geq ek/2$ and $R_{k,\xi}(u_0)$ for the integral in Lemma \ref{statphase_sc_Klsums_nohyponp}, we have
\begin{equation}\label{intermediate_family_step1}I_{\xi}(m,p^k) = p^{2k}\sum_{\substack{u_0 \in \cO_E^\times /U_E(\lceil ek/2\rceil) \\ \Nm(u_0) \equiv m \shortmod{p^k}}}  \xi(u_0)\psi_E(-u_0p^{-k})R_{k,\xi}(u_0).\end{equation} Here, and in similar situations below (e.g.\ \eqref{statphase_bounds_1}) the sum on the right hand side runs over $$\{u_0 \in \cO_E^\times /U_E(\lceil ek/2\rceil) : \exists \text{ a lift } \tilde{u}_0 \in (\cO_E/p^{k}\cO_E)^\times \text{ of } u_0 \text{ with }  \Nm(\tilde{u}_0)\equiv m \shortmod{p^k}\}.$$ Write $u_0=a+b\alpha_0$. 
We claim that $\supp (R_{k,\xi}(u_0)) \cap \cO_E^\times \subseteq \cO_F^\times U_E(n)$, so that \eqref{intermediate_family_step1} matches the right hand side of \eqref{intermediate_family_goal0_pgeneral}.

Set $\widetilde{b}= 2b\Nm(\alpha_0) + a \Tr(\alpha_0)$ as in the proof of Lemma \ref{statphase_sc_Klsums_nohyponp}. Suppose first that $v(\widetilde{b})<\lfloor \frac{k+(e-1)}{2} \rfloor$. Then, Lemma \ref{statphase_sc_Klsums_nohyponp}(1) shows that $a+b\alpha_0 \in \supp (R_{k,\xi}(u_0))$ only if $$v(b) \geq \min(\lfloor \frac{k-d}{2}\rfloor , k-c_0-d+v(2)).$$ The second of these two possibilities is $\geq \lfloor n/e\rfloor$ by the case hypothesis, while for the first we have
$$ \frac{k-d}{2} = \frac{k-d-c_0+v(2)}{2} +\frac{c_0}{2} - \frac{v(2)}{2}\geq \frac{1}{2}\lfloor \frac{n}{e}\rfloor +\frac{c_0}{2} - \frac{v(2)}{2}\geq \lfloor \frac{n}{e}\rfloor - \frac{v(2)}{2}+\frac{1}{2} \geq \lfloor \frac{n}{e}\rfloor.$$ Then, 
$$v_E(b\alpha_0) = ev(b)+(e-1) \geq e\lfloor \frac{n}{e}\rfloor +(e-1) \geq n.$$ Since $0=v_E(a+b \alpha_0) =\min(v_E(a),v_E(b\alpha_0))$, we must have $v(a)=0$. By Lemma \ref{OFxU_En_lem}, those $u_0 = a+b \alpha_0 \in \supp (R_{k,\xi}(u_0)) \cap \cO_E^\times$ with $v(\widetilde{b})<\lfloor \frac{k+(e-1)}{2} \rfloor$  lie in $\cO_F^\times U_E(n)$.  

Now suppose that $v(\widetilde{b})\geq \lfloor \frac{k+(e-1)}{2} \rfloor$. We need some casework so refer to Table \ref{table:a'b'cases}. Since $\lfloor \frac{k+(e-1)}{2} \rfloor \geq 1$, and is $\geq 2$ when $d=2$ by the hypothesis that $c(\sigma)\geq 5$ when $p=2$, we have that $\supp (R_{k, \xi}(u_0))$ is only non-empty in the cases  $d=3$, ($p \neq2$, $d=0$, $v(a)=0$), and ($p \neq 2$, $d=1$). In these cases, $\widetilde{b}$ and $b$ are related by $v(\widetilde{b})= v(b)+e-1+v(2)$. So,
$$v(b) \geq \lfloor \frac{k+(e-1)}{2} \rfloor -(e-1) -v(2)$$ and $$ k\geq c_0 + d  + \lfloor \frac{n}{e} \rfloor -v(2)\geq 2 \lfloor \frac{n}{e} \rfloor+d-v(2) + 1,$$
so that $$v(b) \geq  \lfloor \frac{n}{e} \rfloor + \lfloor \frac{d-v(2)+1+(e-1)}{2}\rfloor -(e-1) -v(2) \geq \lfloor \frac{n}{e} \rfloor.$$
Therefore, $v_E(b\alpha_0) \geq n$ and so $\supp (R_{k,\xi}(u_0)) \cap \cO_E^\times \subseteq \cO_F^\times U_E(n)$ by Lemma \ref{OFxU_En_lem}. 
\end{proof}

Now write $\xi'$ for a twist-minimal character of $E^\times$ for which there exists a character $\chi$ of $F^\times$ with $\xi = \xi' \chi_E$, following Section \ref{sec:egSupercuspidal_even}. Recall that if $p \neq 2$ or $d=3$, then we may take $\xi'=\xi$ and if $p=2$ and $d=0$ or $2$, then we have that $c(\xi')=c(\xi)-1$, see Table \ref{table:xiprime}.
\begin{myprop}\label{intermediate_family_mainprop2}
 Set $i=1$ if the $E$ on which $\xi$ is defined is the unramified quadratic extension of $\Q_2$ and $i=0$ otherwise. Suppose $i\leq n < c(\xi')$ and $k\geq 2$.  If $k< c_0+\lceil d/2 \rceil -i + \lfloor \frac{n}{e}\rfloor$, then
\begin{equation}\label{intermediate_family_goal_pgeneral2}\frac{1}{[\xi[n]:\xi[i]]} \sum_{\xi_1 \in \xi[n]/\sim_i}  I_{\xi_1}(m,p^k)  \\ 
= 
0 .\end{equation}
\end{myprop}
\begin{proof}

The conditions for $n$ can be rewritten as
	$$  e\lb k-c_0-\lceil d/2\rceil+i+1\rb \leq n<c(\xi')$$
	For such $n$ to exist, we have in view of Theorem \ref{KlSumThm},
	\begin{equation}\label{Eq:krange}
		\lceil c(\sigma)/2\rceil\leq k< c_0+\lceil d/2\rceil-i-1+c(\xi')/e
	\end{equation}
	We first reduce to the case
	\begin{equation}\label{Eq:n0}
		n=n_0:= e\lb k-c_0-\lceil d/2\rceil+i+1\rb .
	\end{equation} 
	Indeed if the result is true for $n_0$, the sum in $\xi_1$ for general $n$ can be divided into a double sum 
	$$ \frac{1}{[\xi[n]:\xi[i]]} \sum\limits_{\xi_1\in\xi[n]/\sim_i} I_{\xi_1}(m,p^k)= \frac{1}{[\xi[n]:\xi[n_0]]} \sum\limits_{\xi_0\in \xi[n]/\sim_{n_0}}\frac{1}{[\xi[n_0]:\xi[i]]}\sum\limits_{\xi_1\in \xi_0[n_0]/\sim_i} I_{\xi_1}(m,p^k)$$
	and the vanishing result for $n_0$ can be applied to get the vanishing result for larger families.
	As the proof inevitably requires case by case checking, we collect here in a table all necessary information combining parameterization of supercuspidal representations with \eqref{Eq:krange} \eqref{Eq:n0}.
	\begin{table}[h!]
		\centering
		\begin{tabular}{ |c|c|c|c|c|c|c| } 
			\hline
			Case	&	$c(\sigma)$ & $c_0$ & $c(\xi)$ & $c(\xi')$ & range of $k$ & $n_0$ \\ 
			\hline\hline
			$p=2,d=0$ & $2j+2$ & $j+1$ & $j+1$ & $j$& $j+1\leq k<2j-1$& $k-j+1$\\ 
			\hline
			$p=2,d=2$ & $2j+2$ & $j$ & $2j$ & $2j-1$& $j+1\leq k<2j$ &$2k-2j$\\ 
			\hline
			$p=2,d=3$ & $2j+1$ & $j-1$ & $2j-2$ & $2j-2$ & $j+1\leq k<2j-1$ & $2k-2j$\\ 
			\hline
			$p>2,d=0$ & $2j$ & $j$ & $j$ & $j$& $j\leq k<2j-1$& $k-j+1$\\ 
			\hline
			$p>2,d=1$ & $2j+1$ & $j$ & $2j$ & $2j$& $j+1\leq k<2j$& $2k-2j$\\ 
			\hline
		\end{tabular}
		\caption{}
		\label{table:n0k}
	\end{table} 
	
To prove the proposition, it suffices to show that the Fourier/Mellin transform 
\begin{equation}\label{FMxform_of_intermediate_family_mp_2} 
\Sigma:= \frac{1}{\varphi(p^k)} \sumstar_{m \shortmod{p^k}} \frac{1}{[\xi[n]:\xi[i]]} \sum_{\xi_1 \in \xi[n]/\sim_i}I_{\xi_1}(m,p^k) \chi(m) 
\end{equation}
of \eqref{intermediate_family_goal_pgeneral2} vanishes.  Moving the sum over $m$ to the inside, we have by Proposition \ref{MellinXform} that 
$$ \Sigma = p^k \frac{1}{[\xi[n]:\xi[i]]} \sum_{\xi_1 \in \xi[n]/\sim_i} \int_{\cO_E^\times} \xi_1\chi_E(u)\psi_E(-p^{-k}u)du.$$
The inner Gauss integral is nonvanishing only if $c(\chi_E\xi_1)=ek-d$. For interpretation of later parts of the proof, it may be helpful to note that if $k>c(\sigma)/2$, then the condition $c(\chi_E\xi_1)=ek-d$ is only attainable when $c(\chi)=k$.

	Writing $u=u_0(1+u')$ with $v_E(u') \geq \lceil \frac{ek-d}{2}\rceil$ and $u_0\in \cO_E^\times/U_E(\lceil \frac{ek-d}{2}\rceil)$, we have
	\begin{align*}
		&\int\limits_{u\in \cO_E^\times}\xi_1\chi_E(u)\psi_E(-p^{-k}u)\,du\\
		=&\sum\limits_{u_0\in\cO_E^\times/U_E(\lceil \frac{ek-d}{2}\rceil)}\xi_1\chi_E(u_0)\psi_E(-p^{-k}u_0)
		\int\limits_{v_E(u') \geq \lceil \frac{ek-d}{2}\rceil}\psi_E((\alpha_{\xi_1}+\alpha_\chi-p^{-k}u_0)u')\,du'.
	\end{align*}
	
	From this we see that the nonzero contribution to inner Gauss sum comes from 
	$u_0$ satisfying 
	$$u_0\equiv p^k(\alpha_{\xi_1}+\alpha_\chi) \pmod{ \fp_E^{\lfloor \frac{ek-d}{2}\rfloor}}.
	$$
	We claim that this congruence requirement is actually independent of $\xi_1\in \xi[n_0]$.
	(Recall that $c(\xi)=-v_E(\alpha_\xi)-d$.) Indeed for $\xi_1,\xi_2\in \xi[n_0]$, we have
	$$v_E(p^k(\alpha_{\xi_1}-\alpha_{\xi_2}))=ek+(-d-c(\xi_1^{-1}\xi_2))\geq ek-d-n_0$$
	which is $\geq \lfloor\frac{ek-d}{2}\rfloor$
	using case by case check that
	\begin{equation}
		\lceil (ek-d)/2\rceil \geq n_0.
	\end{equation}

	We can thus fix $\xi_0\in \xi[n_0]$, impose the congruence condition for $u$ in $\Sigma$ and swap the order of sum and integral, getting
	\begin{multline*}
	\Sigma= p^k\int\limits_{u\equiv p^k(\alpha_{\xi_0}+\alpha_\chi)\shortmod{ \fp_E^{\lfloor \frac{ek-d}{2}\rfloor}}}\frac{1}{[\xi[n]:\xi[i]]} \sum_{\xi_1 \in \xi[n]/\sim_i}\xi_1\chi_E(u)\psi_E(-p^{-k}u)\,du\\
	= 	p^k\int\limits_{\substack{u \in \cO_F^\times U_E(n) \\ u\equiv p^k(\alpha_{\xi_0}+\alpha_\chi)\shortmod{ \fp_E^{\lfloor \frac{ek-d}{2}\rfloor}}}}\xi\chi_E(u)\psi_E(-p^{-k}u)\,du
	\end{multline*}
	by \eqref{intermediate_family_ortho_rln}. 
	We claim now that the two conditions on the integral are disjoint, i.e.\ that for any $u$ satisfying $u\equiv p^k(\alpha_{\xi_0}+\alpha_\chi) \pmod{\fp_E^{\lfloor \frac{ek-d}{2}\rfloor}}$, we have 
	$u\notin \cO_F^\times U_E(n_0)$. 
	
	To prove the claim, we write
	$$p^k(\alpha_{\xi_0}+\alpha_\chi)
	=p^k(z(A/2+\alpha_0)+\alpha_\chi)
	=p^k(zA/2+\alpha_\chi)+p^kz\alpha_0.
	$$
	If $p$ is odd, then $A=0$ so $v_E(\alpha_{\xi_0}) = ev(z)+(e-1)$ and then $v(z) = -c(\xi_0)/e -d$ by Lemma \ref{postnikov}.   
	If $p=2$, then $v(z)$ is given by Lemma \ref{Q2lem2}. We have that $v(p^k(zA/2+\alpha_\chi))\geq 0$ and that $v_E(z\alpha_0)$ is directly related to $c(\xi')$ by Proposition \ref{prop:twistminimalQ2}, while checking case by case shows that 
		\begin{equation}\label{intermediate_family_sufficient_condition}
		v_E(p^kz\alpha_{0})=ek-d-c(\xi_0')<\lfloor \frac{ek-d}{2}\rfloor.
	\end{equation}
	The inequality in \eqref{intermediate_family_sufficient_condition} reduces the problem to checking that $p^k(\alpha_{\xi_0}+\alpha_\chi)\notin \cO_F^\times U_E(n_0)$, as anything from $\fp_E^{\lfloor \frac{ek-d}{2}\rfloor} $ does not affect the criterion in Lemma \ref{OFxU_En_lem}.
	
	Then the claim follows from Lemma \ref{OFxU_En_lem} and checking case by case that
	\begin{equation}
		v_E(p^kz\alpha_{0})=ek-d-c(\xi_0')<n_0.
	\end{equation} 
\end{proof}

 We give one last application of the $p$-adic stationary phase Lemma \ref{statphase_sc_Klsums_nohyponp}. Suppose $\sigma$ is as in Theorem \ref{KlSumThm} and $H_p(m,n,p^k)$ is the associated generalized Kloosterman sum therein. We have the following crude bound. 
 \begin{myprop}\label{statphase_bounds}
Suppose that $k\geq \max(\lceil c(\sigma)/2\rceil,2)$. We have 
 \begin{equation}\label{eqn:statphase_bounds} |H_p(m,1,p^k)| \leq 64 \zeta_p(1) f_{\xi}(1) p^{\frac{k+\fa}{2}+\lfloor \frac{1}{2} \min(v(\frac{(p^k \Tr \alpha_0 \alpha_\xi)^2}{D} +m),\lceil \frac{k}{2} \rceil)\rfloor},\end{equation}
 where $\fa = \frac{1-(-1)^{k+d}}{2}.$
 If $p\neq 2$, the leading constant 64 may be replaced by 2. 
 \end{myprop}
 \begin{myrema}
  Proposition \ref{statphase_bounds} does not exclude the possibility that $H_p(m,n,p^k)$ has worse than square-root cancellation. First of all, if $k+d$ is odd then there is an extra factor of $p^{1/2}$. It may be possible to remove this factor by working with the quadratic terms in the Postnikov formula Lemma \ref{postnikov} as in \cite[Lem.\ 12.3]{IK}, but we leave this aside. Second, if $k= c(\sigma)/2\geq 4$, $p \nmid m$, and $m \equiv -\frac{(p^k \Tr \alpha_0 \alpha_\xi)^2}{D} \pmod{p^2}$, then the bound in Proposition \ref{statphase_bounds} is worse than square-root by a factor of at least $p$. 
 \end{myrema}
 
 \begin{proof}
  Combining \eqref{eq:KlSumThm} and \eqref{intermediate_family_step1}, we have
 \begin{equation}\label{statphase_bounds_1}H_p(m,1,p^k) =  \overline{\gamma}\zeta_p(1)f_{\xi}(1) p^{2k-\frac{d}{2}} \sum_{\substack{u_0 \in \cO_E^\times /U_E(\lceil ek/2\rceil) \\ \Nm(u_0) \equiv m \shortmod{p^k}}}  \xi(u_0)\psi_E(-u_0p^{-k})R_{k,\xi}(u_0) \end{equation}
 with $u_0=a+b\alpha_0$ and $R_{k, \xi}(u_0)$ given by Lemma \ref{statphase_sc_Klsums_nohyponp}. Accordingly, split the sum on the right hand side of \eqref{statphase_bounds_1} as $L+U$ with $$L = \sum_{\substack{u_0 \in \cO_E^\times /U_E(\lceil ek/2\rceil) \\ \Nm(u_0) \equiv m \shortmod{p^k} \\ v(\widetilde{b})<\lfloor \frac{k+(e-1)}{2}\rfloor}}  \xi(u_0)\psi_E(-\frac{u_0}{p^{k}})R_{k,\xi}(u_0)\, \text{ and } \, U = \sum_{\substack{u_0 \in \cO_E^\times /U_E(\lceil ek/2\rceil) \\ \Nm(u_0) \equiv m \shortmod{p^k} \\ v(\widetilde{b})\geq \lfloor \frac{k+(e-1)}{2}\rfloor}}  \xi(u_0)\psi_E(-\frac{u_0}{p^{k}})R_{k,\xi}(u_0).$$
  
 By Lemma \ref{statphase_sc_Klsums_nohyponp}(1) we have that  
 \begin{equation*} |L |\leq p^{-\lceil \frac{3k-d}{2}\rceil} |S_L| \end{equation*}
 where $S_L$ is the set defined by 
 \begin{equation*} S_L =  \\ \{ u_0 \in \cO_E^\times /U_E(\lceil ek/2\rceil): \Nm(u_0) \equiv m \shortmod{p^{\lceil k/2 \rceil }}, \, b D\equiv 2 p^k \Tr \alpha_0 \alpha_\xi  \shortmod{p^{\lfloor \frac{k+d}{2}\rfloor}}\}.\end{equation*}
 The congruence $b D\equiv 2 p^k \Tr \alpha_0 \alpha_\xi  \pmod{p^{\lfloor \frac{k+d}{2}\rfloor}}$ determines (modulo $p^{\lceil \frac{k-(e-1)}{2}\rceil}$)
  \begin{itemize}
  \item exactly $p^\fa$ values of $b$ if $d=0$ or $1$,
  \item exactly $2$ values of $b$ if $d=2$, and 
  \item at most $4$ values of $b$ modulo if $d=3$.  
  \end{itemize}
  Next we estimate the size of the set 
  $$S_{L, b} = \{ a \in \cO/p^{\lceil k/2\rceil} \cO : \Nm(a+ b \alpha_0) \equiv m \shortmod{p^{\lceil k/2 \rceil }} \}$$
  for $b \equiv \frac{2 p^k \Tr \alpha_0 \alpha_\xi}{D} \pmod{p^{\lfloor \frac{k-d}{2}\rfloor}}$. We proceed by cases. Let us write $S(\ell,n)$ for the number of integers $x$ modulo $n$ for which $x^2 -\ell \equiv 0 \pmod{n}$.
  
 If $d=1$ or $d=0, p\neq 2$, and $k$ is even, then the congruence in $S_{L, b}$ is $$ a^2 - \frac{(p^k \Tr \alpha_0 \alpha_\xi)^2}{D}\equiv m \shortmod{p^{\lceil k/2 \rceil}}$$ 
 since $\Tr(\alpha_0)=0$ and $\Nm(\alpha_0)  p^{\lfloor \frac{k-d}{2}\rfloor}\equiv 0 \pmod{p^{\lceil k/2 \rceil }}$ in these cases. Thus, by e.g.\ \cite[Lem.\ 10]{KaplanPetrow2} we have $$|S_{L,b}|\leq S(\frac{(p^k \Tr \alpha_0 \alpha_\xi)^2}{D} +m,p^{\lceil \frac{k}{2}\rceil})\leq 2p^{\lfloor \frac{1}{2} \min(v(\frac{(p^k \Tr \alpha_0 \alpha_\xi)^2}{D} +m),\lceil \frac{k}{2} \rceil)\rfloor}.$$

  Now consider the case that $p=2$. We may complete the square to find that any $a \in S_{L, b}$ satisfies 
  $$ \left(a+ \frac{p^k \Tr \alpha_0 \alpha_\xi}{D} \Tr \alpha_0\right)^2- \frac{(p^k \Tr \alpha_0 \alpha_\xi)^2}{D} -m \equiv 0 \shortmod{p^{\lfloor \frac{k-d}{2}\rfloor}}.$$
  Thus, $|S_{L,b}|\leq S(\frac{(p^k \Tr \alpha_0 \alpha_\xi)^2}{D} +m,p^{\lfloor \frac{k-d}{2}\rfloor})$. Since $\lceil \frac{k}{2}\rceil- \lfloor \frac{k-d}{2}\rfloor \leq 2$, we have by e.g.\ \cite[Lem.\ 10]{KaplanPetrow2} that $$S(\frac{(p^k \Tr \alpha_0 \alpha_\xi)^2}{D} +m,p^{\lfloor \frac{k-d}{2}\rfloor}) \leq 4S(\frac{(p^k \Tr \alpha_0 \alpha_\xi)^2}{D} +m,p^{\lceil \frac{k}{2}\rceil}) \leq 16p^{\lfloor \frac{1}{2} \min(v(\frac{(p^k \Tr \alpha_0 \alpha_\xi)^2}{D} +m),\lceil \frac{k}{2} \rceil)\rfloor}.$$

  Lastly, let us consider the case that $d=0$, $p \neq 2$, and $k$ is odd. Writing $b_0 = \frac{2 p^k \Tr \alpha_0 \alpha_\xi}{D}$, we parametrize the possible values of $b$ by $b=b_0+ xp^{\lfloor k/2\rfloor}$, where $x$ runs modulo $p$. Then, 
  $$|S_{L,b}| = S(m-b_0^2\Nm(\alpha_0)-2xb_0p^{\lfloor k/2 \rfloor}\Nm(\alpha_0), p^{\lceil k/2 \rceil}).$$ If $v(b_0)>0$, then $b_0p^{\lfloor k/2 \rfloor} \equiv 0 \pmod{p^{\lceil k/2 \rceil}}$ and we have
  $$|S_{L,b}|\leq 2 p^{\lfloor \frac{1}{2} \min(v(m-b_0^2\Nm(\alpha_0))),\lceil \frac{k}{2} \rceil)\rfloor}$$ by a direct application of \cite[Lem.\ 10]{KaplanPetrow2}. So, we may assume $v(b_0)=0$ in the following. 
  
   If $v(m-b_0^2\Nm(\alpha_0))\leq \lceil \frac{k}{2}\rceil-2$, then $v(m-b_0^2\Nm(\alpha_0)-2xb_0p^{\lfloor k/2 \rfloor}\Nm(\alpha_0)) = v(m-b_0^2\Nm(\alpha_0))$, so that by loc.\ cit.\ $$|S_{L,b}| \leq 2p^{\lfloor \frac{1}{2} \min(v(m-b_0^2\Nm(\alpha_0)),\lceil \frac{k}{2} \rceil)\rfloor}.$$ If $v(m-b_0^2\Nm(\alpha_0))\geq \lceil \frac{k}{2}\rceil-1$ and $\lceil k/2 \rceil$ is odd, then similarly
  $$ |S_{L,b}| = p^{\frac{1}{2} ( \lceil \frac{k}{2}\rceil -1)} S\left( \frac{m-b_0^2\Nm(\alpha_0)-2xb_0p^{\lfloor k/2 \rfloor}\Nm(\alpha_0)}{p^{ \lceil \frac{k}{2}\rceil -1}},p\right)\leq 2 p^{\lfloor \frac{1}{2} \min(v(m-b_0^2\Nm(\alpha_0))),\lceil \frac{k}{2} \rceil)\rfloor}.$$  
  If $v(m-b_0^2\Nm(\alpha_0))\geq \lceil \frac{k}{2}\rceil-1$ and $\lceil k/2 \rceil$ is even, then $$ |S_{L,b}| = p^{\frac{1}{2} ( \lceil \frac{k}{2}\rceil -2)} S\left( \frac{m-b_0^2\Nm(\alpha_0)-2xb_0p^{\lfloor k/2 \rfloor}\Nm(\alpha_0)}{p^{ \lceil \frac{k}{2}\rceil -2}},p^2\right)$$ and $$v\left(\frac{m-b_0^2\Nm(\alpha_0)-2xb_0p^{\lfloor k/2 \rfloor}\Nm(\alpha_0)}{p^{ \lceil \frac{k}{2}\rceil -2}}\right) = \begin{cases} \geq 2 & \text{ if } x = \frac{m-b_0^2\Nm(\alpha_0)}{2b_0p^{\lfloor k/2 \rfloor}\Nm(\alpha_0)}, \\ 1 & \text{ otherwise.}\end{cases}$$ Thus by loc.\ cit.\
  $$|S_{L,b}| = \begin{cases}p^{\frac{1}{2}  \lceil \frac{k}{2}\rceil} & \text{ if } b= \frac{1}{2}(b_0+ \frac{m}{b_0\Nm(\alpha_0)}),  \\ 0 & \text{ otherwise,}\end{cases}$$
  and
  $$ \frac{1}{2}  \lceil \frac{k}{2}\rceil =\lceil \frac{1}{2} \min(v(m-b_0^2\Nm(\alpha_0)),\lceil \frac{k}{2} \rceil)\rceil \leq \lfloor \frac{1}{2} \min(v(m-b_0^2\Nm(\alpha_0)),\lceil \frac{k}{2} \rceil)\rfloor +1.$$
  Of course, $m-b_0^2\Nm(\alpha_0)= m+\frac{(p^k \Tr \alpha_0 \alpha_\xi)^2}{D}.$  
 
Drawing these cases together, we conclude that $$|S_L| \leq 2^6 p^{\fa + \lfloor \frac{1}{2} \min(v(\frac{(p^k \Tr \alpha_0 \alpha_\xi)^2}{D} +m),\lceil \frac{k}{2} \rceil)\rfloor},$$ where $2^6$ may be replaced by $2$ if $p\neq 2$. 
  
 Now let us consider the sum $U$. In similar fashion to the proof of Proposition \ref{intermediate_family_mainprop}, the condition $v(\widetilde{b}) \geq \lfloor \frac{k+(e-1)}{2}\rfloor$ excludes all cases except $d=3,1$, or ($p \neq 2, d=0$ and $v(a)=0$). In these cases, by Table \ref{table:a'b'cases} we have that $v(\widetilde{b})= v(b)+(e-1)+v(2)$.  If $k \geq 2c_0+d$, then by Proposition \ref{Hpmnpk_is_classicalKl}, the sum $H_p(m,1,p^k)$ is a classical Kloosterman sum, so that \eqref{eqn:statphase_bounds} holds by the classical Weil bound. We may therefore assume that $k< 2c_0+d$ for the remainder of the proof. 
 
 First let us suppose that $p$ is odd, which ensures that all terms in the sum $U$ with $v(a) \neq 0$ vanish. Moreover, we have by Lemma \ref{statphase_sc_Klsums_nohyponp}(2) that $U$ vanishes unless $\lceil (k-(e-1))/2\rceil \geq c_0$, so that the only case left to consider is when $k=2c_0+e-2$. Thus, when $p\neq 2$ and $k<2c_0+d$, the sum $U$ either vanishes, or 
 $$U =  p^{-\lceil \frac{3k-d}{2}\rceil } \sum_{\substack{b \in \cO/p^{c_0}\cO \\ v(b) \geq c_0-1}} \sum_{\substack{a \in (\cO/p^{c_0}\cO)^\times \\ a^2 + b^2 \Nm(\alpha_0) \equiv m \shortmod{p^{k}}}}  \xi(a+b\alpha_0)\psi(-\frac{2a}{p^{k}}).$$
 For $m \in \cO^\times$ the domain of summation on $a$ is 
 \begin{multline*}
 \{a \in (\cO/p^{c_0}\cO)^\times : \exists \text{ a lift } \tilde{a} \in (\cO/p^k\cO)^\times \text{ of } a \text{ with } \tilde{a}^2 \equiv m-b^2\Nm(\alpha_0) \shortmod{p^k}\} \\ 
 = \{ a \in (\cO/p^{c_0}\cO)^\times : a^2 \equiv m \shortmod{p^{c_0}}\}
 \end{multline*}
 by Hensel's lemma.  
   In particular, the domain is independent of $b$. The result of these transformations is 
 \begin{multline*} U= p^{-\lceil \frac{3k-d}{2}\rceil } \sum_{\substack{a \in (\cO/p^{c_0}\cO)^\times \\ a^2 = m \shortmod{p^{c_0}}}}  \sum_{\substack{b \in \cO/p^{c_0}\cO \\ v(b) \geq c_0-1}}  \xi(a+b\alpha_0)\psi(-\frac{2a}{p^{k}}) \\
  = p^{-\lceil \frac{3k-d}{2}\rceil } \sum_{\substack{a \in (\cO/p^{c_0}\cO)^\times \\ a^2 = m \shortmod{p^{c_0}}}}  \xi(a)\psi(-\frac{2a}{p^{k}}) \sum_{\substack{b \in \cO/p^{c_0}\cO \\ v(b) \geq c_0-1}} \psi(\frac{b}{a} \Tr \alpha_\xi \alpha_0) = 0.
  \end{multline*}

 If $d=3$, then essentially the same argument as for $p\neq 2$ goes through to show that $U=0$ when $k<2c_0+d$. We quickly note the necessary changes. Lemma \ref{statphase_sc_Klsums_nohyponp}(2) shows that $U$ vanishes, except possibly in the cases $2c_0\leq k\leq 2c_0+2$. We have $$ \begin{cases} b \in \cO/p^{c_0}\cO, v(b)\geq c_0-2  & \text{ if } k=2c_0, \\ b \in \cO/p^{c_0}\cO, v(b)\geq c_0-1  & \text{ if } k=2c_0+1, \\ b \in \cO/p^{c_0+1}\cO, v(b)\geq c_0-1 & \text{ if } k=2c_0+2,\end{cases}  \text{ and }  a \in \begin{cases} (\cO/p^{c_0}\cO)^\times & \text{ if } k=2c_0, \\ (\cO/p^{c_0+1}\cO)^\times & \text{ if } k\geq 2c_0 +1.\end{cases} $$
  Lastly, we use the hypothesis $c(\sigma)\geq 9$ from Theorem \ref{KlSumThm} to ensure that $b^2 \Nm(\alpha_0) \equiv 0 \pmod{p^{c_0}}$ and that Hensel's lemma continues to work in residue characteristic 2.
 \end{proof}
Proposition \ref{statphase_bounds} does not apply in the case $c(\sigma)=2$ and $k=1$ (i.e.\ $c_0=1$ and $E/F$ unramified) since the decomposition \eqref{intermediate_family_step1} is tautological in that case. Instead, we have the following bounds from $\ell$-adic cohomology.
\begin{myprop}[Deligne, Katz]\label{Katz_AG_bounds}
Suppose $E/F$ is an unramified quadratic extension, $q=|k_F|$, $c(\xi)\leq1$ and $\psi\neq 1$ an additive character of $F$ of conductor 0. For all $m$ we have
\begin{equation} 
\Big| \sum_{\substack{u \in (\cO_E/p\cO_E)^\times \\ \Nm(u)\equiv m \shortmod{p}}} \xi(u) \psi_E(-u p^{-1})\Big| \leq 2 \sqrt{q}.
\end{equation}
\end{myprop}
\begin{proof}
Let $\ell$ be a prime invertible in the residue field $k_F$. Deligne \cite[Sommes Trig.\ Rem.\ 7.18]{SGA4.5} suggested and Katz \cite[8.8.5 Thm.]{Katz} proved that there exists a lisse $\overline{\Q}_\ell$-sheaf $\Kl(\Res_{k_E/k_F}\G_m, \psi_E,\xi)$ of rank $2$ on $\G_{m,k_F}$, pure of weight 1, with trace function 
$$t_{\Kl(\Res_{k_E/k_F}\G_m, \psi_E,\xi)} (m) = - \sum_{\substack{u \in k_E^\times \\ \Nm_{k_E/k_F}(u)=m}}\xi(u)\psi_E(u/p).$$
Then, we have that $|t_{\Kl(\Res_{k_E/k_F}\G_m, \psi_E,\xi)} (m)| \leq 2 \sqrt{q}$ for all $m \in k_F^\times$ (see e.g.\ \cite[(3.4)]{FKMS}). If $m=0$ the sum clearly vanishes. 
\end{proof}
 
\section{Examples}\label{sec:examples}

\subsection{Classical family}\label{sec:egclassical}
Choose $c \in \Z_{\geq 0}$ and let \begin{equation}\label{classical_test_fcn} f_{\leq c} = \nu(p^c)1_{ZK_0(p^c)}.\end{equation} The function $f_{\leq c} \in \cH_p$ is the classical choice of test function matching \cite{knightly_kuznetsovs_2013}. 
\subsubsection{Geometric and Spectral Assumptions}\label{classic_gsassumptions}
It is clear that $f_{\leq c}$ satisfies Geometric Assumptions \eqref{geo2} and \eqref{geo3} with $y=p^{i}$, any $i \leq c$. It also satisfies the spectral assumption, by definition. 
\subsubsection{Local family}\label{sec:locfam-classical}
The operator $\pi(f_{\leq c}):V_\pi \to V_\pi$ is the orthogonal projection onto the space $V_\pi^{K_0(p^c)}$ of $K_0(p^c)$-fixed vectors in $V_\pi$. Therefore the local family $\cF_{\leq c}:=\cF_p(f_{\leq c})$ consists of $\pi \in {\overline{G}}(\Q_p)^{\wedge}$ that admit a non-zero $K_0(p^c)$-fixed vector. Equivalently, by newform theory 
\begin{equation}\label{localfam:classical}
\cF_{\leq c}= \{ \pi \in {\overline{G}}(\Q_p)^{\wedge}: c(\pi) \leq c\}.
\end{equation}
\subsubsection{Level}
It is clear that the local level $N_p$ of $f_{\leq c}$ satisfies $N_p = p^c$. 
\subsubsection{Diagonal weights}
By definition 
\begin{equation}\label{diagwt:classical}
\delta_p:= \int_{c(\pi) \leq c} \dim V_\pi^{K_0(p^c)}\, d\widehat{\mu}(\pi) = \int_{{\overline{G}(\Q_p)}^{\wedge}} \Tr \pi(f_{\leq c})\,d\widehat{\mu}(\pi),
\end{equation}
which by the Plancherel formula equals 
\begin{equation}  \int_{\Q_p} f_{\leq c} \left( \begin{smallmatrix} 1 & t \\ 0 & 1 \end{smallmatrix} \right) \psi_p(-mt)\,dt = \nu(p^c)  = f_{\leq c}(1)
\end{equation}
for any $m \in \Z_p^\times$. 
\subsubsection{Local Generalized Kloosterman Sums} 
\label{section:localgenKloostSums}
We have that 
\begin{equation}
\label{eq:localgenKloostSums}
H_p(m,n;p^k) = \begin{cases} 0 & \text{ if } k< c, \\
\delta_p S(m,n;p^k) & \text{ if } k \geq c \end{cases}
\end{equation}
by e.g.\ \cite[Prop.\ 3.7]{KLPetersson}.

\subsubsection{Local Geometric Conductor}\label{kp:classical}
Equation \eqref{eq:localgenKloostSums} shows that the local geometric conductor $k_p$ satisfies 
$k_p=c$ by the generic non-vanishing of classical Kloosterman sums. 
\subsubsection{Hypotheses from Section \ref{sec:intro:LSI}}\label{OLSI:classical}
Hypothesis \ref{hypCvF} (CvF) holds for  the classical family $\cF_{\leq c}$, since $p^{k_p} = p^c \geq \frac{2}{3} \nu(p^c) =\frac{2}{3} f_{\leq c}(1).$

To verify Hypothesis \ref{hypFTB} (FTB), we compute the Fourier/Mellin transform of $H_p$. A simple calculation shows that when $k\geq c$ and $c(\chi)\leq k$ 
\begin{equation}\label{Hhatp_def_examplessec}
\widehat{H}_p(\chi,k) := \frac{1}{\varphi(p^k)}\sumstar_{m \shortmod{p^k}} H_p(m,1;p^k)\overline{\chi(m)} = \nu(p^c) \frac{\tau(\overline{\chi})^2}{\varphi(p^k)},
\end{equation}
where $$\tau(\chi) = \sumstar_{m \shortmod{q}}\chi(m) e(m/q)$$ is the classical Gauss sum of $\chi$ as in e.g.\ \cite[Lem.\ 7.1]{PetrowYoungCoset}. In particular, we have
$$|\widehat{H}_p(\chi,k)| = \begin{cases} 
0 & \text{ if } k<c,  \\
f_{\leq c}(1) \zeta_p(1)  & \text{ if } c(\chi) = k \geq c, \\ f_{\leq c}(1) \zeta_p(1) p^{-1} & \text{ if } c(\chi)=0 \text{ and } k=1\geq c, \\ 0 & \text{ if } 0<c(\chi)< k \text{ and } k\geq c, \\ 0 & \text{ if } c(\chi)=0 \text{ and } k\geq 2, \end{cases}$$
so that Hypothesis \ref{hypFTB} (FTB) follows. As a side comment, the sum in \eqref{Hhatp_def_examplessec} is meaningless if $c(\chi)>k$, but the integral in \eqref{intro:HhatLocalDef} for $\widehat{H}_p(\chi,k)$ does make sense and returns 0 for $c(\chi)>k$.

\subsection{Principal series families}\label{sec:egPS}
Let $\chi$ be a character of $\Z_p^\times$ with $\chi^2$ non-trivial, i.e.\ a primitive non-quadratic Dirichlet character to some $p$-power modulus. Write $c=c(\chi)$, and if $p=2$ assume in addition that $c\geq 4$. We define a test function $f_\chi \in \cH_p$ by 
\begin{equation}\label{fchi_def}
f_\chi(g) := \frac{1}{\varphi(p^c)}\sum_{a,a' \in (\Z/p^c\Z)^{\times}} f_{\chi, a,a'} ,
\end{equation}
where 
\begin{equation}
f_{\chi, a,a'} = \chi(a)^{-1} \chi(a')  f_{\chi,0}(n(a'p^{-c})^{-1} g n(a p^{-c})),
\end{equation}
and
\begin{equation}
f_{\chi,0} (g) := \nu(p^c) 1_{ZK_0(p^c)} \overline{\chi(\alpha/\delta)} \quad \text{ for } \quad g= \left(\begin{smallmatrix} \alpha  & \beta \\ \gamma & \delta \end{smallmatrix}\right) \in G(\Q_p).
\end{equation}
Note in particular that $f_\chi(1)=f_{\chi,0}(1) = \nu(p^c)$. 

As we will see, the relative trace formula Theorem \ref{theoGeomSpec} associated to the choice $f_\chi$ at all ramified places matches the Bruggeman-Kuznetsov trace formula for $\cup_{m \mid q} (\cH(m, \chi^2) \otimes \overline{\chi})$ 
derived by classical means by the second and third authors in \cite{PetrowYoungWeyl} (here $\cH(m,\chi^2)$ is a basis of Hecke-Maass newforms of level $m \mid q$ and central character $\chi^2$, where $\chi$ is a primitive Dirichlet character modulo $q$). Note that in \cite{PetrowYoungWeyl}, the family used had $\cH(m, \overline{\chi}^2) \otimes \chi$ instead of $\cH(m, \chi^2) \otimes \overline{\chi}$, but of course these are identical.  

\subsubsection{Geometric and Spectral Assumptions}\label{PS_gsassumptions}
We can check by an explicit calculation that for any $\alpha,\alpha' \in (\Z/p^c\Z)^\times$, the support of $f_{\chi, 0}(n(\alpha'p^{-c})^{-1} g n(\alpha p^{-c}))$ is contained in $a(p^c)^{-1}Z K_p a(p^c)$. Therefore, $f_\chi$ satisfies Geometric Assumption \eqref{geo3} with $y=p^{c}$.  

In the case that $p$ is odd, the spectral assumption for $f_\chi$ was established by the first named author \cite[\S 3.3]{Hu}. Precisely, by Proposition 3.28 and the first sentence of Corollary 3.24 of loc.\ cit.\ we have that $\pi(f_\chi):V_\pi \to V_\pi$ is an orthogonal projection onto the line of the newform in $V_\pi$ if $\pi$ is isomorphic to a principal series representation $\pi(\mu,\mu^{-1})$ with $\mu \vert_{\Z_p^\times} = \chi$ and $\pi(f_\chi)= 0$ if $\pi$ is not such a representation (recall we have assumed that $\chi$ is not quadratic). Therefore $f_\chi$ is a newform projector, and hence satisfies the spectral assumption.

If $p=2$ then we may argue along similar lines to show that $f_\chi$ is a newform projector. We briefly give the details now. First, note that $c(\chi^2) = c-1$ since we have assumed $c\geq 4$, as can be seen by e.g.\ Lemma \ref{postnikov}. Next, denote by $\tilde{\theta}'$ the function on $ZK_0(p^{c})$ given by
$$\tilde{\theta}'\left(\begin{smallmatrix} \alpha & \beta \\ \gamma & \delta \end{smallmatrix} \right)=\chi^{-2}(\delta).$$
\begin{mylemma}\label{Lem:localizedfortwist}
	Let $\pi'$ be an irreducible smooth admissible representation of $\GL_2(\Q_2)$ and $c\geq 4$. Then the subspace of $V_{\pi'}$ on which $ZK_0(p^{c})$ acts by the character $\tilde{\theta}'$ is nontrivial only when $\pi'\simeq \pi(\nu,\nu^{-1}\chi^{-2})$ for some unramified character $\nu$, in which case it is two dimensional with a basis given by the newform $\varphi_0'\in V_{\pi'}$ and its translate $\varphi_1'=\pi'(a(p))\varphi_0'$.
\end{mylemma}
\begin{proof}
	If the subspace of $V_{\pi'}$ on which $K_0(p^{c})$ acts by the character $\tilde{\theta}'$ is nontrivial, it is necessary that $\pi'=\pi(\eta_1,\eta_2)$ (see e.g.\ \cite[Pf.\ of Prop.\ 2]{Casselman_restriction}) with $\sum c(\eta_i)\leq c$ and $c(\eta_1\eta_2)=c-1$. As there is no character over $\Q_2$ with level $1$, and the central character is determined, we have $\pi'\simeq \pi(\nu,\nu^{-1}\chi^{-2})$ for some unramified characters $\nu$. In that case we have $c(\pi')=c(\chi)-1$, thus by newform theory the corresponding subspace is 2-dimensional, spanned by the newform and its diagonal translate.
\end{proof}
Now we twist back. Denote by $\tilde{\theta}$ the following character on $ZK_0(p^{c})$
\begin{equation}\label{Eq:thetatilde}
	\tilde{\theta}\left(\begin{smallmatrix} \alpha & \beta \\ \gamma & \delta \end{smallmatrix}\right)=\chi(\alpha/\delta).
\end{equation}
\begin{mylemma}\label{Lem:Whittakerbeforetwist}
		Let $\pi$ be an irreducible smooth admissible representation of $\GL_2(\Q_2)$ and $c\geq 4$. Then the subspace of $V_\pi$ on which $ZK_0(p^{c})$ acts by the character $\tilde{\theta}$ is nontrivial only when $\pi\simeq \pi(\nu\chi,\nu^{-1}\chi^{-1})$ for some unramified characters $\nu$, in which case it is two dimensional with a basis $\{\varphi_0,\varphi_1\}$  given in the Whittaker model by
		$$W_0\lb\zxz{x}{0}{0}{1}\rb=\sqrt{1-p^{-1}}\begin{cases}
		p^{-v \lb x \rb /2}\chi\nu(x), &\text{\ if $v \lb x \rb \geq 0$,}\\
		0, &\text{\ otherwise.}
	\end{cases}$$
	$$W_1\lb\zxz{x}{0}{0}{1}\rb=\sqrt{1-p^{-1}}\begin{cases}
		p^{-(v \lb x \rb+1) /2}\chi\nu(x), &\text{\ if $v \lb x \rb \geq -1$,}\\
		0, &\text{\ otherwise.}
	\end{cases}$$
\end{mylemma}
A similar statement to Lemma \ref{Lem:Whittakerbeforetwist} was previously given by Nelson \cite[Lem.\ 48]{nelson_microlocal_2016}.

The following is an analogue of \cite[Lem.\ 3.25]{Hu}:
\begin{mylemma}\label{Lem:projnewform}
	For $\pi$ as in Lemma \ref{Lem:Whittakerbeforetwist} and $i=0,1$, 
	$$\sum\limits_{a\in (\Z/p^{c}\Z)^\times}\chi(a)\pi\lb \zxz{1}{\frac{a}{p^{c}}}{0}{1}\rb W_{i}$$
	is a non-zero scalar multiple of the newform in $V_\pi$.
\end{mylemma}
\begin{proof}
	The proof is essentially the same as for \cite[Lem.\ 3.25]{Hu} using the Whittaker functions from Lemma \ref{Lem:Whittakerbeforetwist}, with the main step being that the Gauss sum
	$$
	\sum\limits_{a\in  \lb \Z/p^{c}\Z \rb ^\times}\psi\lb\frac{a x }{p^{c_1}}\rb \chi \lb a x  \rb 
	$$
	is non-vanishing only if $v(x)=0$, in which case the value is independent of $x$.
\end{proof}
For the purpose of comparison with \cite{Hu}, note that
$$f_{\chi,0} = \nu(p^c) \overline{\widetilde{\Phi}}_{0,0} = \frac{1}{\vol(ZK_0(p^c)/Z)}  \overline{\widetilde{\Phi}}_{0,0}, \, \text{ and } \, f_{\chi, a,a'}  =  \nu(p^c)  \overline{\widetilde{\Phi}}_{a,a'}= \frac{1}{\vol(ZK_0(p^c)/Z)}  \overline{\widetilde{\Phi}}_{a,a'},$$
where  $\widetilde{\Phi}_{0,0}$ and $\widetilde{\Phi}_{a,a'}$ are as in Definition 3.26 of loc.\ cit..
 Recall that $\pi$ is unitary with the pairing $\langle \cdot , \cdot \rangle$ given in the Kirillov model by \eqref{GinvtInnerProd}. 
\begin{mylemma}\label{p=2ps_orthogonality}
For $\pi$ as in Lemma \ref{Lem:Whittakerbeforetwist} and $u = n(\alpha p^{-c})$ with $\alpha \not \equiv 0 \pmod {p^c}$, 
\begin{enumerate}
	\item $\Span\{\pi(u)\varphi_0,\pi(u)\varphi_1\}\perp \Span\{\varphi_0,\varphi_1\},$ and
	\item   if $v\perp \Span\{\varphi_0,\varphi_1\}$, then $v\in \ker \pi\left(f_{\chi,0}\right)$.
\end{enumerate}
\end{mylemma}
\begin{proof}
To verify (1), one can use the unitary pairing $\langle \cdot, \cdot \rangle$ on the Kirillov model and the expression of Whittaker functions from Lemma \ref{Lem:Whittakerbeforetwist}. To see (2), we note that
\begin{align*}
	\langle \varphi_i, v \rangle = \langle \pi\left(f_{\chi,0}\right) \varphi_i, v \rangle = \langle \varphi_i,\pi\left(f_{\chi,0}\right) v\rangle,
\end{align*}
where the last equality follows from the fact that $\tilde{\theta}$ is a character on the support with $|\tilde{\theta}|=1$. Thus $v\perp \Span\{\varphi_0,\varphi_1\}$  if and only if $\pi\left(f_{\chi,0}\right) v\perp \Span\{\varphi_0,\varphi_1\}$, if and only if $\pi\left(f_{\chi,0}\right)=0$. The last equivalence follows from the fact that $\langle\cdot,\cdot \rangle$ is non-degenerate on $\imag \pi\left(f_{\chi,0}\right)$. Indeed by Lemma \ref{Lem:Whittakerbeforetwist} we have
$$
\zxz{\langle W_0,W_0\rangle}{\langle W_0,W_1\rangle}{\langle W_1,W_0\rangle}{\langle W_1,W_1\rangle}=\zxz{1}{p^{-1/2}}{p^{-1/2}}{1}
$$
which has nonzero determinant and is thus non-degenerate. 
\end{proof}
\begin{mylemma}\label{Lem:convolutionlaw}
	$$f_{\chi,a,a'}*f_{\chi,b,b'}=\begin{cases}
		f_{\chi,b,a'} & \text{ if }a\equiv b'\pmod{ p^{c}},\\
		0 & \text{ otherwise}.
	\end{cases}$$
	\end{mylemma}
\begin{proof}
	By definition and change of variable, 
\begin{align*}
		&f_{\chi,a,a'}*f_{\chi,b,b'}(g)\\
		=&\overline{\chi}\left(\frac{ab}{a'b'}\right)\int\limits_{h\in G}f_{\chi,0}\left(\zxz{1}{-a'p^{-c}}{0}{1} gh^{-1}\zxz{1}{ap^{-c}}{0}{1}
		\right)f_{\chi,0}\left(\zxz{1}{-b'p^{-c}}{0}{1} h\zxz{1}{bp^{-c}}{0}{1}\right)dh\\
		=&\overline{\chi}\left(\frac{ab}{a'b'}\right)\int\limits_{h\in G}f_{\chi,0}\left(\zxz{1}{-a'p^{-c}}{0}{1} g\zxz{1}{bp^{-c}}{0}{1}h^{-1}
		\right)f_{\chi,0}\left(\zxz{1}{(a-b')p^{-c}}{0}{1}h \right)dh.
\end{align*}
The conclusion is clear when $a\equiv b'\pmod{ p^{c}}$. We need to show that when $a \not\equiv b' \pmod {p^{c}}$, the integral is always vanishing for any $g$, i.e.\
$$ 
f_{\chi,0}*f_{\chi,0,b'-a}=0.
$$
By the Plancherel formula \eqref{PlDef},  it suffices to prove that $$\pi\left(f_{\chi,0}*f_{\chi,0,b'-a}\right)=\pi\left(f_{\chi,0}\right)\pi\left(f_{\chi,0,b'-a}\right)$$ is vanishing for any $\pi$. From Lemma \ref{Lem:Whittakerbeforetwist}, we can restrict to the case $\pi\simeq \pi(\nu\chi,\nu^{-1}\chi^{-1})$. In this case let $\varphi_0,\varphi_1$ be as in Lemma \ref{Lem:Whittakerbeforetwist}. Then by a change of variable, 
$$\imag\left(\pi\left(f_{\chi,0,b'-a}\right)\right)=\Span\{\pi(u)\varphi_0,\pi(u)\varphi_1\}
$$
for the unipotent matrix $$u=\zxz{1}{(b'-a)p^{-c}}{0}{1}.$$ 
The required vanishing now follows from Lemma \ref{p=2ps_orthogonality}. 
\end{proof}
\begin{myprop}
If $\pi \simeq \pi(\mu,\mu^{-1})$ with $\mu \vert_{\Z_p^\times} = \chi$, then $\pi(f_\chi)$ is a projection operator onto the space of newforms in $V_\pi$, and otherwise $\pi(f_\chi)=0$. 
\end{myprop}
\begin{proof}
First, note that $f_\chi *f_\chi =f_\chi$ by the definition of $f_\chi$ and Lemma \ref{Lem:convolutionlaw}. Thus, $\pi(f_\chi)$ is a projection operator. Next, note that for any $v\in V_\pi$, 
\begin{equation}\label{pifchiexpansion} \pi(f_\chi)v 
 = \frac{1}{\varphi(p^c)}\sum_{a' \in (\Z/p^c\Z)^{\times}} \chi(a') \pi(n(a'p^{-c})) \pi(f_{\chi,0}) \sum_{a \in (\Z/p^c\Z)^{\times}} \chi(a)^{-1}\pi(n(-a p^{-c}))v.\end{equation}
Note that $ZK_0(p^c)$ acts on $\imag \pi(f_{\chi, 0})$ through the character $\widetilde{\theta}$, so by Lemma \ref{Lem:Whittakerbeforetwist} $\pi(f_\chi)=0$ unless $\pi \simeq \pi(\mu,\mu^{-1})$ with $\mu \vert_{\Z_p^\times} = \chi$. If $\pi$ is such a principal series, then by Lemmas \ref{Lem:Whittakerbeforetwist} and \ref{Lem:projnewform}, the operator $\pi(f_\chi)$ has image in the line of the newform. Lastly, choose any $a_0 \in (\Z/p^c\Z)^\times$ and let $v_0 = \pi(n(a_0p^{-c}))\varphi_0$. We have by Lemma \ref{p=2ps_orthogonality} that $$ \pi(f_{\chi,0})\sum_{a \in (\Z/p^c\Z)^{\times}} \chi(a)^{-1}\pi(n(-a p^{-c}))v_0 = \chi(a_0)^{-1} \pi(f_{\chi,0})\varphi_0$$ and $\pi(f_{\chi,0})\varphi_0 = \varphi_0$, so that
$$\pi(f_\chi)v_0  
= \frac{\chi(a_0)^{-1}}{\varphi(p^c)}\sum_{a' \in (\Z/p^c\Z)^{\times}} \chi(a') \pi(n(a'p^{-c}))\varphi_0,$$ which is non-zero by Lemma \ref{Lem:projnewform}.
\end{proof}

 By Lemma \ref{specControlonH}(2), $f_\chi$ also satisfies Geometric Assumption \eqref{geo2}. 
\subsubsection{Local family}\label{sec:locfam-PS} 
Given $\chi$ a character as above, define
\begin{equation}\label{localfam:PS}
\cF_\chi := \{ \pi(\mu,\mu^{-1}) \in {\overline{G}}^{\wedge}: \mu \vert_{\Z_p^\times} = \chi\}. 
\end{equation}
By the discussion in Section \ref{PS_gsassumptions}, we have that the local family $\cF_p(f_\chi) = \cF_\chi$. 
\subsubsection{Level}\label{sec:PSlevel}
The local level $N_p$ of $f_{\chi}$ satisfies $N_p = p^{2c}$. Indeed, since $f_\chi$ is a newform projector, Proposition \ref{mainprop} applies, and $f_\chi$ is bi-$K_0(p^{2c})$-invariant as $c(\pi(\mu,\mu^{-1}))=2c$. Thus, $f_\chi$ is bi-$K(p^{2c})$-invariant and so $N_p \mid p^{2c}$. On the other hand, suppose $f_\chi$ were bi-$K( p^{2c-1})$-invariant. Then, it would be bi-invariant by the product $K(p^{2c-1})K_0(p^{2c})= K_0(p^{2c-1})$. Indeed, the inclusion $\subseteq$ is clear and for the other direction, note that $$ \zxz{1}{0}{\frac{c}{a}p^{2c-1}}{1-\frac{bc}{ad} p^{2c-1}}\zxz{a}{b}{0}{d}= \zxz{a}{b}{cp^{2c-1}}{d} =g$$ for any $g \in K_0(p^{2c-1})$.  But then $\pi(f_\chi)V_\pi$ for $\pi \simeq \pi(\mu,\mu^{-1})$ would have a non-zero $K_0(p^{2c-1})$-fixed vector, which it does not. Thus $N_p=p^{2c}$. 
\subsubsection{Diagonal weights}
By definition, 
\begin{equation}\label{diagwt:PS}
\delta_p:= \int_{\cF_\chi} \frac{1}{\mathcal{L}_\pi(1)}\, d\widehat{\mu}(\pi) = (1-p^{-1})^{-1} \widehat{\mu}(\cF_\chi),
\end{equation}
since $\mathcal{L}_\pi(1)= (1-p^{-1})$ is constant on $\cF_\chi$, as $c(\pi)\geq 2$ for all $\pi \in \cF_\chi$ (recall \eqref{Appidef} for the definition)
By the Plancherel formula,  
\begin{equation}  \delta_p = (1-p^{-1})^{-1} f_\chi(1) = \frac{\nu(p^c)}{1-p^{-1}}.
\end{equation}
\subsubsection{Local Generalized Kloosterman Sums}
The local generalized Kloosterman sums $H_p(m,n;c)$ associated to $f_\chi$ were computed in \cite[Cor.\ 4.12]{Hu} and go through in the case $p=2$. We have 
\begin{equation}\label{PSgenKLsum}
H_p(m,n;p^k) = \begin{cases}\delta_p \overline{\chi(m)} \chi(n) S_{\chi^2}(m,n;p^k) & \text{ if } k \geq c(\chi) \text{ and } (p,mn)=1, \\ 0 & \text{ if } k<c(\chi) \text{ or } p \mid mn.\end{cases}
\end{equation}
For comparison to the supercuspidal Kloosterman sums below, it is pleasing to note that \begin{equation}\label{PSgenKLsum2}  \overline{\chi(m)} \chi(n) S_{\chi^2}(m,n;p^k) =  \sumstar_{\substack{x,y \shortmod{p^k} \\ xy =mn}} \overline{\chi(x)}\chi(y)e\left( \frac{x+y}{p^k}\right).\end{equation}

\begin{myrema} Note that the formula \eqref{PSgenKLsum} for the generalized Kloosterman sums differs from the Kloosterman sums that appear via the classical procedure (as in \cite{PetrowYoungWeyl}) by the factor of $(1-p^{-1})$. The extra factor of $(1-p^{-1})^{-1}$ may be accounted for by the observation that the harmonic weights in Theorem \ref{theoGeomSpec} and the harmonic weights in the classically derived formula are not exactly the same. The former are attached to forms of conductor $2c$ and trivial central character, while the latter are attached to forms of level $c$ and non-trivial central character. \end{myrema}
\subsubsection{Local Geometric Conductor}\label{kp:PS}
The previous subsection shows that the local geometric conductor $k_p$ satisfies $k_p=c$ by the generic non-vanishing of classical Kloosterman sums. 
\subsubsection{Hypotheses from Section \ref{sec:intro:LSI}}\label{OLSI:PS}
Hypothesis \ref{hypCvF} (CvF) holds for the family $\cF_{\chi}$, since $p^{k_p} = p^c \geq \frac{2}{3} \nu(p^c) =\frac{2}{3} f_{\chi}(1).$

To verify Hypothesis \ref{hypFTB} (FTB), we compute the Fourier/Mellin transform of $H_p$.  A simple calculation shows that when $k\geq c$ and $c(\alpha)\leq k$ 
\begin{equation}
\widehat{H}_p(\alpha) := \frac{1}{\varphi(p^k)}\sumstar_{m \shortmod{p^k}} H_p(m,1;p^k) \overline{\alpha(m)}= \delta_p \frac{\tau(\overline{\alpha \chi})\tau(\overline{\alpha}\chi)}{\varphi(p^k)},
\end{equation}
where $$\tau(\chi) = \sumstar_{m \shortmod{q}}\chi(m) e(m/q)$$ is the classical Gauss sum of $\chi$ as in e.g.\ \cite[Lem.\ 7.1]{PetrowYoungCoset}. In particular, since $c(\overline{\alpha}\chi)$ and $c(\overline{\alpha \chi})$ are 
both 
$\leq k$ whenever $H_p(m,1;p^k) \neq 0$ (see \eqref{PSgenKLsum})
we have
$$|\widehat{H}_p(\alpha)| \leq (1-p^{-1})^{-2} f_{\chi}(1)$$
for all characters $\alpha$ of $\Z_p^\times$ so that Hypothesis \ref{hypFTB} (FTB) follows. 

\subsection{Supercuspidal families}\label{sec:examples_supercuspidal}
Let $F=\Q_p$ with ring of integers $\cO=\Z_p$. If $p \neq 2$ suppose we are given an admissible pair $(E/F,\xi) \in \P_2(F)$ with $\xi \vert_{F^\times} = \eta_{E/F}$, and if $p=2$ suppose we are given $(E/F, \xi) \in \P_2(F)_{\geq 9}^1$, and moreover that $c(\xi)\geq 8$ when $d=3$. 

Let $\sigma$ be the supercuspidal representation corresponding to the pair $(E/F,\xi)$ by Theorem \ref{TameParametrizationTheorem} or Corollary \ref{cor_p2bijection} and $\Phi= \Phi_\sigma$ the diagonal matrix coefficient of an $L^2$-normalized newform in $\sigma$. Recall $c_0 = c(\xi)/e$, $d=v_p(\disc E/F)$, and the compact open subgroups $K_0(m,n)$ from \eqref{K0mn_def}. Following Theorems \ref{cor:SpecAssumption_supercuspidal} and \ref{cor:SpecAssumption_supercuspidal_p2} we set
\begin{equation}\label{sec7_fxi_def}
f_\xi = \frac{\overline{\Phi} \vert_{ZK_0(m,n)}}{\|\Phi \vert_{ZK_0(m,n)}\|_2^2 },
\end{equation}
with $$(m,n) = \begin{cases} (c_0,-c_0) &\text{ if } d=0,  \\
(c_0+1,-c_0) & \text{ if }  d=1, \\ 
(c_0+1,-c_0-1) & \text{ if } d=2,\\
(c_0+2, -c_0-1) & \text{ if } d=3. \end{cases}$$

\subsubsection{Geometric and Spectral Assumptions}\label{SC_gsassumptions}
It is clear from its definition \eqref{sec7_fxi_def} that $f_\xi$ satisfies Geometric Assumption \eqref{geo3}. By Theorems \ref{cor:SpecAssumption_supercuspidal} and \ref{cor:SpecAssumption_supercuspidal_p2} $f_\xi$ satisfies the spectral assumption, \`a fortiori Geometric Assumption \eqref{geo2} by Lemma \ref{specControlonH}(2). 

\subsubsection{Local family}\label{sec:locfam-SC} 
With hypotheses as above, by Theorems \ref{cor:SpecAssumption_supercuspidal} and \ref{cor:SpecAssumption_supercuspidal_p2} we have
\begin{equation}\label{localfam:SC} \cF_p(f_\xi) = \cF_\xi := \begin{cases} \{\sigma\}  & \text{ if } p \neq 2 \text{ and } d=0, \\
 \{\sigma, \sigma \times \eta\} & \text{ if } d \geq 1,\\
 i( \xi[1]) & \text{ if } p=2 \text{ and } d=0,\end{cases}
 \end{equation}
where $i$ is the map in Corollary \ref{cor_p2bijection}. Note, if $p=2$ and $d=0$, then $|\cF_\xi|=3$ and $\sigma \in \cF_\xi$.

\subsubsection{Level} The local level $N_p$ of $f_\xi$ satisfies $N_p = p^{c(\sigma)}$. Indeed, since $f_\xi$ is a newform projector, Proposition \ref{mainprop} applies, and so $f_\xi$ is bi-$K_0(p^{c(\sigma)})$-invariant, in particular bi-$K(p^{c(\sigma)})$-invariant so that $N_p \mid p^{c(\sigma)}$. On the other hand, if $f_\xi$ were bi-$K( p^{c(\sigma)-1})$-invariant, then it would be bi-invariant by the product $K(p^{c(\sigma)-1})K_0(p^{c(\sigma)})= K_0(p^{c(\sigma)-1})$ (see Section \ref{sec:PSlevel}). But then $\pi(f_\xi)$ would project into the space of $K_0(p^{c(\sigma)-1})$-fixed vectors, which it does not. Thus $N_p=p^{c(\sigma)}$.

\subsubsection{Diagonal weights}
By definition, 
\begin{equation}
\delta_p= \int_{\cF_\xi} \frac{1}{\mathcal{L}_\pi(1)}\, d\widehat{\mu}(\pi) = (1-p^{-1})^{-1} \widehat{\mu}(\cF_\xi),
\end{equation}
since $\mathcal{L}_\pi(1)= (1-p^{-1})$ is constant on $\cF_{\xi}$ (recall \eqref{Appidef} for the definition). By the Plancherel formula, \eqref{sec:7.3fp(1)} and \eqref{f1p=2_formula},  
\begin{equation}\label{diagwt:SC}  \delta_p = (1-p^{-1})^{-1} f_{\xi}(1)  = \begin{cases} 
p^{c_0} & \text{ if } p \neq 2 \text{ and } d=0,\\
\nu( p^{c_0+1}) & \text{ if } p \neq 2 \text{ and } d=1,\\
\nu( p^{c_0+1}) & \text{ if } p = 2 \text{ and } d\neq 3,\\
\nu( p^{c_0+2}) & \text{ if } p = 2 \text{ and } d= 3 . \end{cases}
\end{equation}

\subsubsection{Local Generalized Kloosterman Sums}\label{Hmnc:SC} A natural expression for the local Kloosterman sum $H_p(m,n;p^k)$ was given formally in Theorem \ref{KlSumThm}. Briefly, we have that
\begin{equation}\label{sec7_kl_sum_formula}
H_p(m , n; p^k)= \begin{cases} \delta_p \overline{\gamma} p^{-\frac{d}{2}}\sum_{\substack{u \in (\cO_E/p^k \cO_E)^\times \\ \Nm(u) \equiv mn \shortmod{p^k}}} \xi(u) \psi\left( - \frac{\Tr( u)}{p^k}\right) & \text{ if } k \geq \lceil c(\sigma)/2 \rceil \text{ and } (mn,p)=1, \\ 
0 & \text{ otherwise.} \end{cases}
\end{equation}
See the beginning of Section \ref{sec:SCKloostermanSum} for an explanation of the notation in \eqref{sec7_kl_sum_formula} and some relevant remarks.

\subsubsection{Local Geometric Conductor}\label{kp:SC}
By Lemma \ref{ccondition} and the Definition \eqref{sec7_fxi_def}, we have \begin{equation}\label{kp_SC_formula} k_p \geq  \begin{cases} c_0 & \text{ if } d=0, \\ c_0+1 & \text{ if } d=1 \text{ or } 2, \\ c_0+2 & \text{ if } d=3.\end{cases}\end{equation}

In fact, the inequality is sharp. We can check this when $p\neq 2$ as follows. Suppose first that $d=0$, i.e.\ $c(\sigma)$ is even. Then applying \cite[Prop.\ 3.1(iii)]{Hu:17a} and Lemma \ref{MxCoeffFurther} with $i=c_0=c(\sigma)/2$, we see that $\Phi\left( \big( \begin{smallmatrix} a & m \\ 0 & 1 \end{smallmatrix}\big) \big( \begin{smallmatrix} 1 & \\ p^{c_0} & 1 \end{smallmatrix} \big)\right) \neq 0$ for some $a \in \cO^\times$ and some $m \in F$ with $v(m)=-c_0$. By the left-$A(\cO)$-invariance of $\Phi$, we have that $\Phi\left( \big( \begin{smallmatrix} 1 & n \\ 0 & 1 \end{smallmatrix}\big) \big( \begin{smallmatrix} 1 & \\ p^{c_0} & 1 \end{smallmatrix} \big)\right) \neq 0$ for some $n \in F$ with $v(n)=-c_0$.  Then Lemma \ref{lem:admmodulus} applies with $c=p^{c_0}$ and $g=  \left(\begin{smallmatrix} 1+np^{c_0} & np^{c_0} \\ 1 & 1 \end{smallmatrix}\right)$ (also using Lemmas \ref{*defect} and \ref{moduli}), so that $k_p \leq c_0$ thus $k_p=c_0$. 

Now suppose that $d=1$, i.e.\ $c(\sigma)$ is odd. Then we apply \cite[Prop.\ 3.1(i),(ii)]{Hu:17a} with $i=c_0+1=\frac{c(\sigma)+1}{2}$, obtaining in similar fashion to the $c(\sigma)$ even case above that $\Phi\left( \big( \begin{smallmatrix} 1 & n \\ 0 & 1 \end{smallmatrix}\big) \big( \begin{smallmatrix} 1 & \\ p^{c_0+1} & 1 \end{smallmatrix} \big)\right) \neq 0$ for some $n \in F$ with $v(n)=-c_0$. Thus Lemma \ref{lem:admmodulus} applies with $c= p^{c_0+1}$ and $g=  \left(\begin{smallmatrix} 1+np^{c_0+1} & np^{c_0+1} \\ 1 & 1 \end{smallmatrix}\right)$ (also using Lemmas \ref{*defect} and \ref{moduli}), so that $k_p \leq c_0+1$ thus $k_p=c_0+1$. 

\subsubsection{Hypotheses from Section \ref{sec:intro:LSI}}\label{hypfrom_intro:SC}

Next we compute the Fourier/Mellin transform of the supercuspidal Kloosterman sum. Recall from \eqref{intro:HhatLocalDef} that 
\begin{equation}\label{examples:HhatLocalDef}
\widehat{H}_p(\overline{\chi},k) :=  \frac{1 }{\varphi(p^k)} \sumstar_{m \shortmod{p^k}} H_p(m,1;p^k)\chi(m) =  \zeta_p(1)\int_{ \cO^\times} H_p(m,1;p^k) \chi(m) \,dm.
\end{equation}

\begin{myprop}\label{MellinXform}
If $k<\max(c(\chi),c(\sigma)/2)$, then $\widehat{H}_p(\overline{\chi},k)=0$. If $ k \geq \max(c(\chi), c(\sigma)/2)$, then 
$$\widehat{H}_p(\overline{\chi},k) =\delta_p (1-p^{-1})^{-1}\overline{\gamma} \frac{p^{k- \frac{d}{2}}}{\zeta_\fp(1)} \int\limits_{ \cO_E^\times}\chi_E\xi(x)\psi_E(-x p^{-k})\,d^\times x,$$
where $(E/F,\xi)$ is as in Theorem \ref{KlSumThm}, $\chi_E = \chi \circ \Nm$ and $\psi_E = \psi \circ \Tr$. In particular, $\widehat{H}_p(\overline{\chi},k) \neq 0$ if and only if $c(\chi_E \xi) = ek-d$, and in this case $|\widehat{H}_p(\overline{\chi},k)| =(1-p^{-1})^{-2}f_\xi(1) =(1-p^{-1})^{-1} \delta_p$. 
\end{myprop}
\begin{proof}
If $2k<c(\sigma)$ then $H_p(y,1;p^k)$ vanishes identically, so $\widehat{H}_p(\overline{\chi},k)$ does as well. If $c(\chi)>k$ then $\widehat{H}_p(\overline{\chi},k)$ vanishes identically by the $p^k$-periodicity of $H_p$. This is the first assertion. 

It remains to consider the case that $c(\chi) \leq k$ and $2k\geq c(\sigma)$. Under these assumptions, $$\widehat{H}_p(\overline{\chi},k) = \overline{\gamma} (1-p^{-1})^{-2}f_\xi(1) p^{2k- \frac{d}{2}}  \int_{\cO^\times} \chi(m) \int\limits_{\Nm(x)\equiv m \shortmod{p^{k}}}\xi(x)\psi(-p^{-k}\Tr(x))\,dx\, dm.$$ Swapping order of integration gives $$\widehat{H}_p(\overline{\chi},k) = \overline{\gamma} (1-p^{-1})^{-2}f_\xi(1) p^{k- \frac{d}{2}} \int_{\cO_E^\times}\chi_E\xi(x)\psi(-p^{-k}\Tr(x))\,dx,$$ which is a Gauss sum over $E$. Switching from additive to multiplicative Haar measure shows the 2nd assertion of the proposition. 

For the third assertion, it suffices to evaluate the Gauss sum, and such evaluations for Gauss sums over non-archimedean local fields are well-known. Note that since $\xi$ is regular, we have that $(\chi_E \xi)^\sigma \neq \chi_E \xi$, so that this character is non-trivial on $\cO_E^\times$. Then, by e.g.\ \cite[Lem.\ 2.3]{CorbettSaha} the Fourier/Mellin transform $\widehat{H}_p(\overline{\chi},k) $ is non-vanishing if and only if $c(\chi_E \xi) = ek-d$ and in this case 
\begin{equation}\widehat{H}_p(\overline{\chi},k) = \overline{\gamma}(1-p^{-1})^{-2}f_{\xi}(1)  \epsilon(1/2, (\chi_E\xi)^{-1}, \psi'_E)(\chi_E\xi)^{-1}(-1),\end{equation}
where $\psi'_E$ is the additive character of conductor 0 defined by $\psi'_E: x \mapsto \psi_E(\varpi_E^{d} x)$  and $\epsilon(1/2, (\chi_E\xi)^{-1}, \psi'_E)$ is the root number associated to $(\chi_E\xi)^{-1}$ and $\psi'_E$. We have in particular that 
\begin{equation}\label{MellinTransformModulus}|\widehat{H}_p(\overline{\chi},k)| = (1-p^{-1})^{-2}f_\xi(1) \delta_{c(\chi_E \xi) = ek-d}.\end{equation}
\end{proof}

Perhaps in practice it is useful to look at \eqref{MellinTransformModulus} in cases depending on $k,c(\sigma)$ and $c(\chi)$. 
If $k>\max(c(\chi),c(\sigma)/2)$, then $$c(\chi_E \xi) \leq \max(c(\chi_E),c(\xi)) \leq \max(\psi_{E/F}(c(\chi))-(e-1), \frac{e}{2} (c(\sigma)-d))<ek-d,$$ where $\psi_{E/F}$ is the Hasse-Herbrand function (see \cite[Ch.\ V]{SerreLocalFields}), so that $\widehat{H}_p(\overline{\chi},k)=0$.  
If $2k>c(\sigma)$ and $c(\chi)=k>d$, then $c(\chi_E \xi) = c(\chi_E) =ek-d$ by loc.\ cit.\ Corollary 3, so $\widehat{H}_p(\overline{\chi},k) \neq 0$. 
If $2k= c(\sigma)$ and $c(\chi)<k$, then $c(\chi_E)\leq ec(\chi)-d<ek-d= c(\xi)$, so $c(\chi_E\xi) = c(\xi)= ek-d$, so $\widehat{H}_p(\overline{\chi},k) \neq 0$. 
If $2k=c(\sigma)$ and $c(\chi)=k>d$, then $c(\xi) = ek-d= c(\chi_E)$ loc.\ cit.\ Corollary 3, so whether $\widehat{H}_p(\overline{\chi},k) = 0$ or not depends on whether the conductor of $\chi_E \xi$ drops or not. 

In particular, the last assertion of Proposition \ref{MellinXform} shows that Hypothesis \ref{hypFTB} (FTB) of Section \ref{sec:intro:LSI} holds for $f_p =  f_{\xi}$. 

From the above case analysis of Proposition \ref{MellinXform}, one can quickly check that the inequality in \eqref{kp_SC_formula} is in fact an equality. 
Therefore, Hypothesis \ref{hypCvF} (CvF) of Section \ref{sec:intro:LSI} holds locally for $f_{\xi}$, since we may check that $p^{k_p} \geq f_\xi(1)$ by comparing e.g.\ \eqref{diagwt:SC} and \eqref{kp_SC_formula}. 

\subsection{Neighborhood of a supercuspidal representation}\label{sec:neighborhood_of_a_SC}
Let $\sigma$ be a trivial central character dihedral supercuspidal representation corresponding to a pair $(E/F,\xi)$ as in Section \ref{sec:examples_supercuspidal}. For $0\leq n < c(\xi)$, recall \eqref{xi_n_def} the neighborhood $\xi[n]$ of characters around $\xi$, and for $0\leq a \leq n$ the equivalence relation $\sim_a$ on $\xi[n]$.

Write $\xi'$ for a twist-minimal character of $E^\times$ for which there exists a character $\chi$ of $F^\times$ with $\xi = \xi' \chi_E$, following Section \ref{sec:egSupercuspidal_even}. Recall that if $p \neq 2$ or $d=3$, then we may take $\xi'=\xi$ and if $p=2$ and $d=0$ or $2$, then we have that $c(\xi')=c(\xi)-1$, see Table \ref{table:xiprime}.

Now, set $a=1$ if the $E$ on which $\xi$ is defined is the unramified quadratic extension of $\Q_2$ and $a=0$ otherwise. 
Suppose that $a\leq n <c(\xi')$, so that no $\xi_1 \in \xi[n]$ is of the form $\chi_E$ for some character $\chi$ of $F^\times$. That is to say, all $\xi_1 \in \xi[n]$ are regular in the sense of Section \ref{sec:parametrization_dihedrals}.
Let $f_{\xi,n}\in \cF_{\rm fin}$ be defined by 
\begin{equation}\label{fxin_def}f_{\xi,n} = \sum_{\xi_1 \in \xi[n]/\sim_a} f_{\xi_1},\end{equation}
where $f_\xi$ is the supercuspidal projection operator defined in \eqref{sec7_fxi_def}.

The test function $f_{\xi,n}$ clearly is a newform projector because each $f_\xi$ is a newform projector. Moreover, since each $\xi_1 \in \xi[n]$ is defined over the same field as $\xi$ and has $c(\xi_1)=c(\xi)$, it follows from the definition of $f_\xi$ that $f_{\xi,n}$ satisfies the geometric assumptions as well.

Clearly,
$$ \cF_p(f_{\xi,n}) =
 i(  \xi[n]/\sim_a) ,$$ where $i$ is the LLC parametrization map of Section \ref{sec:parametrization_dihedrals}. 
 Since all $\pi \in \cF_p(f_{\xi,n}) $ have the same conductor exponent, the diagonal weight \eqref{eq:deltapnewformproj} is given by 
 \begin{equation}\label{diagwt:nbhd_of_a_SC}\delta_p = [\xi[n]:\xi[a]] \zeta_p(1) f_\xi(1).\end{equation} An explicit formula for $\zeta_p(1) f_\xi(1)$ was given in \eqref{diagwt:SC}. 
 
The local generalized Kloosterman sums corresponding to $f_{\xi,n}$ are computed by combining Theorem \ref{KlSumThm}, and Propositions \ref{intermediate_family_mainprop} and \ref{intermediate_family_mainprop2}. Writing $H_{\xi,p}(m_1,m_2;c)$ for the generalized Kloosterman sum attached to $\xi$ as in \eqref{Hmnc:SC}, the result is that
\begin{equation}\label{intermediate_family_sc_sum_section7}H_p(m_1,m_2;p^k) = \begin{cases}[\xi[n]:\xi[a]]H_{\xi,p}(m_1,m_2,p^k) & \text{ if } k\geq c_0+\lceil d/2 \rceil -a + \lfloor \frac{n}{e}\rfloor, \\ 
0 & \text{ if } k< c_0+\lceil d/2 \rceil -a + \lfloor \frac{n}{e}\rfloor.\end{cases}\end{equation} In particular, by referring to the results of Sections \ref{kp:SC} and \ref{hypfrom_intro:SC} we obtain that the local geometric conductor of $f_{\xi,n}$ is 
\begin{equation}\label{kp_nbhd_of_a_SC_formula} k_p = c_0+\lceil d/2 \rceil -a + \lfloor \frac{n}{e}\rfloor.\end{equation}
With \eqref{intermediate_family_sc_sum_section7} in hand, the details of the Fourier/Mellin transform of $H_p(m,1,p^k)$ can be read off directly from Section \ref{hypfrom_intro:SC}. In particular, the local version \eqref{hyp_FTB_local_version} of Hypothesis \ref{hypFTB} (FTB) is merely  that of $f_\xi$ times $[\xi[n]:\xi[a]]$ on both sides. Meanwhile, the local version \eqref{hypCvF_local_version} of Hypothesis \ref{hypCvF} (CvF) follows from \eqref{kp_nbhd_of_a_SC_formula}, \eqref{fxin_def}, \eqref{diagwt:SC}, and Remark \ref{size_thetaell_rem}. 

\subsection{Representations of a given conductor exponent $\geq 3$}\label{sec:f=c}
Let $c\geq 3$ and recall the definition of $K_0(m,n)$ from \eqref{K0mn_def}. Set 
$$ f_{m,n} = \frac{1}{\vol( Z \backslash ZK_{0}(m,n))} 1_{ZK_0(m,n)},$$ and define
\begin{equation}\label{f=c_def} f_{=c} = f_{c,0} - f_{c,-1}- f_{c-1,0} + f_{c-1,-1} .\end{equation}
Then, by \cite[Cor.\ 5]{Nelson_analytic_isolation}, the test function $f$ is a newform projector onto irreducible generic representations $\pi$ with $c=c(\pi)$. It clearly satisfies the geometric assumptions. 

The test function $f_{=c}$ has support controlled by $y=p^{c-1}$, so that by Lemma \ref{ccondition}, we have $k_p\geq c-1$. On the other hand, applying Lemma \ref{lem:admmodulus} with $N=M=p^c$ and $g= \left( \begin{smallmatrix} 1+p^{c-2} & p^{c-2} \\ 1 & 1 \end{smallmatrix}\right)$ shows that $p^{c-1} \in \cC(\cF(f_{=c}))$ is an admissible modulus. Thus $k_p=c-1$. 

The local generalized Kloosterman sums assocated to $f_{=c}$ can be deduced from \cite[(4)]{Nelson_analytic_isolation}. See Section \ref{sec:holomorphicunrefined} for our definition of Fourier coefficients and the Petersson formula, and \eqref{Petersson_Normalization} for our normalization of Petersson inner products. One finds 
\begin{equation}
H_p(m,n;p^k) = \sum_{d|(m,n,p^c)} \mu(d) d^{2} \sum_{e|p^c} \mu(e) \nu \Big( \frac{p^c}{de}\Big) \sum_{\substack{r \equiv 0 \shortmod{p^c/de} \\ dr=p^k}} S\Big(\frac{m}{d},\frac{n}{d}; r\Big).
\end{equation}
In particular, if $p \nmid (m,n)$, then 
\begin{equation} 
H_p(m,n,p^k) = \begin{cases} 0 & \text{ if } k\leq c-2, \\ 
- \nu(p^{c-1}) S(m,n;p^k) & \text{ if } k=c-1, \\
p^c(1-\frac{1}{p^2})S(m,n;p^k) & \text{ if } k\geq c.\end{cases}
\end{equation}
Now, the local Fourier/Mellin transform (see \eqref{intro:HhatLocalDef}) of $H_p(m,n;p^k)$  is given by 
\begin{equation}
\widehat{H}_p(\overline{\chi},k) = \begin{cases} 0 & \text{ if } k\leq c-2, \\ 
- \nu(p^{c-1}) \frac{\tau(\chi)^2}{\varphi(p^k)} & \text{ if } k=c-1, \\
p^c(1-\frac{1}{p^2})\frac{\tau(\chi)^2}{\varphi(p^k)} & \text{ if } k\geq c,\end{cases}
\end{equation}
using e.g.\ \eqref{Hhatp_def_examplessec}. Meanwhile, from the definition of $f_{=c}$ we have
\begin{equation}
f_{=c}(1) = p^c \left(1-\frac{1}{p^2}\right)\left(1- \frac{1}{p}\right),
\end{equation}
so that $|\widehat{H}_p(\chi,k)|\leq 4f_{=c}(1)$ for all $p, \chi, k$, i.e.\ the local version \eqref{hyp_FTB_local_version} of Hypothesis \ref{hypFTB} (FTB) holds for $f_{=c}(1)$. 

On the other hand, we have
$$\frac{p^{k_p}}{f_{=c}(1)} =\frac{1}{p \left(1-\frac{1}{p^2}\right)\left(1- \frac{1}{p}\right)},$$ so that 
local version \eqref{hypCvF_local_version} of Hypothesis \ref{hypCvF} (CvF) fails in ``horizontal'' ($p \to \infty$) aspect for $f_{=c}$.

\section{Statements}
\subsection{Conflicts of Interest}
Conflicts of interest: none.

\subsection{Financial Support}
Y.H.\ is supported by the National Key Research and Development Program of China (No.\ 2021YFA1000700). I.P.'s work on this paper is supported by the Engineering and Physical Sciences Research Council award EP/W009838/1.
This material is based upon work supported by the National Science Foundation under agreement No.\
DMS-2302210 (M.Y.). Any opinions, findings and conclusions or recommendations expressed in this material
are those of the authors and do not necessarily reflect the views of the National Science Foundation.

\newcommand{\etalchar}[1]{$^{#1}$}


\begin{thebibliography}{BBD{\etalchar{+}}17}

\bibitem[Ass19]{Assing_padic_Whittaker}
Edgar Assing.
\newblock On the size of {$p$}-adic {W}hittaker functions.
\newblock {\em Trans. Amer. Math. Soc.}, 372(8):5287--5340, 2019.

\bibitem[BBD{\etalchar{+}}17]{BBDDM}
Owen Barrett, Paula Burkhardt, Jonathan DeWitt, Robert Dorward, and Steven~J.
  Miller.
\newblock One-level density for holomorphic cusp forms of arbitrary level.
\newblock {\em Res. Number Theory}, 3:Art. 25, 21, 2017.

\bibitem[BH06]{BushnellHenniart:06a}
C.~Bushnell and G.~Henniart.
\newblock {\em The {Local} {Langlands} {Conjecture} for $\rm{GL}(2)$}.
\newblock Springer-Verlag, Berlin, 2006.

\bibitem[BM15]{BlomerMilicevic2ndMoment}
Valentin Blomer and Djordje Mili{\'c}evi{\'c}.
\newblock The second moment of twisted modular {$L$}-functions.
\newblock {\em Geom. Funct. Anal.}, 25(2):453--516, 2015.

\bibitem[BM24]{BrumleyMilicevic}
Farrell Brumley and Djordje Mili\'cevi\'c.
\newblock Counting cusp forms by analytic conductor.
\newblock {\em Ann. Sci. \'Ec. Norm. Sup\'er. (4)}, 57(5):1473--1597, 2024.

\bibitem[Bum97]{Bump}
Daniel Bump.
\newblock {\em Automorphic forms and representations}, volume~55 of {\em
  Cambridge Studies in Advanced Mathematics}.
\newblock Cambridge University Press, Cambridge, 1997.

\bibitem[Car79]{Cartier}
P.~Cartier.
\newblock Representations of {$p$}-adic groups: a survey.
\newblock In {\em Automorphic forms, representations and {$L$}-functions
  ({P}roc. {S}ympos. {P}ure {M}ath., {O}regon {S}tate {U}niv., {C}orvallis,
  {O}re., 1977), {P}art 1}, Proc. Sympos. Pure Math., XXXIII, pages 111--155.
  Amer. Math. Soc., Providence, R.I., 1979.

\bibitem[Cas73a]{Casselman_onAL}
William Casselman.
\newblock On some results of {A}tkin and {L}ehner.
\newblock {\em Math. Ann.}, 201:301--314, 1973.

\bibitem[Cas73b]{Casselman_restriction}
William Casselman.
\newblock The restriction of a representation of {${\rm GL}\sb{2}(k)$} to
  {${\rm GL}\sb{2}({\mathfrak{o}})$}.
\newblock {\em Math. Ann.}, 206:311--318, 1973.

\bibitem[\v{C}NS24]{CesnaviciusNeururerSaha}
Kęstutis Česnavičius, Michael Neururer, and Abhishek Saha.
\newblock The {M}anin constant and the modular degree.
\newblock {\em J. Eur. Math. Soc. (JEMS)}, 26(2):573--637, 2024.


\bibitem[CS18]{CorbettSaha}
Andrew Corbett and Abhishek Saha.
\newblock On the order of vanishing of newforms at cusps.
\newblock {\em Math. Res. Lett.}, 25(6):1771--1804, 2018.

\bibitem[Del77]{SGA4.5}
P.~Deligne.
\newblock {\em Cohomologie \'etale}, volume 569 of {\em Lecture Notes in
  Mathematics}.
\newblock Springer-Verlag, Berlin, 1977.
\newblock S\'eminaire de g\'eom\'etrie alg\'ebrique du Bois-Marie SGA
  $4\frac{1}{2}$.

\bibitem[DI83]{DI}
J.-M. Deshouillers and H.~Iwaniec.
\newblock Kloosterman sums and {F}ourier coefficients of cusp forms.
\newblock {\em Invent. Math.}, 70(2):219--288, 1982/83.

\bibitem[Dix69]{Dixmier}
Jacques Dixmier.
\newblock {\em Les {$C^{\ast} $}-alg\`ebres et leurs repr\'{e}sentations}.
\newblock Cahiers Scientifiques, Fasc. XXIX. Gauthier-Villars \'{E}diteur,
  Paris, 1969.
\newblock Deuxi\`eme \'{e}dition.

\bibitem[Don82]{MR664496}
Harold Donnelly.
\newblock On the cuspidal spectrum for finite volume symmetric spaces.
\newblock {\em J. Differential Geometry}, 17(2):239--253, 1982.

\bibitem[Duf82]{Duflo}
Michel Duflo.
\newblock Th\'{e}orie de {M}ackey pour les groupes de {L}ie alg\'{e}briques.
\newblock {\em Acta Math.}, 149(3-4):153--213, 1982.

\bibitem[FKMS19]{FKMS}
\'Etienne Fouvry, Emmanuel Kowalski, Philippe Michel, and Will Sawin.
\newblock Lectures on applied {$\ell$}-adic cohomology.
\newblock In {\em Analytic methods in arithmetic geometry}, volume 740 of {\em
  Contemp. Math.}, pages 113--195. Amer. Math. Soc., [Providence], RI, [2019]
  \copyright2019.

\bibitem[Gal70]{GallagherLargeSieve}
P.~X. Gallagher.
\newblock A large sieve density estimate near {$\sigma =1$}.
\newblock {\em Invent. Math.}, 11:329--339, 1970.

\bibitem[Gel75]{Gelbart}
Stephen~S. Gelbart.
\newblock {\em Automorphic forms on ad\`ele groups}.
\newblock Princeton University Press, Princeton, N.J.; University of Tokyo
  Press, Tokyo, 1975.
\newblock Annals of Mathematics Studies, No. 83.

\bibitem[GHL94]{HLAppendix}
Dorian {Goldfeld}, Jeffery Hoffstein, and Daniel Lieman.
\newblock Appendix: An effective zero-free region.
\newblock {\em Ann. of Math. (2)}, 140(2):177--181, 1994.

\bibitem[God18]{Godement_Notes_on_JL}
Roger Godement.
\newblock {\em Notes on {J}acquet-{L}anglands' theory}, volume~8 of {\em CTM.
  Classical Topics in Mathematics}.
\newblock Higher Education Press, Beijing, 2018.
\newblock With commentaries by Robert Langlands and Herve Jacquet.

\bibitem[GR07]{GR}
I.~S. Gradshteyn and I.~M. Ryzhik.
\newblock {\em Table of integrals, series, and products}.
\newblock Elsevier/Academic Press, Amsterdam, seventh edition, 2007.
\newblock Translated from the Russian, Translation edited and with a preface by
  Alan Jeffrey and Daniel Zwillinger, With one CD-ROM (Windows, Macintosh and
  UNIX).

\bibitem[Her08]{HerzigSmoothReps}
Florian Herzig.
\newblock Smooth representations of $p$-adic reductive groups.
\newblock Notes for the Instructional Conference on Representation Theory and
  Arithmetic, Northwestern, 2008.

\bibitem[HN18]{HN18}
Yueke Hu and Paul~D. Nelson.
\newblock New test vector for {W}aldspurger's period integral, relative trace
  formula, and hybrid subconvexity bounds.
\newblock {\em arXiv:1810.11564v2}, 2018.

\bibitem[HNS19]{HuNelsonSaha:17a}
Yueke Hu, Paul~D. Nelson, and Abhishek Saha.
\newblock Some analytic aspects of automorphic forms on {$\rm GL(2)$} of
  minimal type.
\newblock {\em Comment. Math. Helv.}, 94(4):767--801, 2019.

\bibitem[HS20]{HuSa:19}
Yueke Hu and Abhishek Saha.
\newblock Sup-norms of eigenfunctions in the level aspect for compact
  arithmetic surfaces, {II}: newforms and subconvexity.
\newblock {\em Compos. Math.}, 156(11):2368--2398, 2020.

\bibitem[Hu17]{hu_triple_2017}
Yueke Hu.
\newblock Triple product formula and the subconvexity bound of triple product
  {$L$}-function in level aspect.
\newblock {\em American Journal of Mathematics}, 139(1):215--259, 2017.

\bibitem[Hu18]{Hu:17a}
Yueke Hu.
\newblock Triple product formula and mass equidistribution on modular curves of
  level {$N$}.
\newblock {\em IMRN}, (9):2899--2943, 2018.

\bibitem[Hu24]{Hu}
Yueke Hu.
\newblock The {P}etersson/{K}uznetsov trace formula with prescribed local
  ramifications.
\newblock {\em Amer. J. Math.}, 146(5):1193--1252, 2024.

\bibitem[IK04]{IK}
Henryk Iwaniec and Emmanuel Kowalski.
\newblock {\em Analytic number theory}, volume~53 of {\em American Mathematical
  Society Colloquium Publications}.
\newblock American Mathematical Society, Providence, RI, 2004.

\bibitem[ILM17]{IchinoLapidMao}
Atsushi Ichino, Erez Lapid, and Zhengyu Mao.
\newblock On the formal degrees of square-integrable representations of odd
  special orthogonal and metaplectic groups.
\newblock {\em Duke Math. J.}, 166(7):1301--1348, 2017.

\bibitem[ILS00]{ILS}
Henryk Iwaniec, Wenzhi Luo, and Peter Sarnak.
\newblock Low lying zeros of families of {$L$}-functions.
\newblock {\em Inst. Hautes \'{E}tudes Sci. Publ. Math.}, (91):55--131 (2001),
  2000.

\bibitem[Iwa90]{IwaniecSmallEvals}
Henryk Iwaniec.
\newblock Small eigenvalues of {L}aplacian for {$\Gamma_0(N)$}.
\newblock {\em Acta Arith.}, 56(1):65--82, 1990.

\bibitem[Iwa97]{IwaniecTopics}
Henryk Iwaniec.
\newblock {\em Topics in classical automorphic forms}, volume~17 of {\em
  Graduate Studies in Mathematics}.
\newblock American Mathematical Society, Providence, RI, 1997.

\bibitem[Jac86]{JacquetRelativeTraceFormula}
Herv\'{e} Jacquet.
\newblock Sur un r\'{e}sultat de {W}aldspurger.
\newblock {\em Ann. Sci. \'{E}cole Norm. Sup. (4)}, 19(2):185--229, 1986.

\bibitem[JL70]{JacquetLanglands}
H.~Jacquet and R.~P. Langlands.
\newblock {\em Automorphic forms on {${\rm GL}(2)$}}.
\newblock Lecture Notes in Mathematics, Vol. 114. Springer-Verlag, Berlin-New
  York, 1970.

\bibitem[JM05]{JutilaMotohashi}
Matti Jutila and Yoichi Motohashi.
\newblock Uniform bound for {H}ecke {$L$}-functions.
\newblock {\em Acta Math.}, 195:61--115, 2005.

\bibitem[Joy90]{Joyner}
David Joyner.
\newblock On the {K}uznetsov-{B}ruggeman formula for a {H}ilbert modular
  surface having one cusp.
\newblock {\em Math. Z.}, 203(1):59--104, 1990.

\bibitem[Jut00]{JutilaSpectralLargeSieve}
Matti Jutila.
\newblock On spectral large sieve inequalities.
\newblock volume~28, pages 7--18. 2000.
\newblock Dedicated to W\l odzimierz Sta\'{s} on the occasion of his 75th
  birthday.

\bibitem[Kat88]{Katz}
Nicholas~M. Katz.
\newblock {\em Gauss sums, {K}loosterman sums, and monodromy groups}, volume
  116 of {\em Annals of Mathematics Studies}.
\newblock Princeton University Press, Princeton, NJ, 1988.

\bibitem[KL06a]{KLPetersson}
Andrew Knightly and Charles Li.
\newblock A relative trace formula proof of the {P}etersson trace formula.
\newblock {\em Acta Arith.}, 122(3):297--313, 2006.

\bibitem[KL06b]{KnightlyLiTracesOfHeckeOperators}
Andrew Knightly and Charles Li.
\newblock {\em Traces of {H}ecke operators}, volume 133 of {\em Mathematical
  Surveys and Monographs}.
\newblock American Mathematical Society, Providence, RI, 2006.

\bibitem[KL13]{knightly_kuznetsovs_2013}
A.~Knightly and C.~Li.
\newblock Kuznetsov's trace formula and the {H}ecke eigenvalues of {M}aass
  forms.
\newblock {\em Mem. Amer. Math. Soc.}, 224(1055):vi+132, 2013.

\bibitem[KM99]{KowalskiMichel}
E.~Kowalski and P.~Michel.
\newblock The analytic rank of {$J_0(q)$} and zeros of automorphic
  {$L$}-functions.
\newblock {\em Duke Math. J.}, 100(3):503--542, 1999.

\bibitem[Kna01]{Knapp}
Anthony~W. Knapp.
\newblock {\em Representation theory of semisimple groups}.
\newblock Princeton Landmarks in Mathematics. Princeton University Press,
  Princeton, NJ, 2001.
\newblock An overview based on examples, Reprint of the 1986 original.

\bibitem[Kni25]{Knightly_counting_supercuspidal_newforms}
Andrew Knightly.
\newblock Counting locally supercuspidal newforms.
\newblock {\em Essent. Number Theory}, 4(2):349--438, 2025.

\bibitem[KP17]{KaplanPetrow2}
Nathan Kaplan and Ian Petrow.
\newblock Elliptic curves over a finite field and the trace formula.
\newblock {\em Proc. Lond. Math. Soc. (3)}, 115(6):1317--1372, 2017.

\bibitem[KST20]{KimShinTemplier}
Ju-Lee Kim, Sug~Woo Shin, and Nicolas Templier.
\newblock Asymptotic behavior of supercuspidal representations and
  {S}ato-{T}ate equidistribution for families.
\newblock {\em Adv. Math.}, 362:106955, 57, 2020.

\bibitem[KY21]{KhanYoung}
Rizwanur Khan and Matthew~P. Young.
\newblock Moments and hybrid subconvexity for symmetric-square {$L$}-functions.
\newblock {\em J. Inst. Math. Jussieu}, pages 1--45, 2021.
\newblock published online.

\bibitem[LPW23]{LuoPiWu}
Zhilin Luo, Qinghua Pi, and Han Wu.
\newblock Bias of root numbers for {H}ilbert newforms of cubic level.
\newblock {\em J. Number Theory}, 243:62--116, 2023.

\bibitem[MV10]{MichelVenkateshGL2}
Philippe Michel and Akshay Venkatesh.
\newblock The subconvexity problem for {${\rm GL}_2$}.
\newblock {\em Publ. Math. Inst. Hautes \'{E}tudes Sci.}, (111):171--271, 2010.

\bibitem[Nel17]{Nelson_analytic_isolation}
Paul~D. Nelson.
\newblock Analytic isolation of newforms of given level.
\newblock {\em Arch. Math. (Basel)}, 108(6):555--568, 2017.

\bibitem[Nel18]{nelson_microlocal_2016}
Paul~D. Nelson.
\newblock Microlocal lifts and quantum unique ergodicity on
  {$GL_2(\Bbb{Q}_p)$}.
\newblock {\em Algebra Number Theory}, 12(9):2033--2064, 2018.

\bibitem[Neu99]{Neukirch}
J\"{u}rgen Neukirch.
\newblock {\em Algebraic number theory}, volume 322 of {\em Grundlehren der
  Mathematischen Wissenschaften [Fundamental Principles of Mathematical
  Sciences]}.
\newblock Springer-Verlag, Berlin, 1999.
\newblock Translated from the 1992 German original and with a note by Norbert
  Schappacher, With a foreword by G. Harder.

\bibitem[Ng12]{NgBasis}
Ming-ho Ng.
\newblock The basis for space of cusp forms and {P}etersson trace formula.
\newblock Masters thesis, University of Hong Kong, 2012.

\bibitem[Pal12]{Palm}
Marc~R. Palm.
\newblock {\em Explicit {${\rm GL}(2)$} trace formulas and uniform, mixed
  {W}eyl laws}.
\newblock PhD thesis, Georg-August-Universit\"{a}t G\"{o}ttingen, 2012.
\newblock arXiv:1212.4282.

\bibitem[Pet18]{PetrowTraces}
Ian Petrow.
\newblock Bounds for traces of {H}ecke operators and applications to modular
  and elliptic curves over a finite field.
\newblock {\em Algebra Number Theory}, 12(10):2471--2498, 2018.

\bibitem[PY19]{PetrowYounghybrid}
Ian Petrow and Matthew~P. Young.
\newblock A generalized cubic moment and the {P}etersson formula for newforms.
\newblock {\em Math. Ann.}, 373(1-2):287--353, 2019.

\bibitem[PY20]{PetrowYoungWeyl}
Ian Petrow and Matthew~P. Young.
\newblock The {W}eyl bound for {D}irichlet {$L$}-functions of cube-free
  conductor.
\newblock {\em Ann. of Math. (2)}, 192(2):437--486, 2020.

\bibitem[PY23]{PetrowYoungCoset}
Ian Petrow and Matthew~P. Young.
\newblock The fourth moment of {D}irichlet {$L$}-functions along a coset and
  the {W}eyl bound.
\newblock {\em Duke Math. J.}, 172(10):1879--1960, 2023.

\bibitem[Rio06]{Rio}
Anna Rio.
\newblock Dyadic exercises for octahedral extensions. {II}.
\newblock {\em J. Number Theory}, 118(2):172--188, 2006.

\bibitem[Rou11]{Rouymi}
D.~Rouymi.
\newblock Formules de trace et non-annulation de fonctions {$L$} automorphes au
  niveau {$\mathfrak{p}^\nu$}.
\newblock {\em Acta Arith.}, 147(1):1--32, 2011.

\bibitem[Sai93]{SaitoOnTunnelsFormula}
Hiroshi Saito.
\newblock On {T}unnell's formula for characters of {${\rm GL}(2)$}.
\newblock {\em Compositio Math.}, 85(1):99--108, 1993.

\bibitem[Sch02]{Schmidt:02a}
Ralf Schmidt.
\newblock Some remarks on local newforms for $\rm{GL}(2) $.
\newblock {\em Journal of the Ramanujan Mathematical Society}, 17(2):115--147,
  2002.

\bibitem[Ser79]{SerreLocalFields}
Jean-Pierre Serre.
\newblock {\em Local fields}, volume~67 of {\em Graduate Texts in Mathematics}.
\newblock Springer-Verlag, New York-Berlin, 1979.
\newblock Translated from the French by Marvin Jay Greenberg.

\bibitem[Sli84]{Sliman}
Mohamed~Hachmi Sliman.
\newblock Th\'{e}orie de {M}ackey pour les groupes ad\'{e}liques.
\newblock {\em Ast\'{e}risque}, (115):151, 1984.

\bibitem[SST16]{SarnakShinTemplier}
Peter Sarnak, Sug~Woo Shin, and Nicolas Templier.
\newblock Families of {$L$}-functions and their symmetry.
\newblock In {\em Families of automorphic forms and the trace formula}, Simons
  Symp., pages 531--578. Springer, [Cham], 2016.

\bibitem[Tun78]{TunnellLLC}
Jerrold~B. Tunnell.
\newblock On the local {L}anglands conjecture for {$GL(2)$}.
\newblock {\em Invent. Math.}, 46(2):179--200, 1978.

\bibitem[Ven79]{VenkovWeylSelberg}
A.~B. Venkov.
\newblock On the remainder term in the {W}eyl-{S}elberg asymptotic formula.
\newblock {\em Zap. Nauchn. Sem. Leningrad. Otdel. Mat. Inst. Steklov. (LOMI)},
  pages 5--24, 180, 1979.
\newblock Analytic number theory and the theory of functions, 2.



\bibitem[Wal88]{Wallach1}
Nolan~R. Wallach.
\newblock {\em Real reductive groups. {I}}, volume 132 of {\em Pure and Applied
  Mathematics}.
\newblock Academic Press, Inc., Boston, MA, 1988.

\bibitem[Wu14]{Wu2014}
Han Wu.
\newblock Burgess-like subconvex bounds for {$\text{GL}_2\times\text{GL}_1$}.
\newblock {\em Geom. Funct. Anal.}, 24(3):968--1036, 2014.

\bibitem[Yos77]{YoshidaOnExtraordinaryReps}
Hiroyuki Yoshida.
\newblock On extraordinary representations of {${\rm GL}_{2}$}.
\newblock In {\em Algebraic number theory ({K}yoto {I}nternat. {S}ympos.,
  {R}es. {I}nst. {M}ath. {S}ci., {U}niv. {K}yoto, {K}yoto, 1976)}, pages
  291--303. Japan Soc. Promotion Sci., Tokyo, 1977.

\bibitem[You17]{YoungHybrid}
Matthew~P. Young.
\newblock Weyl-type hybrid subconvexity bounds for twisted {$L$}-functions and
  {H}eegner points on shrinking sets.
\newblock {\em J. Eur. Math. Soc. (JEMS)}, 19(5):1545--1576, 2017.

\bibitem[You19]{YoungExplicit}
Matthew~P. Young.
\newblock Explicit calculations with {E}isenstein series.
\newblock {\em J. Number Theory}, 199:1--48, 2019.

\bibitem[Zag81]{ZagierEisSeriesAndSTF1}
Don Zagier.
\newblock Eisenstein series and the {S}elberg trace formula. {I}.
\newblock In {\em Automorphic forms, representation theory and arithmetic
  ({B}ombay, 1979)}, volume~10 of {\em Tata Inst. Fund. Res. Studies in Math.},
  pages 303--355. Tata Institute of Fundamental Research, Bombay, 1981.

\end{thebibliography}
\end{document}